\documentclass[11pt]{amsart}
\usepackage{amsmath}
\usepackage{amsthm}
\usepackage{amsbsy}
\usepackage{amsopn}
\usepackage{amsmath,amscd}
\usepackage{amssymb}
\usepackage{amsfonts}
\usepackage{multirow}
\usepackage{mathrsfs} 
\usepackage[official ]{eurosym}

\newcommand\cA{\mathcal{A}}

\newcommand\cI{\mathcal{I}}

\newcommand\cO{\mathcal{O}}
\newcommand\cP{\mathcal{P}}

\newcommand\fri{\mathfrak{i}}

\newcommand\frp{\mathfrak{p}}

\newcommand\sX{\mathscr{X}}
\newcommand\sY{\mathscr{Y}}
\newcommand\sA{\mathscr{A}}

\newcommand\sN{\mathscr{N}}

\newcommand\sP{\mathscr{P}}

\newcommand\sT{\mathscr{T}}
\newcommand\sR{\mathscr{R}}
\newcommand\sU{\mathscr{U}}

\newcommand{\cont}{\mathit{cont}}

 \usepackage[all]{xy}
                \usepackage{amsthm,amsfonts,amsmath,amscd,amssymb,epsfig,verbatim,
}

\usepackage{graphicx}
\usepackage{epstopdf}
\date{\today} 
 
\usepackage{graphicx}

\usepackage[dvipsnames,svgnames,x11names,hyperref]{xcolor}
\usepackage{url,graphicx,verbatim,amssymb,enumerate,stmaryrd}
\usepackage[pagebackref,colorlinks,citecolor=Bittersweet,linkcolor=Bittersweet,urlcolor=Bittersweet,filecolor=Bittersweet]{hyperref}
\usepackage{tikz} 
\usetikzlibrary{arrows,decorations.pathmorphing,backgrounds,fit,positioning,shapes.symbols,chains,calc}
\usepackage[all]{xy}
\xyoption{matrix}
\xyoption{arrow}
\usepackage[hmargin=3cm,vmargin=3cm]{geometry}
\usepackage{setspace,kantlipsum}
\tikzset{help lines/.style={step=#1cm,very thin, color=gray},
help lines/.default=.5} 

\usepackage{booktabs}

\theoremstyle{plain}
\newtheorem{theorem}{Theorem}[section]
\newtheorem{thm}[theorem]{Theorem}
\newtheorem{lemma}[theorem]{Lemma}
\newtheorem{proposition}[theorem]{Proposition}
\newtheorem{prop}[theorem]{Proposition}
\newtheorem{corollary}[theorem]{Corollary}
\newtheorem{cor}[theorem]{Corollary}

\theoremstyle{definition}
\newtheorem{definition}[theorem]{Definition}
\newtheorem{inductive step}[theorem]{Inductive step}

\newtheorem{inductive lemma}[theorem]{Inductive Lemma}
\theoremstyle{remark}
\newtheorem{example}[theorem]{Example}
\newtheorem{remark}[theorem]{Remark}
\newtheorem*{remark*}{Remark}

\newtheorem*{example*}{Example}

\newcommand{\wt}{\widetilde}
\newcommand{\wh}{\widehat}
\newcommand{\ol}{\overline}
\newcommand{\oL}{\overline{L}}

\newcommand{\oC}{\overline{C}}

\newcommand{\ul}{\underline}
\newcommand{\p}{\partial}

\newcommand{\om}{\omega}
\newcommand{\Om}{\Omega}
\newcommand{\eps}{\varepsilon}

\newcommand{\Skel}{\mathrm {Skel}}
\newcommand{\Cone}{\mathrm {Cone}}

\newcommand{\Z}{{\mathbb{Z}}}
\newcommand{\R}{{\mathbb{R}}}
\newcommand{\C}{{\mathbb{C}}}

\newcommand{\st}{{\rm st}}

\newcommand{\Int}{{\rm Int\,}} 

\renewcommand{\min}{{\rm min}}
\renewcommand{\max}{{\rm max}}

\newcommand{\Id}{\mathrm {Id}}

\newcommand{\Span}{\mathrm{Span}}
\newcommand{\So}{\mathrm{Soul}}

\newcommand{\Pos}{\mathrm{Pos}} 

\newcommand{\II}{\mathcal{I}}

\newcommand{\cL}{\mathcal{L}}
\newcommand{\cT}{\mathcal{T}}
\newcommand{\cS}{\mathcal{S}}
\def\Op{{\mathcal O}{\it p}\,}
\newcommand{\h}{{\mathfrak h}}
\renewcommand{\t}{{\mathfrak t}}

\newcommand{\mM}{\mathring{M}}
\newcommand{\mB}{\mathring{B}}

\newcommand{\pprec}{\mathop{\prec}\limits}
\newcommand{\psucc}{\mathop{\succ}\limits}

\newcommand{\Pol}{\mathrm{Pol}}
\numberwithin{figure}{section}

\newcommand{\symp}{\mathit{symp}}

 \newcommand{\Arb}{\mathrm{Arb}} 

\title{Positive arborealization of polarized Weinstein manifolds}

\author{Daniel  \'Alvarez-Gavela}
\address{Department of Mathematics \\ Massachusetts Institute of Technology \\ Cambridge, MA, 02139}
\email{dgavela@mit.edu}
\thanks{DA was partially supported by NSF grant DMS-1638352 and the Simons Foundation}

\author{Yakov Eliashberg }
\address{Department of Mathematics\\Stanford University \\ Stanford, CA 94305}
\email{eliash@stanford.edu}
\thanks{YE was partially supported by NSF grant DMS-1807270. }

\author{David Nadler}
\address{Department of Mathematics\\University of California, Berkeley\\Berkeley, CA  94720-3840}
\email{nadler@math.berkeley.edu}
\thanks{DN was partially supported by NSF grant DMS-1802373.}

\begin{document}
\begin{abstract}

Let $X$ be a Weinstein manifold. We show that the existence of a global field of Lagrangian planes in $TX$  is equivalent to the existence of a positive arboreal skeleton for the Weinstein homotopy class of $X$.

 \end{abstract}

\maketitle

 \onehalfspacing
 \tableofcontents

 \section{Introduction}\label{sec:intro}
  
\subsection{Arborealization program}
The symplectic topology of the cotangent bundle $T^*M$ of a smooth manifold is determined by the smooth topology of its Lagrangian zero-section $M$. A general {\em Weinstein symplectic manifold} $X$ (see precise definitions  below), which is the symplectic counterpart of a Stein complex manifold, can be viewed as a symplectic thickening of its {\em skeleton}, which is a singular isotropic subcomplex $\Skel\,X\subset X$. By analogy with the special case $X=T^*M$, $\Skel \, X = M$, one would like to view $X$ as the ``cotangent bundle" of  $\Skel\, X$, and  characterize the smooth symplectic topology of $X$ in terms of the differential topology of $\Skel\, X$. However, in general the singularities of $\Skel\, X$ are too complicated to be amenable to a differential topological treatment.

In the paper~\cite{N13}, the third author  introduced a class of Lagrangian singularities, called {\em arboreal}. For arboreal singularities, the smooth topology of the singularity determines the symplectic topology of its neighborhood.  
In particular,  it is possible to calculate the local symplectic invariants  of a neighborhood of an arboreal singularity in terms of the singularity.   This was shown in~\cite{N13} for microlocal sheaves; 
  for Fukaya categories, one can apply Lefschetz fibration calculations~\cite{Se08} to the plumbing  characterization of~\cite{Sh18}. Going further, if a Weinstein manifold $W$ has a skeleton with arboreal singularities, then its global symplectic invariants can be effectively computed knowing the smooth  topology of the skeleton, see~\cite{N16} for microlocal sheaves and~\cite{GPS17} for Fukaya categories.

  The paper  ~\cite{N13} initiated a program to determine which Weinstein manifolds admit arboreal skeleta, or as we say, {\em which Weinstein manifolds could be arborealized}. As evidence that this program might achieve the arborealization of a large class of Weinstein manifolds, it was shown in \cite{N15} that germs of  Whitney stratified Lagrangians can always be deformed  to arboreal Lagrangians in a non-characteristic fashion, i.e.~without changing their microlocal invariants. 
  
  The question of whether a Weinstein manifold can be arborealized via a homotopy of its Weinstein structure is more subtle. In two dimensions, the story is classical: generic ribbon graphs provide arboreal skeleta. In four dimensions, Starkston proved that arboreal skeleta always exist \cite{St18}.
 In this paper,  we establish that  any  Weinstein manifold admitting a polarization,
  i.e.~a Lagrangian plane field, or equivalently a  reduction of structure group of the tangent bundle from $Sp(2n)$ to $GL(n)$,
    can be arborealized,   see Theorem \ref{intro:main thm} below.
Moreover,  our proof yields skeleta with the more specific class of  {\em positive} arboreal singularities. Conversely,
Weinstein manifolds with  positive arboreal skeleta admit polarizations.
On the other hand  it turns out not all Weinstein manifolds can be arborealized, see the discussion in Section \ref{sec:arb-progr} below. Perhaps this is not surprising: there are homotopical obstructions to defining many symplectic invariants, and so any combinatorial route to realizing them must also encounter these obstructions.

\subsection{Main results}\label{sec:main res}

   Let $(W, \lambda)$ be a $2n$-dimensional Liouville  domain. We recall the definition: $W$ is a compact $2n$-manifold with boundary;
    $\lambda$ is a 1-form, called the Liouville form, such that $\omega = d\lambda$ is a symplectic form; and the corresponding  vector field $Z = \omega^{-1}(\lambda)$, called the  Liouville vector field, is 
    outward transverse to $\partial W$. 
  
%
   
  By definition, the skeleton of a Liouville domain $(W, \lambda)$ is the 
    attractor of the negative flow of the Liouville vector field:
   $$
   \Skel(W, \lambda) = \bigcap_{t>0} Z^{-t}(W).
   $$
While the $2n$-dimensional Lebesgue measure of $\Skel(W, \lambda)$  is  equal to $0$, in general $\Skel(W, \lambda)$  can be quite large if no additional conditions are imposed.

   A Weinstein domain is a  Liouville domain $(W, \lambda)$ which admits a Morse Lyapunov function $\phi:W\to \R$, i.e.~the critical points of $\phi$ are non-degenerate and   $Z$ is gradient-like for $\phi$. Sometimes it is convenient to relax the Morse condition to generalized Morse or Morse-Bott. We do not consider the Lyapunov function as part of the defining data of a Weinstein domain, we merely require its existence.
    
The skeleton of a  Weinstein domain $(W, \lambda)$   is known to be the union of  stable manifolds 
 $$
\xymatrix{
   \Skel(W, \lambda) = \bigcup_{\lambda_p=0} S_p, & S_p = \{x \in W \, |\, \lim_{t\to \infty} Z^t(x) = p\}.
}   $$

Each stable manifold is $\lambda$-isotropic, hence $\omega$-isotropic and so at most half-dimensional. A result of F. Laudenbach \cite{L92} states that if $Z$ is Morse-Smale, i.e.~its stable and unstable manifolds intersect transversely, and if moreover
 $Z$ is the Euclidean gradient near critical points with respect to the coordinates which give the Morse normal form, then $\Skel(W, \lambda)$ 
can be Whitney stratified. 

However, even if $Z$ satisfies the above conditions, in general the singularities of $\Skel(W, \lambda)$ are quite complicated, and their smooth topology does not determine the symplectic topology of their neighborhoods. For example, in the simplest case when $W$ is obtained by attaching a handle to a ball along a Legendrian sphere, the skeleton has a unique conical singular point and the  symplectic topology depends on the Legendrian isotopy class of the link, not its smooth isotopy class (which is always trivial for manifolds of dimension $>4$).
However, for the class of arboreal singularities described in Definition \ref{def:arb intr} below the situation is different.

First, we introduce some auxilliary notions. A closed 
subset of a symplectic  or  contact  manifold is called {\em isotropic} 
 if it is stratified by  isotropic submanifolds. It is called Lagrangian or Legendrian if it is isotropic and purely of the maximal possible dimension. The germ at the origin of a  locally simply-connected isotropic subset
   $L\subset T^*\R^n$ of the cotangent bundle with its standard Liouville structure $\lambda = pdq$ admits a unique   lift to an isotropic germ at the origin $\wh L\subset J^1\R^n=T^*\R^n\times\R$ of the 1-jet bundle. 
   Given an isotropic subset $\Lambda\subset S^*\R^n$ of the cosphere bundle, its Liouville cone $C(\Lambda) \subset T^*\R$, i.e.~the closure of its saturation by trajectories of the Liouville vector field $Z = p\frac{\p}{\p p}$,
   is an isotropic subset.
    
    \begin{definition}\label{def:arb intr}
  {\it Arboreal  Lagrangian (resp.~Legendrian) singularities}  form the smallest class $\Arb^{\symp}_n$  (resp.~$\Arb^{\cont}_n$) of germs of    closed isotropic subsets in  $2n$-dimensional symplectic (resp.~$(2n+1)$-dimensional   contact) manifolds such that the following properties are satisfied:
    \begin{enumerate}
    \item(Invariance) $\Arb^\symp_n$ is invariant with respect to symplectomorphisms and $\Arb^\cont_n$ is invariant with respect to contactomorphisms.
    \item (Base case) $\Arb^\symp_0$ contains $pt = \R^0 \subset T^*\R^0 = pt$.
    
    \item (Stabilizations) If  $L \subset (X, \omega)$ is  in $\Arb^\symp_n$,  then  the  product $L \times \R \subset (X \times T^*\R, \omega + dp\wedge dq)$
    is in $Arb^\symp_{n+1}$. 
    \item (Legendrian lifts) If $L\subset T^*\R^n$ is in $\Arb^\symp_n$, then its Legendrian lift $\wh L\subset J^1\R^n$ is in $\Arb^\cont_n$.
    \item (Liouville cones) Let $\Lambda_1,\dots, \Lambda_k \subset  S^*\R^n$ be a  finite disjoint  union of arboreal Legendrian germs from $Arb^\cont_{n-1}$ centered at points $z_1,\dots, z_k \in S^* \R^n$. Let $\pi:S^*\R^n\to \R^n$ be   the front projection. Suppose
    \begin{itemize}\item[-] $\pi(z_1)=\dots=\pi(z_k)$.
    \item[-] For any $i$, and smooth submanifold $Y \subset \Lambda_i$, the restriction  $\pi|_Y:Y\to\R^n$ is an embedding (or equivalently, an immersion, since we only consider germs).
     \item[-] For any distinct $i_1, \ldots, i_\ell$, and any  smooth submanifolds $Y_{i_1}\subset \Lambda_{i_1}, \dots,Y_{i_\ell}\subset \Lambda_{i_\ell}$,  the restriction 
    $\pi|_{Y_{i_1}\cup \dots\cup Y_{i_\ell}}: Y_{i_1}\cup \dots\cup Y_{i_\ell} \to \R^n$ is self-transverse.
    \end{itemize} 
    Then the union $\R^n \cup C(\Lambda_1)\cup\dots\cup C(\Lambda_k)$ of the Liouville cones  with the zero-section form an arboreal Lagrangian germ from    $\Arb^\symp_n$.
    \end{enumerate}

 With the above classes defined, we can also allow boundary by additionally taking the product  $L \times \R_{\geq 0} \subset (X \times T^*\R, \omega + dp\land  dq)$
    for any arboreal Lagrangian $L \subset (X, \omega)$, and similarly  for arboreal Legendrians.  

    \end{definition}
  
 \begin{figure}[h]
\includegraphics[scale=0.5]{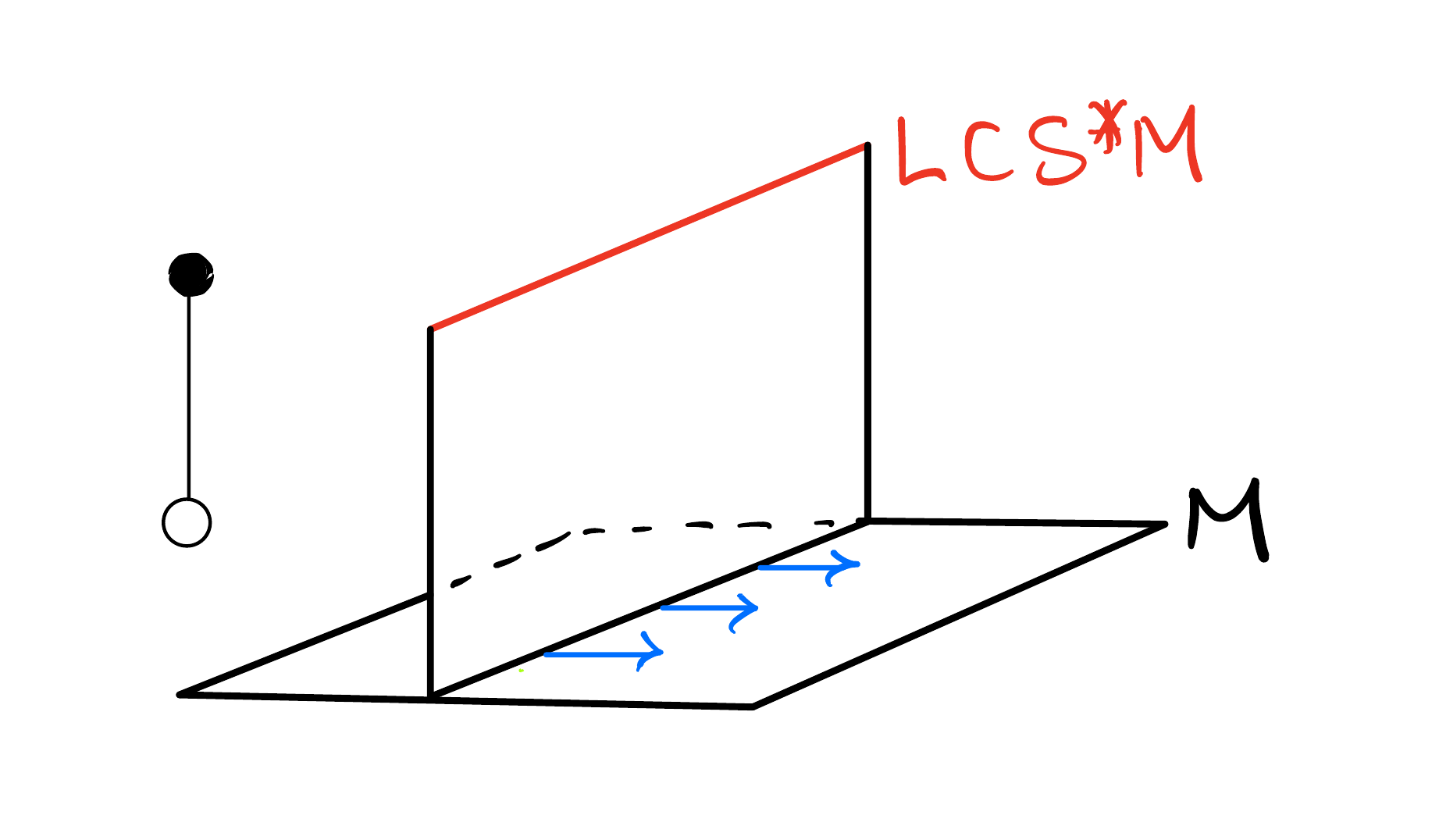}
\caption{The simplest non-smooth arboreal singularity is the zero section union the positive conormal of a smooth co-oriented hypersurface.}
\label{fig:A1stab}
\end{figure}
 \begin{figure}[h]
\includegraphics[scale=0.42]{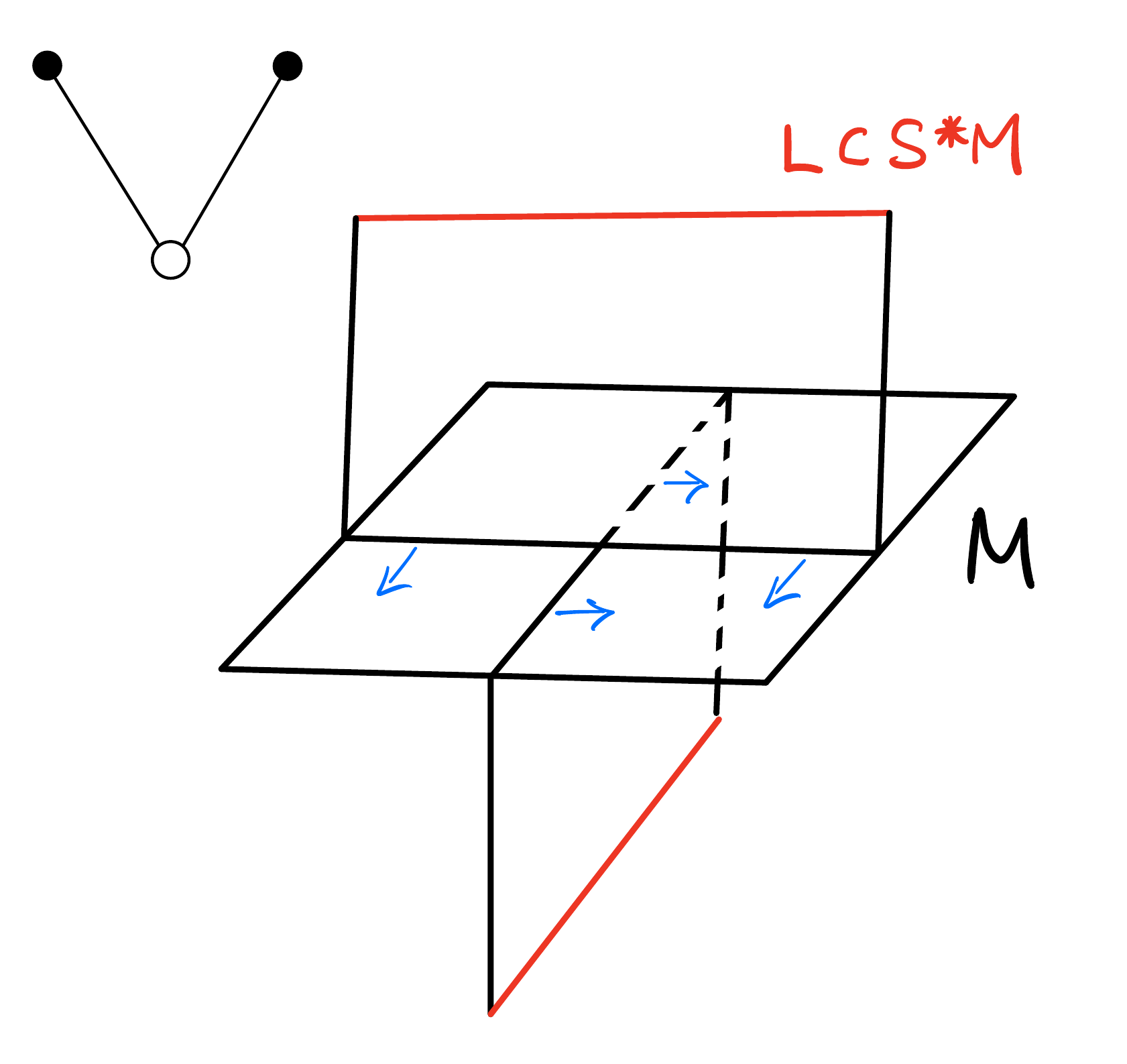}
\caption{This arboreal singularity consists of the zero section union the positive conormal of two smoothly intersecting co-oriented hypersurfaces.}
\label{fig:NotAroot}
\end{figure}
 \begin{figure}[h]
\includegraphics[scale=0.5]{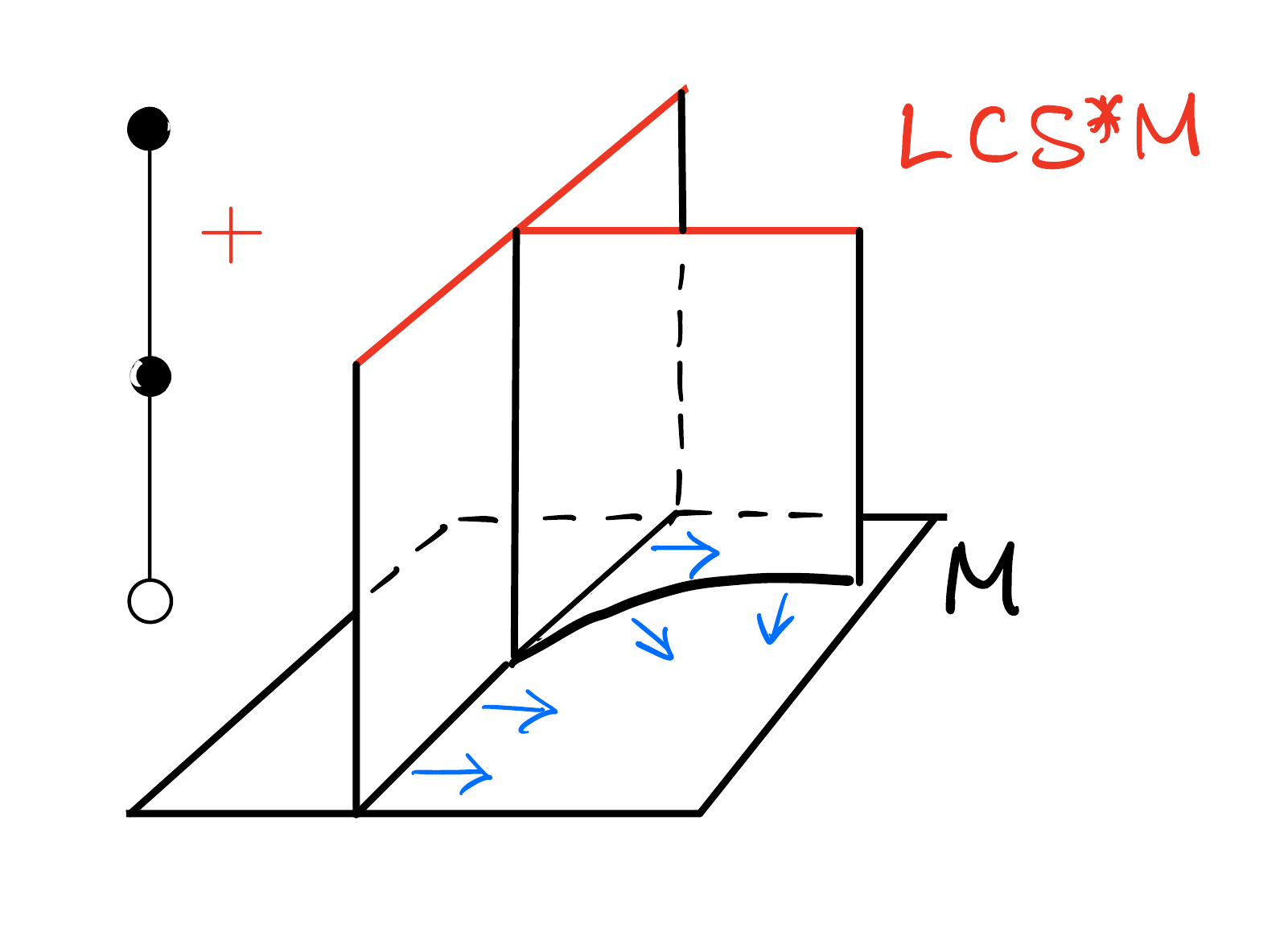}
\caption{This arboreal singularity consists of the zero section union the positive conormal of a singular co-oriented hypersurface (namely the front of the arboreal singularity of Figure \ref{fig:A1stab}).}
\label{fig:A2}
\end{figure}
    
In our paper \cite{AGEN20a} we proved that for fixed dimension $n$ the  class of arboreal singularities contains only finitely many local models up to ambient symplectomorphism or contactomorphism.  More precisely, to each member of the class  $\Arb^{\symp}_n$ one can assign a signed rooted tree $\cT= (T, \rho, \eps)$ 
    with at most $n+1$ vertices.
    Here $T$ is a  finite acyclic graph, $\rho$
      is a distinguished root vertex, and  $\eps$ 
           is a sign function on the edges of $T$ not adjacent to $\rho$. This discrete data completely determines the germ:
         
         \begin{theorem}[ \cite{AGEN20a}]
         If two arboreal singularities $L \subset (X, \omega)$, $L' \subset (X', \omega')$
         have the same dimension and  signed rooted tree $\cT$, then there is (the germ of) a symplectomorphism 
         $(X, \omega) \simeq (X', \omega')$ identifying $L$ and $L'$.

         \end{theorem}

           Within all arboreal singularities, there is the distinguished class with positive sign $\eps \equiv +1$. 
           
       \begin{figure}[h]
\includegraphics[scale=0.55]{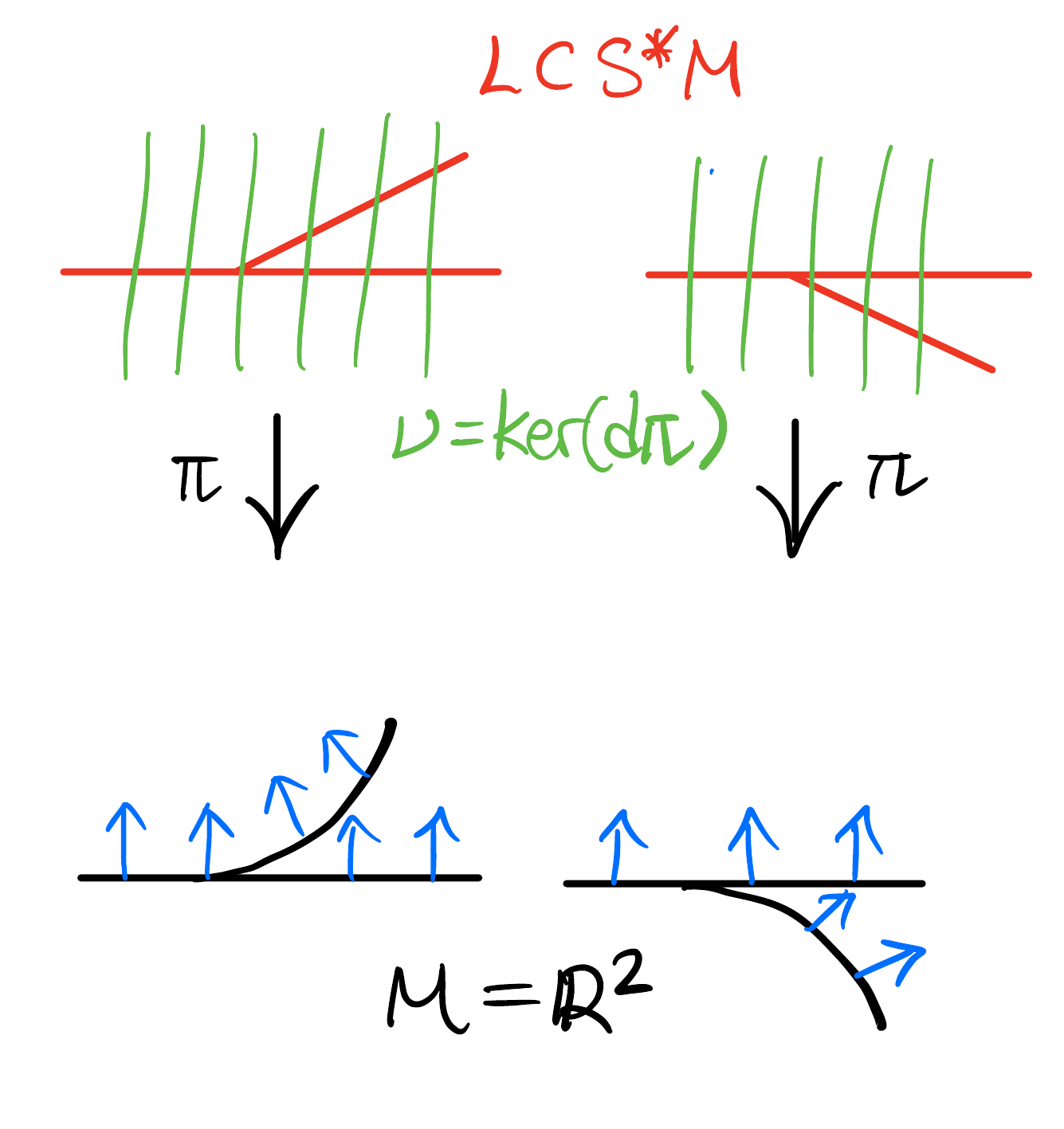}
\caption{One can obtain a new arboreal singularity from the arboreal singularity of Figure \ref{fig:A1stab} in two different ways depending on how it is embedded in $S^*M$ relative to the vertical distribution $\nu= \ker(d \pi)$. One yields a positive arboreal singularity and the other does not.}
\label{fig:POSneg}
\end{figure}

    \begin{definition}

    An {\it arboreal Lagrangian} $L$ (with boundary) in a symplectic manifold  $(X, \omega)$ is a piecewise smooth Lagrangian with arboreal singularities (with boundary), i.e.~locally modeled on the above class (with boundary). When $L$ is an arboreal Lagrangian whose boundary $\p L$ is a smooth manifold, we will say that $L$ has {\it smooth boundary}.
    
        A {\it positive arboreal Lagrangian} $L$ (with boundary) in a symplectic manifold  $(X, \omega)$ is a Lagrangian with positive arboreal singularities (with boundary), i.e.~locally modeled on the  distinguished class  with positive sign $\eps \equiv +1$. 
\end{definition}

 Given any positive arboreal Lagrangian $L\subset (X, \omega)$, possibly with boundary, its neighborhood $U \subset X$ admits a canonical, up to homotopy,  Lagrangian plane field $\xi \subset TU$. This homotopy class admits a representative $\xi$ which is transverse to $L$, i.e.~transverse to the closure of each smooth piece. 
 
    \begin{definition}
   A {\it polarization} of a  symplectic  manifold $(X, \omega)$ is the choice of a Lagrangian plane field $\xi \subset TX$. 
    \end{definition}

If $(W, \lambda)$ is a Weinstein domain with positive arboreal skeleton, then $(W, \lambda)$, which retracts onto an arbitrarily small neighborhood of its skeleton, admits a polarization. The main result of this paper is that the converse also holds:    

          \begin{figure}[h]
\includegraphics[scale=0.5]{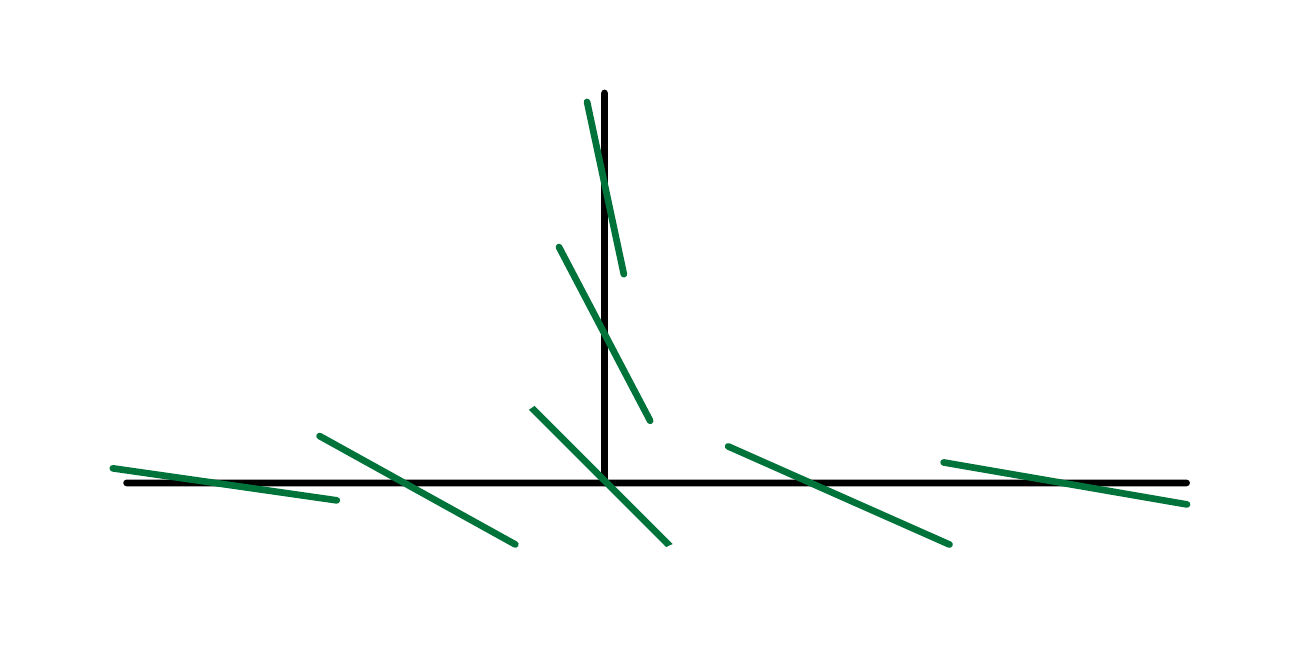}
\caption{The canonical polarization in the neighborhood of an $A_2$ singularity, where the positivity condition is vacuous. One could also turn the tangent field to the zero section counter-clockwise instead of clockwise. The choice is determined by what we call an orientation structure, which in this case is equivalent to a co-orientation of the origin inside the zero-section.}
\label{fig:POSdusta2}
\end{figure}

          \begin{figure}[h]
\includegraphics[scale=0.7]{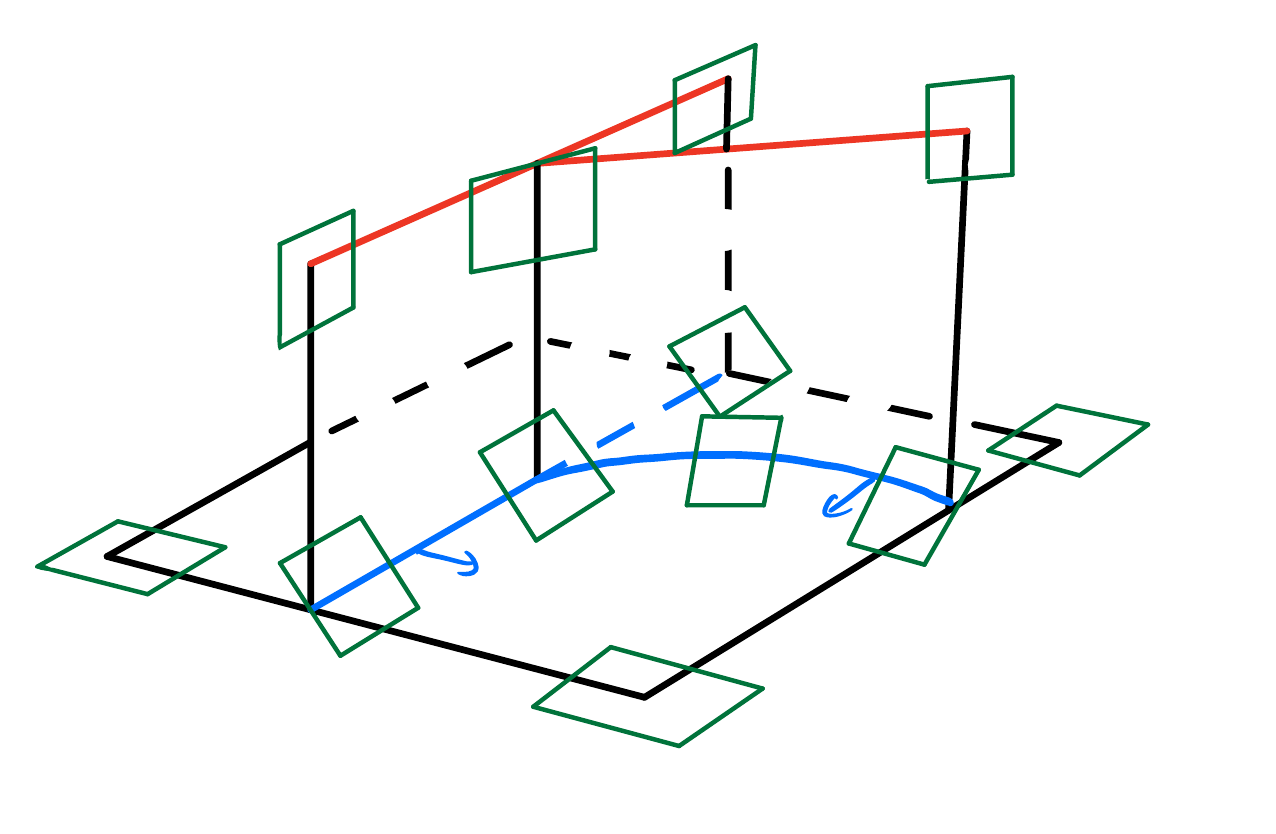}
\caption{The canonical polarization in the neighborhood of a positive $A_3$ singularity, where the orientation structure is given by the co-orientation of the $A_3$ front.}
\label{fig:Polara3}
\end{figure}

  \begin{theorem}\label{intro:main thm}
  A Weinstein domain $(W, \lambda)$ is homotopic to a Weinstein domain whose skeleton is positive arboreal with smooth boundary if and only if $(W, \lambda)$ admits a polarization.
  \end{theorem}  
  
  \begin{remark} The conclusion of Theorem \ref{intro:main thm} can be refined in several ways, some of which we sketch below. For precise formulations see Theorem \ref{thm:existence-arb-pos} and its corollaries, which use the language of Wc-manifolds, see Definition \ref{def:Wc-man}.
  \begin{enumerate}
    \item In fact, for $\xi$ a polarization of $(W, \lambda)$ we can arrange it so that $\xi$ is the canonical polarization in the neighborhood of the positive arboreal skeleton, so in particular $\xi$ is transverse to the skeleton.
  \item  If so desired, the smooth boundary of the positive arboreal skeleton can be pushed out to the boundary of $W$, in which case it serves as a skeleton for a {\it Weinstein pair}, see Section \ref{sec:W-hyp}. The Weinstein hypersurface in the pair is thus the ribbon of a smooth Legendrian, and can also be thought of as a stop.
  \item Theorem \ref{intro:main thm} also holds with input a Weinstein pair $(W,A)$ instead of a Weinstein domain. In this case the conclusion is that the Weinstein hypersurface $A$ can be enlarged by the ribbon of a smooth Legendrian so that the resulting Weinstein pair has a positive arboreal skeleton. 
    \item Moreover, the version of Theorem \ref{intro:main thm} for Weinstein pairs also holds in relative form if the Weinstein hypersurface in the Weinstein pair already has a positive arboreal skeleton and the polarization restricts on it to the canonical polarization. In this case the Weinstein structure can be kept fixed near the Weinstein hypersurface.
    \item Theorem \ref{intro:main thm} and its refinements hold more generally for arbitrary Weinstein manifolds, not necessarily of finite type (i.e. not completions of Weinstein domains), and the result is proved in the same way since the inductive argument stabilizes. We restrict the discussion to the finite type case for simplicity of exposition.
  \end{enumerate}

  \end{remark}


 Existence of a polarization for a symplectic manifold $(X^{2n},\omega)$ is equivalent to asking that the classifying map $X \to BU_n$ of the tangent bundle lifts to $BO_n$, where we endow $TX$ with the homotopically unique almost complex structure compatible with $\omega$.   

\begin{definition} A {\it stable polarization}  of a  symplectic  manifold $(X, \omega)$ is the choice of a  Lagrangian plane field $\xi \subset TX \oplus \C^d$ for some $d \geq 1$. \end{definition}

The existence of a stable polarization is equivalent to asking that the classifying map $W \to BU$ of the stable tangent bundle lifts to $BO$. Recall that a Weinstein  domain $(W, \lambda)$ has the homotopy type of a half-dimensional CW complex. Recall also that $BO_k \to B O_{k+1}$ is $k$-connected and $BU_k \to BU_{k+1}$ is $(2k+1)$-connected. Therefore, it follows that for Weinstein domains the existence of a polarization is equivalent to the existence of a stable polarization.   
  From the above, one can check that many   Weinstein  domains\ admit polarizations, notably smooth complete intersections in complex affine space. Indeed, if $X \subset \C^N$ is a smooth complete intersection, then its normal bundle is trivial, hence its tangent bundle is stably trivial and in particular $X$ admits a stable polarization. This proves:

\begin{corollary}
Let $X \subset \C^N$ be a complete intersection. Then $X$ is Weinstein homotopic to a Weinstein manifold whose skeleton is positive arboreal with smooth boundary.
\end{corollary}

   An arboreal skeleton with boundary comes with a natural additional orientation structure, induced from the ambient symplectic structure. We formalize this  in the notion of an arboreal space and prove the following counterpart to Theorem \ref{intro:main thm}.
   
  \begin{theorem}\label{intro:weinst thick}
  Any compact  arboreal space with boundary arises as the skeleton of a  Weinstein domain $(W, \lambda)$,
  unique up to  Weinstein homotopy.
  \end{theorem}   
    \begin{figure}[h]
\includegraphics[scale=0.43]{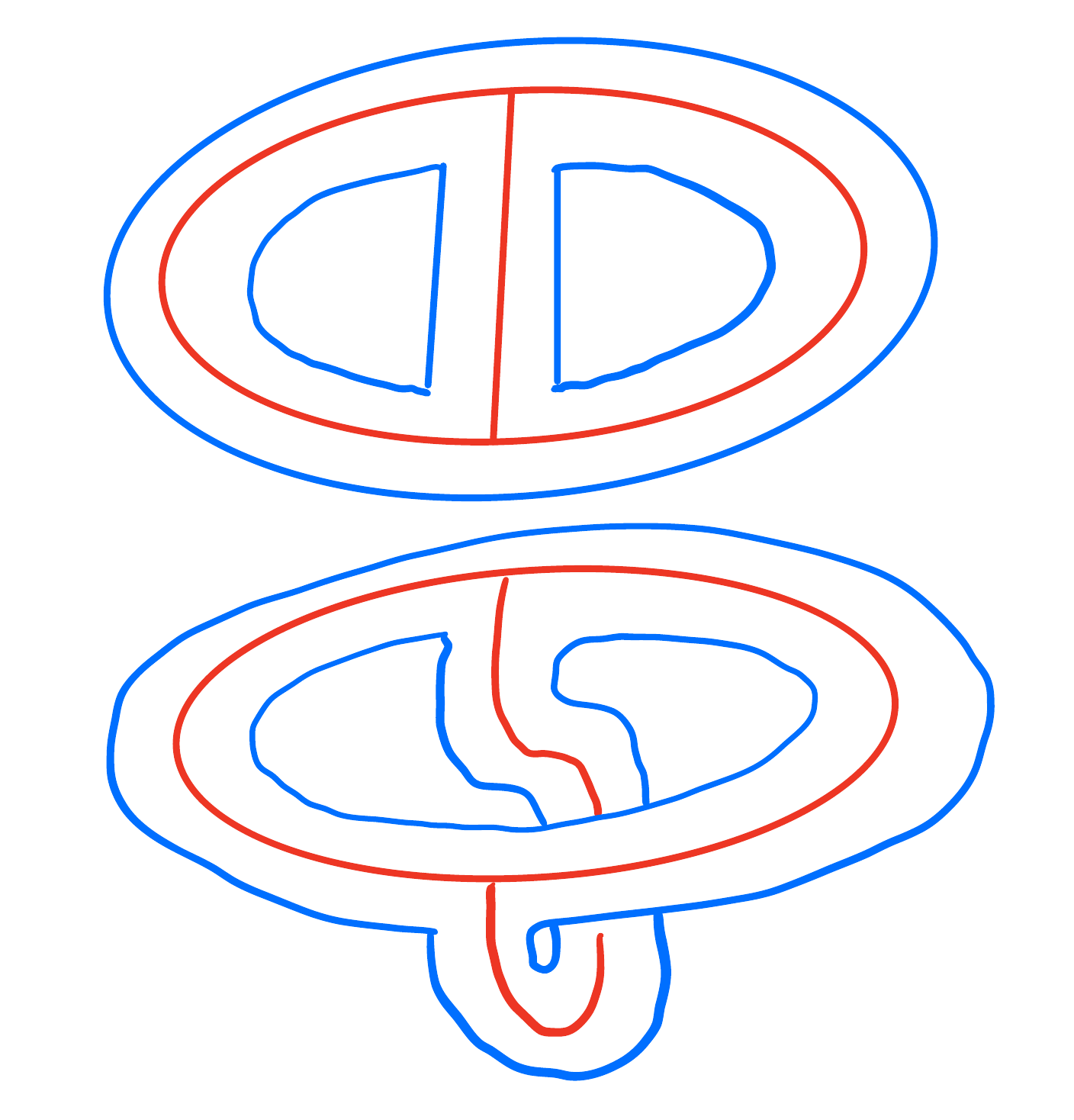}
\caption{The simplest example of two non-equivalent arboreal spaces with diffeomorphic underlying arboreal Lagrangians. These are skeleta for the pair of pants and the once-punctured torus repsectively. The arboreal spaces are distinguished by the orientation structure, which is additional data, and in this dimension reduces to the usual cyclic structure on ribbon graphs. We thank A. Oancea for pointing out this example. }
\label{fig:orientationstr}
\end{figure}

 Hence, up to Weinstein homotopy, those Weinstein domains which admit polarizations are precisely the Weinstein thickenings of positive arboreal spaces. Moreover, the relative form of Theorem \ref{intro:main thm} implies a uniqueness result for positive arboreal skeleta. To state it we need the notion of a Weinstein concordance, which is best defined in the language of Wc-manifolds, see Definition \ref{def:Wc-man}. Briefly, if $\lambda_1$ and $\lambda_2$ are two Weinstein structures on the same symplectic domain $(W,\omega)$ with Liouville fields $Z_1,Z_2$, respectively, then a Weinstein concordance between them is a Weinstein-type primitive $\lambda$ for $\omega + du \land dt$ on $W \times T^*[0,1]$ such that the Liouville field $Z$ of $\lambda$ is tangent to the boundary $W \times T_0^*[0,1] \cup W \times T_1^*[0,1]$ and restricts to $Z_1+ u\p/\p u$ and $Z_2 + u \p / \p u$ over $W \times T^*_0[0,1]$ and $W \times T_1^*[0,1]$ respectively. 
 
 \begin{definition} A {\it positive arboreal concordance} is a Weinstein concordance $(W \times T^*[0,1],\lambda)$ whose skeleton is a positive arboreal Lagrangian with boundary. We assume that the boundary of the skeleton is smooth in the interior of $W \times T^*[0,1]$ but we allow simple corners with one face contained in $W \times T^*_0[0,1]$ or $W \times T^*_1[0,1]$ and the other face transverse to it. \end{definition}
 
 \begin{remark}
 We must allow simple corners so that for example if $(W,\lambda)$ is a Weinstein manifold whose skeleton is a positive arboreal with nonempty smooth boundary, then the trivial product $(W \times T^*[0,1], \lambda + udt)$ is a positive arboreal concordance. \end{remark}
%
%

 The uniqueness statement then reads as follows:
 
 \begin{theorem}\label{intro:conc} Suppose that $\lambda_1$ and $\lambda_2$ are two homotopic Weinstein structures on $W$ whose skeleta are positive arboreal Lagrangians with smooth boundary, transverse to polarizations $\xi_1$ and $\xi_2$ respectively. Then $\xi_1$ and $\xi_2$ are stably homotopic if and only if there is a positive arboreal concordance between $\lambda_1$ and $ \lambda_2$.
 \end{theorem}

We remark that the classification of polarizations on a Weinstein manifold up to stable homotopy is in general weaker than the classification up to homotopy.

%
%
%
%
%

\subsection{ Further development of the arborealization program}\label{sec:arb-progr}
 
Work in progress of the authors extends the results of this paper in two directions. First, in our forthcoming work  \cite{AGEN21}  we consider  a $1$-parametric version of the  problem, by introducing a suitable notion of ``arboreal homotopy of positive arboreal spaces", i.e.~the minimal sequence of combinatoral Reidemeister type moves  necessary to connect  two homotopic   Weinstein structures with positive arboreal skeleta that have homotopic canonical polarizations. Note that substantial progress in this direction appeared in Zorn's thesis \cite{Z18}, where local mutations of arboreal skeleta are classified.  

The second direction is an extension of  the arborealization program  to  a more general class of Weinstein manifolds. While most Weinstein manifolds considered in applications such as mirror symmetry are complete intersections and hence polarized, not all Weinstein manifolds admit polarizations. For example a cohomological obstruction to existence of a polarization is that all odd Chern classes are 2-torsion. However, this cohomological condition does not obstruct existence of a not necessarily positive arboreal skeleton.

But even allowing for non-positive arboreal skeleta, there are still obstructions to their existence.   Indeed, while existence of a polarization is not a necessary condition for   existence of a non-positive arboreal skeleton, there is a closely related necessary condition: existence of an {\em $(n,n-1)$-polarization}. This is, roughly speaking, a polarization which is allowed to degenerate to an $(n-1)$-dimensional isotropic subspace.   More precisely, it is defined as follows. 

Given a symplectic vector space $(E,\om)$ let $L_n(E)$ denote its Lagrangian Grassmanian of non-oriented Lagrangian planes,   and $I_{n-1}(E)$ denote the Grassmanian of non-oriented isotropic (n-1)-dimensional subspaces. We further denote  by $F_{n,n-1}(E)$ the flag manifold $\{(\lambda, \mu); \lambda\in L_n(E), \mu\in I_{n-1}(E), \mu\subset\lambda\}$. Let $\pi_L$ and $\pi_I$ be  the  tautological projections $\pi_L:F_{n,n-1}(E)\to L_n(E)$ and $\pi_I:F_{n,n-1}(E)\to I_{n-1}(E)$.   Consider the double cone $C_n(E):=(F_{n,n-1}\times[0,1]\cup L_n(E)\cup I_{n-1}(E))/\{f\times0\sim
\pi_L(f),f\times1\sim\pi_I(f)\}.$ Given any symplectic vector {\em bundle} $E$,  we will use the notation $C_n(E)$ for the associated cone bundle.
An $(n,n-1)$-polarization on a symplectic manifold $V$ is by definition a {\em section} of the cone bundle $C_n(T(V))$. 

One can show   that for $n\leq 5 $ any $2n$-dimensional Weinstein manifold admits a  $(n,n-1)$-polarization, while for $n=6$ there are obstructions (e.g. one must have $c_3(V)=c_1(V)c_2(V)$).\footnote{We thank John Pardon for suggesting to  us   the notion of a $(n,n-1)$-polarization, and to S{\o}ren Galatius  for computations of obstructions.}
It is possible that this condition is also sufficient for existence of an arboreal skeleton. Note in particular that a Weinstein manifold of dimension four always admits an $(n,n-1)$-polarization, so this discussion is consistent with Starkston's result \cite{St18} that four-dimensional Weinstein manifolds always admit arboreal skeleta.

For completely arbitrary Weinstein manifolds it is not clear whether one should expect there to be  a simple class of singularities which the skeleton can always be arranged to have, but if it exists such a class would have to be more general than the arboreal class.  In any case, even the most optimistic hopes of constructing a reasonable skeleton for a Weinstein manifold will in all likelihood require the existence of ``Maslov data", i.e. the homotopical trivialization required to define Fukaya categories or microlocal sheaves. 

\subsection{Flexibility of caustics and the ridgification theorem}

We discuss a toy example to illustrate our strategy of proof. Consider a Weinstein domain $(W^{2n},\lambda)$ which is obtained from the standard Darboux ball $B \subset (\R^{2n}, pdq-qdp)$ by attaching an index $n$ handle, i.e. the neighborhood of a Lagrangian disk, along a Legendrian sphere $\Lambda \subset \p B = (S^{2n-1} ,\xi_{\text{std}})$. The skeleton of $W$ is homeomorphic to the $n$-sphere: it consists of the Lagrangian disk union the Liouville cone on $\Lambda$ with respect to the radial Liouville field on $\R^{2n}$, union the origin.  

Replace the radial Liouville structure $pdq-qdp$ on $\R^{2n}$ with the Morse-Bott canonical Liouville structure $pdq$ on the cotangent bundle $T^*\R^n$, or rather on $T^*D$ for $D \subset B$ a Lagrangian $n$-disk bounding a Legendrian unknot $\p D \subset S^{2n-1}$ which we assume disjoint from $\Lambda$. Then the radial Liouville cone on $\Lambda \subset S^{2n-1}$ in $B$ gets replaced with the fibrewise Liouville cone on $\Lambda \subset S^*D$ in $T^*D$, where $S^*D$ denotes the cosphere bundle, see Figure \ref{fig:Blowingup}. 

Observe that as a result, the singularity gets spread out over the disk, i.e. instead of a cone on a point it is now a fibrewise cone over the disk. The skeleton of the new Weinstein structure on $W$ consists of the Lagrangian handle, union the Liouville cone on $\Lambda$ with respect to the canonical Liouville structure on $T^*D$, union the zero section $D$. The singularities of this new Lagrangian skeleton are thus related to the singularities of the map $\pi|_\Lambda: \Lambda \to D$, where $\pi: S^*D \to D$ is the front projection. These singularities correspond to the tangencies of $\Lambda$ with the distribution $\nu= \ker(d \pi) \subset T (S^*D)$ and are also known as caustics in the literature. For example, if $\pi|_{\Lambda}: \Lambda \to D$ has no caustics, then it is an immersion and so by axiom (v) of Definition \ref{def:arb intr}, after a generic perturbation (to ensure self-transversality), the skeleton of $W$ is arboreal.
\begin{figure}[h]
\includegraphics[scale=0.5]{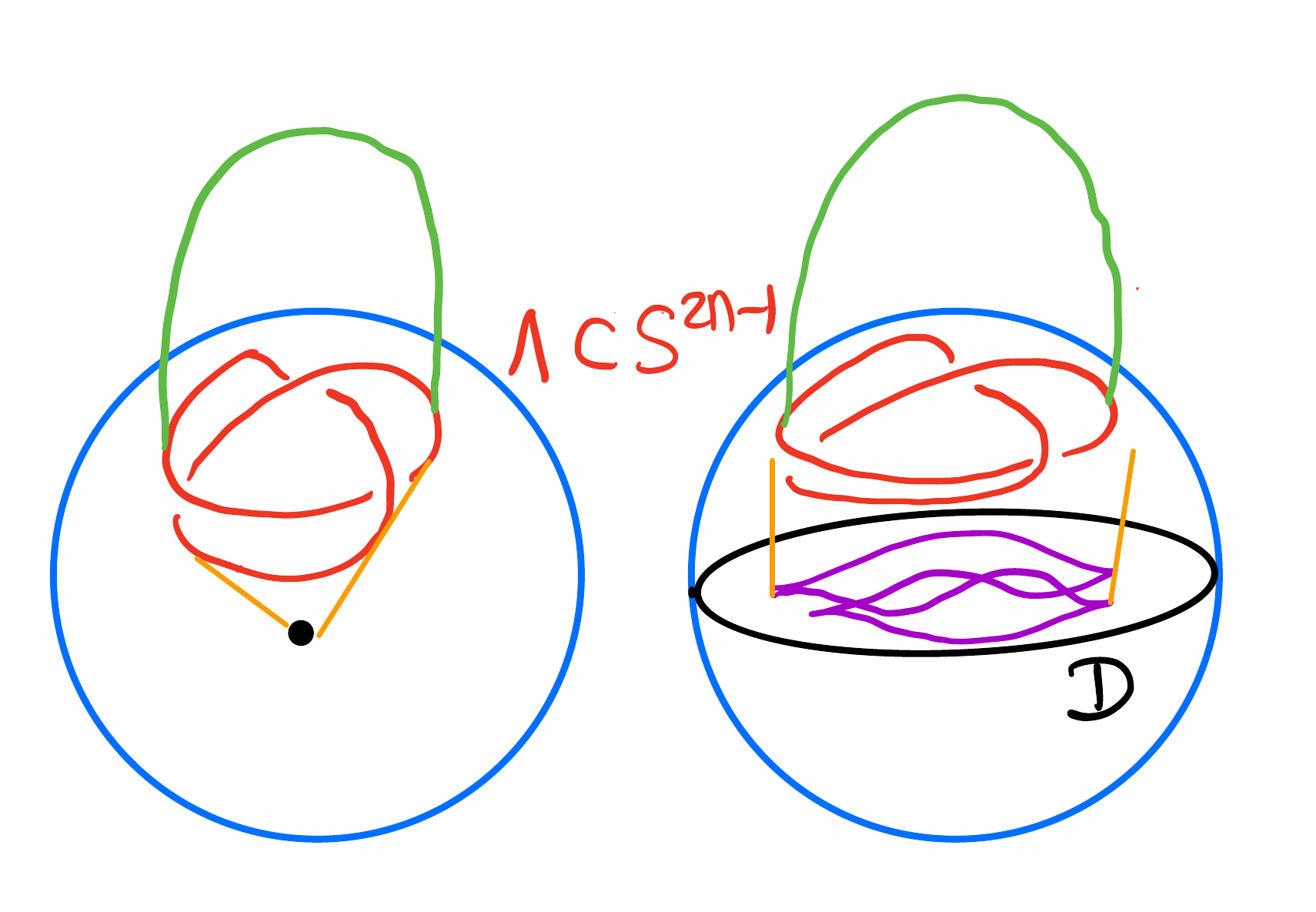}
\caption{The blow-up procedure for a Weinstein manifold $W^{2n}$ with two critical points: one of index 0 and one of index $n$.}
\label{fig:Blowingup}
\end{figure}

The idea of blowing up conical Lagrangian singularities was already present at the inception of the arborealization program, but for our purposes it is necessary to perform the blowing up procedure globally, at the level of Weinstein structures. This strategy was successfully implemented by Starkston in~\cite{St18} for the $4$-dimensional case. Starkston moreover showed how to explicitly arborealize the semi-cubical cusp singularities which generically appear in the front projections of 1-dimensional Legendrians, thus leading to her result that 4-dimensional Weinstein manifolds always admit arboreal skeleta. In higher dimensions this strategy encounters the difficulty that there is no classification theorem for the generic singularities of front projections, nor can anything remotely close to a satisfactory classification be hoped for.

Nevertheless, there holds an h-principle for the simplification of caustics. The h-principle states that if there is no homotopy theoretic obstruction to the simplification of caustics, then the simplification can be achieved by means of a Legendrian isotopy (or for Lagrangians, by means of a Hamiltonian isotopy). Results in this direction first appeared in work of Entov \cite{En99}, and the full h-principle was established in work of the first author \cite{AG18a}, \cite{AG18b}. Returning to our toy example, if we know that $W$ admits a polarization as in the hypothesis of our main Theorem \ref{intro:main thm} and we further assume that the dimension of $W$ is congruent to $2$ modulo $4$, then it follows from the h-principle that there exists a Legendrian isotopy which deforms $\Lambda$ to a Legendrian whose front projection $\pi|_\Lambda : \Lambda \to D$ only has cusp singularities. In this case we can proceed as in~\cite{St18} to conclude that the skeleton of $W$ can be arborealized by a Weinstein homotopy. When the dimension of $W$ is congruent to $0$ modulo $4$ there is an additional homotopical subtlety to consider, even in this toy example, but we do not pursue this further since one also encounters additional and more serious difficulties when attempting to implement the above strategy on a Weinstein manifold with a more complicated handlebody presentation. For example, one would need to know that the relevant homotopical obstruction to applying the h-principle vanishes for every handle, yet it is unclear how to arrange for this even under the global assumption of existence of a polarization.

So instead of attempting a direct application of the h-principle, we will exploit the fact that for the purposes of arborealization we can allow deformations which are more general than a Hamiltonian isotopy, namely we can deform the skeleton of a Weinstein manifold by any homotopy of the Weinstein structure. This allows the skeleton to develop genuine singularities (i.e. non-smooth points) along which its field of tangent planes jumps discontinuously. By introducing these jumps we destroy the homotopical obstruction to simplifying the singularities with the distribution $\nu = \ker( d \pi)$, and in fact we can always completely eliminate these tangencies. 

This strategy was made precise in \cite{AGEN19}, where we introduced a class of singular Lagrangians, called ridgy, and proved an h-principle {\it without pre-conditions}, which allows one to deform any Lagrangian so that it becomes transverse to any given Lagrangian distribution, at the expense of developing ridgy singularities. This result, which we refer to as the {\it ridgification theorem}, uses the h-principle for the simplification of caustics as one of the key ingredients in its proof, but is better suited for the task at hand since there are no hypotheses needed for its application. 
    \begin{figure}[h]
\includegraphics[scale=0.4]{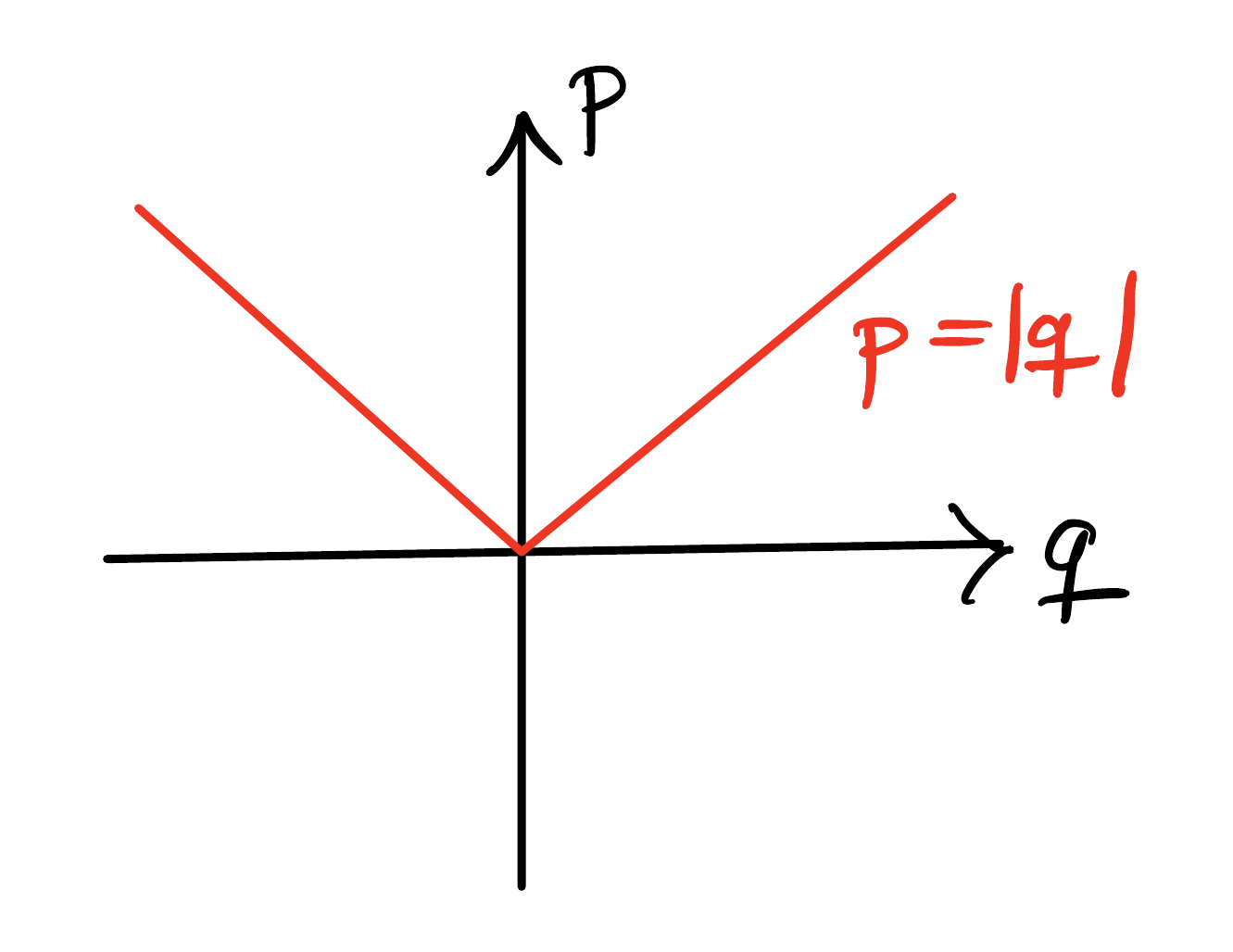}
\caption{The ridge singularity in one dimension.}
\label{fig:ridgy1D}
\end{figure}
         \begin{figure}[h]
\includegraphics[scale=0.5]{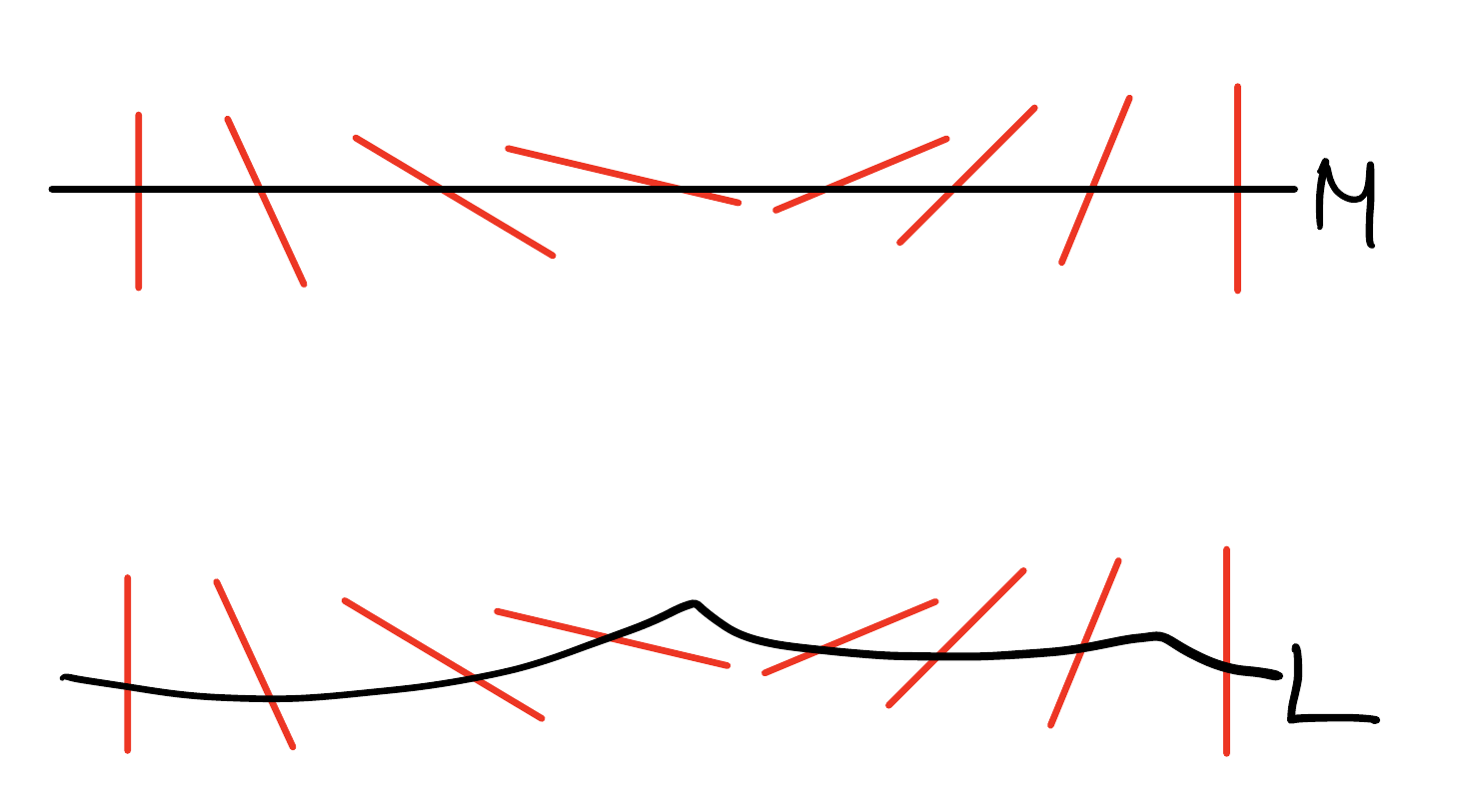}
\caption{The ridgification theorem in action: the zero section $M \subset T^*M$ gets deformed to a ridgy Lagrangian $L \subset T^*M$ which is transverse to the red distribution. Of course in this simple example a single ridge point would suffice, but the proof always produces a ridge locus which is the union of homologically trivial hypersurfaces, in this case two ridge points.  }
\label{fig:ridgy}
\end{figure}

Ridgy singularities are extremely simple: they consist of the corner $\{p=|q| \} \subset \R^2$ together with its products and stabilizations. Moreover, ridgy singularities can be explicitly arborealized while maintaining transversality to any given Lagrangian distribution, see Figure \ref{fig:arb rid}. So our revised strategy to arborealize a Weinstein manifold $W$ is the following: (1) blow up the singularities as before to replace Liouville cones on a point with Liouville cones on Lagrangian disks, (2) apply the ridgification theorem to get rid of all singularities of the resulting front projections, and finally (3) arborealize the ridgy Lagrangians produced by the ridgification theorem, while maintaining transversality to the relevant distributions, to obtain an arboreal skeleton for $W$.

\begin{figure}[h]
\includegraphics[scale=0.5]{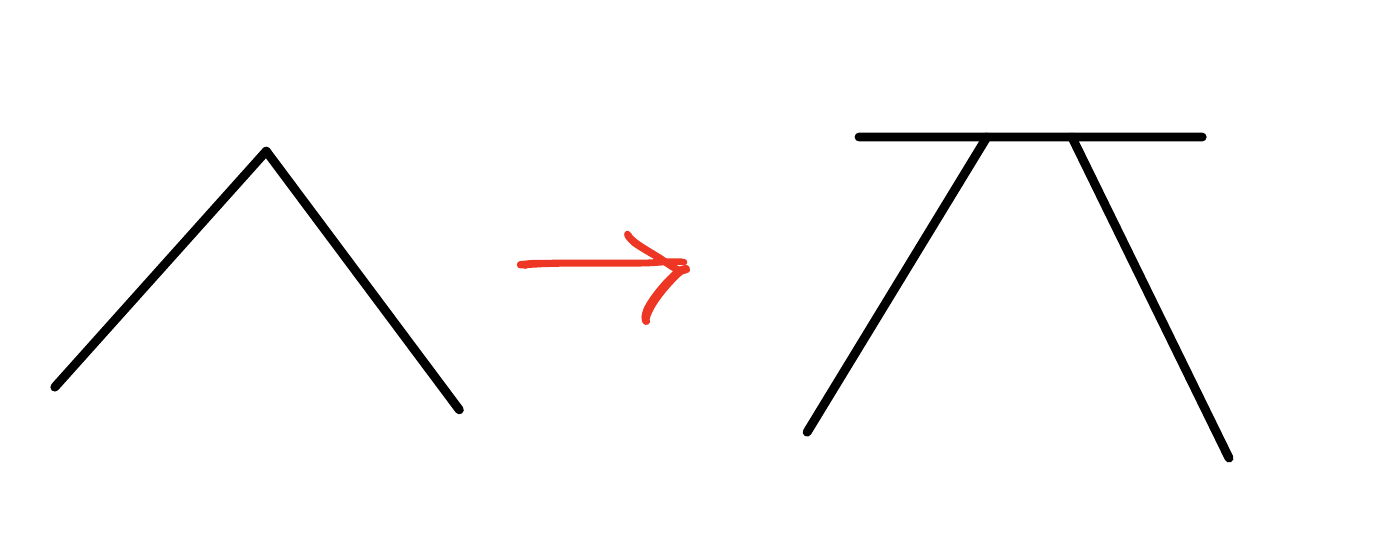}
\caption{Arborealization of an order 1 ridge.}
\label{fig:arb rid}
\end{figure}

         \begin{figure}[h]
\includegraphics[scale=0.5]{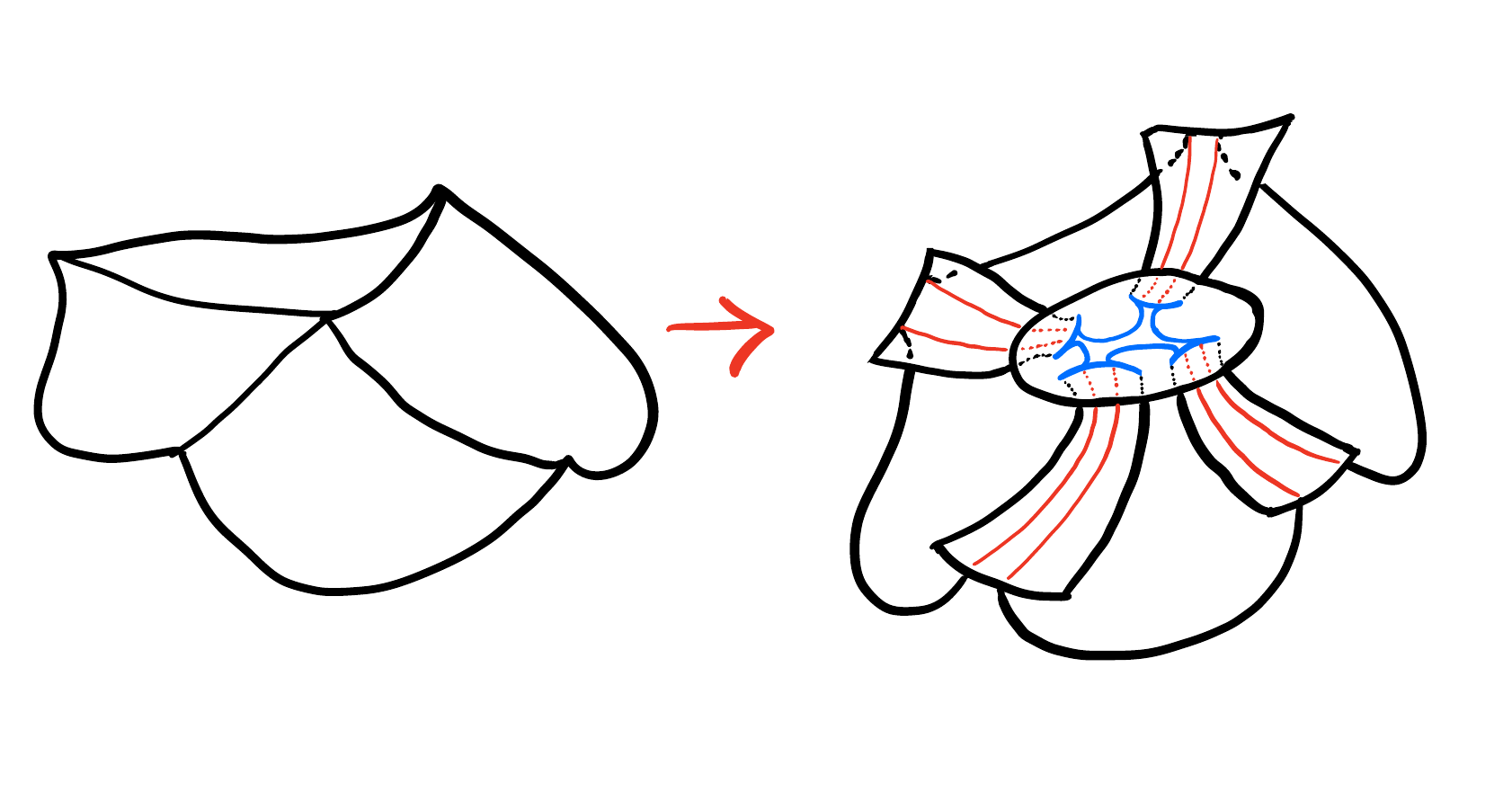}
\caption{Arborealization of an order 2 ridge. }
\label{fig:arborealizationridge2}
\end{figure}

Finally, a word on how we deal with the issue of compatibility at the interaction of three or more strata of the skeleton: this is where the notion of positivity comes up in an essential way. The key fact is that if a Lagrangian distribution $\eta$ on $T^*\R^n$ corresponds to a family of positive definite quadratic forms, then it is transverse to the conormal $T^*_H\R^n \subset T^*\R^n$ of {\it any} immersed hypersurface $H \subset \R^n$ in a neighborhood of the zero section. For the interaction of say three strata, we use the ridgification theorem to achieve transversality for two strata at a time and then conclude that transversality also holds for the third stratum {\it for free} using the above observation on positivity. This yields the desired compatibility. The existence of a global polarization for $W$ will play a key role in guiding the positivity condition globally, thus enabling the successful implementation of the above strategy.

\subsection{Structure of the article}

The paper is organized as follows:

In Section \ref{sec:W-prelim} we give the necessary preliminaries regarding Weinstein manifolds. This includes a review of notions in the literature as well as the introduction of a new framework for working with Weinstein structures. The key notions are that of a Weinstein manifold with boundary and corners, which we call a Wc-manifold, and that of a cotangent building, which provides a way of presenting a Wc-manifold in terms of successive attachments of cotangent bundles. The concept of a cotangent building gives an alternative to the standard Weinstein handlebody presentation for Weinstein manifolds, and is better suited for our purposes.

In Section \ref{sec: nbhd} we prove the existence and uniqueness of symplectic neighborhoods of arboreal spaces. The existence part is an easy consequence of the stability theorem for arboreal singularities, which we proved in \cite{AGEN20a}. The uniqueness part consists of a Darboux-Weinstein theorem for the symplectic neighborhoods of arboreal Lagrangians.

In Section \ref{sec:positivity}  we discuss the notion of positivity, which is the key ingredient that allows us to control the interaction of three or more strata in the Lagrangian skeleta under consideration. We first discuss the notion of positivity for ordered tuples of Lagrangian planes, then discuss positivity for cotangent buildings and finally we relate this to the notion of positivity for arboreal spaces. We show that existence of a positive arboreal skeleton for the homotopy class of a Weinstein structure implies the existence of a polarization, proving one half of our main theorem \ref{intro:main thm}.

In Section \ref{sec: ridgy} we recall the notion of a ridgy Lagrangian from our previous work \cite{AGEN19}  and state the main result of that paper, which we call the ridgification theorem, as well as its non-integrable counterpart, which we call the formal ridgification theorem. The formal ridgification theorem also plays an important role since it holds in extension form. The main point is that by allowing for ridgy singularities we can deform any Lagrangian submanifold so that it becomes transverse to any given Lagrangian distribution, and moreover the deformation can be achieved by means of a Weinstein homotopy.

Finally, in Section \ref{sec:corner-arboreal} we show how the ridgy singularities produced by the ridgification theorem can be further deformed into arboreal singularities. Furthermore, the deformation is realized by a Weinstein homotopy and the end result is a positive arboreal Lagrangian which serves as skeleton for the deformed Weinstein structure. This proves the other half of our main theorem \ref{intro:main thm}.

\subsection{Acknowledgements}
We are very grateful to Laura Starkston who collaborated with us on the initial stages of this project. We are also grateful to  John Pardon who helped us to crystalize  the notion of an $(n, n-1)$-polarization and to S{\o}ren Galatius who explained to us how to compute obstructions to its existence.

The first author thanks the Institute for Advanced Study and Princeton University for a great research environment and thanks CRM Montreal for its hospitality. The second author  thanks RIMS Kyoto and ITS ETH Zurich for their hospitality. The third author thanks MSRI for its hospitality.

Finally, we are very grateful for the support of the American Institute of Mathematics, which hosted a workshop on the arborealization program in 2018 from which this project has greatly benefited.

\section{Wc-manifolds and cotangent buildings}\label{sec:W-prelim}

In this section we recall the basic definitions concerning Liouville and Weinstein manifolds, and introduce the language of Wc-manifolds and cotangent buildings, which will be used throughout the text.
         
 \subsection{Liouville manifolds}\label{sec:L-man} 
 
 \subsubsection{Liouville domains}
 
An exact symplectic manifold $(W, \lambda)$ is a compact manifold with boundary equipped with an exact symplectic form $\omega$, together with a choice of primitive $\lambda$, i.e. $\omega=d \lambda$. Since $\omega$ is non-degenerate, the 1-form  $\lambda$ corresponds to a vector field  $Z$ under the contraction $\iota_Z \omega = \lambda$.

\begin{definition}
A {\em Liouville domain} is an exact symplectic manifold $(W,\lambda)$ such that $Z$ is outwards transverse to
 $\p W$.
\end{definition}
Equivalently, $(W, \lambda)$ is Liouville if $\lambda|_{\p W}$ is a contact form and the orientation defined on $\p W$ by the volume form $\lambda\wedge (d\lambda)^{n-1}|_{\p W} $ coincides with the orientation of $\p W$ as the boundary of $W$, where $W$ is itself oriented by the volume form $\omega^{\land n}$.

We call $\lambda$ the {\em Liouville form} and $Z$ the {\em Liouville field}. By Cartan's formula for the Lie derivative, the condition $ \omega = d \lambda$ is equivalent to $\cL_{Z} \omega = \omega$. This means that the vector field $Z$ is symplectically conformally expanding, i.e.  $(Z^{t})_*\omega=e^{t} \omega$, $t\geq 0$, for $Z^t$ the flow of $Z$. Equivalently, the vector field $-Z$ is symplectically conformally contracting and this viewpoint will sometimes be more natural in what follows.
Note that we also have $(Z^{t})_*\lambda=e^{t} \lambda$. The object of interest in this paper is the following.

\begin{definition}  The {\em skeleton} of a Liouville domain $(W,\lambda)$ is the subset  $$\Skel(W,\lambda)=\bigcap\limits_{t>0}Z^{-t}(W).$$   \end{definition}


In other words, $\Skel(W,\lambda)$ is the attractor of the 
positive flow of the contracting field $-Z$. We will be going back and forth between Liouville domains and the closely related notion of a Liouville manifold, whose definition we recall next. 

\subsubsection{Liouville manifolds}

Recall that the {\em symplectization} $S(Y,\xi)$ of a contact manifold $(Y, \xi)$ equipped with a contact form $\alpha$ for $\xi$ is the exact symplectic manifold $Y \times \R$ equipped with the primitive $\lambda=e^t \alpha$, where $t \in \R$. We can equivalently consider $Y \times \R^+$ with $\lambda = s \alpha$ for $s=e^t \in \R^+$ the multiplicative coordinate.

The symplectization of a contact manifold $(Y, \xi)$  can be defined invariantly without choosing a contact form. Assuming that $Y$ is connected and $\xi$ is co-oriented, we define $S(Y,\xi)$ as follows. Let $N(\xi)\subset T^*Y$ be the total space of the conormal line bundle to $\xi$ and $N_+(\xi)$ the component of $N(\xi)\setminus Y$ consisting of $1$-forms defining the given co-orientation of $\xi$.  Then $\lambda_\xi=pdq|_{N_+(\xi)}$ is a primitive of the exact symplectic form $\omega= d \lambda_\xi$ and we define $S(Y,\xi)$ as $N_+(\xi)$ endowed with this structure. A choice of a contact form $\alpha$ for $\xi$ identifies  $N_+(\xi)$ with $Y \times \R^+$ and
 $\lambda_\xi$ with $s\alpha$. For example, the symplectization of the cosphere bundle $S^*M$ is the complement of the zero section in $T^*M$. 
 
 In the other direction, the {\em contactization} of an exact symplectic manifold $(N, \mu)$ is the manifold $M \times \R$ equipped with the contact form $\mu-du$, where $u \in \R$. For example, the contactization of the cotangent bundle $T^*M$ is the 1-jet bundle $J^1M$. 


\begin{definition} A {\em finite type Liouville manifold} $(X, \lambda)$ consists of a (non-compact) boundaryless exact symplectic manifold such that the following conditions hold.
\begin{itemize}
\item[(L1)]  The vector field $Z$ which is $\omega$-dual to $\lambda$  is complete for $t\to\pm\infty$. 
\item[(L2)] There exists a compact domain $W\subset X$ such that $Z$ is outward  transverse  to the boundary $\p W$, i.e. $(W, \lambda|_W)$ is a Liouville domain, and such that the union of forward trajectories of $Z$ starting at $\p W$ is equal to $X\setminus\Int W$. \end{itemize}
\end{definition}

 We call Liouville domains as     in (L2)  {\em defining}. 

\begin{definition} The {\em skeleton} of a finite type Liouville manifold $(X,\lambda)$ is defined to be 
$$\Skel(X,\lambda)=\bigcup\limits_{C} \mathop{\bigcap}\limits_{t>0}Z^{-t}( C), \qquad \text{$C \subset X$ is compact}. $$ \end{definition}

Clearly, $\Skel(X,\lambda)=\Skel(W,\lambda)$ for any defining Liouville domain $W\subset X$. 
Hence the skeleton of a finite type Liouville manifold is compact. Note that for any such $W$ we may use the Liouville flow starting at $\p W$ to parametrize $X \setminus \text{int} W$ as  $\partial W \times [0,\infty)$. Since the Liouville field conformally expands $\lambda$ we have $\lambda|_{\p W \times [0,\infty)} = e^t (\lambda|_{\p W})$ for $t \in [0,\infty)$ the coordinate parametrizing the Liouville flow, so $X \setminus \text{int} W$ is the positive part of the symplectization of the contact manifold $(\p W , \lambda|_{\p W})$.

Conversely,  any Liouville domain $(W,\lambda)$ can be completed to a Liouville manifold
by attaching  to it the positive part of the symplectization of  $(\p W,\lambda|_{\partial W} )$. Explicitly:
$$X=W\mathop{\bigcup}\limits_{\p W\sim\p W\times 0} \p W\times[0,\infty) ,$$  where the Liouville form $\lambda$ on $W$ extends to to $\p W\times[0,\infty)$ as $e^t\left(\lambda|_{\p W}\right)$.

\begin{definition} The {\em ideal boundary}  of a finite type Liouville manifold $(X,\lambda)$ is defined to be $\partial_{\infty}X = \big(X \setminus \Skel(X, \lambda) \big) / \R^+$, where the $\R^+$ action is given by the positive flow of $Z$. 
\end{definition}
Since $\lambda(Z)=0$, the form $\lambda$ descends to  $\partial_\infty X$  as  a contact structure $\xi_\infty$. A choice of a contact form $\alpha$ for
$\xi_\infty$ yields an isomorphism $j_\alpha:S(\p_\infty X,\xi_\infty)=(\p_\infty X\times\R, e^t\alpha)\to (X\setminus \Skel(X,\lambda),\lambda)$ which maps $\p_\infty X\times 0$ contactomorphically onto the boundary of a defining Liouville domain.
We denote by $\pi_\infty : X\setminus \Skel(X,\lambda)\to\p_\infty X$  the projection along the trajectories of $Z$.

\subsubsection{Deformations of Liouville structures}
 
\begin{definition} A {\it homotopy of finite type Liouville structures} $(X,\lambda_t)$ is a family $\lambda_t$, $t\in[0,1]$, of finite type Liouville structures on a manifold $X$  admitting a   smooth  family $W_t \subset X$ of defining Liouville domains. \end{definition}
Given such a homotopy  $(X,  \lambda_t)$ there exists an isotopy $\phi_t:\ X\to  X$ such that $\phi_t^*\om_t=\om_0$, where $\omega_t=d\lambda_t$. Moreover, one can always arrange that $\phi_t^*\lambda_t=\lambda_0+dH_t$ for uniformly compactly supported  functions $H_t$, see \cite{CE12}, Sections 11.1 and 11.2. In particular, it is always sufficient to consider homotopies fixing the symplectic form, and moreover changing the Liouville form by adding a compactly supported exact form. Note that for any fixed $\omega$ the space of Liouville structures on $(X,\omega)$ which are fixed at infinity is convex, hence any two such Liouville structures are canonically homotopic. 

\begin{remark} Unless explicitly stated otherwise, all Liouville structures under consideration in this paper are of finite type and hence we will stop mentioning this below. \end{remark}

We will often abuse terminology and say that two Liouville manifolds $(X_1,\lambda_1)$ and $(X_2,\lambda_2)$ are Liouville homotopic if there exists a symplectomorphism $F:(X_1,d\lambda_1) \to (X_2, d\lambda_2)$ such that $(X_1,F^*\lambda_2)$ is a Liouville manifold which is Liouville homotopic to $(X_1,\lambda_1)$. In all cases where such terminology will be abused, the symplectomorphism will be obvious and unique up to contractible choice.

If $F^*\lambda_2=\lambda_1$ on the nose, we say that $(X_1,\lambda_1)$ and $(X_2,\lambda_2)$ are {\it Liouville isomorphic}. For $(X, \lambda)$ and $(Y, \mu)$ Liouville domains or manifolds we call a smooth, proper embedding $F:(X, \lambda) \to (Y, \mu)$ a {\it Liouville embedding} if $F^* \mu=\lambda$.

\subsubsection{Liouville germs}

It will often be convenient to express our constructions in the language of germs.

\begin{definition}  A {\em Liouville germ }  $(\sX,\lambda)$ is the equivalence class of defining Liouville domains for a fixed Liouville manifold $ (X,\lambda)$. 
\end{definition}

\begin{definition} A Liouville embedding of a Liouville germ $(\sX, \lambda)$ into a Liouville domain or manifold $(Y, \mu)$ is an equivalence class of Liouville embeddings $F_W : (W, \lambda|_W) \to (Y, \mu)$, where $W \subset X$ is a defining domain for $(X, \lambda)$ and we identify $F_{W_0}$ with $F_{W_1}$ if they agree on a defining domain $W_2$ for $(X, \lambda)$ contained in both $W_0$ and $W_1$.  \end{definition}



 



\begin{example}
The cotangent bundle $X=T^*M$ of a closed manfiold $M$ is a Liouville manifold with $\lambda=pdq$ the canonical 1-form. The Liouville field is the fibrewise radial vector field $Z=p \p_p$ and the skeleton is the zero section $M \subset T^*M$. As a Liouville domain we can take the unit disk bundle $W=\{ \| p \|  \leq 1 \}$ relative to any Riemannian metric on $M$. The ideal boundary is the cosphere bundle $S^*M$, i.e the (positively) projectivized cotangent bundle, which is contactomorphic to the unit sphere bundle $\partial W = \{ \|p \|=1 \}$ for any choice of Riemannian metric. We will use the notation $\sT^*M$ for the germ of the cotangent bundle $T^*M$.
\end{example}

\subsection{Weinstein manifolds}\label{sec:W-man} 
The notion of {\em Weinstein manifold} was first  introduced in \cite{EG91}  building on \cite{E90} and \cite{W91}. 
\subsubsection{Lyapunov functions}

Let $(X, \lambda)$ be a Liouville manifold. While $\Skel(X,\lambda)$ always has its $2n$-dimensional Lebesgue  measure equal to $0$, its dimension can nonetheless be quite  large if no  extra conditions are imposed on the Liouville structure. For example, McDuff constructed in  \cite{MD91}   a Liouville structure on $T^*\Sigma_g\setminus \Sigma_g$ for  $\Sigma_g$ a closed surface of genus $g>1$ whose skeleton has   codimension 1. To tame the dynamics of the Liouville flow we require existence of a Lyapunov function.
\begin{definition} We say that $\phi:X \to \R$ is a {\em Lyapunov function} for the Liouville manifold $(X, \lambda)$ if $Z$ is gradient-like for $\phi$, i.e. if there holds
\begin{itemize}
\item[(W1)] $d\phi(Z)\geq \delta (\|Z\|^2+ \|d \phi \|^2)$ for some Riemannian metric on $X$ and some constant $\delta >0$.
\end{itemize} 
\end{definition}
Note that the space of Lyapunov functions for a given Liouville structure is convex, hence contractible as soon as it is nonempty. 
Consider the set of critical points $\text{Crit}(\phi)= \{ d \phi = 0 \}$  of the Lyapunov function, which by (W1) is the same as the zero set  $\{ \lambda = 0\}$ of the Liouville form. The conditions (L1) and (W1) imply that $\Skel( X,\lambda)$ is the union of the $Z$-stable manifolds of the critical points of $\phi$, i.e. points converging to $\text{Crit}(\phi)$ in forward time. However, as far as we know, we have to add some further constraints  on the Lyapunov function $\phi$  in order to deduce any meaningful properties.

\begin{definition}
A {\em Morse-Weinstein manifold} is a Liouville manifold $(X, \lambda)$ for which there exists a
 Morse  Lyapunov function $\phi:X \to \R$. 
\end{definition}
Under this assumption it was shown in \cite{CE12}, see also \cite{EG91, E95}, that
  \begin{itemize}
\item[(W2)] $\Skel(X, \lambda)$ is the union of submanifolds which are isotropic for $\lambda$, and hence for $\om$.
\end{itemize}  
   \begin{figure}[h]
\includegraphics[scale=0.5]{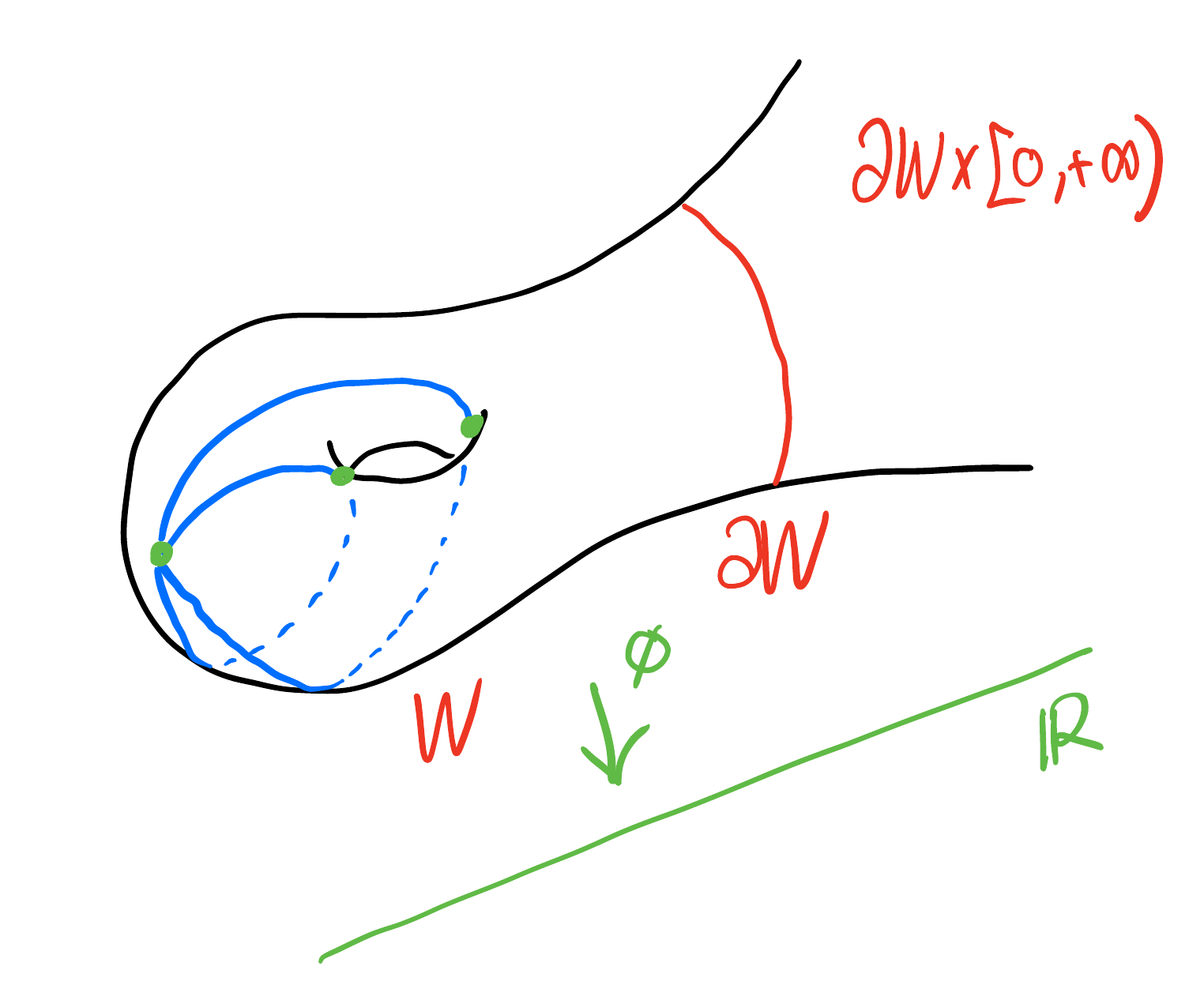}
\caption{The skeleton of an open Riemann surface with Morse Lypapunov function.}
\label{OpenRS}
\end{figure}

\begin{remark}
Note that $\lambda$-isotropic is equivalent to $\omega$-isotropic and $Z$-conic.
\end{remark} 
Without any assumptions on $\phi$, is not known whether a triple $(X,\lambda,\phi)$ which satisfies (L1), (L2) and (W1) is homotopic in the class of such triples to one with a Morse Lyapunov function.
However, condition (W2) holds for a larger class  of taming functions, for example when  $\phi$ is generalized Morse, see \cite{CE12}. Moreover, it is often most natural to consider a Lyapunov function which is not Morse or even generalized Morse, such as the Morse-Bott function $\phi=\|p\|^2$ on $(T^*M,pdq)$. 

\subsubsection{Morse-Bott Lyapunov functions}

In this paper we will consider a version of Morse-Bott Weinstein structures which allows the critical set of the Lyapnuov function to have nonempty boundary, and even corners. Our definition is a slight variation of Starkston's definition in \cite{St18} of a Weinstein structure which is Morse-Bott with boundary. 

Let $E^+ \oplus E^0 \oplus E^-$ be the decomposition of the tangent space to $X$ at a zero of $Z$ into generalized eigenspaces of the differential $dZ$ with eigenvalue having positive, zero or negative real part respectively. Suppose first that a connected component $C$ of $\text{Crit}(\phi)$ is a smooth, proper submanifold of $X$ without boundary. Then the Morse-Bott condition on a triple $(X, \lambda, \phi)$ along the component $C$ is the following:
\begin{itemize}
\item[(MB1)] $Z$ is non-degenerate in the directions normal to $C$, i.e. $T C = E^0|_{C}$. 
\end{itemize}

When $\p C \neq \varnothing$ we moreover demand that in a neighborhood of $\p C$ the triple $(X, \lambda,  \phi)$ takes a special form, which we now describe. Recall the Weinstein normal form for Morse critical points in a Weinstein manifold of dimension $2$. There are two possibilities depending on the index: $k=0$ or $k=1$.


\begin{itemize}


\item[($k=1$)] $\lambda_1=2pdq+ qdp$, $\phi_1=p^2-\frac{1}{2}q^2$.


\item[$(k=0$)] $\lambda_0=\frac{1}{2}(p dq - qdp)  $, $ \phi_0 = \frac{1}{4}(p^2 + q^2)$.

\end{itemize}

We also recall the standard Morse-Bott normal form corresponding to the canonical Liouville structure on $T^*\R$.

\begin{itemize}
\item[(std)] $\lambda_{\text{std}} = p dq$,  $\phi =\frac{1}{2}p^2$.
\end{itemize}

We construct local models for {\it Morse-Bott with boundary} Weinstein structures on $T^*\R$ with the half line $\{p=0 , q \geq 0 \}$ as the critical set. The construction consists of patching up the above Morse normal forms on the half plane $\{ q \leq 0 \}$ with the standard Morse-Bott normal form on the half plane $\{ q \geq 0 \}$. Take any smooth function $\psi: \R \to [0,1]$ satisfying the following properties:
\begin{itemize}
\item [(i)] $\psi =0 $ on $[0, \infty)$
\item [(ii)] $\psi = 1$ on $(-\infty, -\eps]$ for some $\eps>0$.
\item [(iii)] $\psi  < 1$ and $\psi'<0$ on $(-\eps, 0)$ for that same $\eps>0$.
\end{itemize}

In the case $k=1$ we set $f_1(q,p)=\psi(q)pq$ and in the case $k=0$ we set $f_0(q,p)=-\frac{1}{2}\psi(q)pq$. In both cases the patched up model is then given by $\lambda_k^{\text{MB}} = \lambda_{\text{std}}+ df_k$, which has Liouville field $Z_k^{\text{MB}}=Z_{\text{std}}+X_{f_k}$ for $X_{f_k}$ the Hamiltonian vector field of $f_k$. Explicitly, when $k=1$ we have
\[ Z_1^{\text{MB}} = \left(1 + \psi(q) + \psi'(q) q \right) p \p_p - \psi(q) q \p_q.\]
Since the coefficient which multiplies $p \p_p$ is positive, we deduce that $Z_1^{\text{MB}}$ admits a Lyapunov function. For example, if we set $\phi= p^2-h(q)$ for
\[ h(q) =  \int_0^q \psi(t) t dt, \]
then $Z_1^{\text{MB}}$ is the Euclidean gradient of $\phi$ along $p=0$ and we have $d\phi(Z_1^{\text{MB}})>0$ on all of $T^*\R$, from which the claim follows. Similarly, one verifies that $Z_0^{\text{MB}}$ also admits a Lyapunov function.

\begin{definition}
We call $(T^*\R , \lambda_k^{\text{MB}})$ the {\it Weinstein normal form for Morse-Bott boundary} of index $k=0,1$.  \end{definition}
\begin{figure}[h]
\includegraphics[scale=0.55]{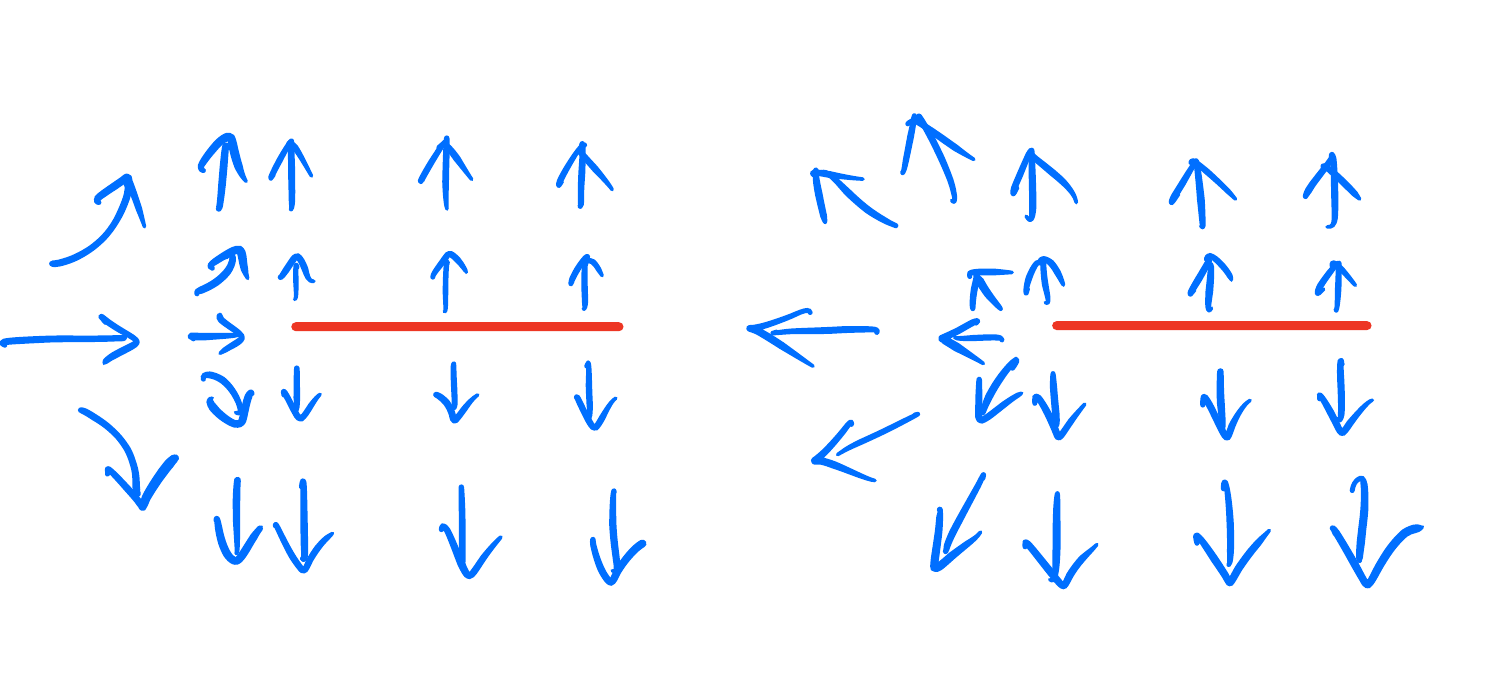}
\caption{The Weinstein normal form for Morse-Bott boundary of index 1 and 0 respectively.}
\label{MorseBott}
\end{figure}

The space of permissible $\psi$ is convex, hence the normal form is well defined up to a contractible choice. Note that the behavior of the Liouville flow along the half line $\{p=0, q \leq 0\}$, which shares its boundary point $\{p=0,q=0\}$ with the critical set $\{p=0, q \geq 0 \}$, can be either repelling or attracting depending on whether $k=0$ or $k=1$. We also consider products of these normal forms. Thus we obtain Weinstein structures on $T^*\R^n$ where the critical set is the positive quadrant $\{ p_j = 0, q_j \geq 0 \}$, which not only has boundary but also has corners of all orders $\leq n$.

\begin{definition} The {\it Weinstein normal form for Morse-Bott $n$-fold corners $(T^*\R^n,\lambda_r^{\text{MB}})$ of index $0 \leq r \leq n$} is the product of $n-r$ copies of $(T^*\R, \lambda_0^{\text{MB}})$ and $r$ copies of $(T^*\R, \lambda_1^{\text{MB}})$. \end{definition}


Going back to a triple $(X,\lambda, \phi)$ as before, consider a component $C \subset X$ of the critical set of $\phi$ such that $C$ is a smooth, proper submanifold of $X$ with nonempty boundary $\p C$ and possibly with higher order corners. 
By definition, if $x \in \p_k C \subset \p C$ is a corner of order $k$, then there is a collar neighborhood $U \times \cI^k$ of $x$ in $C$, where $\cI$ is the germ of the interval $[0,1)$ at $0$ and $U \subset \p_k C$ is a neighborhood of $x$ in the order $k$ corner locus. For such a collar neighborhood we have a local symplectomorphism near $x$ between $X$ and $T^*U \times T^*\R^k \times  \C^{m}$, with $C$ corresponding to $U \times \cI^k \times 0$.
The condition we require is that for each such $k$-fold point $x \in \p_kC$ there exists a collar neighborhood $U \times \cI^k$ as above such that with respect to the splitting $T^*U \times T^*\R^k \times \C^m$ there holds:

\begin{itemize}
\item [(MB2)] $\lambda$ splits as a direct sum with the summand corresponding to $T^*U$ given by the canonical Liouville form and the summand corresponding to $T^*\R^k$ given by the Weinstein normal form for Morse-Bott $k$-fold corners $\lambda_r^{\text{MB}}$ for some $r \leq k$.
\end{itemize}


\begin{definition}
A  {\em generalized Morse-Bott} Weinstein manifold is a Liouville manifold $(X, \lambda)$ which admits a Lyapnuov function $\phi$ whose critical set is a union of smooth, proper submanifolds with corners $C \subset X$  satisfying (MB1) and (MB2).
\end{definition}

\begin{remark} It was shown by Starkston \cite{St18} that in the boundary repelling Morse-Bott case the intersection of the $Z$-stable manifold of a component $C$ of the critical set of $\phi$ with a neighborhood of $C$ is a smoothly embedded isotropic submanifold of $X$, even without requiring standard local models. To check that this property also holds near the boundary attracting or mixed behavior points of the Morse-Bott structures considered above is straightforward,  since it suffices to consider the explicit local model.  \end{remark}

\begin{lemma}\label{lm:to-Morse}
Suppose that $(X,\lambda_0, \phi_0)$ is a generalized Morse-Bott Weinstein manifold. Then there exists a homotopy $(\lambda_t,\phi_t)$, fixed outside of a compact subset, such that $(\lambda_t,\phi_t)$ is Morse-Weinstein for $t>0$.
\end{lemma}

\begin{proof}
We give the proof in the case where each component $C \subset \Sigma$ of the critical set of $\phi$ is of the critical dimension $n=\frac{1}{2}\dim(X)$, the general case is similar and in any case will not be needed for our purposes.
Suppose first that  $C \subset \Sigma$ has empty boundary. Since $C$ is Lagrangian, by the Weinstein neighborhood theorem it has a neighborhood in $X$ symplectomorphic to a neighborhood of the zero section in $(T^*C,\lambda_{\text{std}}=pdq)$. Therefore we may from the onset replace $X$ by $T^*C$ with the pulled back Liouville form which we still denote by $\lambda_0$. 

Let $f:C \to \R$ be a Morse function and let $\eta: T^*C \to [0,1]$ be a cutoff function such that $\eta=1$ on $\Op(C)$ and $\eta=0$ outside of a slightly bigger neighborhood. 
Fix an auxiliary Riemannian metric on $C$ and put $h=\eta  \lambda_{\text{std}}(\nabla f)$, a function on $T^*C$. The restriction of the Hamiltonian vector field of the function $h$ to $C$ is $\nabla f$. Hence $\phi_0+\eta f$ is a Lyapunov function for $\lambda_0 + dh$ in a neighborhood of $C$ and therefore on all of $T^*C$ if $f$ is taken sufficiently $C^1$ small. Moreover, for such $h$ the Liouville form $\lambda_0 + t dh$ is Morse-Weinstein for any $t>0$.


If $\partial C$ is non-empty, then we need to modify the construction somewhat. Consider a slight enlargement $\widehat{C}$ of $C$ in $X$ obtained by attaching an isotropic collar along $\partial C$ so that the Liouville field $Z_0$ is tangent to $\widehat{C}$. Then we should require that our function $f: \widehat{C} \to \R$ is Morse on $C$, compactly supported in $\widehat{C}$, has no critical points on $\p C$ and moreover has the following behavior on $\p C$. Let $P$ be the closure of a component of the order 1 corner locus $\p_1C=\p C \setminus \bigcup_{k > 1} \p_k C$. If the Liouville structure is boundary attracting $(k=1)$ along $P$, then we demand that $df(v)>0$ for any inwards pointing vector $v \in TC|_{\p C}$. If the Liouville structure is boundary repelling ($k=0$) along $P$, then we demand that $df(v)>0$ for any outwards pointing vector $v \in TC|_{\p C}$. With this restriction on $f$ the proof proceeds just as before.  \end{proof}

  \begin{remark}
  As we will see, if a Morse-Bott Weinstein manifold has arboreal skeleton, then one can choose the Morsifying homotopy such that  $\Skel(X,\lambda_t)=\Skel(X,\lambda_0)$ for all $t\geq 0$.  
  \end{remark}
  
  \subsubsection{Weinstein manifolds}

It is easy to check that Lemma \ref{lm:to-Morse} also holds for the generalized Morse condition instead of Morse-Bott, see Proposition 9.13 in \cite{CE12}. Therefore we will take the following as our working definition, which includes Morse, generalized Morse and Morse-Bott Weinstein structures.

\begin{definition}\label{def: weinstein mfld} A Liouville manifold $(X, \lambda_0)$ is  called {\em Weinstein} if there exists a Liouville homotopy $\lambda_t$ and a family of Lyapunov functions $\phi_t$ for  $Z_t$ such that $(X,\lambda_t,\phi_t)$ is Morse-Weinstein for $t>0$.
\end{definition}

The notions of  a Weinstein domain,  defining Weinstein domain  for a Weinstein manifold and a Weinstein germ are entirely analogous to the corresponding notions in the Liouville case. We emphasize that the Lyapunov function $\phi$ is not part of the structure of a Weinstein manifold, we merely require its existence. For example, Weinstein embeddings are just Liouville embeddings of Weinstein type Liouville manifolds (or domains, or germs). We do not require the embedding to pull back one Lyapunov function to the other. Similarly, a Weinstein isomorphism (resp. homotopy) is a Liouville isomorphism (resp. homotopy) of Liouville manifolds of Weinstein type.

\begin{example}
The standard Liouville structure $\lambda=pdq$ on the cotangent bundle $T^*M$ of a closed manifold $M$ is  Weinstein. Given a Riemannian metric on $M$, the function $\phi=\|p\|^2$ is a Morse-Bott Lyapunov function for $\lambda$.  By the procedure explained in Lemma \ref{lm:to-Morse}, we can Morsify the Lyapunov function $\phi$ using a Morse function on $M$. By inspection, this procedure does not change the skeleton of $T^*M$, which is the zero section.
  \end{example}
  
  \begin{remark} Although Definition \ref{def: weinstein mfld} allows for rather degenerate Weinstein structures, in the Weinstein homotopies constructed in this paper for the purposes of arborealization of a Morse-Weinstein manifold the Weinstein structures are always of Morse-Bott type except for a finite number of times at which there is a birth/death of Morse-Bott critical components.
\end{remark}

 \subsection{Wc-manifolds}\label{sec:Wbc}
 
 \subsubsection{Manifolds with boundary and corners}


Recall that an $n$-dimensional smooth manifold $M$ has a corner of order $k \leq n$ at $x \in M$ if there is a neighborhood of $x$ in $M$ diffeomorphic to a neighborhood of the origin in $\cI^k \times \R^{n-k}$, where $\cI$ denotes the germ of the interval $[0,1)$ at $0$. We denote the locus of order $k$ corners by $\partial_k M$. 
The closure $P$ of a connected component of $\p_kM$ is called a {\it boundary $k$-face.} For $k=1$ we will more simply call $P$ a {\it boundary face.} Sometimes we will call a manifold with boundary and corners (of any order) a {\em bc-manifold} for short.

We always assume that each component of $\p_k M$ is regularly embedded, so that each $k$-face is itself a manifold with corners. Under this assumption, if $Q$ is a $k$-face, then there is an embedded collar neighborhood $Q \times \cI^k \subset M$. We consider the germs of these collars as part of the structure. In particular, near each point $x \in \partial_k M$ we have canonical collar coordinates $x=(y,t)$, where $y \in \partial_k M$ and $t=(t_1, \ldots , t_k) \in \cI^k$. Note that in a neighborhood of $x$ we have $\partial_k M$ cut out by $t_1 = \cdots = t_k=0$. More generally, for $j\leq k$ the components of $\partial_j M$ whose closure contains $x$ are given by setting exactly $j$ of the coordinates $t_i$ equal to zero. 

We demand compatibility of these collars in the sense that the remaining $n-j$ coordinates $t_i$ give the collar structure for $\partial_j M \times \cI^{n-j}$ near $x$. We call such a compatible system of collar germs a {\it corner structure on $M$}. Henceforth all manifolds with corners will be implicitly equipped with a fixed but otherwise arbitrary corner structure. 

\begin{figure}[h]
\includegraphics[scale=0.5]{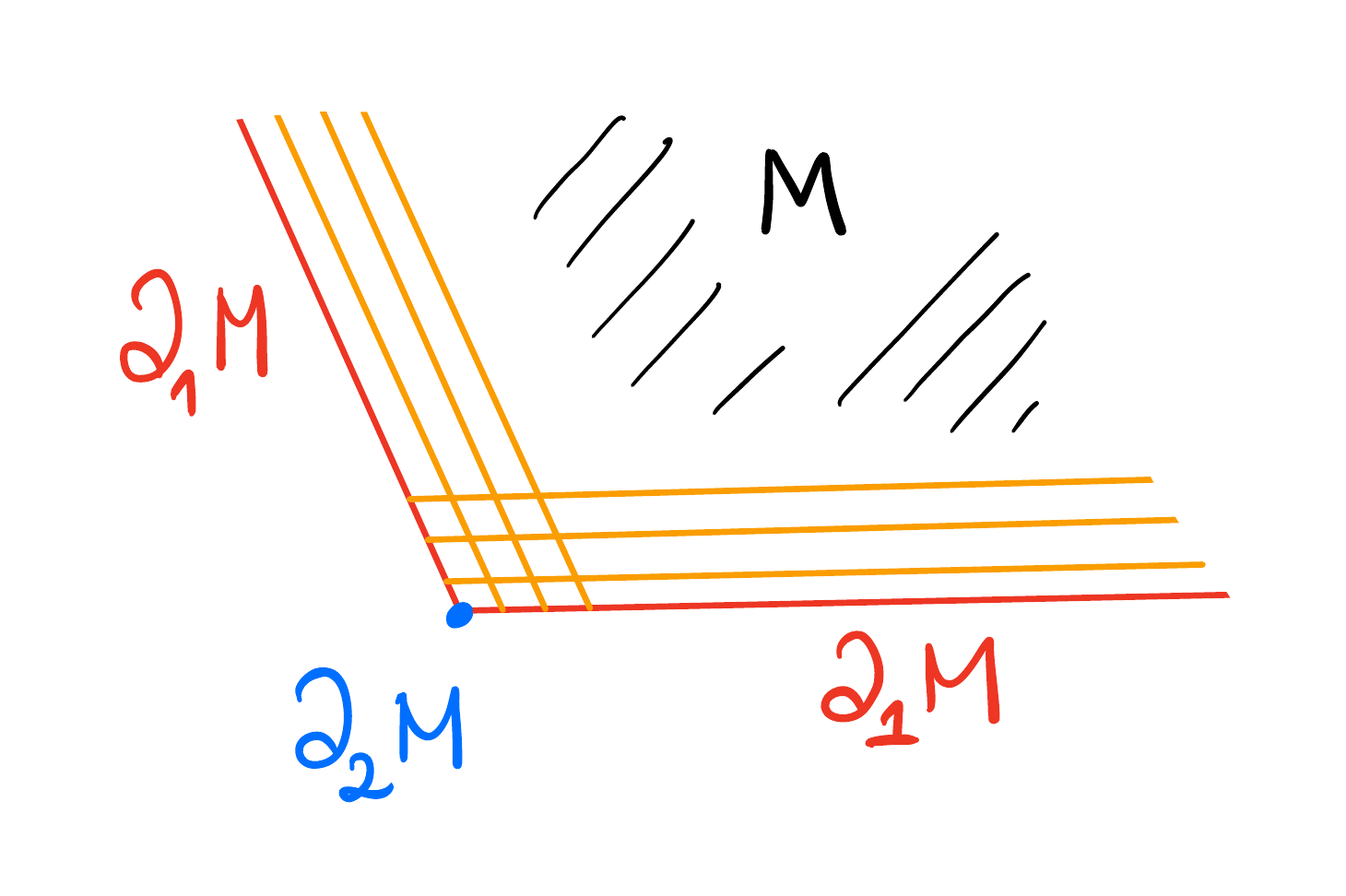}
\caption{A corner structure on $M$.}
\label{fig:collarstr}
\end{figure}


 Let us choose a partial ordering of the boundary faces of a manifold $M$ with corners such that the faces at all corners are ordered in a compatible way, i.e. the order is that induced by the order of the coordinates $t_i$ in the collar structure. We will assume this order to be part of the structure of a manifold with corners.
 
 

Given manifolds with corners $W$ and $X$, we consider embeddings of $W$ into $X$ which satisfy the property that each $j$-face of $W$ is properly embedded into some $k$-face of $X$, where $k \geq j$. 

\subsubsection{Liouville manifolds with boundary and corners}

We now give an inductive definition of the notion of a Liouville manifold with corners of order $k \leq  \frac{1}{2} \dim X$, starting with the base case $k=0$ which is just our previous definition.

 \begin{definition}
A \emph{Liouville manifold with corners of order $k$} consists of an exact symplectic manifold $(X^{2n},\lambda)$ with smooth corners of order $k \leq n$ such that the following properties hold.

\begin{figure}[h]
\includegraphics[scale=0.5]{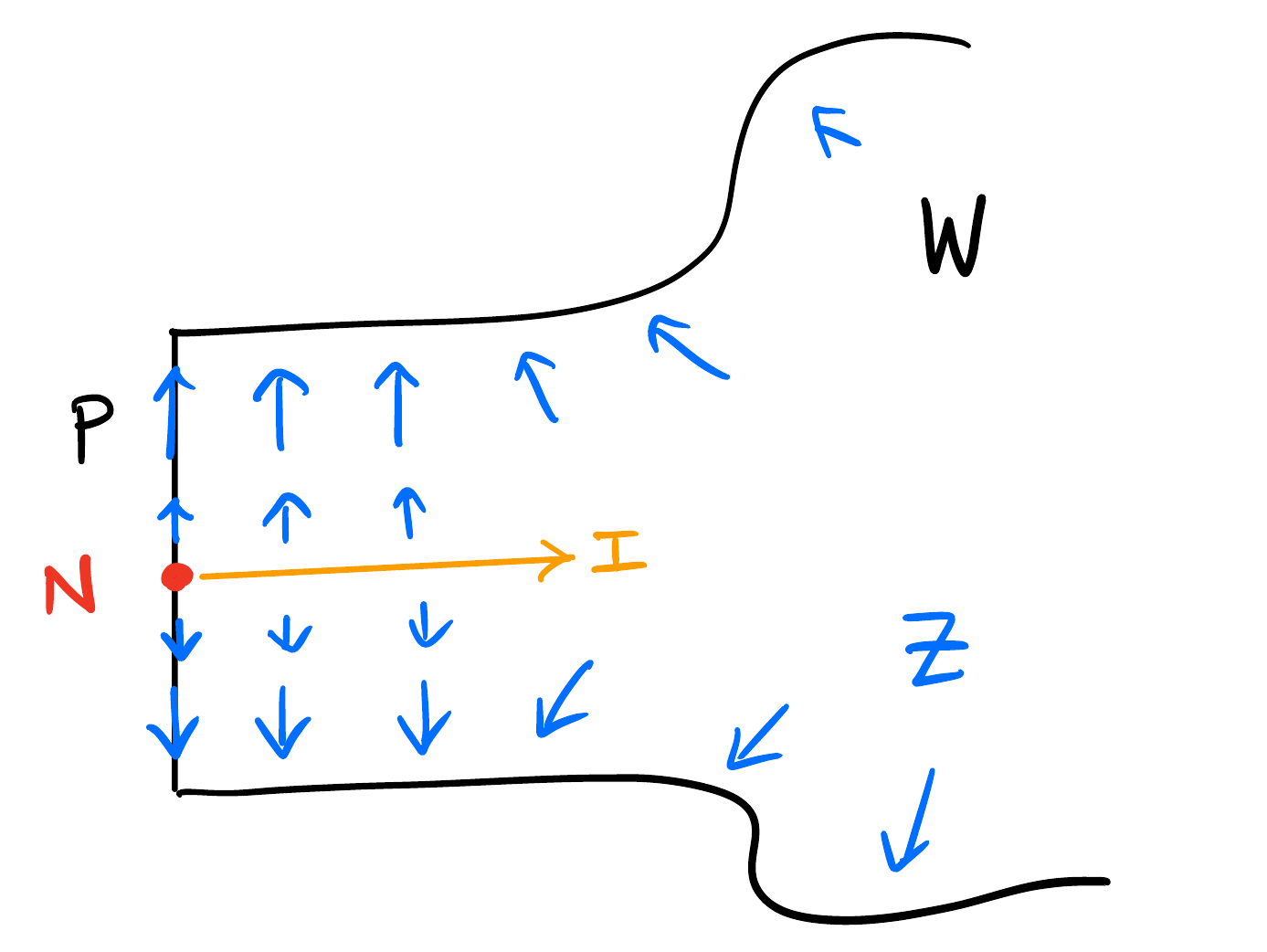}
\caption{A Liouville manifold with coners $W$ near one of its faces $P$ with nucleus $N$.}
\label{fig:Lbcmfld}
\end{figure}

\begin{itemize}
\item[(L1)] The vector field $Z$ which is $\omega$-dual to $\lambda$  is complete for $t \to \pm \infty$.
\item[(L2)] There exists a compact domain with corners $W \subset X$ such that   such that $Z$ is outwards transverse to its {\em vertical boundary}  $\partial_v W$, which is defined to be the closure of $ \p W \cap \text{int} X$, and such that the union of forward trajectories of $Z$ starting at $\partial_v W$ is  equal to $\partial_v W \cup (X \setminus W)$.
\item[(L3)] For $r=1, \ldots, k$, each $r$-face $P$ of the {\em horizontal boundary} $\p_h W$, which is defined to be $\p W \cap \p X$, contains a Liouville submanifold with corners $N \subset P$, called  the {\em nucleus} of $P$, such that the collar neighborhood of $P$ in $X$ is of the form $N \times T^*\II^r$, where $P$ is the product of $N$ with the cotangent fibre over $0 \in \II^r$, and in this neighborhood the form $\lambda$ is the direct sum of a Liouville form on $N$ and the canonical 1-form on $T^*\II^r$.
\end{itemize}
\end{definition}
  
It follows from the definition that the Liouville field $Z$ is tangent to every component of $\partial_k X$, $k= 1, \ldots , n$. A Liouville manifold with corners has a {\em horizontal boundary} $\partial_h X = \partial X$ and an  {\em ideal vertical boundary } $\partial_\infty X = \big(X \setminus \Skel(X, \lambda)\big)/\R^+$, where just as before the action of $\R^+$ is defined by the flow of $Z$ and the skeleton $\Skel(X,\lambda)$ is defined to be the attractor of the negative flow of $Z$. 
 
 We call $(W, \lambda|_W)$ as in (L2) a defining Liouville domain with corners for the Liouville manifold  with corners $(X,\lambda)$. Note that the horizontal boundary of $W$ is $\partial_h W = W \cap \partial_h X$ and its vertical boundary $\partial_v W$, the closure of $\p W \setminus \p_hW$,  is a contact manifold with corners which is naturally contactomorphic to $\partial_\infty X$ under the projection $\pi_\infty: \partial_v W \to \partial_\infty X$. Just as before, the equivalence class of defining domains for Liouville manifold with corners is called the Liouville germ of $(X,\lambda)$ and is denoted by $(\sX,\lambda)$. Notions of isomorphism and homotopy of Liouville manifolds with corners are defined exactly the same way as without corners. 

 \subsubsection{Wc-manifolds}
We now introduce Weinstein structures for Liouville manifolds with corners, which we call {\it Wc-manifolds}. We restrict our discussion to the Morse-Bott category since this is the framework in which we will work throughout this paper. 
 
 \begin{definition}\label{def:Wc-man}
 A {\it Wc-manifold} is a Liouville manifold with corners $(X, \lambda)$ whose Liouville vector field admits a Morse-Bott Lyapunov function $\phi:X \to \R$, i.e. satisfying (MB1) and (MB2) in the interior of $X$, and such that on the collars $N \times T^*\II^k$ of (L3) $\phi$  has the   form $\phi_N + \sum_{i=1}^k u_i^2$ for $\phi_N$ a Morse-Bott Lyapunov function on $(N, \lambda|_N)$ and $u_1, \ldots , u_k$ the momentum coordinates on $T^*\II^k$. 
 \end{definition}
 
  \begin{figure}[h]
\includegraphics[scale=0.55]{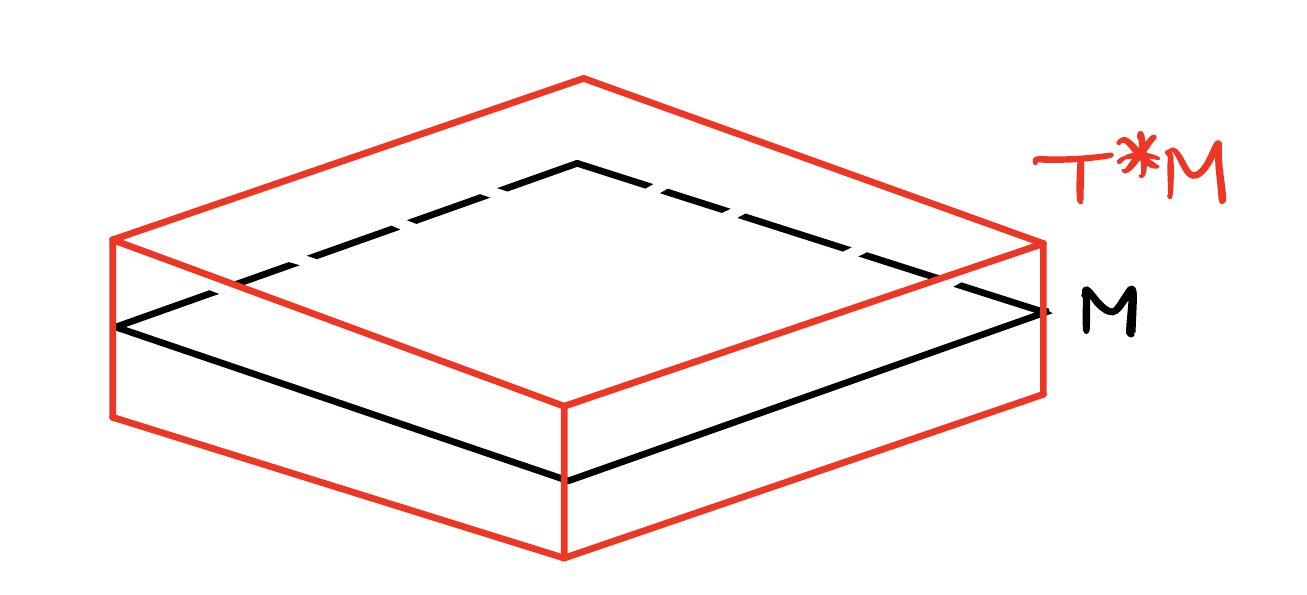}
\caption{The prototypical example of a Wc-manifold is the cotangent bundle of a manifold $M$ with nonempty boundary, or even corners of any order.}
\label{WCMfld}
\end{figure}
 

Defining Liouville domains for Wc-manifolds will be called defining Wc-domains, and the equivalence class of such will be called a Wc-germ. Notions of Wc-isomorphism, embedding and homotopy are the same as the corresponding notions for Liouville manifolds with corners, but demanding that all objects in sight are of Weinstein type.
  
\begin{example}
 Let $M$ be a smooth $n$-dimensional manifold with corners. Then $T^*M$ is a Liouville manifold with corners, where the structure (L3) is induced by a corner structure on $M$.
 We call   the  germ $\sT^*M$ of $T^*M$ 
 a  {\em cotangent   block}.     The corresponding Lagrangian distribution tangent to cotangent  fibers    will be  denoted by  $\nu_M$. We denote by $\lambda_M$ and $Z_M$ the corresponding Liouville form and the Liouville vector field.
The notation $S^*M$ will be used for the positively projectivized cotangent bundle of $M$, which is the ideal boundary of $\sT^*M$. 
 For $Q$ a $k$-face of $M$ we have $P=T^*M|_{Q}\subset\p_k T^*M$ a $k$-face of $T^*M$ and $N=T^*Q$ is the nucleus  of $P$. The collar coordinates near $Q$ yield the decomposition $\Op P= N \times T^*\II^k$. A choice of Riemannian metric on $M$ yields  defining Liouville domain with corners $W=\{ \| p \| \leq \varepsilon \}$ and a Morse-Bott Lyapunov function $\phi=\|p\|^2$.
   \end{example}
 
 The rest of this section can be equally discussed in the Liouville and Weinstein categories. We restrict the discussion only to the Weinstein case as this is the one needed for our further applications.

\subsection{Wc-hypersurfaces}\label{sec:W-hyp}

\subsubsection{Weinstein hypersurfaces}

We now turn our attention to Weinstein hypersurfaces with corners, or Wc-hypersurfaces. They   were introduced  (without corners) in \cite{E18}. These are special cases of the Liouville hypersurfaces introduced by Avdek in \cite{A12}. This notion is   similar  to the ``stops" of Sylvan, \cite{Sy16} and  the Liouville sectors of Ganatra-Pardon-Shende, \cite{GPS17}. Related constructions are also considered in    Ekholm-Lekili's paper \cite{EL17}. We first recall the definition without corners.

 \begin{definition} A  {Weinstein hypersurface} in a Weinstein manifold  $(X^{2n},\lambda)$ is a Weinstein embedding $(A^{2n-2},\lambda'=\lambda|_A) \hookrightarrow
  (X\setminus\Skel(X,\lambda),\lambda)$ such that $\pi_\infty|_{A}:A\to \p_\infty X$ is an embedding.  
 \end{definition}
 
Note that any Weinstein hypersurface is contained in a boundary $\p W$ of some defining domain $W$. Note also that the Reeb vector field of the contact form  $\lambda|_{\p W}$ is  transverse to $A$. Note also that the boundary $\p A$ is a codimension 2 contact submanifold of $(\p W,\lambda|_{\p W})$.  By the assumption that $(A, \lambda|_A)$ is Weinstein, the skeleton $\Skel(A,\lambda|_A)$ can be decomposed as a union of isotropic submanifolds, which in the top dimension $n-1$ are Legendrian for the contact structure $\xi=\ker(\lambda|_{\p W})$. 

 \begin{lemma}  The skeleton   $\Skel(A, \lambda|_{A})$ depends only on the projection $\pi_\infty(A)\subset \p_\infty X$. \end{lemma}
 \begin{proof} Indeed, the Liouville fields corresponding to the Liouville  structures $\lambda_{A}$ and $f \lambda|_{A} $ for   a positive function  $f$  are proportional. In fact, as computed in Lemma 12.1 of \cite{CE12} the form $d(f\lambda)|_A$ is symplectic if and only if $k:=\inf(f+df(Z))>0$, where $Z$ is the expanding field for $\lambda|_A$, and in that case the restriction of the form $f \lambda$ to $A$ is automatically Liouville, with the expanding vector field for $f\lambda|_A$ equal to $\frac1kZ$. Therefore $\Skel(A,\lambda|_{A})$ only depends on the contact structure $\xi_\infty$.   \end{proof} 
 
Note that the space of  functions  $f>0$ for which $f\lambda$ is  Liouville (and hence in the considered case Weinstein) is convex, therefore contractible. 

\begin{definition}
We say that two Weinstein hypersurfaces $A_0,A_1$ in a Weinstein manifold $(X, \lambda)$ are related by a {\it translation} if $A_1=Z^f (A_0)$ for some function $f:A_0 \to \R$. \end{definition}

Note that given a defining domain $W \subset X$, any Weinstein hypersurface $A$ of $X$ can be translated to a Weinstein hypersurface contained in $W$. However, it may not be possible to translate it to a Weinstein hypersurface contained in $\p W$. Indeed, this is equivalent to the condition $\text{inf}(f+df(Z) )>0$ for $f$ the conformal factor between the contact forms corresponding to $A$ and $\p W$.

\begin{definition} A Weinstein hypersurface germ $\sA \subset X \setminus \Skel(X, \lambda)$ is an equivalence class of Weinstein hypersurfaces $(A, \lambda'= \lambda|_A) \subset (X \setminus \Skel(X, \lambda) , \lambda)$ for $(A, \lambda')$ defining domains for a fixed Liouville manifold, where two embeddings are equivalent if they agree on a smaller defining domain, possibly after translation. \end{definition}

One can think of $\sA$ as living in the ideal boundary $\p_\infty X$, or one can think of $\sA$ as the germ of a Weinstein embedding $\sA \to \sX \setminus \Skel(\sX,\lambda)$. The two viewpoints are equivalent and we will go back and forth between them.

  \begin{example} An important example of a Weinstein hypersurface is the {\em ribbon} of  a Legendrian submanifold.  Let $\Lambda$ be a Legendrian submanifold in the contact boundary of a Liouville domain $(W, \lambda)$. Then it admits a Darboux neighborhood $U(\Lambda) \subset \p W$ contactomorphic to  $\{ \|p\|^2+ z^2 \leq\eps^2
\} \subset J^1(\Lambda)$ for some Riemannian metric on $\Lambda$. The strip $\Sigma(\Lambda):=U(\Lambda)\cap\{z=0\}$ is a Weinstein hypersurface symplectomorphic
  to the cotangent disk bundle of $\Lambda$, whose skeleton is precisely $\Lambda$. The space of all ribbon extensions of a  given Legendrian $\Lambda$ is contractible. An important example to keep in mind is the case where$ \sX=\sT^*M$ for $M$ a smooth manifold, $W=\{ \|p\|\leq 1\}$ with respect to any Riemannian metric on $M$ and $\Lambda$ is the unit conormal to a co-oriented hypersurface $H \subset M$.
     \end{example}
 
        \begin{figure}[h]
\includegraphics[scale=0.5]{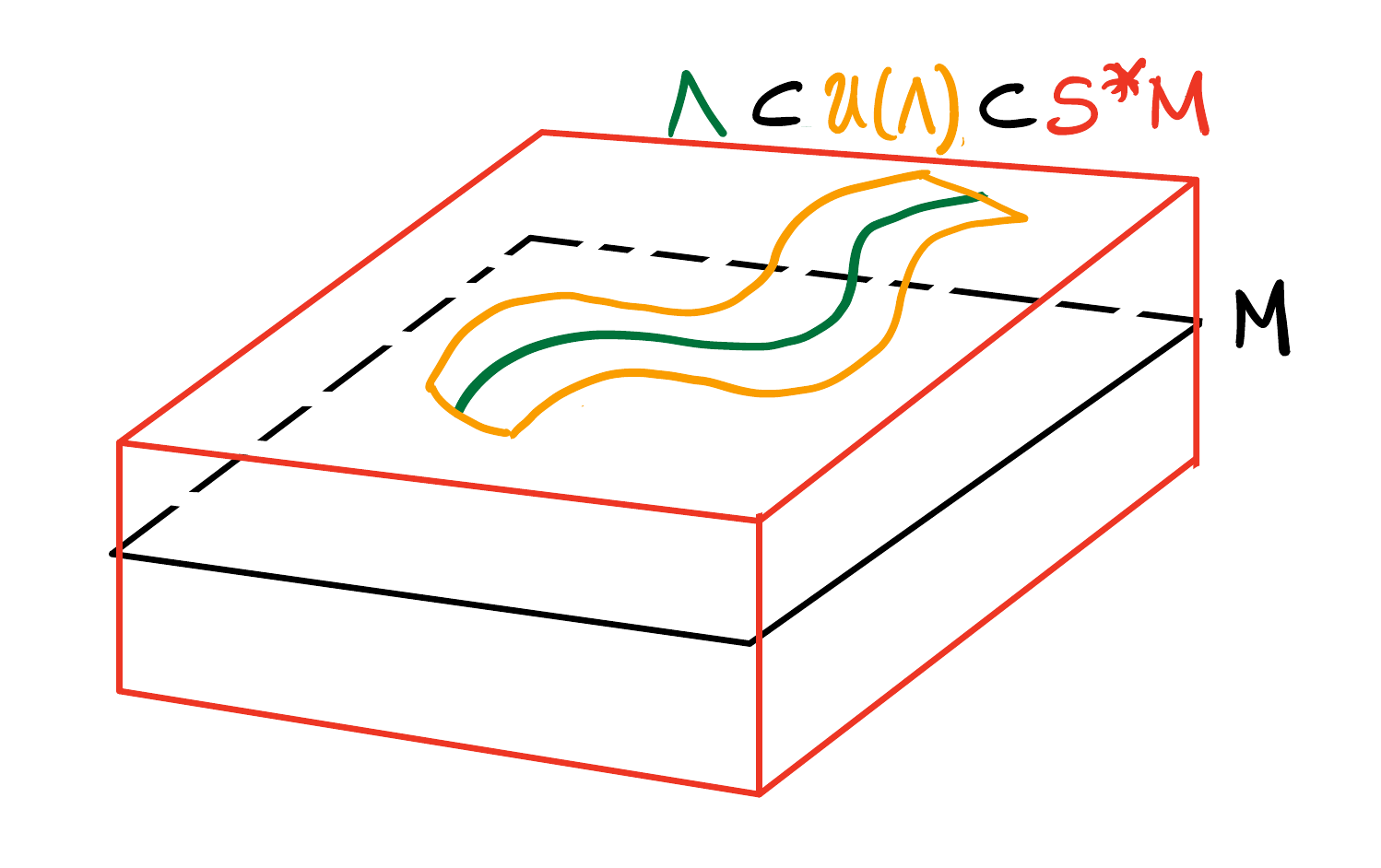}
\caption{A Legendrian $ \Lambda \subset S^*M$ and its ribbon $U(\Lambda)$.}
\label{fig:ribbon}
\end{figure}
 
\subsubsection{Weinstein pairs}
 
\begin{definition} A  {\it Weinstein pair} $(X,A,\lambda)$ consists of a Weinstein manifold  $(X, \lambda)$ together a Weinstein hypersurface $ A \subset X\setminus\Skel(X,\lambda)$.  
\end{definition}

We will also call $(X,\sA,\lambda)$ a Weinstein pair when $\sA$ is the germ of a Weinstein hypersurface $ \sA \subset X\setminus\Skel(X,\lambda)$.

 \begin{definition}
Let $(X, \lambda)$ be a Liouville manifold  and $Y \subset X\setminus\Skel(X,\lambda)$ any subset. The {\em Liouville cone} on $Y$ is the saturation of $A$ by the Liouville trajectories:
\[ \Cone(Y, \lambda)= \bigcup_{t \in \R } Z^{-t}(Y) \subset X.\]
The {\em positive Liouville cone} on $A$ is the saturation of $A$ by the backwards Liouville trajectories:
\[ \Cone^+(Y, \lambda)= \bigcup_{t \geq 0 } Z^{-t}(Y) \subset X.\]
\end{definition}

Note that if $Y$ is compact, then $\text{Cone}^+(Y,\lambda)$ has compact closure while $\text{Cone}(Y, \lambda)$ does not. The skeleton of a Weinstein pair $(X,A,\lambda)$ is defined as the compact subset:
$$\Skel(X, A, \lambda) = \Skel(X, \lambda) \cup \Cone^+\left(\Skel(A,\lambda|_A), \lambda \right).$$ 

Note that $\Skel(X,A,\lambda) \subset W$ for any defining domain $W$ such that $A \subset \p W$, moreover $\p W \cap \Skel(X,A, \lambda) = \Skel(A, \lambda|_A)$. The notion of a pair at the level of germs is defined as follows.

\begin{definition} A Weinstein (germ) pair $(\sX,\sA,\lambda)$ consists of a Weinstein germ $(\sX,\lambda)$ together with a Weinstein hypersurface germ $\sA \subset \sX \setminus \Skel(\sX,\lambda)$. \end{definition}

Note that for any representative $(W,\lambda)$ of $(\sX,\lambda)$ we may take a representative $A$ for $\sA$ contained in $W$. The skeleton of a Weinstein pair $(\sX,\sA,\lambda)$ is defined as the non-compact subset: $$\Skel(X, \sA, \lambda) = \Skel(X, \lambda) \cup \Cone\left(\Skel(\sA,\lambda|_\sA), \lambda \right).$$

Equivalently in terms of germs, we can think of $\Skel(X,\sA,\lambda)$ as the equivalence class of the compact subsets $\Skel(X,A,\lambda) \subset W$ for all defining domains $(W,\lambda|_W)$ and representatives $A$ for $\sA$ such that $A \subset W$. 
 
 \subsubsection{Wc-hypersurfaces}
 
 Finally, we extend the above to include boundary and corners.
 
 \begin{definition} A  {Wc-hypersurface} in a Wc-manifold $(X^{2n},\lambda)$ is an equivalence class of embeddings of Wc-germs $(\sA^{2n-2},\lambda'=\lambda|_{\sA})\hookrightarrow
  (X\setminus\Skel(X,\lambda),\lambda)$ such that $\pi_\infty|_{\sA};\sA\to \p X_\infty$ is an embedding, with the equivalence given by translation. 
 \end{definition}
 
All the previous discussion carried over word by word to the Wc-setting. The notion of a Weinstein pair in the Wc-category is called a Wc-pair.  

\begin{definition} A Wc-pair $(\sX,\sA,\lambda)$ consists of a Wc-germ $(\sX,\lambda)$ together with a Wc-hypersurface germ $\sA \subset \sX \setminus \Skel(\sX,\lambda)$. \end{definition}

\subsection{Conversion between Wc-hypersurfaces and face nuclei}\label{sec:conversion}

\subsubsection{Nucleus-to-hypersurface conversion}

We start with a Wc-manifold $(X, \lambda)$. Let $P$ be a boundary face of $X$ and $N \subset P$ its nucleus. Let $U$ be a collar neighborhood of $P$ in $X$ as in the definition of a Wc-manifold. So $U$ is of the form $N \times T^*\cI$, with $\lambda = \lambda|_N + udt$ where $t$ is the defining coordinate for $P$. Attach to $N \times T^*\cI$ the half-infinite collar $N \times T^*(-\infty, 0]$ along their common intersection and on the union consider the Weinstein structure given by the direct sum of $\lambda|_N$ and the Weinstein normal form for Morse-Bott boundary of index $k=0$, which agrees with $\lambda_N + udt$ on $N \times T^*\cI$. The existence of a Morse-Bott Lyapnuov function is immediate from the local model. Hence the result is a Wc-manifold $(X^{N}, \lambda^N)$. Let $\sN$ be the germ of $N$. For $\eps>0$ note that $\sN \times \{t=-\eps \} \subset \sN \times T^*(-\infty, 0]$ is a Wc-hypersurface for $\lambda^N$. Any two differ by a translation and all are canonically identified with $\sN$. We will denote this Wc-hypersurface by $\sA^N \subset X^N \setminus \Skel(X^N, \lambda^N)$. 

Note that by construction we have $\Skel(X^N,\sA^N,\lambda^{N})=\Skel(X,\lambda)$.
The operation $(X,N) \mapsto (X^{N}, \sA^N)$ only depends on the contractible choice of a collar neighborhood and the contractible choice of the Weinstein normal form for Morse-Bott boundary. Hence it produces a Wc-germ $(\sX^N, \lambda^N)$ which is well defined up to Wc-homotopy.

\begin{definition}
The Wc-pair $(\sX^{N}, \sA^{N}, \lambda^{N})$ is called the {\em nucleus-to-hypersurface} conversion of the Wc-germ $(\sX,\lambda)$ and the boundary nucleus $N$.
\end{definition}
  \begin{figure}[h]
\includegraphics[scale=0.55]{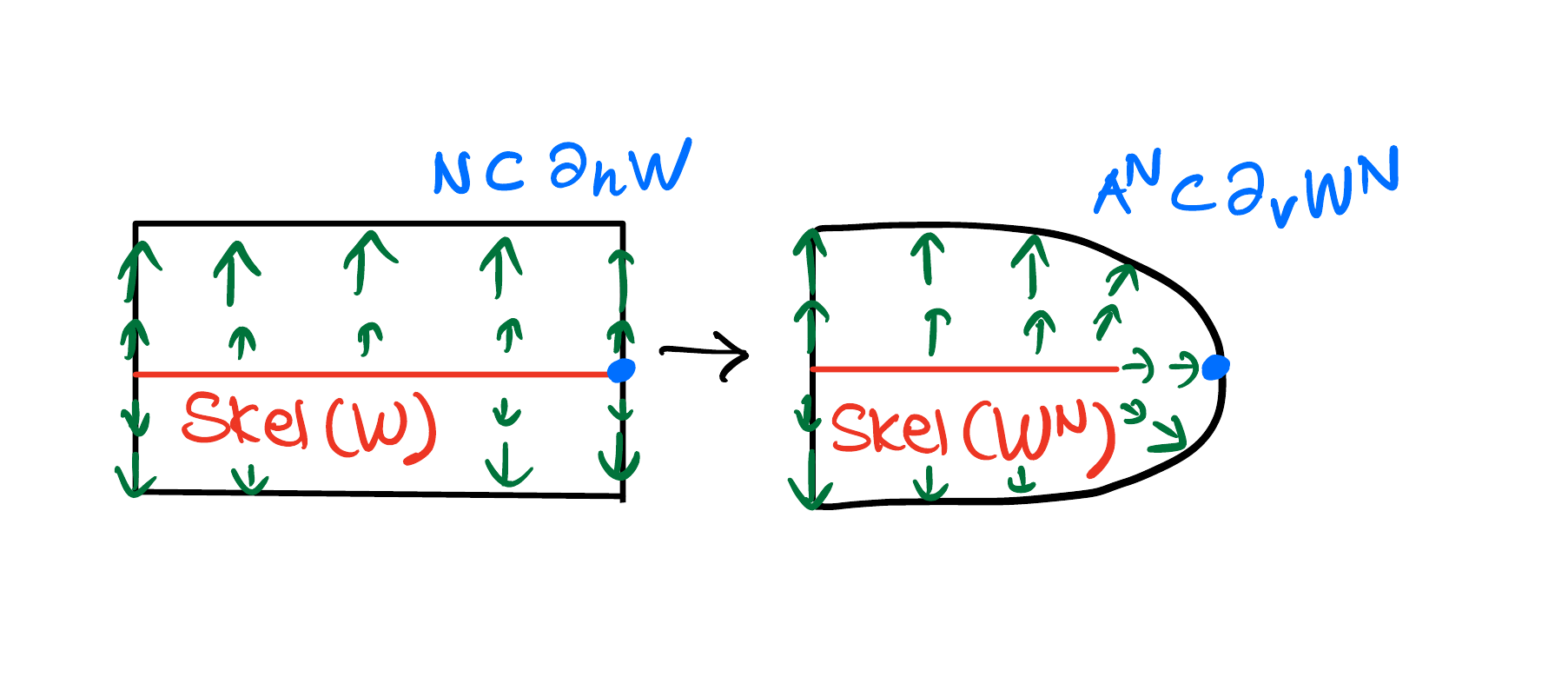}
\caption{Converting a boundary nucleus to a Wc-hypersurface.}
\label{fig:hyptonuc}
\end{figure}

\subsubsection{Underlying Weinstein manifold of a Wc-manifold}

Let $X$ be a Wc-manifold and let $P$ be a boundary face. The resulting of converting the nucleus $N$ of $P$ to a Wc-hypersurface is a Wc-manifold which has one less boundary face than $X$. We can continue this process inductively over all faces. After forgetting the resulting Wc-hypersurfaces, we end up with a Weinstein manifold without boundary which we denote $\wh X$. It is clear that the Weinstein homotopy class of the Weinstein manifold $\wh X$ does not depend on the order in which the boundary faces are chosen. If one wishes to, one can further deform the Weinstein structure on $\wh X$, which at this point is generalized Morse-Bott, so that it becomes Morse-Weinstein.

\begin{definition} The Weinstein manifold $\wh X$ obtained from a Wc-manifold $X$ by converting all of its boundary nuclei to Wc-hypersurfaces is called the {\it underlying Weinstein manifold}. \end{definition}

\begin{example} the underlying Weinstein manifold of the Wc-manifold $(T^*[0,1]^n,pdq)$ is Weinstein homotopic to $(\R^{2n},pdq-qdp)$.  \end{example}

\begin{remark}
A Wc-manifold may be thought of as a Weinstein manifold with `stops'. From this viewpoint, taking the underlying Weinstein manifold of a Wc-manifold is simply removing the stops.
\end{remark}

\subsubsection{Hypersurface-to-nucleus conversion}

 The converse operation to the nucleus-to-hypersurface conversion is the {\em hypersurface-to-nucleus} conversion, which we discuss next. Here and below we denote $U_{a,b}=\{0<u \leq a , \, -b < t < b \} \subset T^*\R$. 
\begin{lemma}\label{lem:embedding extension}
The inclusion $(A,\lambda'=\lambda|_{A})\hookrightarrow (X,\lambda)$ of a Wc-hypersurface extends, uniquely up to contractible choice, to a Liouville embedding
$(A\times U_{1,\eps},\lambda' + udt)\to (X,\lambda)$, where $A$ is identified with $A\times (u=1,t=0)\subset A\times U_{1,\eps}$.
\end{lemma}

\begin{proof} Let $W \subset X$ be a defining domain such that $A \subset \p W$ and let $A \times (-\eps,\eps) \subset \p W$ be its contact surrounding, as in \cite{E18}, so that the restriction of $\lambda$ to this neighborhood is of the form $\pi^* \lambda|_{A} + dt$, where $\pi: A \times (-\eps, \eps) \to A$ is the projection and $t \in (-\eps,\eps)$. For every $x \in X \setminus \Skel(X,\lambda)$, let $x^s$ denote the backwards flow of $Z_{X}$ starting at $x$, where  $s \in (-\infty, 0]$. For every $y \in A \times (-\eps, \eps)$, let $y^s \in U_{1,\eps}$ denote the backwards flow of $Z_{A} + u \frac{\partial}{\partial u}$ starting at $y$, where $s \in (-\infty, 0]$. We extend the contact embedding $A \times (-\eps, \eps) \to \p W$ which sends $y \mapsto x$ to a Liouville embedding $A \times U_{1,\eps} \to W$ by sending $y^s \mapsto x^s$, so that by construction $Z_{A} + u \frac{\partial }{\partial u}$ is identified with $Z_{X}$. It is clear that the embedding is determined by the contact surrounding and its parametrization, which are unique up to contractible choice. \end{proof}

We now explain the {\em hypersurface-to-nucleus} conversion. Let $(X,\lambda)$ be a  Wc-manifold and let $\sA \subset X\setminus \Skel(X,\lambda)$ be a  Wc-hypersurface. 
Fix defining Wc-domains $A$ and $W$ for $\sA$ and $X$ respectively so that we have an embedding $A \to \p W$ which extends to an embedding $A \times U_{1,\eps} \to W$ as in Lemma \ref{lem:embedding extension}. Consider the exact symplectic manifold which is the result of gluing $W$ and $A \times U_{3,\eps}$ along $A \times U_{1,\eps}$. Consider the Weinstein structure on $U_{3,\eps}$, thought of as the $\eps$-neighborhood of the zero section in $T^*(0,3)$, which is the restriction of the Weinstein normal form for Morse-Bott boundary of index $k=1$ translated two units in the $u$ direction, so that the critical set is $\{t=0, u \geq 2\}$. The restriction of this structure to $U_{1,\eps}$ is $udt$, hence we obtain a patched up structure on the union of $W$ and $A \times U_{3,\eps}$. Moreover, by inspection the new Liouville field points outwards at the boundary. So we may replace $ U_{3,\eps}$ with a smoothing $ U_{3,\mu}=\{-\eps<t<\eps, \, \, 0<u<\text{min}(3,\mu(|t|)) \}$ for $\mu:(0,\eps) \to [1,\infty)$ a function satifying
\begin{itemize}
\item[(i)] $\mu(x) \to \infty $ as $x \to 0$
\item[(ii)] $\mu=1$ on $(\eps/2,\eps)$
\item[(iii)] $\mu'<0$ on $(0,\eps/2)$
\end{itemize}
and the Liouville field is still outwards pointing for the smoothed boundary. Thus we obtain a defining domain for a Wc-manifold $(X^{A},\lambda^{A})$. Indeed the only think that needs to be checked is the existence of a Morse-Bott Lyapunov function, which is clear from the local models since without loss of generality we can start with a defining domain $W$ for $X$ such that $\p W$ is a regular level set of the Morse-Bott Lypanuov function $\phi$ for $X$. Note that $A \times \{ t= 0, u=3 \}$ is the nucleus of the new boundary face of $X^A$.

By construction we have $\Skel(X^{A},\lambda^{A})=\Skel(X,A,\lambda)$, or more precisely $$ \Skel(X^A, \lambda^A)=\Skel(X,A,\lambda) \cup (\Skel(A, \lambda|_A) \times [1,2]).$$ 
In other words, the backwards $\lambda$-Liouville cone on $A$ has a trivial collar attached to it, which does not change the topology of $\Skel(X,A, \lambda)$. Note that the space of possible smoothings $\mu$ is convex, hence contractible. The choice of Morse-Bott Weinstein normal form is also contractible. Similarly the choice of collars used in the construction is contractible. Hence we get a Wc-germ $(\sX^A,\lambda^A)$ which is well defined up to Wc-homotopy. Finally, note that different representatives $A$ for $\sA$ will also result in homotopic Wc-manifolds $(X^A,\lambda^A)$. Hence the conversion operation only depends on the Wc-hypersurface germ $\sA$ up to Wc-homotopy.

\begin{definition} The Wc-germ $(\sX^{\sA}, \lambda^{\sA})$ is called the {\em hypersurface-to-nucleus} conversion of the Wc-germ $X$ and the Wc-hypersurface $\sA$.
\end{definition}

   \begin{figure}[h]
\includegraphics[scale=0.55]{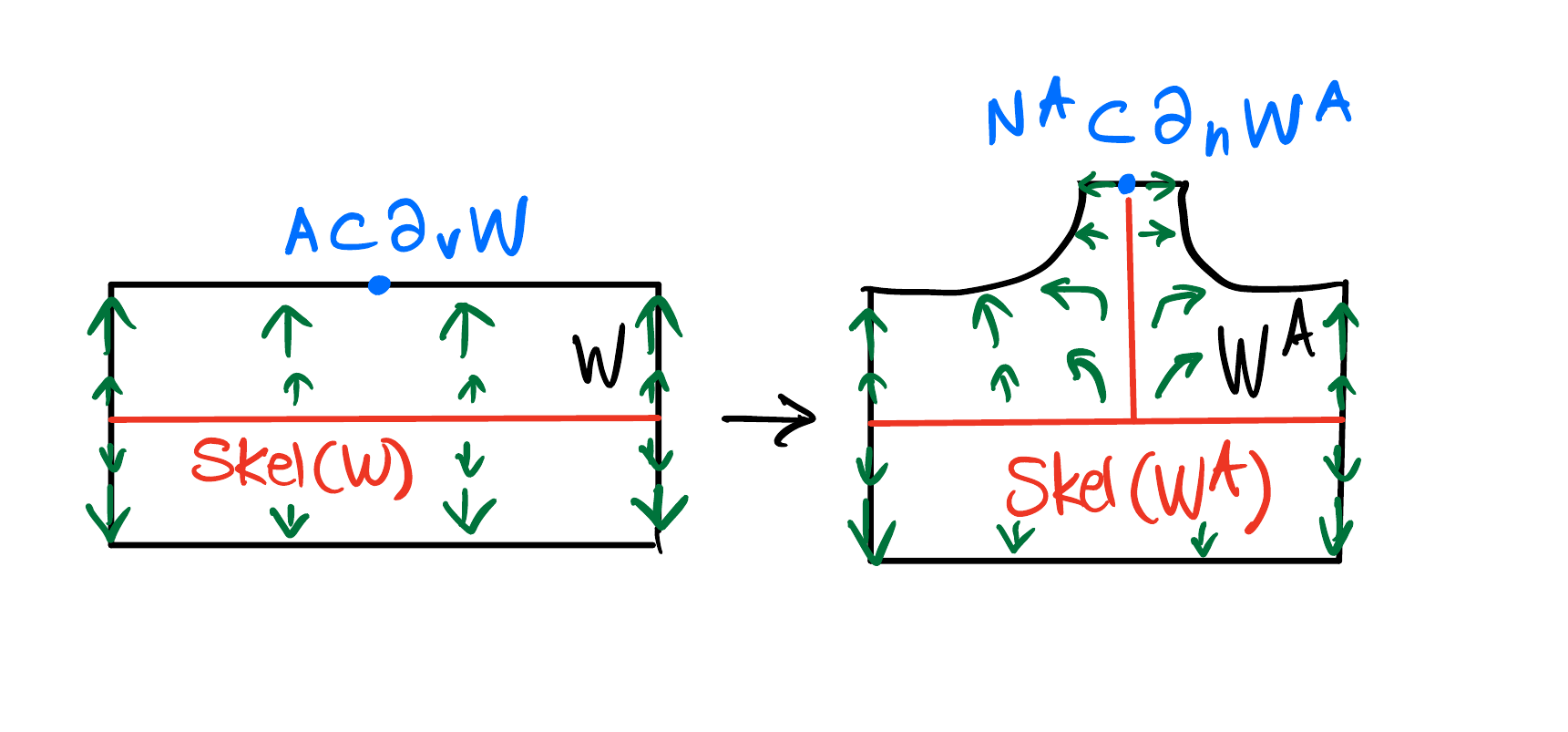}
\caption{Converting a boundary nucleus to a Wc-hypersurface.}
\label{fig:nuctohyp}
\end{figure}

\subsubsection{Cancellation of conversions}

The hypersurface-to-nucleus conversion is inverse to the nucleus-to-hypersurface conversion at the level of Weinstein homotopy classes, as follows from a parametric Smale cancellation of Morse-Bott index $0$ and $1$ critical points, the existence of which is immediate from the local models. This implies the following obvious but important lemma.

\begin{lemma}
Let $(X^\sA,\lambda^\sA)$ be the result of {\it hypersurface-to-nucleus} conversion of a Wc-hypersurface $\sA \subset X \setminus \Skel(X,\lambda)$. Then the underlying Weinstein manifolds of $(X,\lambda)$ and $(X^\sA,\lambda^\sA)$ are Wc-homotopic.
\end{lemma}
 


\subsection{Gluing and splitting}\label{sec:Lbc-gluing}

\subsubsection{Horizontal gluing}

We will now describe the basic gluing operation for Wc-manifolds. Given two Wc-manifolds $(X_0,\lambda_0)$ and $(X_1,\lambda_1)$, we define their
{\em horizontal gluing} as follows. Pick
  some {\em non-adjacent} boundary faces $P_1,\dots, P_l$ of $X_1$ and 
   $Q_1,\dots, Q_l$ of $X_0$, and  for each $j$ suppose we are given a Liouville isomorphism $\phi_j: N_j \to M_j$, where   $N_j$ and $M_j$ are the nuclei or $P_j$ and $Q_j$ respectively. We have collar neighborhoods $N_j \times T^*\cI \subset X_1$ and $M_j \times T^*\cI \subset X_2$ and we extend $\phi_j$ to $\Phi_j:  P_j\to Q_j$ as $ \Phi_j = \phi_j \times \text{id} _\R: N_j \times \R \to M_j \times \R$. 
   \begin{definition} The horizontal gluing of $(\sX_1,\lambda_1)$ and $(\sX_2,\lambda_2)$ along $\{\phi_j\}_j$ is defined to be the Wc-germ of the Wc-manifold $X_1 \cup X_2$ where we identify $P_j\sim Q_j$ via $\Phi_j$. 
   \end{definition}
   
   \begin{example}
   Suppose $M, N$ are two manifolds with boundary and we have a diffeomorphism $\psi: \p N \to \p M$. Then $\psi$ lifts to a Liouville isomorphism $\phi$ between $T^*(\p N)$ and $T^*( \p M)$ and we can perform the horizontal gluing $T^*N \cup_\phi T^*M$ which is the same as $T^*(N \cup_\psi M)$ for $N \cup_\psi M$ the smooth gluing of $N$ and $M$ given by $\psi$.
   
   \end{example}
      \begin{figure}[h]
\includegraphics[scale=0.4]{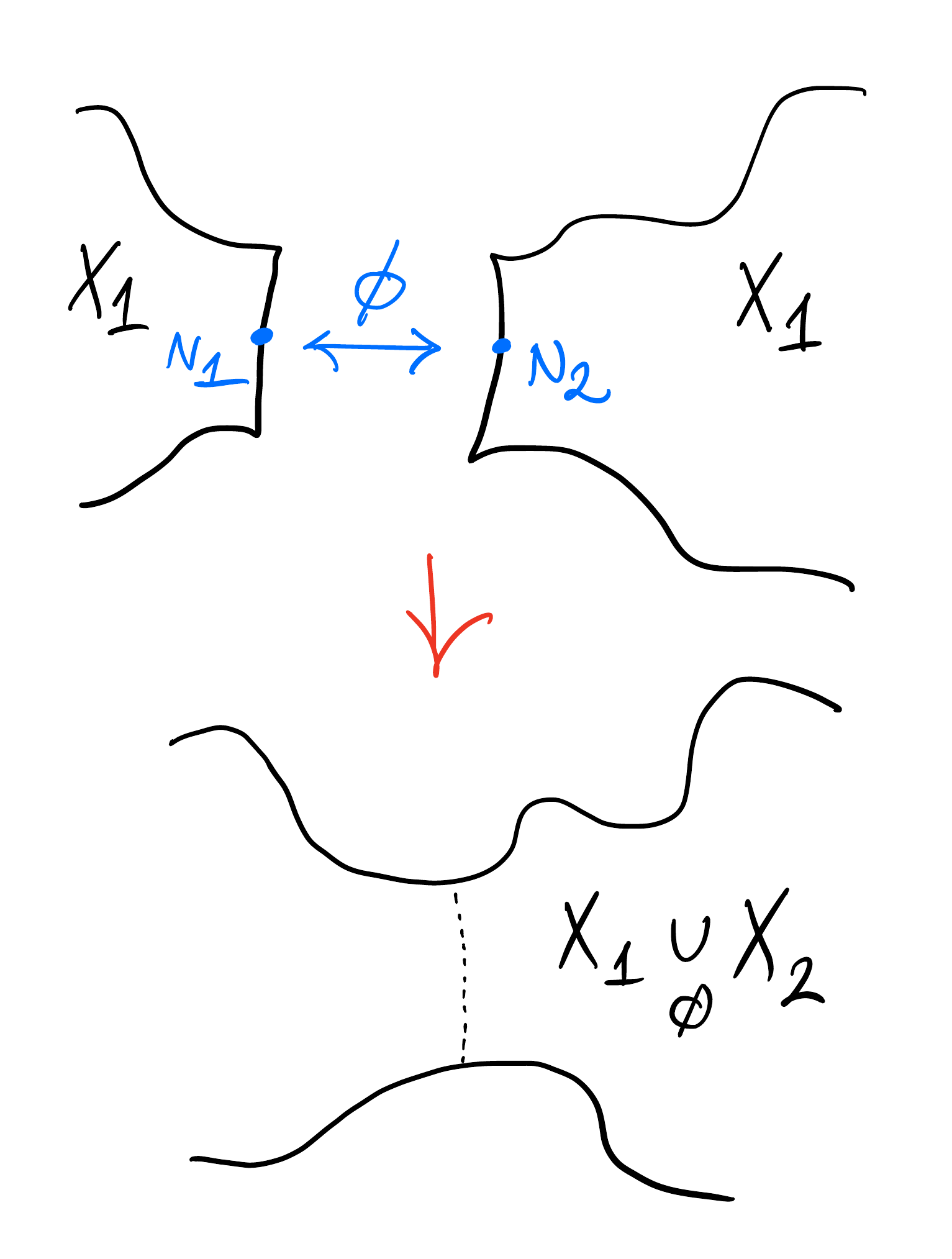}
\caption{Horizontal gluing of two Wc-manifolds.}
\label{fig:Horgluing}
\end{figure}
   
Note that the skeleton of the horizontal gluing of $X_1$ and $X_2$ is equal to  the union of the skeleta of $X_1$ and $X_2$. 

\subsubsection{Vertical gluing} We can also glue two Wc-manifolds $X_1$ and $X_2$ in a different way. 

\begin{definition} Suppose we have a Wc-hypersurface $\sA$  in $\sX_1$ and a Liouville isomorphism $\phi: \sA \to \sN$ for $\sN$ the germ of the nucleus $N$ of a boundary face of $X_2$. The {\em vertical gluing} of $\sX_1$ and $\sX_2$ is defined to be the composition of two steps: a conversion a Wc-hypersurface into a boundary hypersurface, see Section \ref{sec:conversion} above, followed by a horizontal gluing.
\end{definition}

       \begin{figure}[h]
\includegraphics[scale=0.45]{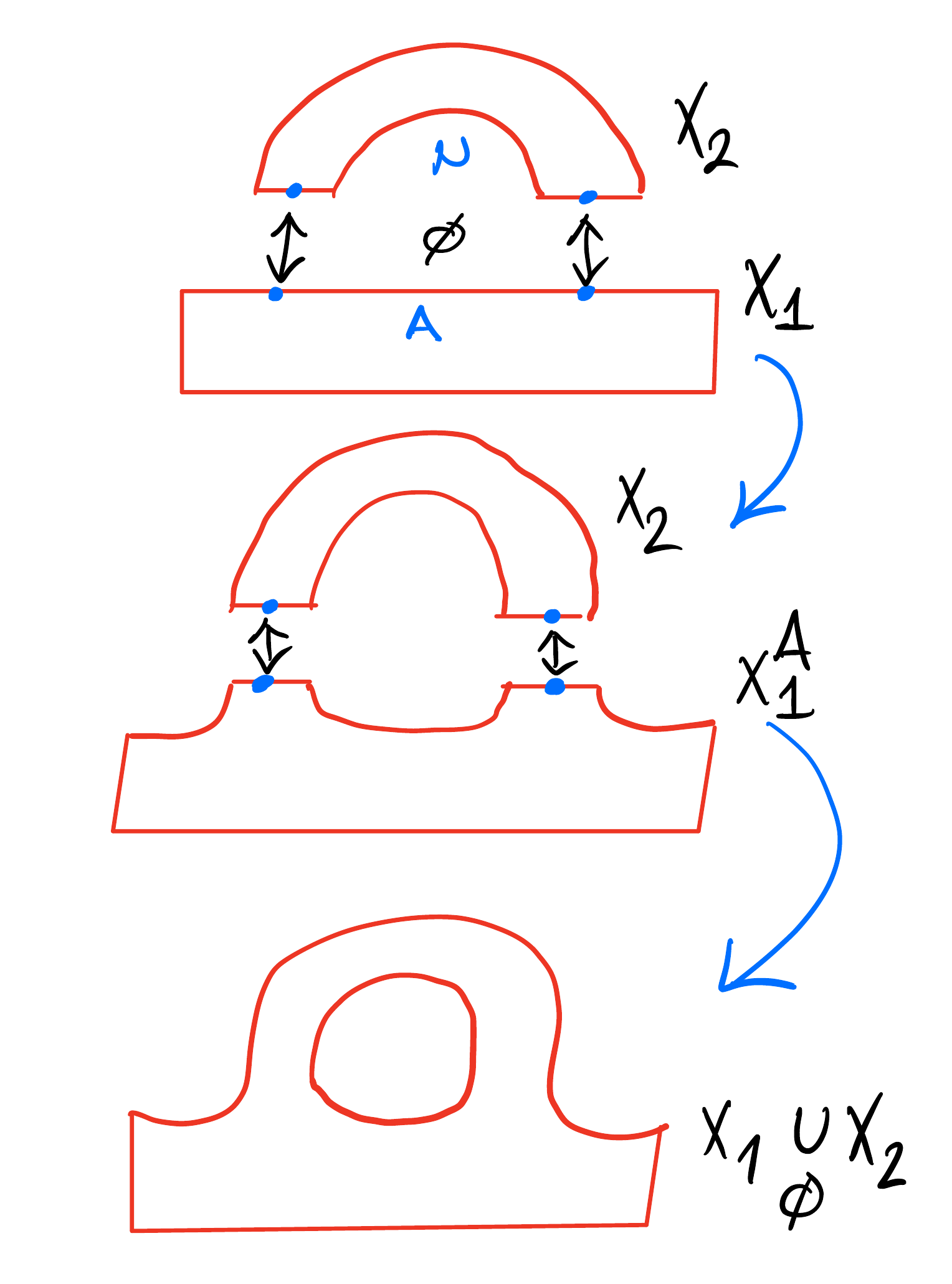}
\caption{Vertical gluing is the composition of two steps: a Wc-hypersurface to boundary nucleus conversion and a horizontal gluing.}
\label{Vertical gluing}
\end{figure}

 Indeed, if we first convert $\sA$ into a boundary nucleus, we can then apply the horizontal gluing determined by $\phi$ as descibed above. The resulting Wc-germ is well defined up to Wc-isomorphism. Note that the skeleton of the vertical gluing is equal to the union of the skeleton of the pair $(\sX_1,\sA)$ and the skeleton of $\sX_2$. We record for future reference the following slight strengthening of this observation, where we denote by $\lambda^P$ the Liouville form $\lambda - d(ut)$ for $t$ the defining coordinate of a boundary face $P$. 
 
 \begin{lemma}\label{lem:maps}
Let $(\sX_i, \lambda_i)$, $i=1,2$ be two Wc-germs, $\sA \subset \sX_2 \setminus \Skel(\sX_2,\lambda_2)$ a Wc-hypersurface, $N$ the nucleus of a boundary face $P$ of $X_1$ and $\sN$ its germ, $\phi: \sA \to \sN$ a Liouville isomorphism and $(\sY, \lambda)$ the result of vertically gluing $\sX_1$ to $\sX_2$ along $\phi$. Then there exists a symplectic embedding of germs $\Psi: \text{int} \sX_1 \to \sY$ such that the following properties hold, where we denote by $\iota: \text{int}  \sX_2 \to \sY$ the restriction of the inclusion to the interior of $\sX_2$.
\begin{itemize}
\item[(i)] $\iota(\Skel(\sX_2 ,\lambda_2)) \cup \Psi(\Skel(\sX_1,\lambda_1) \cap \text{int} \sX_1)=\Skel(\sY,\lambda)$
\item[(ii)] $\iota(\Skel(\sX_2, \lambda_2))\cap \Psi( \text{int} \sX_1) = \varnothing$.
\item[(iii)] $ \Psi^{-1}( \iota(\text{int}  \sX_2 ) ) = \Op \sN \setminus \p \sX_1 \subset \sX_1$ and $(\Psi)^*( \iota)_* \lambda_2=\lambda_1^P$.
\end{itemize}
\end{lemma}

\begin{proof}
We recall from the gluing construction that we attached $\sN \times (-\eps,\eps)$ to the top of a cylinder $\sA \times U_{3,\eps} \subset Y$. Since a neighborhood $\Op N$ of $N$ in $X_1$ is symplectomorphic to $N \times T_\eps^*[0, \delta)$, we can map $\sN \times T^*_\eps(0,\delta)$ symplectomorphically to $\sA \times U_{3,\eps} \cup \sN \times T^*_\eps[0,\delta) $, relative to $\sN \times T^*_\eps (\delta/2,\delta)$ say, so that the symplectomorphism extends to the rest of $\text{int} \sX_1$ as the identity. Moreover, the symplectomorphism can be obtained from a reparametrization of the $t$ coordinate, hence the Liouville forms pull back as desired.
\end{proof}

\subsubsection{Splitting}
 
The converse operation to gluing is splitting.        
\begin{definition} Given a Wc-manifold $(X,\lambda)$,  a regularly embedded  hypersurface 
$(H,\p H)\subset (X,\p X)$ with corners   called   {\em splitting} for  $X$ if it  satisfies the following conditions:
\begin{itemize}
\item[(i)]  $H$  divides $X$  into    two  parts,  $X=X_+\cup X_-$ and $X_-\cap X_+=H$, which are manifolds with corners.
\item[(ii)] the Liouville vector field $Z$ of $X$ is tangent to $H$.
\item[(iii)] there exists a hypersurface with corners $(S,\p S)\subset (H,\p H)$ tangent to $Z$.
\item[(iv)]  $(S,\lambda|_S)$ is a Wc-manifold and there exists a retraction $\pi:H\to S$ such that $\lambda|_H=\pi^*(\lambda|_S)$.
\end{itemize} \end{definition}

\begin{remark}
That $H$ is regularly embedded implies in particular that $\Op H = H \times (-\eps, \eps)$.
\end{remark}

We will refer to $S$ as the  {\em   soul} of the  splitting hypersurface $H$ and   denote it by $\So(H)$.
Note that  $\Skel(S, \lambda|_S)=\Skel(X, \lambda)\cap H$.

\begin{example}
Let $M$ be a manifold with corners and $N$ a regularly embedded codimension $1$ submanifold with corners, so that $\Op N = N \times (-\eps, \eps)$. Assume that $N$ divides $M$ into two parts: $M_+$ and $M_- $, so that $M_+\cup M_-=M$. Then $H=T^*M|_N$ is a splitting hypersurface. Given a decomposition of the tubular neighborhood of $N$, $\Op(N)=N\times(-\eps,\eps)$, we have $S:=T^*N\times 0\subset T^*M=T^*N\times T^*(-\eps,\eps)$ the soul of $H$.
\end{example}

The  splitting hypersurface  $S$ can be used to split $X$ into two Wc-manifolds  in two ways, {\em horizontal} and {\em vertical}. For the {\em horizontal  splitting} we 
 view  $X_\pm$ as Wc-manifolds with an additional component $H$ of $\p_1X_\pm$. The Wc-manifold $S$  serves as the nucleus of the additional boundary component $H$.  The {\em vertical splitting} consists of a horizontal splitting followed by a nucleus-to-hypersurface conversion  operation, transforming  the nucleus $S$ into a Wc-hypersurface.
 
\subsubsection{Partial horizontal gluing}
 
One more basic operation that we will need is the creation of corners and its associated partial horizontal gluing. Let $X$ be a Wc-manifold, $P$ a boundary face of $X$ and $N$ the nucleus of $P$. Let $\Sigma \subset N$ be a splitting hypersurface with soul $T$. Then there is a deformation of the Wc-structure on $X$ which creates an additional corner along $\Sigma$ with its nucleus $T$.

More precisely, take a collar neighborhood $\Op P = N \times T^*\cI$, so that $\lambda = \lambda|_N + udt$ and extend it to $N \times T^*(-\eps,\eps)$ with the same Liouville form. Let $\Sigma \times (-\delta, \delta) \subset N$ be a tubular neighborhood of $\Sigma$ in $N$ and consider a smooth function $\mu:(0,\delta) \to [0,\infty)$ satisfying the following properties:
\begin{itemize}
\item[(i)] $\mu(x) \to \infty$ as $x \to 0$.
\item[(ii)] $\mu(x)=0$ on $[\delta/2,\delta)$
\item[(iii)] $\mu'(x)<0$ on $(0,\delta/2)$.
\end{itemize}

Set $\eta(x)=\min( \mu(x), \epsilon/2)$ and consider the subset $\Omega \subset (\Sigma \times (-\delta,\delta) ) \times T^*(-\eps,\eps)$ consisting of points $(\sigma, s ,t,u)$ such that $\eta(s)+t  \geq 0$. Then the union of $\Omega$ and $X$ along $P$ is a new Wc-manifold which has a new corner along a parallel copy of $\Sigma$.

       \begin{figure}[h]
\includegraphics[scale=0.5]{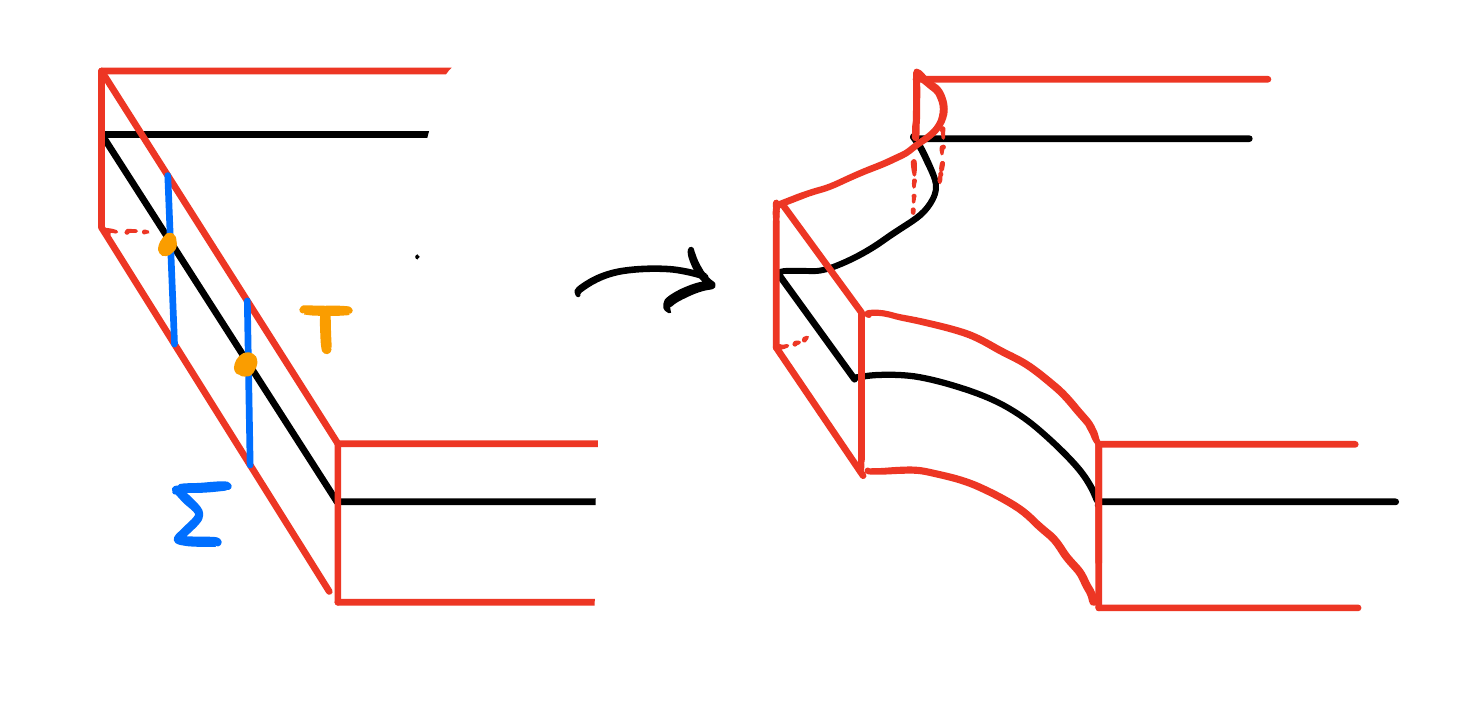}
\caption{A creation of additional corners on $\sT^*M$ for $M$ a closed manifold with corners is induced from a creation of additional corners on $M$.}
\label{creationcorner}
\end{figure}

Next, consider two Wc-manifolds $X_0, X_1$ with boundary faces $P_0,P_1$ and respective nuclei $N_0, N_1$. Let $\Sigma$ be a dividing hypersurface for $N_0$ with its soul $T$, which splits $N_0$  into Wc-manifolds $N_0'$ and $N_0''$, which share a common face $\Sigma$ with the nucleus $T$. Suppose there is a Liouville isomorphism $\phi:N_1\to N_0'$. 
\begin{definition} The partial horizontal gluing of $\sX_1$ to $\sX_0$ using $\phi$ is the Wc-germ obtained from the following two operations:
\begin{itemize}
\item[(a)] Creating a corner along $\Sigma$, i.e. adding $\Sigma$ to $\p_2X_0$, and replacing the face $N_0$ with two  faces $N_0', N_0''$;
\item[(b)] The horizontal gluing $X_0$ and $X_1$ along $\phi$.
\end{itemize}
\end{definition}
 
 As before, the partial horizontal gluing is well defined up to Wc-homotopy.

\subsection{Wc-buildings}
\label{sec:building}

\subsubsection{Iterated gluing}

We now introduce the iterated vertical gluing of Wc-manifolds, which will be a convenient notion for the constructions in this paper. Consider Wc-germs  $(\sX_1,\lambda_1),\dots,(\sX_k,\lambda_k)$. We inductively define a $k$-level {\em Wc-building}, which we will denote by $(  \sX_k\to \sX_{k-1} \to \dots  \to \sX_1)$, as the following presentation of a Wc-germ.
\begin{definition} A $1$-level  Wc-building  is  $(\sX_1,\lambda_1)$ itself.
Suppose that the Wc-building $$( \sY,\mu):= (  \sX_k\to \sX_{k-1}\dots\to \sX_2)$$ is already defined. Consider a nucleus $(\sN,\mu':=\mu|_\sN)$ of one of the boundary faces  $\sP$ of $( \sY,\mu)$ and let $\phi$ be an embedding $\sN$ as  a Wc-hypersurface  in $\sX_1\setminus\Skel(\sX_1,\lambda_1)$. 
The building  $$ ( \sX_k\to \sX_{k-1} \to \cdots \to \sX_2\to \sX_1) := (\sX_k \to \sX_{k-1} \to \cdots \to \sX_2) \to \sX_1 $$ is by definition the result of a vertical gluing of  $( \sY,\mu)$ and  $(\sX_1,\lambda_1)$  along $\phi$. \end{definition}

          \begin{figure}[h]
\includegraphics[scale=0.5]{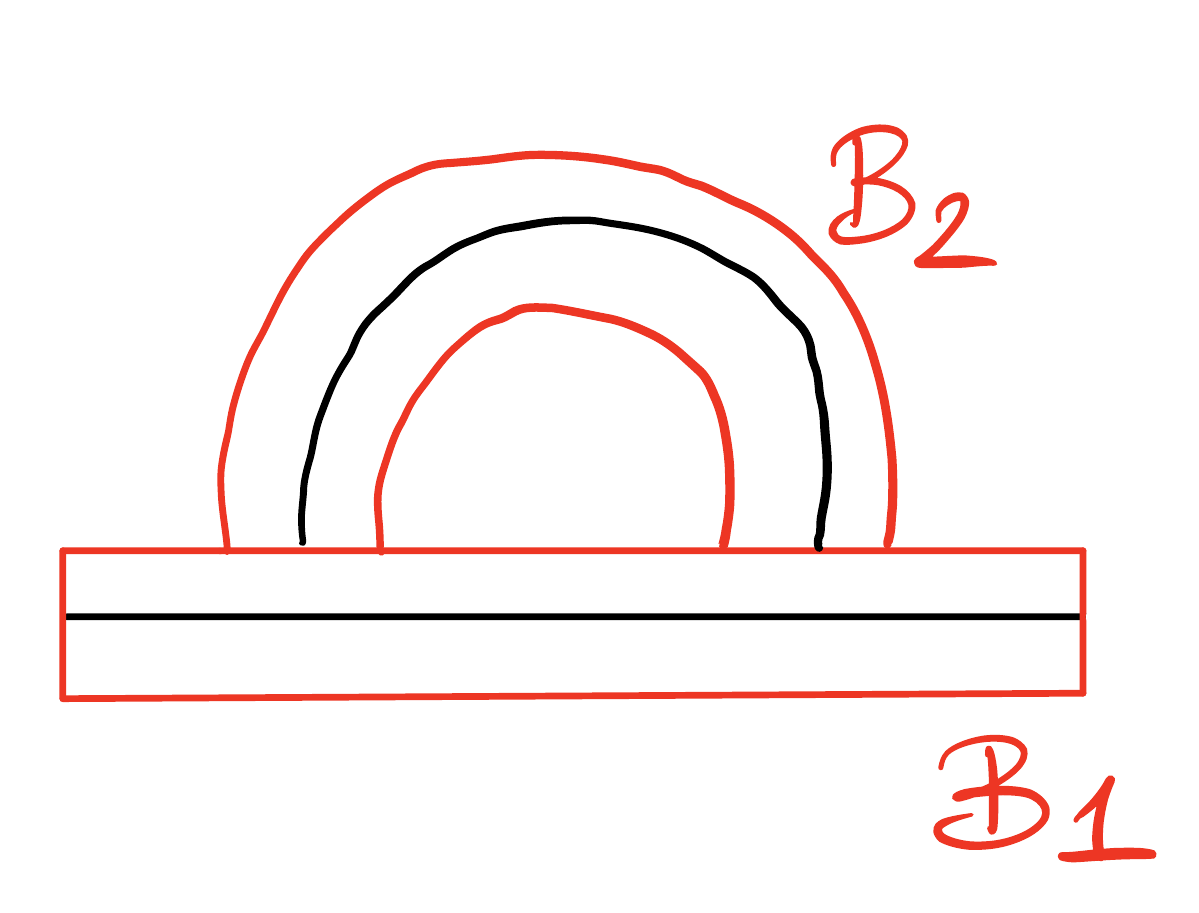}
\caption{A Wc-building.}
\label{fig:2blocks}
\end{figure}

So a Wc-building is a Wc-germ $\sY$ equipped with a choice of presentation of its Wc-deformation class as the iterated vertical gluing $\sY=(\sX_1 \to \cdots \to \sX_k)$. 

\subsubsection{Distinguished cover of a Wc-building}
A Wc-building come equipped with a symplectic cover by the Wc-manifolds it is built out of. The relevant properties of this cover are summarized in the following lemma. Recall that given a Wc-manifold   $(X,\lambda)$ the Liouville form $\lambda$ restricted to a neighborhood of a face $P$ has the form $\lambda:=\lambda|_N+udt,$ where $N$ is the nucleus of $P$ and $t $ is the defining coordinate for the face $P$. We recall the notation $\lambda^P$ for the form $\lambda|_N-tdu=\lambda-d(ut)$ on $\Op N$.
 
 \begin{lemma}\label{lm:complex}
 Suppose  $(\sX ,\lambda )$  is presented as a $k$-level building $ (\sX_k\to \sX_{k-1}\to \dots \to  \sX_1)$.
 Then there exist symplectic embeddings $\Phi_j: \text{int} \sX_j\to \sX$, $j=1,\dots, k$, such that
 \begin{itemize}
 \item[(i)] $\bigcup_{j=1}^k(\Phi_j( \Skel(\sX_j,\lambda_j) \cap \text{int} \sX_j)=\Skel(\sX,\lambda)$;
 \item[(ii)] $\Phi_j(\text{int}{\sX_j})\cap\Phi_i(\Skel(\sX_i,\lambda_i) \cap \text{int} \sX_i)=\varnothing$ if $j>i$;
 \item[(iii)] If  $\Phi_j(\text{int}{\sX_j})\cap\Phi_i(\text{int}{\sX_i}) \neq \varnothing$ for $j>i$ then
 there exists a nucleus $\sN$ of a boundary face of $\sX_j$ such that 
 $\Phi_j^{-1}(\Phi_i(\text{int} \sX_i))=\Op \sN  \cap \text{int}{\sX_j}$ and $\Phi_j^*(\Phi_i)_*\lambda_i=\lambda_j^P.$
 \end{itemize}
 \end{lemma}
       \begin{figure}[h]
\includegraphics[scale=0.5]{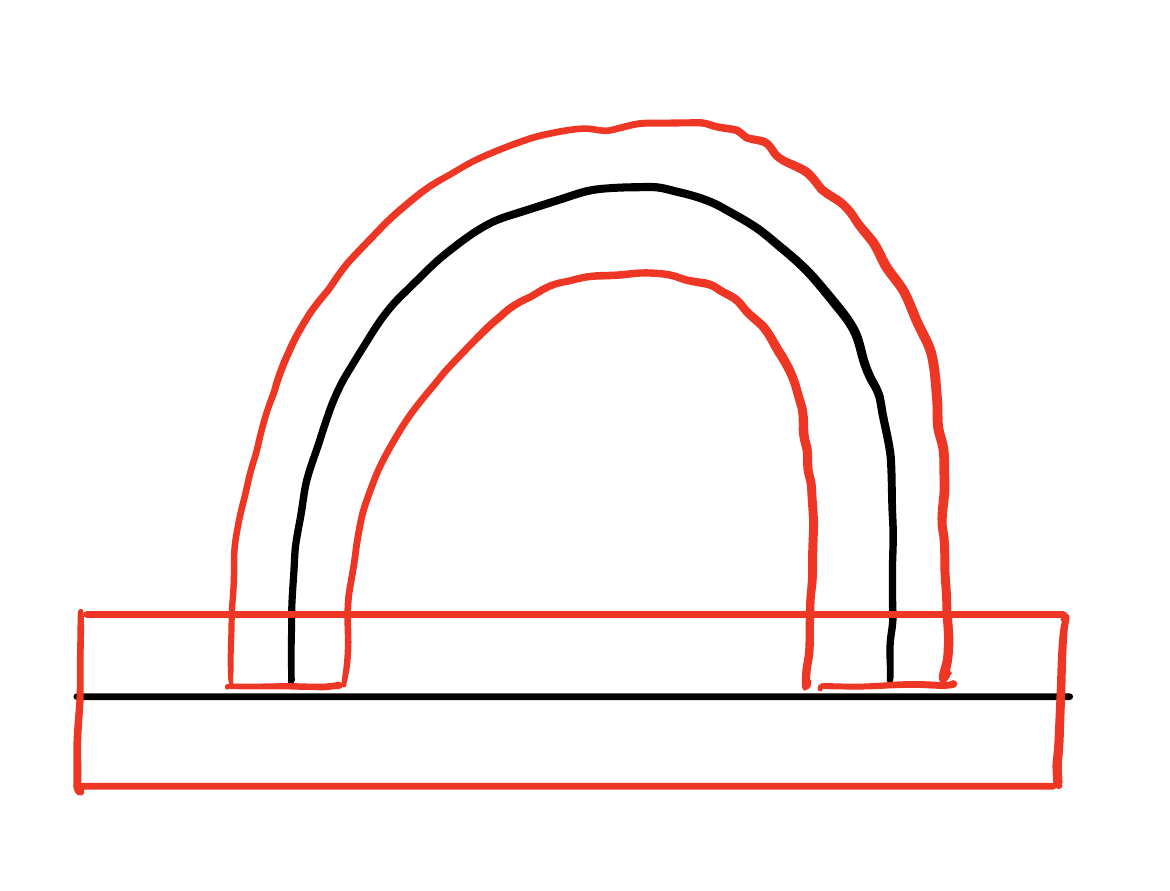}
\caption{Presenting a Wc-building as the union of Wc-manifolds.}
\label{Union}
\end{figure}

 \begin{proof} We inductively construct the maps $\Phi_{i,j}: \sX_i \to (\sX_k \to \cdots \to \sX_j)$, $i\geq j$, as follows, at the end of the inductive process we will set $\Phi_i=\Phi_{i,1}$ . The map $\Phi_{k,k}$ is just the inclusion on $\sX_1$ into $\sX$. Suppose we have defined maps $\Phi_{i,j}$ for $i \geq j$. Recall that $ \sY=(\sX_1 \to \cdots \to \sX_j)$ is vertically glued to $\sX_{j-1}$ by a Liouville isomorphism between the nucelus $\sN$ of a boundary face $\sP$ of $ \sY$ and a Wc-hypersurface $\sA$ in $\sX_{j-1}$.  Let $\Psi :  \sY \to (\sX_1 \to \cdots \to \sX_j \to \sX_{j-1})$ be the symplectic embedding from Lemma \ref{lem:maps}, set $\Phi_{i,j-1}=\Psi \circ \Phi_{i,j}$ and set $\Phi_{j-1,j-1}$ to be the restriction of the inclusion $\sX_{j-1} \hookrightarrow (X_1 \to \cdots \to X_j \to X_{j-1})$ to the interior of $\sX_{j-1}$. The desired properties for the resulting $\Phi_i$ follow from the properties stated in Lemma \ref{lem:maps}.  \end{proof}

 \begin{remark}
Note that on the overlap $\Phi_{i_1}(\sX_{i_1}) \cap \cdots \cap \Phi_{i_k}(\sX_{i_k}) \subset \sX$ the pushed-forward Liouville fields $Z_{i_1}, \ldots , Z_{i_k}$ coming from $\sX_{i_1}, \ldots , \sX_{i_k}$ all commute with each other.
 \end{remark}
 
 \subsubsection{Cotangent buildings}
 
 Finally, we introduce a particular class of Wc-buildings which will be particularly relevant for our purposes. 
  
\begin{definition} A Wc-building $(\sX_k \to \cdots \to \sX_1)$ for which each Wc-germ consists of a cotangent block $\sT^*M_j$ is called a {\em cotangent building}. \end{definition}

We will later show that any finite type Weinstein manifold can be realized as the underlying Weinstein manifold of a cotangent building.
  

\subsection{Legendrian submanifolds}\label{sec:Leg-in-blocks}

\subsubsection{Legendrians in cotangent bundles}
Consider first a cotangent bundle $T^*M$, where $M$ an $n$-dimensional compact manifold with corners. Let $\Lambda$ be an $(n-1)$-dimensional manifold with corners.

\begin{definition} An embedding  $\Lambda \to T^*M\setminus M$  is called {\em Legendrian} if its projection to $S^*M$ is a Legendrian embedding.
\end{definition}

With every Legendrian $\Lambda \subset T^*M$ we associate its Lagrangian Liouville cone $L:=\text{Cone}(\Lambda,pdq) $ formed by forward and backwards trajectories of  the Liouville field $Z$. Note that $L $ is embedded.

   The  splitting $P \times \cI^k$ near a $k$-face $P$ of $M$ provide   canonical splittings \begin{equation}\label{eq:bc-splitting}
 T^*{M}|_{\Op{P}}=T^*P\times T^*\II^k  \end{equation}

  We will  call a  Legendrian embedding   $\phi:  L \to \sT^*M\setminus M$  
  {\em adapted}  to the block $T^*M$ if    $f(\Op\p L)\subset \Op\p M$ and  for each stratum $P\subset\p_kM$  we have $L\cap \Op P =L_k\times \cI^k $ for a Legendrian $L_k\subset \sT^*P$.

 
  By   conical  Lagrangians $L\subset T^*M\setminus M$  we will always mean Lagrangian Liouville cones over    Legendrians.   
  
  Denote by   $\pi:T^*M\to M$ the cotangent bundle projection.
  
\begin{definition} An adapted Legendrian $\Lambda$ is called {\em regular} if $\pi|_\Lambda:\Lambda\to M$ is an   immersion with transverse self-intersections.

\end{definition}

      \begin{figure}[h]
\includegraphics[scale=0.5]{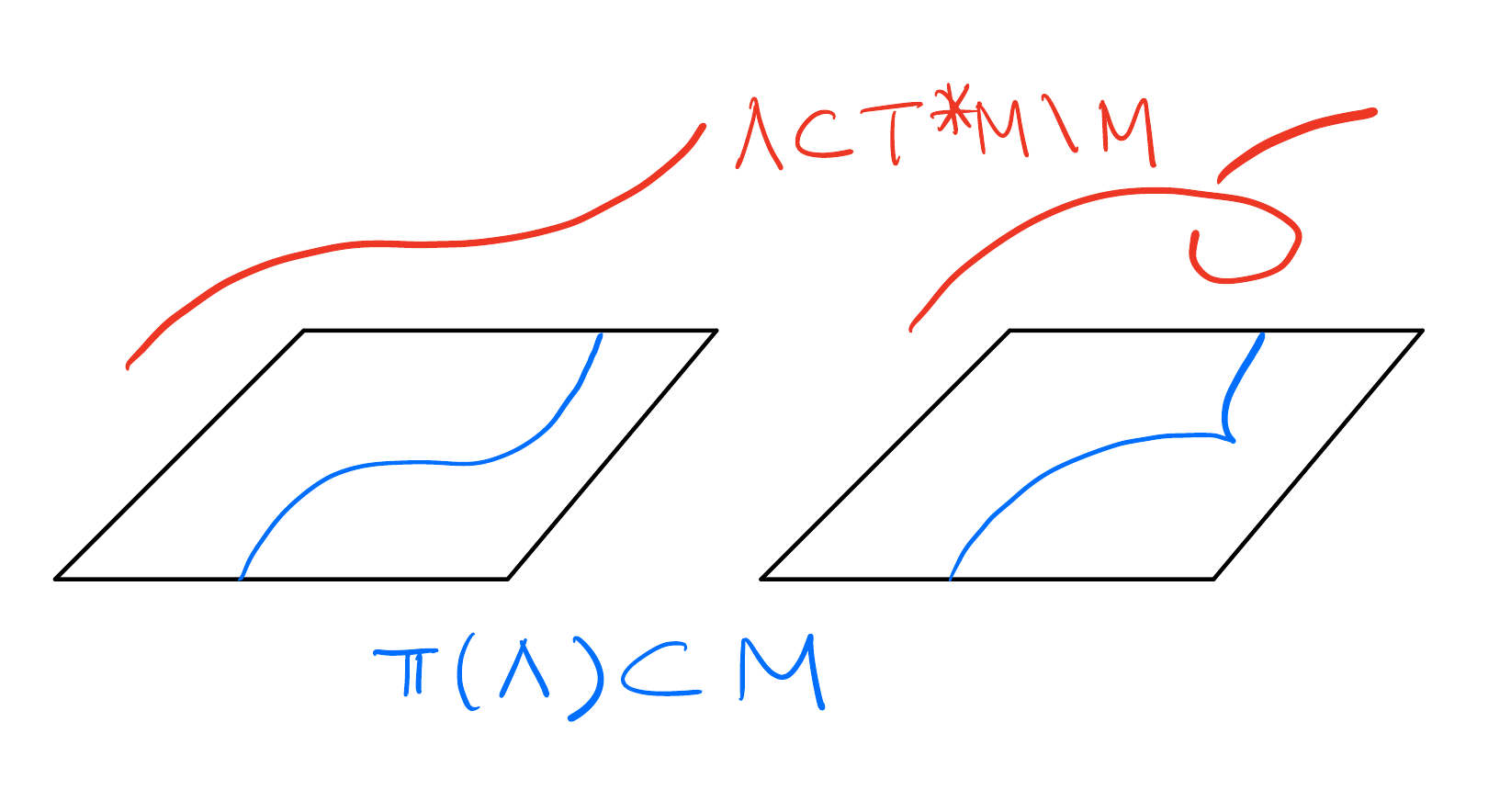}
\caption{The Legendrian on the left is regular, while the one on the right is not.}
\label{fig:regular}
\end{figure}

\subsubsection{Legendrians in Wc-manifolds}


Let $(X,\lambda)$ be a $2n$-dimensional Wc-manifold.
\begin{definition} An  $(n-1)$-dimensional submanifold  $\Lambda\subset 
   X\setminus\Skel(X,\lambda)$ is called {\em Legendrian} if   it projects  to $\p_\infty X$ as an embedded Legendrian submanifold of $\p_\infty X$. Its Liouville cone $L=\Cone(\Lambda,\lambda)$ is a called a {\em conical Lagrangian}.  \end{definition}

   We  say that a Legendrian  $\Lambda$ is  {\em adapted} to the Wc-structure of $(X,\lambda)$ if $\p_k \Lambda \subset \p_k X$ and moreover in a neighborhood of   any corner  $P\subset \p_k X$ we have $\Lambda=\Lambda_N\times  \cI^k \subset P\times T^*\II^k$, where $\Lambda_N$ is a Legendrian in the nucleus $N$ of the corner $P$.
   
       \begin{figure}[h]
\includegraphics[scale=0.5]{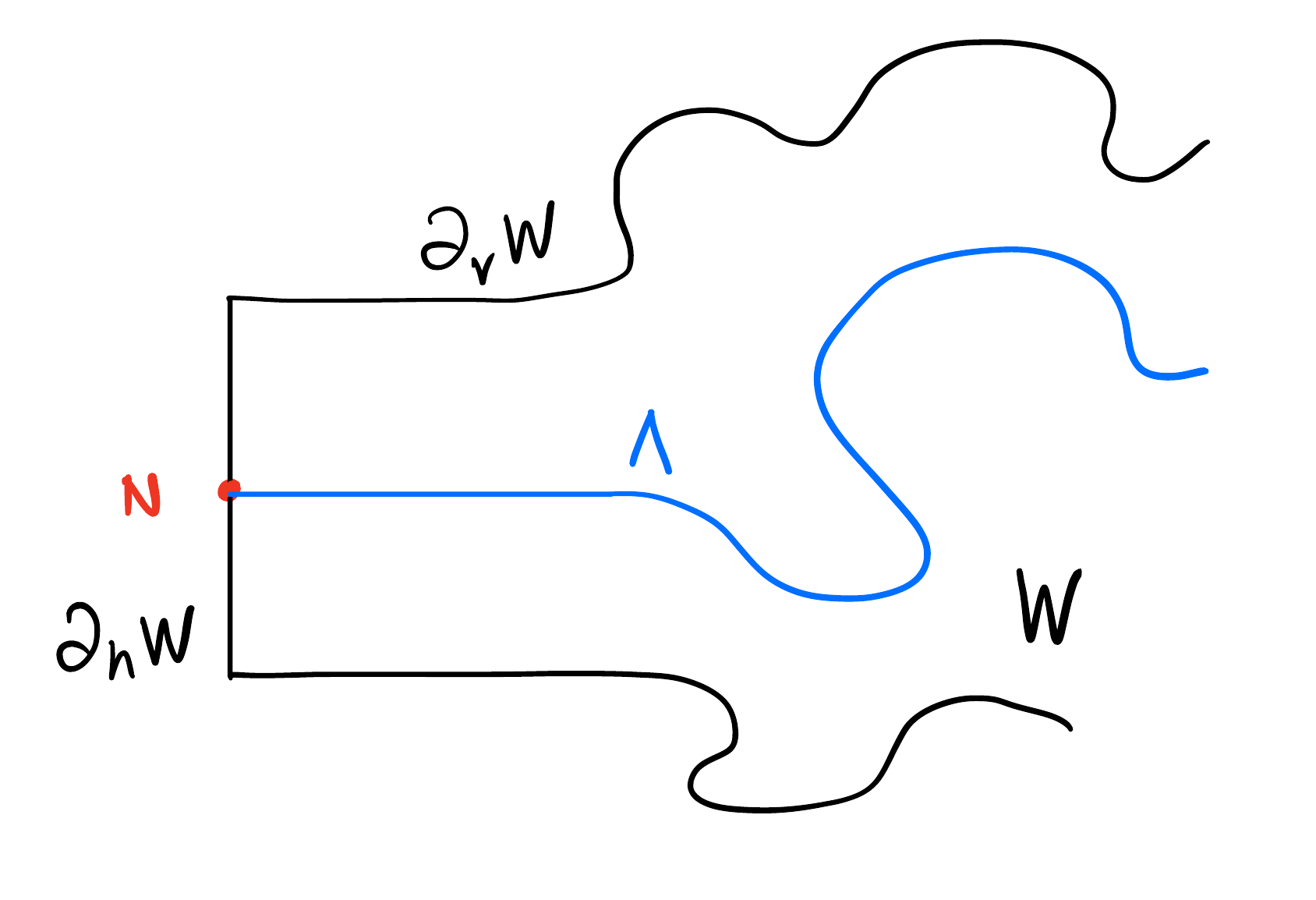}
\caption{An adapted Legendrian $\Lambda$ in a Wc-manifold $W$.}
\label{fig:adatedLeg}
\end{figure}

       \begin{example}
       Let $M$ be a manifold and $N\subset M$ a co-oriented hypersurface. Then any submanifold $\Lambda \subset T^*M\setminus M$ which  1-1 projects to the  conormal lift of $N$ in $S^*M$ (and is never tangent to the Liouville flow) is   Legendrian. If $N$ is adapted for the corner structure of $M$, then one can construct a Legendrian lift $\Lambda$ which is adapted for the Wc-structure of $T^*M$.
              \end{example}
              
              We can also think of Legendrians in $T^*M$ as equivalence classes of Legendrian embeddings into defining domains for $T^*M$, where the equivalence is given by translation along the Liouville flow. This gives us the notion of a Legendrian in a cotangent block $\sT^*M$. Similarly, we have the notion of a Legendrian in a Wc-germ $\sX$. For a Legendrian $\Lambda$ in a defining domain $(W,\lambda|_W)$ for $(\sX,\lambda)$, we only take its positive Liouville cone $\text{Cone}^+(\Lambda, \lambda)$, which is contained in $W \setminus \Skel(W,\lambda)$. One can think of the equivalence class of all positive Liouville cones over all Legendrians equivalent to $\Lambda$ as the whole Liouville cone $\text{Cone}(\Lambda, \lambda) \subset X \setminus \Skel(X, \lambda)$.
    

   \subsection{Gluing and splitting of Legendrians}\label{sec:Leg-split}
   
   \subsubsection{Gluing Legendrians}
 
\begin{lemma}\label{lm:eff-of-conv-on-Leg} Let $(\sX,\lambda)$ be a  Wc-germ, $\Lambda\subset \sX\setminus(\Skel(\sX,\lambda) $ an  adapted  Legendrian and $L$ the $\lambda$-Liouville cone over $\Lambda$. Let $(\sX^N,\lambda^N)$ be the result of a nucleus-to-hypersurface conversion of the nucleus $N$ of  one  of the boundary faces of $X$. Then  $\Lambda\subset\sX^N$ remains Legendrian  and its $ \lambda^N$-Liouville cone coincides with $L$. 
 \end{lemma}
 
 \begin{proof} In a neighborhood of the face $P$ of $X$ whose nucleus $N$ we convert to a boundary face we have $\Op P = T^*N \times T^*\cI$ and $\Lambda=\Lambda_N \times \cI$. The conversion changes $\lambda = \lambda_N + udt$ only by the differential of a function of $(u,t)$, moreover the resulting Liouville field on $T^*\cI$ is tangent to $\cI$ (and outwards pointing at the boundary). Hence $\Lambda$ remains Legendrian and its Liouville cone is unchanged.
 \end{proof}

\begin{lemma}\label{lm:Leg-horizont-gluing}
Let $\sX_0, \sX_1$ be two Wc-germs, $P_0, P_1$ two of their boundary faces and $N_0,N_1$ the nuclei of $P_0, P_1$.  Let $\Lambda_0\subset \sX_0\setminus\Skel(\sX_0,\lambda_0)$ and  $\Lambda_1\subset \sX_1\setminus\Skel(\sX_1,\lambda_1)$ be two adapted Legendrians. Denote $\Lambda'_i:=\Lambda_i\cap N_i$, $i=0,1$. Let $\phi:N_0\to N_1$ be a Liouville isomorphism such that $\phi(\Lambda'_0)=\Lambda'_1$. Then the horizontally glued Wc-manifold
$\sX_{01}=\sX_0\mathop{\cup}_\phi \sX_1$ contains a glued Legendrian
$$\Lambda_{01}=\Lambda_0\mathop{\cup}\limits_{\Lambda_1'=\phi|_{\Lambda_0'}} \Lambda_1 . $$
\end{lemma}

\begin{proof} This is obvious by definition of adapted. \end{proof}

 
 \subsubsection{Splitting Legendrians}

Let $X$ be a Wc-manifold and $P\subset X$ a splitting  hypersurface with the soul $S$.
A Legendrian $\Lambda\subset X\subset\Skel (X,\lambda)$ is called  in {\em in splitting position} if  $\Lambda$ is transverse to $P$ and $\Lambda\cap P\subset S$.
\begin{lemma}\label{lm:Legendrian-splitting}
Let $X$ be a Wc-manifold,  $P\subset X$ a splitting hypersurface and  consider a Legendrian $\Lambda\subset X\setminus\Skel (X,\lambda)$ which is  in splitting position.
Then under a horizontal  splitting of $X$ we can arrange it so that $\Lambda$  splits into  Legendrians with corners
$\Lambda_\pm\subset X_\pm\setminus\Skel (X_\pm,\lambda_\pm)$.  \end{lemma}

\begin{proof} It suffices to take a neighborhood of $P$ in $\sX$ so that $\Lambda$ is of product form with respect to the neighborhood, which is clearly possible. \end{proof}

Let us recall that given a contact manifold $(M,\xi)$ a hypersurface $N\subset M$ is called {\em convex} if it admits a tranverse to it contact vector field $v$.  Consider a split a tubular neighborhood $U_\eps:=N\times(-\eps,\eps)$ of a contact hypersurface $N$ such that the contact vector field $v$ coincides with $\frac{\p}{\p t}$, where $t$ is the coordinate corresponding to the second factor.
The contact structure $\xi$ on $U_\eps$ can be then defined by a contact form $\alpha=fdt+\lambda$, where $f:N\to\R$ is a smooth function and $\lambda$ is a 1-form on $N$. The set $Q=\{f=0\}=\{x\in N;\;v(x)\in\xi\}$ is called the {\em dividing set}.   It is always  a regular level set of the function $f$.  Denote $N_+:=\{f>0\}, N_-=\{f<0\}$. We also denote $$\wh Q:=Q\times(-\eps,\eps)=\bigcup\limits_{t\in(-\eps,\eps)}v^t(Q)$$ and call it the {\em dividing wall}.
The contact condition implies that $\beta:=\lambda|_{\Sigma}$ is a contact form on $Q$ and that the form $\lambda_f:=\frac1f\lambda$ is Liouville on $N\setminus Q=N_+\cup N_-$.
The corresponding Liouville fields $Z_\pm$ are complete of both halves $N_+$ and $N_-$ and for any point $x\in N_\pm\setminus\Skel(N_\pm)$  there exists a limit  $\lim\limits_{s\to\infty} Z_\pm^s(x)\in Q$.
We say that  the pair $(N,v)$, where $N$ is a convex hypersurface and $v$  a transverse vector field, is of {\em Weinstein type}, if $(N_\pm,\lambda_f)$ is Weinstein.
\begin{lemma}\label{lm:Leg-convex}
Let $(M,\xi)$ be a contact manifold, $N\subset (M,\xi)$ a Weinstein type convex hypersurface and $\Lambda\subset (M,\xi)$ a Legendrian transverse to $N$. Then there exists a contact isotopy $\phi_s:M\to M$, $t\in[0,1]$ such that
\begin{itemize}
\item[(i)] $\phi_s$ fixes $Q$;
\item[(ii)] $\phi_s$    leaves $N$ invariant;
\item[(iii)]   there exists $\eps>0$ such that $\phi_1(\Lambda\cap U_\eps)\subset Q_\eps$;
\item[(iv)] $\phi_s$ is supported in a neighborhood $U'\Supset U_\eps$.
\end{itemize}
\end{lemma}
\begin{proof} Note that  a contact vector field $Y:=t\frac{\p}{\p t}+Z_\pm$ on $N_\pm\times(-\eps,\eps)$ smoothly extends to $N\times(-\eps,\eps)$ as equal to $t\frac{\p}{\p t}$ on $Q\times(-\eps,\eps)$.
Let us choose $\eps$ sufficiently small so that the coordinate $t|_{\Lambda\cap U_\eps}$ has no  critical points.
Denote $\Gamma:=\Lambda\cap N$. A general position argument allows us to  $C^\infty$-perturb  $\Lambda$ so that $\Gamma\cap(\Skel(N_+)\cup\Skel(N_-))=\varnothing$ and  the  projection of $\pi:\Gamma\to Q$ along Liouville  trajectories of $N_\pm$ is an embedding. Note that  $\lambda|_\Gamma=0$, and hence $\pi(\Gamma)\subset Q$ is Legendrian for the contact structure $\ker\beta$ on $Q$. Let us parameterize $\Lambda\cap U_\eps$ by an embedding
$h_0:\Gamma\times(-\eps,\eps)\to \Lambda\cap U_\eps$ such that $h_0(\Gamma\times t)\subset N\times t$. Define a Legendrian isotopy $h_s:\Gamma\times(-\eps,\eps)\to \Lambda\cap U_\eps$, $s\in[0,\infty)$ by the formula
$$h_s(x,t)=Y^s(h_0(x,te^{-s})).$$
Let us observe that $h_s(x,t)=(g_{s,t}(x),t) $, with $g_{s,0}(x)$ flowing a $\lambda$-isotropic submanifold  $\Gamma$ to $\pi(\Gamma)$ along trajectories of $Z_\pm$. The flow is fixed on $Q$.
Let us analyze the isotopy in a neighborhood of $Q$. The form $\lambda_f$ on $\Op Q$ can be written as $\frac1u\beta+dt$ for the coordinate $u=f$ on $\Op Q$, and hence, $Y=t\frac{\p}{\p t}-u\frac{\p}{\p u}$. The flow $Y^t$ is given by $Y^s(x,u,t)=(x,e^{-s}u, e^s t)$. Assuming that $T$ is large enough so that $h_T(\Gamma\times(-\eps,\eps))\subset \Op Q\times(-\eps,\eps)$ we get $h_{T+s}(x,u,t)=(x, e^{-s}x, t)\mathop{\to}(x,0,t)$ as $s\to\infty$.
Hence by reparameterizing the isotopy to the interval $[0,1)$ and extending it to the end-point $1$  as $(x,0,1)$. we get a smooth Legendrian isotopy  $\phi_s:\Gamma\times(-\eps,\eps)$ such that $\phi_1(\Gamma\times(-\eps,\eps))\subset Q\times(-\eps,\eps)$. The isotopy  $\phi_s$ can be realized by an ambient contact isotopy which keeps $N$ invariant. By cutting off this isotopy outside a smaller neighborhood of $N$ we get the required  Legendrian isotopy of $\Lambda$.
\end{proof}

 \begin{lemma}\label{lm:Leg-adjustment}
 Consider a Wc-manifold $(X,\lambda)$ with a dividing hypersurface
 $P \subset X$  and its soul $S  $.
  Let $\Lambda_0\subset   X\setminus\Skel (X,\lambda)$ be any Legendrian.
 Then there exists a Legendrian isotopy $\Lambda_t\subset   X\setminus\Skel (X,\lambda)$, $t\in[0,1]$, such that $\Lambda_1$ is in a splitting position. \end{lemma}
 \begin{proof}
 We can realize $\p_\infty X$ as a smooth boundary of a neighborhood of $\Skel(X,\lambda)$ and view $\Lambda, \p_\infty S$ and $\p_\infty P$ as   submanifolds of $\p_\infty X$. 
 Then the boundary $\p_\infty P\subset\p_\infty X$ is a convex   hypersurface in the contact manifold $(\p_\infty X, \lambda)$, and its dividing set $\p_\infty S$ divides $\p_\infty P$ into two copies $S_\pm$ of $S$.  Therefore the conclusion follows from the previous lemma.
\end{proof}

\subsubsection{Legendrians in Wc-buildings}
 
 Let $(\sX,\lambda)$ be  a Wc-germ with a structure of a $k$-block building  $\sX_k\to\dots\to \sX_1$. We say that a Legendrian $\Lambda\subset\sX\setminus\Skel(\sX,\lambda)$ is {\em compatible with the building structure} if its Liouville cone $L$ is invariant with respect to all Liouville fields $Z_i$, $i=1,\dots, k$, which we can think of as living in $\sX$ via Lemma \ref{lm:complex}.

   \begin{lemma}\label{lm:Leg-adjustment}
  Let $(X,\lambda)$ be presented as a Weinstein building
  $\sX_k\to\dots\to\sX_1$.  Let $\Lambda\subset \sX\setminus\Skel(\sX,\lambda)$ be  a Legendrian.  Then there exists a    Legendrian  isotopy $\Lambda_t$ starting from $\Lambda_0=\Lambda$ such that  $\Lambda_1$ is compatible with  the  building, i.e. its $Z$-Liouville cone is invariant with respect to all Liouville fields $Z_i$ where they are defined.
  \end{lemma}
  \begin{proof}
  By definition we have filtration    by  Wc-manifolds
 $\sX\supset\sX_{\geq 2},\dots, \supset\sX_{\geq k-1}\supset \sX_k$, where
 $\sX_{\geq j}= \sX_k\to\dots\to X_j$. Note that we equivalently view $\sX$ as a 2-block building $\sX_{\geq2}\to \sX_1$. By definition the building $\sX_{\geq2}\to \sX_1$ can be split into the blocks
 $\sX_1$ and  $\sX_{\geq2}$ along a splitting hypersurface $H$ which can be defined as follows. Recall that  $\sX_{\geq2}\to \sX_1$ is obtained from 
 $\sX_1$ and  $\sX_{\geq2}$ by taking the nucleus $N$ of one of the boundary faces of $\sX_{\geq2}$, realizing it as Wc-hypersurface in $\sX_1$,  converting it to a nucleus of a boundary face and then performing the horizontal gluing along the face  $H$ of $\sX_1$ and the converted face 
 of  $\sX_{\geq2}$.  The collar  neighborhood $\Op H$ of the face $H$ in $\sX_1$ is split as  $\Op H=H \times\II$. The inverse splitting can be done by using the $H \times\eps$ as a splitting surface. 
     Note  that   making the  Legendrian  $\Lambda$ compatible with the building
  $\sX_{\geq2}\to \sX_1$ is equivalent to moving it to a splitting position with respect to the splitting hypersurface $H \times\eps$.
   Hence, using Lemma \ref{lm:Leg-adjustment} we can move $\Lambda$ by a Legendrian isotopy to  make it compatible with the  building $\sX_{\geq2}\to \sX_1$. The splitting defines a Legendrian $\Lambda_{\geq2} \subset \sX_{\geq2}\setminus\Skel(\sX_{\geq2},\lambda)$.  We continue by applying again  \ref{lm:Leg-adjustment} we move $\Lambda_{\geq 2}$ by a Legendrian isotopy to make it compatible with    the  building $\sX_{\geq2}\to \sX_2$. Continuing the process we prove the lemma.  \end{proof}

\subsection{Weinstein handlebodies revisited}

\subsubsection{Rounded handles}

Consider an $n$-ball $D^n\subset S^{n-k}\times D^k$ and denote  $$K^n_k:=(S^{n-k}\times D^k)\setminus \Int D^n.$$ Thus, $\p K_k^n$ is manifold with two boundary components $\p_-K_k^n=S^{n-k}\times\p D^{k-1}$ and $\p_+K_k^n:=S^{n-1}$.
 
Consider  a Weinstein manifold with boundary $G^n_k:=T^*K^n_k$. It has two boundary faces $M_-$ and $M_+$  with nuclei $P_-:=T^*(S^{n-k}\times\p D^{k})$ and $P_+:=
T^*S^{n-1}$. 

\begin{definition} We  call  $G^n_k$ a  {\em rounded  subcritical Weinstein handle}.
\end{definition}

\begin{remark} Our rounded handles $G^n_k$ have nothing to do with {\em round} handles of index $k$. \end{remark}

\begin{lemma}\label{lm:rounded-handles}
Let $W^{2n}$ be Weinstein domain and   $\phi:\p D^k\times D^{n-k} \to  \p W$   a Legendrian embedding, $k \leq n$.  Let $W'$ be the Weinstein domain obtained by a subcritical  index $k$ Weinstein   handle attachment along the isotropic sphere $\phi(\p D^k\times 0)$ with the framing given by the Legendrian extension.
 Suppose that the embedding $\phi$ is extended to a Legendrian embedding $\Phi:\p D^k\times S^{n-k}\to \p W$. Let $\Sigma=T^*(\p D^k\times S^{n-k})$ be the Weinstein     ribbon of this embedding. Then the Weinstein structure obtained by the  vertical gluing of the rounded handle  $G_k^n$  along $\Sigma$ is homotopic to  the Weinstein structure on $W'$.
 \end{lemma} 
 
 \begin{proof}
 The vertical gluing consists of the conversion of $\Sigma$ to a boundary nucleus and the horizontal gluing to $G_k^n$ along the new nucleus. The conversion is achieved by taking the union of $W$ with a collar $T^*( \p D^k \times S^{n-k} \times [0,1])$ and deforming the Liouville field of $W$ near $\Sigma$ so that the result is a Weinstein manifold with boundary. Hence the vertical gluing is the union of $W$ with the cotangent bundle of $\wh K_k^n=(\p D^k \times S^{n-k} \times [0,1] )\cup (K_n^k \times  1)$. Let $L := D^k \times D^{n-k} $ be attached to $\p D^k \times S^{n-k}$ along $\p D^k \times D^{n-k}$, i.e. the image of $\phi$. Then $\wh K_k^n $ is diffeomorphic to $T_k^n:=(\p D^k \times S^{n-k} \times [0,1] ) \cup ( L \times 1)$ after smoothing corners. The underlying Weinstein structure of the vertically glued Weinstein manifold is given along the boundary of $T^*\wh K_k^n$ by the Weinstein normal form for Morse-Bott boundary of index $0$. Before smoothing, on $T^*T_k^n$ this corresponds to the Weinstein normal form for Morse-Bott boundary, also of index 0. By a Weinstein homotopy we can deform the structure to cancel the critical points on $T^*(\p D^k \times S^{n-k} \times [0,1])$ along the $[0,1]$ direction, similar to a parametric Smale cancellation, where we think of $\p D^k \times S^{n-k}$ as the parameter space. We are left with the Weinstein normal form for Morse-Bott boundary and corners on $T^*L \simeq T^*(D^k \times D^{n-k})$ which has index $1$ on the boundary face $\p D^k \times D^{n-k}$ and index $0$ on the boundary face $D^k \times \p D^{n-k}$. This Weinstein structure is evidently homotopic to the standard Morse-Weinstein normal form of index $k$ on $T^*(D^k \times D^{n-k})$, which is by definition the Weinstein surgery of index $k$. 
 \end{proof}
 
         \begin{figure}[h]
\includegraphics[scale=0.6]{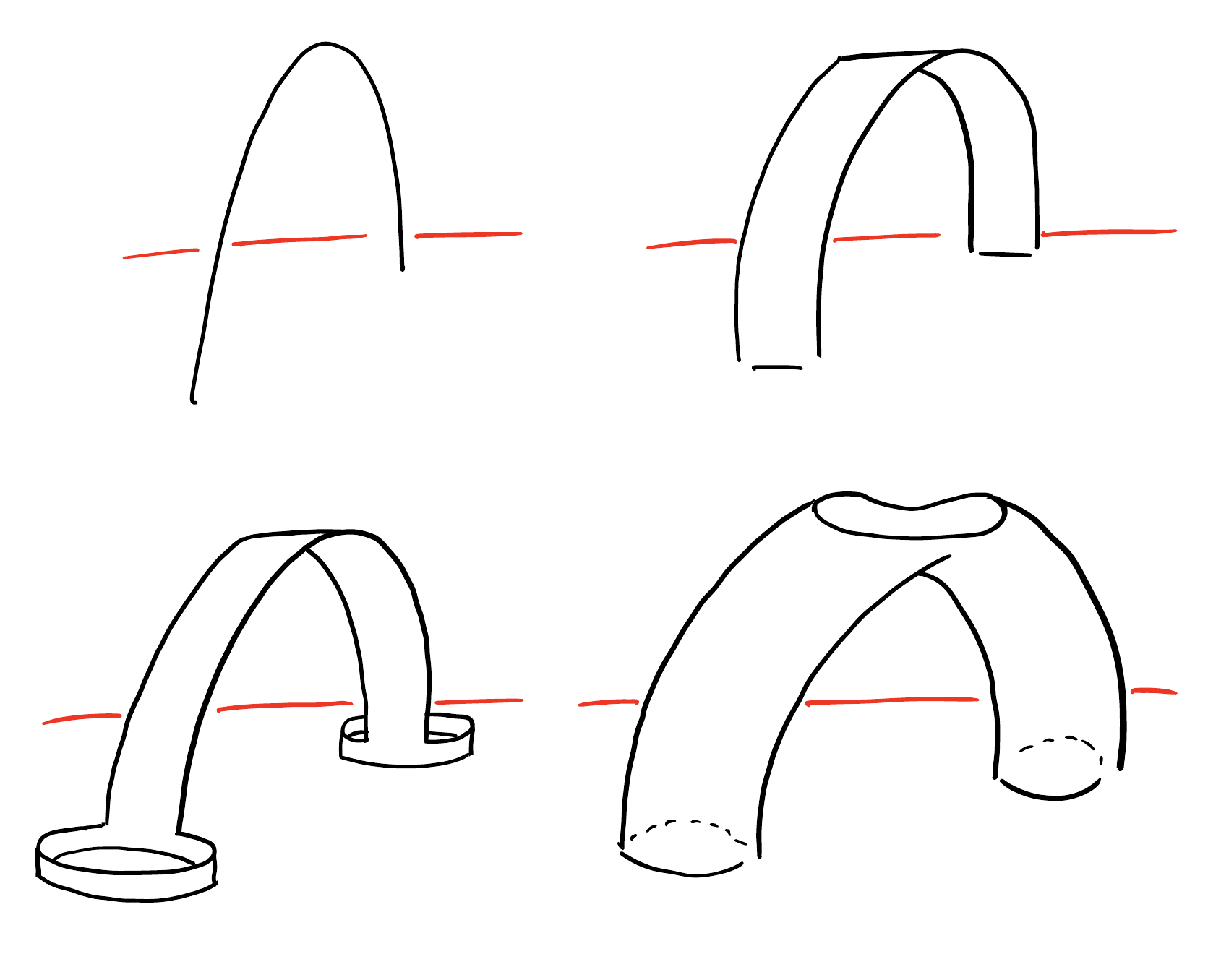}
\caption{Attaching a subcritical isotropic handle is the same as attaching a rounded handle.}
\label{fig:roundedhandle}
\end{figure}
                                                                                                                                                                                                                                                                                                                                                                                                                                                                                                                                                                                                                             
 \subsubsection{Weinstein manifolds as cotangent buildings}\label{sec:W-Wbc}
  
Recall that   a cotangent building  is a Wc-building $W=(B_k\to\dots \to B_1)$ consisting of cotangent blocks $B_i:=\sT^*M_i$.

     
 \begin{prop}\label{prop:WctoW}  
Any Weinstein manifold germ $(\sX,\lambda)$ admits the structure of a cotangent building with blocks $B_j=\sT^*M_j$, where each $M_j$ is diffeomorphic to either a  disc  $D^n$ or  one of the manifolds $K^n_j$, $j=1,\dots,n-1$ with corners added to its boundary. \end{prop}
 
\begin{proof} Consider a handle decomposition $H_k\to H_{k-1}\to\dots\to H_1$ of a Weinstein germ $(\sX,\lambda)$.  If $k=0$ then we  just replace the first index $0$ handle by $T^*D^n$. Suppose that  the Weinstein structure on the handlebody $W_{\leq k}:=H_{k-1}\to \dots\to H_1$ is already compatible   with a structure of a cotangent building.
If the index of the handle $H_k$ is equal $m$ for $0<m<n$ we  use Lemma \ref{lm:rounded-handles} to replace the handle $H_k$ by the rounded  handle $G^n_m=T^*K^n_m$.  If $m=n$ we keep  it as is equal to $T^*D^n$. In both cases we denote by $T^*M_k$  the block we need to attach with the  nucleus of the attaching  face $N$ equal to $T^*\p D^n$ in the latter case and $\p_-G^n_m=T^*(S^m\times S^{n-m})$ in the former one.
 For the attaching, $N$ is realized as a Weinstein hypersurface in $W^{<k}\setminus\Skel(W^{<k} )$. We denote by $\Lambda$ the core Legendrian of that hypersurface.
    Let $L$ be the   $(Z_{\leq k})$-cone of $\Lambda$. Using Lemma \ref{lm:Leg-adjustment} we deform $\Lambda$ via a Legendrian isotopy to make the cone $L$ invariant with respect to the contracting fields  on the blocks $B_1,\dots, B_k$. Note that this automatically gives $\wh M_k:=L\cup M^k$   a new structure of a  Wc-manifold with additional corners. The neighborhood of $\wh M_k$ then gets the structure of a cotangent block $B_k=\sT^*(\wh M_k)$ and $W$ can be obtained from the vertical gluing of $B_k$ to $W_{\leq k}$. \end{proof}
    
        \begin{figure}[h]
\includegraphics[scale=0.4]{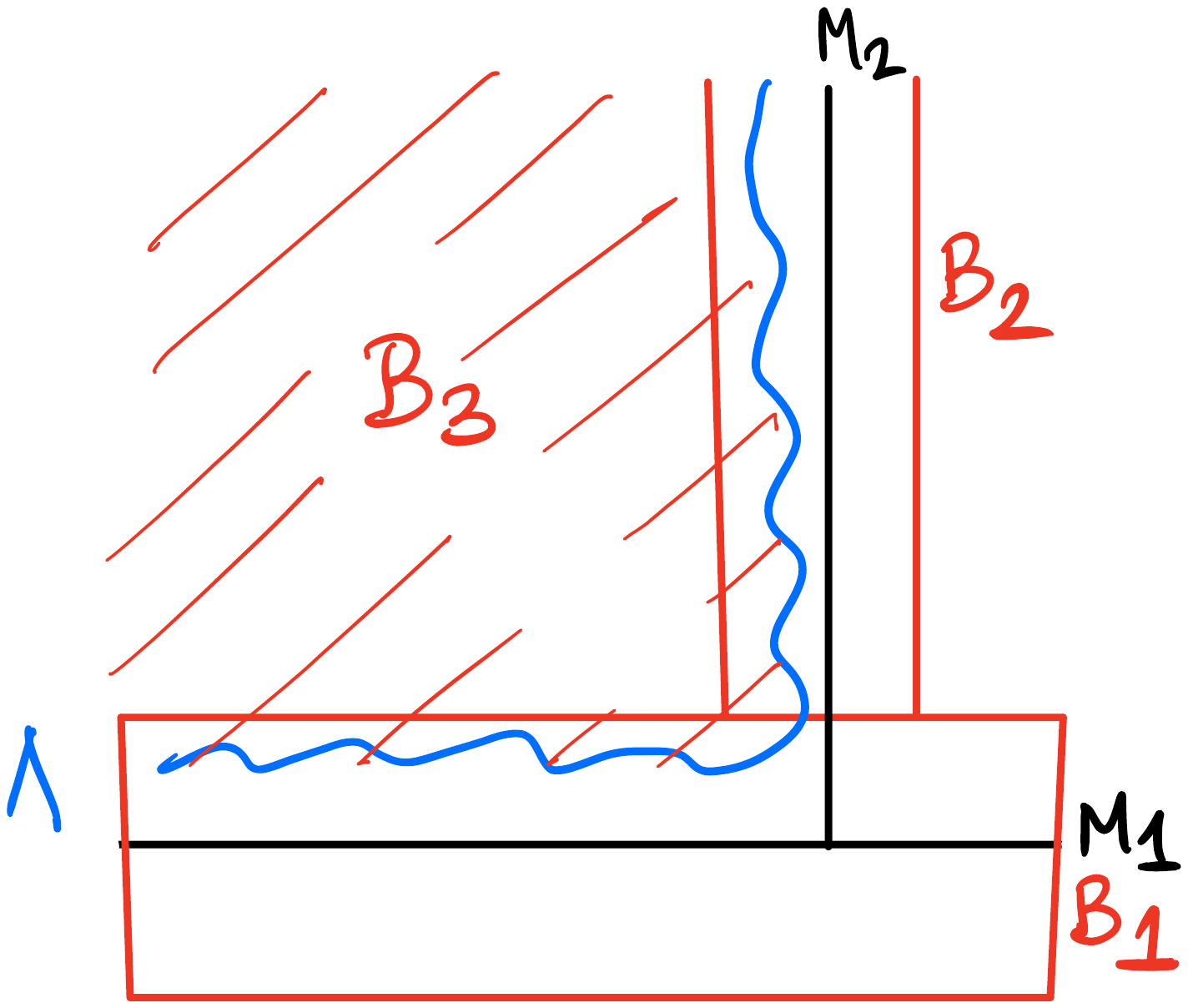}
\caption{If $B_2 \to B_1$ already has the structure of a cotangent building and we wish to extend this structure to a third block $B_3$, we need to make sure the Legendrian $\Lambda$ along whose ribbon the block $B_3$ will be attached is adapted, i.e. its Liouville cone is invariant under the Liouville flows of the blocks $B_1$ and $B_2$.}
\label{fig:Adapted}
\end{figure}

More generally, we have the following result:

\begin{proposition}\label{prop:Wctocot}
Any Wc-manifold admits the structure of a cotangent building.
\end{proposition}

To prove this we will inductively apply the following lemma.

\begin{lemma}\label{lem:rel cot}
Given any Wc-pair $(\sX,\sA,\lambda)$, where $\sX$ and $\sA$ are Wc-germs which admit the structure of a cotangent building, the hypersurface-to-nucleus conversion $(\sX^\sA,\lambda^\sA)$ also admits the structure of a cotangent building. 
\end{lemma}

\begin{proof}
Suppose that $(\sX,\lambda)$ is a Wc-manifold which admits the structure of a cotangent building $B_k \to \cdots \to B_1$ and $\sA \subset \sX \setminus \Skel(\sX,\lambda)$ is a Wc-hypersurface which also admits the structure of a cotangent building $B_r' \to \cdots \to B_1'$. Then  it is straightforward to verify that the hypersurface-to-nucleus conversion $(\sX^\sA,\lambda^\sA)$ inherits an induced structure of a cotangent building $(B_r' \times \sT^*[0,1]) \to \cdots \to (B_1' \times \sT^*[0,1]) \to B_k \to \cdots \to B_1$. 
\end{proof}

\begin{proof}[Proof of Proposition \ref{prop:Wctocot} ]
Given any Wc-manifold $(\sX,\lambda)$, we may convert all of its boundary nuclei to Wc-hypersurfaces to obtain its underlying Weinstein manifold $\wh X$, and $\sX$ can be recovered from $\wh X$ by converting back the Wc-hypersurfaces to face nuclei. By Proposition \ref{prop:WctoW} we know that $\wh X$ admits the structure of a cotangent building, and we may assume by induction on dimension that the Wc-hypersurfaces also admit the structure of a cotangent building. Hence it follows from Lemma \ref{lem:rel cot} that $\sX$ admits the structure of a cotangent building.
\end{proof}

\section{Symplectic neighborhoods of arboreal Lagrangians}\label{sec: nbhd}

In this section we introduce the notion of an arboreal space and prove that any arboreal space has a unique symplectic thickening.

\subsection{Arboreal spaces}

\subsubsection{Pre-Ac spaces}
In \cite{AGEN20a} we associated to each signed rooted tree $\cT = (T, \rho, \eps)$, an arboreal Lagrangian $L_\cT \subset T^*\R^{n(\cT)}$ which is a union of smooth pieces, $L_\cT = \bigcup_{\alpha \in v(\cT)} L_\alpha$.

For fixed $c\geq 0$, and $m\geq n(\cT) + c$, set $d = m -  n(\cT) + c$, and consider the product
$$
\xymatrix{
L(\cT, m, c) := L_\cT \times \R^d \times \R^c_{\geq 0}
}
$$
We also have the smooth pieces $L(\cT, m, c)_\alpha := L_\alpha \times \R^d \times \R^c_{\geq 0}$.

We equip $L(\cT, m, c)$ with the structure sheaf $\cO(\cT, m, c)$ of functions restricted from smooth functions on $T^*\R^{n(\cT)}\times \R^d\times\R^c_{\geq 0}$ by the inclusion
$L(\cT, m, c)\hookrightarrow T^*\R^{n(\cT)}\times \R^d\times\R^c_{\geq 0}$.

\begin{definition}
An $m$-dimensional {\em pre-arboreal space with corners (pre-Ac space)} $(X, \cO)$  is a locally ringed Hausdorff topological space locally modeled  on 
$$(L(\cT, m, c), \cO(\cT, m, c))
$$
 for varying signed rooted trees $\cT$, and  $c \leq m - n(\cT) $.

A {\em smooth sheet} $Y \to X$ of a pre-Ac space $(X, \cO)$  is a finite proper map from a manifold with corners such that the map is an embedding on the interior of $Y$, and
 locally the map identifies the local connected components of $Y$ with smooth pieces
$$
\xymatrix{
L(\cT, m, c)_\alpha  \subset L(\cT, m, c)
}$$
for varying $\alpha \in v(\cT)$.

\end{definition}

In this paper we shall only need to consider the cases where $c=0$ or $c=1$. When $c, d = 0$, we write $(L_\cT, \cO_\cT)$ in place of $(L(\cT, m, c), \cO(\cT, m, c))$.

      \begin{figure}[h]
\includegraphics[scale=0.5]{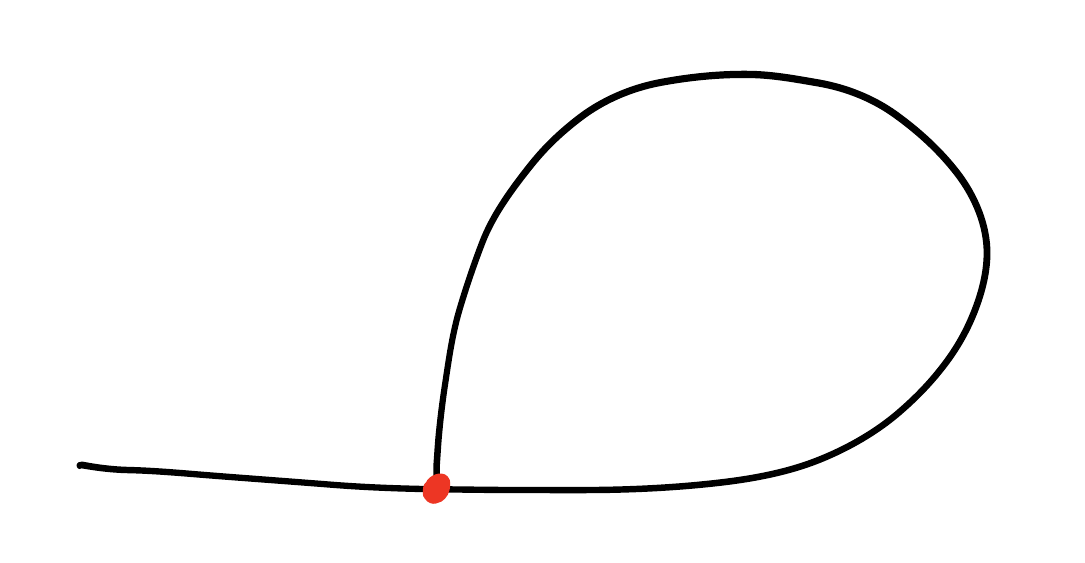}
\caption{We can take $X$ to be the interval $[-1, 1]$ with $1$ glued to $0$ as an $\cA_2$-singularity. Then $[-1, 1]\to X$ is a smooth sheet. We will introduce the notion of Ac-building to avoid the unnecessary complications of this kind of gluing. }
\label{fig:nonex}
\end{figure}

\begin{remark}
Given a pre-Ac space $(X, \cO)$, we can speak about a closed embedding into a smooth manifold $(M, C^\infty_M)$. These are embeddings with closed image $X \to M$ with a surjective  map of sheaves of local rings $C^\infty_M|_X \to \cO$. 
\end{remark}

\begin{remark}
One can relax the definition of a pre-arboreal space by choosing  the alternative structure sheaf
of
  continuous real-valued functions on $L(\cT, m, c)$ whose restrictions to each smooth piece
 $
 L(\cT, m, c)_\alpha, 
 $
 for $\alpha \in v(\cT)$,
are  smooth. One can show that every such ``weak" pre-arboreal space     is diffeomorphic to a strong one,  and any  weak diffeomorphism between strong pre-arboreals is isotopic to a strong one.
\end{remark}

\subsubsection{Ac spaces}
 In \cite{AGEN20a} we proved that any germ of an arboreal Lagrangian is symplectomorphic to exactly one of the canonical models $L_\cT \subset T^*\R^{n(\cT)}$, hence in particular isomorphic as the germ of a locally ringed Hausdorff topological space to the corresponding $(L_\cT,\cO_\cT)$. Going the other way, if we are given a local model $(L_\cT, \cO_\cT)$ it is not possible to reconstruct the symplectic geometry of
$L_\cT \subset T^*\R^{n(\cT)}$ unless we keep track of additional structure. This is due to the fact that not every automorphism of $(L_\cT,\cO_T)$ can be lifted to an automorphism of $L_\cT \subset T^*\R^{n(\cT)}$, i.e. a germ of a symplectomorphism preserving $L_\cT$. We now explain the additional structure that we need to keep track of.

For a  signed rooted tree $\cT = (T, \rho, \eps)$, recall if we remove the root $\cT \setminus \rho$ what remains is a disjoint union of rooted trees  with signs along their edges. Let us write $ \cT' =  (T, \rho, \eps')$ for the 
signed rooted tree  where
we negate all of the signs along the edges in any collection of components of $\cT\setminus \rho$. 
Then  negating corresponding coordinates in our local models provides an isomorphism of pre-Ac spaces
$$
(L_\cT, \cO_\cT)
\simeq
(L_{\cT'}, \cO_{\cT'})
$$
that does not extend to an ambient symplectomorphism unless $\cT = \cT'$.
Additionally, in the case of the tree $\cA_2$ with a single unsigned edge, the nontrivial involution  of the pre-Ac space $(L_{ \cA_2}, \cO_{\cA_2})$ that reverses the orientation on the smooth base piece  $L_\rho \subset L_{\cA_2}$ can only be extended to 
an ambient anti-symplectomorphism.  We will see in Lemma~\ref{lem: ambiguity} below these are the only ambiguities.

We   introduce here additional structure to overcome the above ambiguities and hence enable us to reconstruct $L_\cT \subset T^*\R^{n(\cT)}$ from
$(L_\cT, \cO_\cT)$. Let us first mention some simple data we could add.
Given $(L_\cT, \cO_\cT)$, consider the base smooth piece $L_\rho \subset L_\cT$, and the hypersurface
$H = \bigcup_{\alpha \not = \rho} L_\alpha \cap L_\rho \subset L_\rho$. Then a coorientation of $H$ provides an embedding  $L_\cT \subset T^*\R^{n(\cT)}$ as the corresponding positive conormal of $H$. In particular, there is the  specific choice of coorientation that gives the canonical embedding $L_\cT \subset T^*\R^{n(\cT)}$ of the local model.

To go forward, we will introduce a version of the above coorientation data that can be easily glued in a global setting. 
For a signed rooted tree $\cT = (T, \rho, \eps)$, and vertex
 $\alpha\in v(\cT)$, consider the orientation bundle
$$
\xymatrix{
\kappa(\cT, m, c)_\alpha := \wedge^m T  L(\cT, m, c)_\alpha 
}
$$
Let  $K(\cT,m, c)$,  denote the space of real line bundles $\kappa \to L(\cT,m, c)$ equipped with identifications
$$
\xymatrix{
\kappa|_{L(\cT,m, c)_\alpha} \simeq \kappa(\cT, m, c)_\alpha  &
\alpha \in v(\cT)
}
$$
Note that given such a $\kappa$, we can extract from it a coorientation of the  hypersurface 
$H = \bigcup_{\alpha \not = \rho} L(\cT,m, c)_\alpha \cap L(\cT,m, c)_\rho$ in the base smooth piece
$L(\cT,m, c)_\rho$. Namely, we can take the coorientation  of $H$ to be that glued under $\kappa$ with  the positive Liouville direction.

\begin{lemma}\label{lem:orient}

%
%
 Each connected component of $K(\cT,m, c)$ is contractible, and $\pi_0(K(\cT,m, c))$ is a $(\Z/2\Z)^{e(T)}$-torsor.
\end{lemma}

\begin{proof}
Suppose the lemma is proved for  signed rooted trees with fewer vertices than $\cT$.

Fix $r>0$,  let $B_r   = \{ |p| = r\} \subset T^* (\R^{m-c}  \times \R^c_{\geq 0})$ denote the radius $r$
open ball bundle, and
 $N_r = B_r \cap L(\cT, m, c)$  the corresponding open neighborhood of 
 the zero-section $L(\cT, m, c)_\rho \simeq \R^{m-c}  \times \R^c_{\geq 0}$. 
Note the Liouville flow of $T^*\R^m$ deformation retracts $N_r$ back to the zero-section.

Consider the  \v{C}ech cover of $L(\cT, m, c)$ given by the open $N_r$ along with the disjoint opens
$$
\xymatrix{
L(\cT_i, m-1, c) \times \R_{>0} & i\in I
}
$$
where $I \subset v(T)$ denotes the set of vertices adjacent to the root $\rho$,
or equivalently the set of edges adjacent to $\rho$, $\cT_i$ is the signed rooted tree given by the connected component of $\cT \setminus \rho$ containing $i \in I$, and
the last factor $\R_{>0}$ of each open is given by Liouville flow.
Note  each intersection $N_r \cap L(\cT_i, m-1, c) \times \R_{>0}$, for $i\in I$, is homeomorphic to $ L(\cT_i, m-1, c) \times (0, r)$, in particular connected.

We see that to give a line bundle on $L(\cT, m, c)$, given its restriction to the above opens, is to give
a sign $\pm 1$ on each intersection $N_r \cap L(\cT_i, m-1, c) \times \R_{>0}$, for $i\in I$.
This inductively establishes the lemma.
\end{proof}

We will now use the symplectic geometry to specify a distinguished element of $\kappa(\cT, m, c)  \in K(\cT,m, c)$, i.e. a trivialization of the $(\Z/2\Z)^{e(T)}$-torsor of the lemma
$$
\xymatrix{
K(\cT,m, c) \simeq (\Z/2\Z)^{e(T)} 
}$$ 
Given the canonical model $L(\cT, m, c)\subset T^*\R^m$,  
we have a canonical coorientation of the  hypersurface 
$H = \bigcup_{\alpha \not = \rho} L(\cT,m, c)_\alpha \cap L(\cT,m, c)_\rho$ in the base smooth piece
$L(\cT,m, c)_\rho$. To identify the orientation bundles of the smooth pieces along this intersection it suffices to specify that {\em the positive codirection to the co-oriented front is glued to the positive Liouville direction}. Note this choice is invariant under symplectomorphism.


\begin{definition}
The {\em neutral element} $\kappa(\cT, m ,c) \in K(\cT,m, c)$ is given by the above inductively specified gluing.
\end{definition}

\begin{remark}\label{rem: pos char}
Let us mention a characterization (we will not use) of when an arboreal model is positive. Assume for simplicity $\cT$ has a single vertex adjacent to the root. Then  
the Legendrian at infinity $L(\cT, m, c)^\infty \subset S^* \R^m$ projects homeomorphically to its front $H\subset \R^m$, and the front has a natural ``tangent" bundle given by the orthogonal to the coorientation. In the case of positive $\cT$,
the restriction of $\kappa(\cT, m, c)$ to $L(\cT, m, c)^\infty$ coincides with the top exterior power of the  ``tangent" bundle of the front under the above homeomorphism. 
\end{remark}

%
%

\begin{definition}
An {\em orientation structure} $\kappa$ for a
{pre-Ac space} $(X, \cO)$
is a real line bundle $\kappa\to X$ equipped with identifications $\kappa|_Y\simeq \kappa_Y$ for each smooth sheet $Y \to X$,  locally modeled on the neutral elements $\kappa(\cT,m,c) \in K(\cT,m, c)$.
\end{definition}

\begin{definition}

An $m$-dimensional {\em arboreal space with corners (Ac space)} $(X, \cO, \kappa)$  is an $m$-dimensional
 pre-arboreal space $(X, \cO)$ equipped with an orientation structure $\kappa$.
\end{definition}

\begin{remark}
When $m=1$, one can compare the notion of orientation structure with ribbon structure for a trivalent graph. At each trivalent vertex we have a ``T'' singularity and in this case an orientation structure is equivalent to a cyclic orientation of the three edges.  
In particular, the orientation structure rigidifies the  pre-Ac space $(L_{ \cA_2}, \cO_{\cA_2})$.
\end{remark}

         \begin{figure}[h]
\includegraphics[scale=0.6]{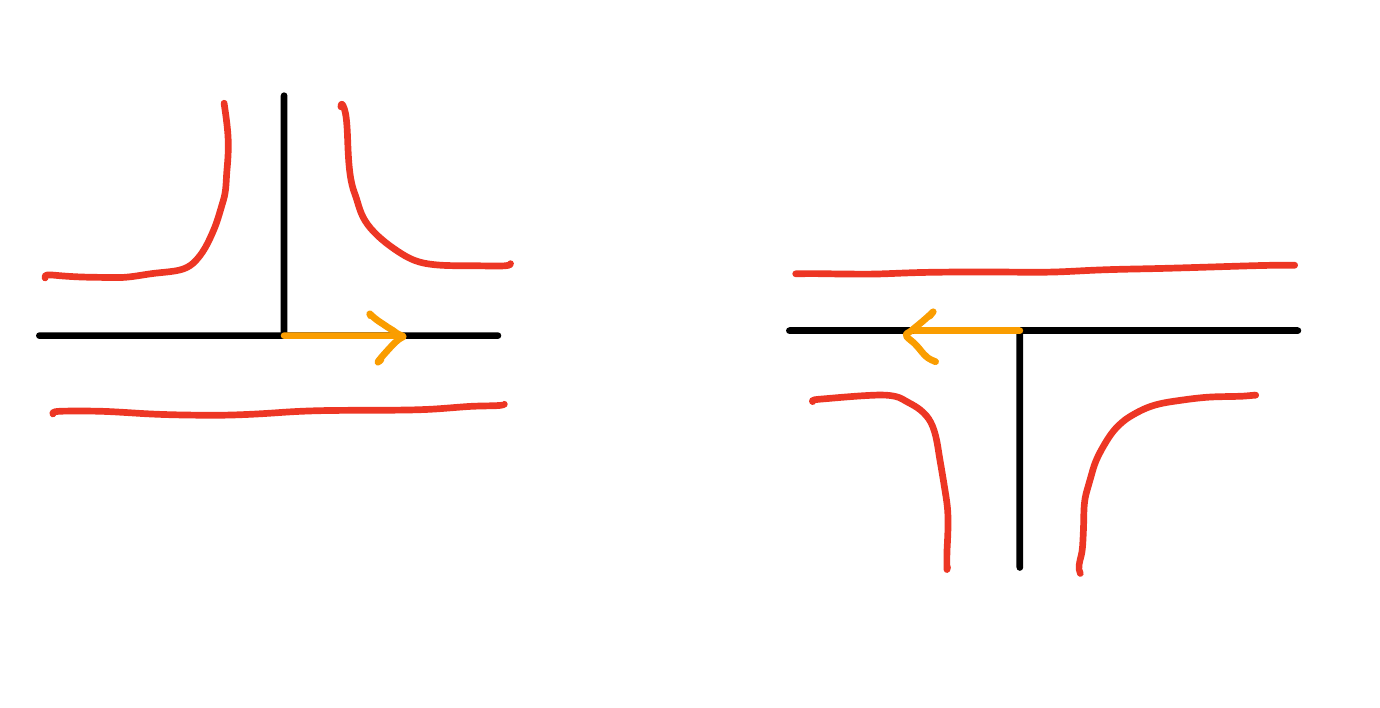}
\caption{For the $A_2$ singularity there are two possible orientation structures, which can be realized by the above embeddings as arboreal spaces into $(\R^2,dp \land dq)$. Since the zero section is fixed, the choice of an orientation structure is naturally equivalent to the choice of a cyclic structure on the three half-edges incident to the vertex. Note that the choice of an orientation structure is also equivalent to a choice of co-orientation for a point inside a line.}
\label{fig:orientstr}
\end{figure}

Recall from \cite{AGEN20a} that an arboreal Lagrangian $L$ in a symplectic manifold $(M, \omega)$ is locally symplectomorphic to a standard model. Since this model has an orientation structure given by its neutral element, and this element is invariant under symplectomorphism, we conclude that any arboreal Lagrangian has an orientation structure induced by the ambient symplectic structure. Hence {\em every arboreal Lagrangian has the structure of an arboreal space}.

Finally, let us return to the discussion at the beginning of this section, and observe that, locally, orientation structures are nothing more than choices of coorientations.
 
\begin{lemma}\label{lem: ambiguity} Let $\cT =(T, \rho, \eps)$ be a signed rooted tree.

There are precisely $2^k$ isomorphism classes of orientation structures on the pre-Ac space $(L_\cT, \cO_\cT)$,
where $k$ is the number of edges adjacent to the root $\rho$.
\end{lemma}

\begin{proof}
By construction, an orientation structure on the local model $(L_\cT, \cO_\cT)$ must result by pullback of the neutral orientation structure under an isomorphism $(L_\cT, \cO_\cT) \simeq (L_{\cT'}, \cO_{\cT'}) $ for some 
$\cT'$.
Moreover, note that $(L_\cT, \cO_\cT)$
is determined  by 
 the hypersurface
$H = \bigcup_{\alpha \not = \rho} L_\alpha \cap L_\rho \subset L_\rho \simeq \R^{n(\cT)}$, and so
we must have  an isomorphism $( \R^{n(\cT)}, H) \simeq ( \R^{n(\cT')}, H')$ of hypersurfaces.
 So we must have $\cT' =(T, \rho, \eps')$ where $\eps'$ results from $\eps$ by
 negating all of the signs along the edges in some collection of components of $\cT\setminus \rho$. 
These choices correspond to the $2^k$ choices of coorientations of the hypersurface $H$. Each such choice provides a distinct Lagrangian realization $L_\cT \subset T^*\R^{n(\cT)}$ with corresponding orientation structure.
\end{proof}

\begin{remark} Note therefore that the data of an orientation structure on an arboreal singularity is equivalent to a fully signed rooted tree, i.e. a signed rooted tree in the previous sense together with additional signs on each of the edges of $\cT$ adjacent to the root, so that now every edge in $\cT$ has a sign. We will however not use this notion of a fully signed rooted tree and will always use the term signed rooted tree for rooted trees with signs on those edges not adjacent to the root. \end{remark}

\subsection{Arboreal Darboux-Weinstein neighborhood theorem}

\subsubsection{Lagrangian subsets}

 Let $L$ be a path connected and locally simply-connected closed subset in a symplectic manifold $(M, \omega)$.
  We call $L$  a {\em Lagrangian subset} if each point of $L$ has a Darboux neighborhood  $U$ such that $\oint_\gamma pdq=0$ for every loop $\gamma$ in $L\cap U$.  
Clearly, any  arboreal Lagrangian $L \subset M$ is  a Lagrangian subset. 
If $f$ is (the germ at a point of) a diffeomorphism $M \to M$, then  $f(L)$ is (the  germ of) a Lagrangian subset if and only if the image under $f$ of each smooth piece of $L$ is Lagrangian.

 If $L  \subset M$ is a pre-Ac space embedded as a Lagrangian subset in a symplectic manifold $(M, \omega)$, then $L$ inherits an orientation structure as follows. Returning to the inductive description of orientation structures given in Lemma~\ref{lem:orient},  along a smooth Lagrangian piece $P \subset L$, given a Lagrangian half-plane $H$ with codimension 1 clean intersection with $P$, we specify that the inward normal direction $\nu$ to $P \cap H$ in $H$ is glued to the normal direction to $P \cap H$ in $P$ which pairs positively with $\nu$ under $\omega$. With this structure, any pre-Ac space embedded as a Lagrangian subset gets an induced structure of an Ac space.

 Note  the above orientation structure on $L$ as a Lagrangian subset agrees with the previously defined orientation structure on $L$ as an arboreal Lagrangian, i.e. the one defined using  local models.

\subsubsection{Darboux-Weinstein theorem}

The main result of this section is the following:

\begin{theorem}\label{thm:arb-Darboux}
Let $V$ be  a manifold with two symplectic forms $\omega_0, \omega_1$.  Suppose $L \subset V$ is a compact   arboreal  space embedded as a Lagrangian subset for both symplectic forms. Let $\kappa_0, \kappa_1$ be the induced orientation structures on $L$.

If the orientation structures are isomorphic $\kappa_0\simeq \kappa_1$, then there exists a diffeomorphism $\varphi:V\to V$,  preserving $L$ as a subset, such that $\varphi^*\om_1=\om_0$. If $\omega_0$ and $\omega_1$ coincide on $\Op A$ for a closed subset $A\subset L$ then $\varphi$ can be chosen to be the identity on $\Op A$.
\end{theorem}

\begin{cor}
Any arboreal space embeds into a symplectic manifold as an arboreal Lagrangian such that the orientation structure induced from the ambient symplectic structure is isomorphic to the given orientation structure.
\end{cor}
\begin{proof} By definition, any point of an arboreal space has a model neighborhood. 
On overlaps, these local symplectic structures have the same orientation structure, and hence  according to Theorem \ref{thm:arb-Darboux} they  are diffeomorphic via a diffeomorphism preserving the arboreal. Moreover, the relative form of that theorem allows us to construct a symplectic neighborhood inductively extending it over neighborhoods of  strata of  certain stratification.
\end{proof}
 
In Theorem \ref{thm:unique-W-for-buildings},  proved below in Section \ref{sec:unique-W-thick}
 below,  we find a Weinstein structure on a symplectic  neighborhood of on arboreal space $L$ for which  $L$ serves as the skeleton.  Moreover,
Theorem \ref{thm:unique-W-for-buildings} ensures uniqueness of this structure  up to  Weinstein homotopy fixing the skeleton.

Theorem \ref{thm:arb-Darboux} is a corollary of the following two propositions. 

\begin{prop}\label{prop:symplectic-normal} Let $V$ be  a manifold with two symplectic forms $\omega_0, \omega_1$.  Suppose $L \subset V$ is a compact   arboreal  space which is a Lagrangian subset   for both symplectic forms. Let $\kappa_0, \kappa_1$ be the induced orientation structures on $L$.
If the orientation structures are isomorphic $\kappa_0\simeq \kappa_1$, then there exists a diffeomorphism  $\phi$ of $\Op(L)$ preserving $L$ pointwise such that the linear interpolation $\omega_t=(1-t)\omega_0 + t \phi^* \omega_1$ is a family of symplectic forms on $\Op(L)$ so that $L$ is a Lagrangian subset for each $\omega_t$. 
If $\omega_0$ and $\omega_1$ coincides on $\Op A$ for a closed subset $A\subset X$ then $\phi$ can be chosen to be the identity on  $\Op A$.\end{prop} 


  \begin{prop}\label{prop:hom-to-Moser}
  Let $L\subset V$ be an arboreal subspace which is a Lagrangian subset for a family $\om_t$, $t\in[0,1]$, of symplectic forms on $L$. Then there exists a diffeotopy $\h_t: \Op(L)\to \Op(L)$, $t\in[0,1]$, such that $h_t(L)=L$ and $h_t^*\om_t=\om_0$.
  If $\om_t=\om_0$ on $\Op A$ for a closed subset $A\subset L$ then the isotopy $h_t$ can be chosen fixed on $\Op A$.   
  \end{prop}
  
 Indeed, to prove Theorem \ref{thm:arb-Darboux} we can apply Proposition \ref{prop:hom-to-Moser} to the output of Proposition \ref{prop:symplectic-normal}. 
  We will prove Propositions \ref{prop:symplectic-normal} and \ref{prop:hom-to-Moser} in the next two subsections. 
  
  \begin{remark}
  As the proof of Proposition \ref{prop:symplectic-normal} will show, the conclusion holds more generally for two symplectic manifolds $(V_i, \omega_i)$, $i=1,2$, together with embeddings of a compact arboreal space $L$ onto a Lagrangian subset of $V_i$, $i=1,2$. In particular it follows that the smooth topology of a neighborhood of a Lagrangian compact arboreal space in a symplectic manifold is uniquely determined by the compact arboreal space. 
  \end{remark}
  
\subsubsection{Proof of Proposition \ref{prop:symplectic-normal}}
Let us revisit the notion of an orientation structure. Consider   a full flag $F= (0=V_0\subset V_1\subset\dots\subset V_{2n}=V)$ of subspaces in a $2n$-dimensional vector  space $V$, $\dim V_j=j$. 
Let  $\Om(F)$ be the space of all symplectic forms  $\om$ on $V$ such that    $V_j$  is  isotropic for $j\leq n$,  $V_j$ is coisotropic for $j\geq n$, and  $V_j^{\perp_\om}=  V_{2n-j} $ for $j\leq n$.

We begin with following elementary assertion.

 \begin{lemma}\label{lm:arb-lin-alg}
 The connected components of $\Om(F)$ are determined by orientations of  the symplectic quotients $V_{2n-j}/V_j$, for $j=0, \ldots, n-1$. Any sequence of orientations can occur and each connected  component is convex.
 \end{lemma}
 
 \begin{proof}
Represent  $\omega$ as a matrix with respect to a basis $u_1, \ldots , u_{2n}$ of $V$ such that $V_j=\text{span}(u_1, \ldots , u_j)$. It has a block form with four $n$ by $n$ blocks, the two diagonal blocks being zero and the two off-diagonal blocks being $A$ and $-A$. Moreover, $A$ is a triangular matrix and its diagonal elements must be non-zero. Their signs naturally correspond to orientations of the symplectic quotients.
\end{proof}
  
We use the notation of \cite{AGEN20a} for the explicit construction of the extended local model 
  $$
\xymatrix{
{}^{n} L= \bigcup_{i = 0}^n {}^n L_i \subset T^*\R^n
}
$$
in terms of equations in the canonical coordinates $x_0, \ldots, x_{n-1}, p_0, \ldots, p_{n-1}$.
Introduce the flag $F({}^n L) = (0 = V_0 \subset V_1 \subset \cdots \subset V_{2n} = T^*\R^n)$ given  by the assignments
  $$
  \xymatrix{
  V_i = T_0(\bigcap_{j = 0}^{n-i} {}^nL_j) = \Span(\p_{x_{n-1}}, \ldots, \p_{x_{n-i}})  & i\leq n
  }
  $$
  $$
  \xymatrix{     
   V_{2n-i} = \Theta_0(\bigcup_{j = 0}^{n-i} {}^nL_j)  =
   \Span(\p_{x_{n-1}}, \ldots, \p_{x_{0}}, \p_{p_0}, \ldots, \p_{p_{2n-i}})  &  i \leq n
  }
  $$
  Here $\bigcap_{j = 0}^{n-i} {}^nL_j$ is a submanifold of $T^*\R^n$, in fact a linear subspace of ${}^n L_0 = \R^n$, and $T_0(\bigcap_{j = 0}^{n-i} {}^nL_j)$ denotes its tangent space at the origin; $\bigcup_{j = 0}^{n-i} {}^nL_j$ is not a submanifold, but has a tangent sheaf  in the sense of algebraic geometry, and $\Theta_0(\bigcup_{j = 0}^{n-i} {}^nL_j) = (\cI_0/\cI_0^2)^\vee$ denotes its fiber at the origin. 
Note these characterizations imply $F({}^n L)$  is preserved by diffeomorphisms of $T^*\R^n$ preserving ${}^n L$ as a subset hence preserving each ${}^n L_i$ as a subset, i.e. $F({}^n L)$ only depends on the sheaf of smooth functions on ${}^nL$. Note also   $V_i$  is  isotropic for $i\leq n$,  $V_i$ is coisotropic for $i\geq n$, and  $V_i^{\perp_\om}=  V_{2n-i} $ for $j\leq n$ for any symplectic form on $\R^{2n}$ such that ${}^nL \subset \R^{2n}$ is Lagrangian.
%
%
%

Recall that the local model for an $\cA_{n+1}$-arboreal Lagrangian is a subset of the extended local model: $L_{\cA_{n+1}} \subset {}^n L$. We have the following reformulation of orientation structures for Ac-space structures on the germ of $L_{\cA_{n+1}}$ at its central point:

\begin{lemma}\label{l:orient reform}
Let $\omega_0$ be a symplectic form on $\R^{2n}$ for which $L_{\cA_{n+1}}$ is a Lagrangian subset.

 If $\omega_1$ is another such form, then the orientation structures $\kappa_0, \kappa_1$ induced by $\omega_0$, $\omega_1$ respectively  on $L_{\cA_{n+1}}$ are isomorphic if and only $\om_0$ and $\om_1$ lie in the same connected component of $\Omega(F({}^n L))$. 

\end{lemma}

\begin{proof}
Let $\omega$ be a symplectic structure on $\R^{2n}$ such that ${}^nL$ is Lagrangian. The sign of the leading diagonal element of the matrix $A$ in the proof of Lemma \ref{lm:arb-lin-alg} applied to $\omega$ is precisely the pairing used to define the orientation structure on ${}^n L$ as a Lagrangian subset of $(\R^{2n}, \omega)$. The other signs are uniquely determined by quadratic information. For example, for  $L_{\cA_3} \subset T^*\R^2$ the third stratum is the positive conormal to $\{x_0=x_1^2 \, \, x_1 \geq 0 \}$. Therefore, the vectors $2x_1 \p/\p x_1 + \p/\p x_0$ and $\p/\p p_1 + 2x_1 \p/\p p_0$ should span a Lagrangian subspace. Their symplectic product is  $2x_1(\lambda- \mu)$ where $\lambda$ is the first diagonal term of the matrix $A$ and $\mu$ is the second diagonal term of the matrix $A$. So in fact $\lambda = \mu$. The general case can be proved similarly by induction. The conclusion is that the connected component of $\Omega(F({}^n L))$ is uniquely determined by the orientation structure, which in this case is just a single sign. 
\end{proof}

For more general trees $\cT$, one can apply the above argument to each factor $T^*\R^n$ in $T^*\R^{n(\cT)}$ corresponding to each $\cA_n$-subtree of $\cT$ with the same root $\rho$ and deduce the general version:

\begin{lemma}\label{l:orient reform}
Let $\cT$ be a signed rooted tree and $\omega_0$ a symplectic form on $\R^{2n(\cT)}$ for which $L_\cT$ is Lagrangian.

 If $\omega_1$ is another such form, then the orientation structures $\kappa_0, \kappa_1$ induced by $\omega_0$, $\omega_1$ respectively on $L_{\cT}$ are isomorphic if and only if for any $\cA_{m+1}$-type subtree $\cA_{m+1} \subset \cT$ with the same root $\rho$, the restrictions of $\om_0$ and $\om_1$ to the factor $\R^{2m}$ of $\R^{2n(\cT)}$ corresponding to the $\cA_{m+1}$-subtree lie in the same connected component of $\Omega(F({}^m L))$.

\end{lemma}

So we have proved Proposition \ref{prop:symplectic-normal} for the germ of the central point of the pre-Ac space $L(\cT,m)$ with $m=n(\cT)$, i.e. not stabilized. In the case $d=m-n(\cT)>0$, the situation changes as follows. Given a symplectic form $\omega_0$ on $\R^{2m}$ for which the standard model $L_{\cT} \times \R^{d} \subset \R^{2 n(\cT)} \times \R^{2d}$ is Lagrangian, consider $E \subset \R^{2m}$, the $(2m-d)$-dimensional coisotropic subspace of $\R^{2m}$ which is symplectically orthogonal to the isotropic $F=0 \times \R^{d}$. Since $L_\cT \times \R^d$ is assumed to be Lagrangian, note that $F$ is spanned by those directions tangent to the smooth part of $L_\cT \times \R^d$ and $E$ is spanned by those directions tangent to $L_\cT \times \R^d$ at $0$, hence both are independent of the symplectic form $\omega_0$.

The symplectic reduction of $(E,\omega_0|_E)$ is a $2(m-d)$-dimensional symplectic vector space $V_0$. The quotient $\R^{2m}/E$ is canonically identified via the fixed symplectic form $\omega_0$ with the dual of $0 \times \R^{d}$. Therefore there exists a canonical symplectic isomorphism $\Phi$ between $\R^{2m}$ and $V_0 \oplus (F \oplus F^*)$.

Now suppose that we have two such symplectic forms $\omega_0$ and $\omega_1$ on $\R^{2m}$ for which $L_\cT \times \R^d$ is Lagrangian. 

\begin{lemma}\label{lem:linear-contractible} Consider $\cS=\cS(\omega_0,\omega_1)$, the space of linear isomorphisms $\Psi:\R^{2m} \to \R^{2m}$ such that:

\begin{enumerate}
\item $\Psi|_{L_\cT \times \R^{d}} = \text{Id}_{L_\cT \times \R^d}$.
\item $\Psi^*\omega_1$ and $\omega_0$ agree on $F \oplus ( \R^{2m}/ E)$.

\end{enumerate}

Then $\cS$ is nonempty and contractible. 

\end{lemma}

\begin{example}
For the smooth arboreal singularity, $\cS$ consists of a single element.
\end{example}

\begin{proof} 
For the existence, write $\R^{2m}=\R^{2n(\cT)} \times \R^d \times \R^d$. Consider a linear isomorphism $\Psi: \R^{2m} \to \R^{2m}$ which is the identity on the first two factors. This induces a linear isomorphism between the $V_0$ and $V_1$ corresponding to $\omega_0$ and $\omega_1$. The further condition that $\Psi$ respects the identification $V_0 \oplus (F \oplus F^*) \simeq V_1 \oplus (F \oplus F^*)$ uniquely determines a linear isomorphism $\Psi$ satisfying (i) and (ii).

For the contractibility, fix $\Psi_0 \in \cS$. For any other $\Psi_1 \in \cS$, we will write the matrix for $\Psi_1^{-1} \Psi_0$ in block form with respect to the decomposition $\R^{2m}=\R^{2n(\cT)} \times \R^d \times \R^d$, with the canonical basis. The matrix will have block form, with the diagonal factors the identity, and with all off-diagonal blocks zero except for the $(3,2)$ block, which is an arbitrary symmetric matrix $A$. The space of such is contractible, so this completes the proof.  \end{proof}

\begin{lemma}\label{lem:convex-interpol}
Let $\omega_0, \omega_1$ be symplectic forms on $\R^{2m}$ such that $L_\cT \times \R^d$ is Lagrangian and which induce isomorphic orientation structures. Let $\Psi \in \cS$. Then $(1-t)\omega_0 + t \Psi^* \omega_1$ is symplectic for all $0 \leq t \leq 1$.
\end{lemma}
\begin{proof}

On $\R^{2m}/V_0$, the symplectic forms induced by $\omega_0$ and $\Psi^* \omega_1$ agree. On $V_0$ the problem is reduced to the non-stabilized case as discussed above: the homotopy $(1-t)\omega_0 + t \Psi^*\omega_1$ will be through symplectic forms whenever the induced orientation structures agree. Indeed, the argument is unaffected by the symplectic isomorphism $\Psi$ since it is the identity on $L_\cT \times \R^d$. 
\end{proof}

Now we are ready to prove Proposition \ref{prop:symplectic-normal}. Let $L \subset V$ be a compact arboreal space which is Lagrangian for two symplectic forms $\omega_0$ and $\omega_1$. By Lemmas \ref{lem:linear-contractible} and \ref{lem:convex-interpol} we may construct a family of linear isomorphisms $\Psi_x:T_xV \to T_xV$, $x \in L$, such that $\Psi_x $ is the identity on directions tangent to $L$ at $x$, and such that $(1-t)\omega_0(x) +t \Psi_x^* \omega_1(x)$ is symplectic for all $x \in X$ and $0 \leq t \leq 1$. We may then integrate the family $\Psi_x$ to a diffeomorphism $\Psi:\Op(L) \to \Op(L)$ such that $\Psi|_L=\text{Id}_L$ and such that $(1-t) \omega_0 + t\Psi^* \omega_1$ is a homotopy of symplectic forms on $\Op(L)$.  This completes the proof.

\subsubsection{Proof of Proposition \ref{prop:hom-to-Moser}}
\begin{lemma}\label{lm:ext-Liouv-restr} 
Let $L\subset T^*M$ be a Lagrangian subset which is the union $M\cup C(\Lambda)$ where $C(\Lambda)$ is the Liouville cone over a Legendrian $\Lambda$.
Let $\lambda$ be a 1-form on $T^*M$ such that $\lambda|_M=0$ and $d\lambda=0$    on $L$.  Then  there exists a smooth function $H:T^*M\to\R$ vanishing on $M$  together with its differential $dH$, and such that
$\lambda$ coincides with $dH$ on vectors tangent to $L$. 
\end{lemma}
\begin{proof}
For any point $p\in T^*_qM$ denote by $\gamma_{p,q}$ the path $t\mapsto (tp,q),\;t\in[0,1]$. Define the function $H$ by the formula
$$H(p,q)=\int\limits_{\gamma_{p,q}}\lambda. $$
 For $(p,q)\in L$ the path $\gamma_{p,q}$ is contained in $L $, and hence   on tangent vectors to $L$ the $1$ forms   $dH$ and $\lambda$ coincide.
\end{proof}
\begin{lemma}\label{lm:conic-Darboux}
Let $L\subset T^*M$ be a Lagrangian subset which is the union $M\cup C(\Lambda)$. Let $A\subset M$ be a closed subset and   $\om_t$ be a family of symplectic forms on $T^*M$ such that $\om_0=d(pdq)$, $\om_t|_{\Op A}=\om_0$, and $L$ is $\om_t$-Lagrangian for all $t$.
Then there exists an isotopy $f_t:T^*M\to T^*M$ which is fixed on $M\cup\Op A$, leaves $L$ invariant and such that $f_t^*\om_t=\om_0$.
\end{lemma}
\begin{proof} 
There exists a family of Liouville forms $\lambda_t$ for $\om_t$ such that $\lambda_0=pdq$,
$\lambda_t|_{T^*M|_M}=0$ and $\lambda_t|_{\Op A}=pdq|_{\Op A}$.
By Lemma \ref{lm:ext-Liouv-restr}    there exists a family of functions $H_t$ on $T^*M$ such that $dH_t|_L=\lambda_t|_L$  $H_t|_{ M\cup \Op A}=0$ and $dH_t|_{T^*M|_M}=0$.  Define $\mu_t:=\lambda_t-dH_t$.
Then $\mu_0=pdq$,
$\mu_t|_{T^*M|_M}=0$, $\mu_t|_{\Op A}=pdq|_{\Op A}$ and $\mu_t|_{L}=0$.
Then Moser's homotopical method yields an isotopy $f_t:T^*M\to T^*M$ which is fixed on $M\cup\Op A$, leaves $L$ invariant and such that $f_t^*\om_t=\om_0=d(pdq)$. \end{proof}

Note that any point of an arboreal space which is realized as a Lagrangian subset  of a symplectic manifold $X$ has a neighborhood   which has the conical form of Lemma \ref{lm:conic-Darboux}. Hence, we can take a finite covering of $X$ by such neighborhoods and successively apply Lemma \ref{lm:conic-Darboux} to extend the isotopy.
This concludes the proof of Proposition \ref{prop:hom-to-Moser}, and with it the proof of Theorem \ref{thm:arb-Darboux}.


\subsection{Weinstein thickening of an Ac-building}

\subsubsection{Ac-buildings}
  
Similar to  the  notion  of a Liouvllle or a Weinstein manifold with  corners, by a {\em symplectic manifold with  corners} we mean a possibly non-compact  manifold with 
corners, such  that near  each corner  point it is symplectomorphic to $\R^{2m}\times T^*\R_+^{n-m}$.

An  {\em arboreal Lagrangian with boundary with corners}, or  {\em ac-Lagrangian} for short, is a closed subset of a symplectic manifold $W$  with  corners  such that the germ of  $(W,L)$ at a  point
$x\in L$ is  symplectomorphic to  the  germ of $(\R^{2n-2c}\times T^*\R^c_+, L_\cT\times   \R_+^c)$  at the origin  for some signed rooted  tree $\cT$ and   an  integer    $c\geq 0$.   The $c$ is called the multiplicity of  the corner.   It follows from the definition  that corners of $L$  are contained in the corresponding  corners  of $W$.
As in the case of manifolds with boundary with corners,  we denote by  $\p_jL$ the strata of  corner  points of order  $j\geq 1$, and  write  $\p L=\bigcup_j\p_jL$.
The closure of each  component of  $\p_jL$ is  itself an ac-Lagrangian in a symplectic submanifold of $W$  of codimension $2k$. The  closures of  components of $\p_1L$ are called  {\em boundary faces} of  $L$.

It will also be useful to consider the notion of an {\em Ac-building}, a special kind of Ac-space whose definition is analogous to that of a cotangent building as a special kind of Wc-manifold. Recall that an Ac-space $L$ is the union $L=\bigcup_j L_j$ of smooth pieces $L_j$, and each point of $L$ belongs to the interior of exactly one $L_j$.

\begin{definition}
An {\em Ac-building} is an Ac-space $L$ whose smooth pieces can be ordered $L_1 , \ldots , L_k$ so that:
\begin{enumerate}
\item Each $L_{>j}=\bigcup_{i>j}L_j$ is an Ac-space with boundary and corners.
\item $L_{>j-1}$ is obtained from $L_{>j}$ by gluing a boundary component $P_j$ of $L_{>j}$ to $L_j$ along a map $P_j \to L_j$.
\end{enumerate}
\end{definition}

\begin{lemma}\label{lem:lifts-Ac}
Let $K=\bigcup_j L_j$ be an Ac-building. Then the gluing maps $ P_j \to L_j$ can be canonically lifted to embeddings $P_j \to T^*L_j \setminus L_j $ onto arboreal Legendrians.
\end{lemma}

\begin{proof}
This follows from the characterization of orientation structures in terms of co-orientations of fronts, as discussed above.
\end{proof}


\begin{thm}\label{thm:unique-W-for-buildings}\label{sec:unique-W-thick}
 Suppose $(W, \omega)$ is a $2m$-dimensional symplectic manifold, and  let
 $L\subset W$ be an Ac-bulding.  Then   \begin{enumerate}
 \item a neighborhood of $L$ admits a Weinstein structure with  $L$ as its skeleton, 
 \item the germ near $L$ of such a structure is unique up to Weinstein homotopy with fixed skeleton $L$.
 \end{enumerate}
\end{thm}
An assumption that $L$ is an Ac-building  rather than a general Ac-space simplifies the proof, but  is not necessary.  We will not need the more general result and will not prove it in this paper.

\subsubsection{Existence}
\begin{proof}[Proof of (i)]
Consider the building presentation $L_k\to L_{k-1}\to\dots\to L_1$ of $L$.
We will inductively construct a cotangent building structure  with $L$ as its skeleton.
 Using Theorem \ref{thm:arb-Darboux} this structure can be transported to $\Op L\subset W$.
 
We argue by induction in the number $k$ of blocks. Suppose that we already constructed  a cotangent  building $W_{>1}:=B_k\to\dots\to B_2$  with the skeleton $L_{>1}$ which consists of blocks $B_j=\sT^*L_j$.
 By Lemma \ref{lem:lifts-Ac} there exists a boundary face $Y$ of $ L_{>1}$ such that  $L$ is obtained by attaching $L_{>1}$ to $L_1$ using a front embedding $Y\to L_1$, i.e. an embedding  which factors  as $Y\mathop{\to}\limits^\phi T^*L_1\setminus L_1$, where $\phi$ as  an embedding onto a Legendrian arboreal. The Ac-space $Y$ serves as the nucleus $Q$ of a boundary face of $W_{>1}$.
 The embedding $\phi$ extends to an embedding $\Phi:Q\to T^*L_1\setminus L_1$ as a Weinstein hypersurface. The required building $L$ can be now obtained by the vertical gluing of $W_{>1}$ to $\sT^*L_1$.
 \end{proof}
 
 \subsubsection{Liouville fields on cotangent bundles}
  
 Before proving the second part of Theorem \ref{thm:unique-W-for-buildings}
  we need to review some facts about Liouville fields on cotangent bundles, see also 
  \cite[Sect. 12.3]{CE12}.

Let $Y$ be an $n$-dimensional bc-manifold, $T^*Y$ its cotangent bundle  with the standard Liouville form $\lambda_{\st} = pdq$ and symplectic form  $\om_{\st}=d(pdq)$, and Liouville field $Z_\st=p\frac{\p}{\p p}$.  Any other Liouville form $\lambda$ has the form $\lambda_\st-dH$  for a function $H:T^*M\to \R$. Respectively, the Liouville field $Z$ of $\lambda$ is given by $Z=Z_\st+X_H$, where $X_H$ is the Hamiltonian field of $H$. 

 We will always suppose in this section  that $Z$ is tangent to $Y$. This is equivalent to the condition $H|_Y=0$. Denote   $Z':=Z|_Y$, $H'=H|_Y$. 
\begin{lemma}
\label{lm:local-near-0}
Let $\lambda=pdq-dH$ a Liouville form on $T^*Y$ with $H|_Y=0$.
Let $a\in Y$ be a zero of $Z'=Z|_Y$. Choose local canonical coordinates $q_1,\dots, q_n,p_1,\dots, p_n$ centered at $a$. Write $H=\frac12\sum\limits_{i,j=1}^na_{ij}p_ip_j
+\sum\limits_{i,j=1}^nb_{ij}p_iq_j+o(|p|^2+|q|^2)$. Then $$Z'=\sum\limits_{i,j=1}^nb_{ij}q_j\frac{\p}{\p q_j}+O(|q|^2), Z=\left(\sum\limits_{i,j=1}^nb_{ij}q_j+\sum\limits_{i,j=1}^na_{ij}p_j\right)\frac{\p}{\p q_i}+\sum\limits_{i,j=1}^n(\delta_{ij}-b_{ij})p_i\frac\p{\p p_j}+O(|p|^2+|q|^2).$$ 
The matrix of  $d_aZ'$ is equal to  $B$ and the matrix of  $d_aZ$ in the basis $\frac{\p}{\p q},\frac{\p}{\p p}$ has the form
$$\left(\begin{matrix} B&A\\0&I-B^T\end{matrix}\right)$$ where $B=(b_{ij})$, $A=(a_{ij})$.
\end{lemma}

\begin{proof} Direct computation. \end{proof}

  \begin{definition} A zero $a\in Y$ of $Z'$ is called {\em transversely non-degenerate}
if the matrix $I-B^T$ has no pure imaginary  eigenvalues.
 \end{definition}

\begin{lemma}\label{lm:arb-core-eigenvalue}
Let $(W,Z)$ be a Liouville structure with an arboreal skeleton $X$.
Let  $a\in X$   be a zero of $Z$ and $Y$  the  smooth piece corresponding to the root of the arboreal singularity at the point $a$. Suppose  that  $a$  is  transversely  non-degenerate with respect to $Y$. Denote by $Z'$ the restriction $Z|_Y$. Then the eigenvalues $\lambda_j$ of the differential  $d_aZ'$ have  real parts $<1$. 
\end{lemma}
\begin{proof}
By assumption  $Z$ is tangent to $Y$, and hence $Z$ can be written near the zero $a$ as $$Z=\left(\sum\limits_{i,j=1}^nb_{ij}q_j+\sum\limits_{i,j=1}^na_{ij}p_j\right)\frac{\p}{\p q_i}+\sum\limits_{i,j=1}^n(\delta_{ij}-b_{ij})p_i\frac\p{\p p_j}+O(|p|^2+|q|^2).$$ The transverse non-degeneracy condition means the matrix $C:=I-B^T=(\delta_{ij}-b_{ji})$ does not have pure imaginary  eigenvalues. If ${\mathrm Re\,} \lambda<0$ for one of eigenvalues of $C$, then the stable manifold of $a$ must contain a curve tangent to  the corresponding eigenvector.  But   the arboreal  $X$  cannot contain  any  (two-sided) curves transverse to the root smooth piece $Y$.
\end{proof}
\begin{lemma}\label{lm:grad-like}
Let $Y$ be a compact manifold, $v$ a vector field, and $\ell:Y\to\R_+$ a positive   Lyapunov function for $v$.
 Let $A$ be the set of critical points of $\ell$. Then for any neighborhood $U\supset A$ and any constant $C>0$ there exists a function $\theta:\R\to\R$ with $\theta'>0$ such that   on $Y\setminus U$ there holds:
 \begin{align}\label{eq:grad}
 d(\theta\circ\ell)(v)>C\, \theta\circ\ell
 \end{align}

\end{lemma}
\begin{proof}
Let $c_0<\dots<c_k$ be critical values of $Y$.  Choose a sufficiently small  $\sigma>0$ and denote $B_j:=\{c_j-\sigma \leq \ell\leq c_j+\sigma\}$, $C_j:= \{c_j+\sigma \leq \ell\leq c_{j+1}-\sigma\}$. Note that each $C_j$ is a trivial cobordism bounded by regular level sets $\p_- C_j =\{\ell=c_j+\sigma\}$ and
 $\p_+C_j=\{\ell=c_{j+1}-\sigma\}$. Denote   $\delta_0:=[c_0, c_0+\sigma], \Delta_0:=[c_0+\sigma, c_1-\sigma],\delta_1:=[c_1-\sigma,c_1+\sigma],  \Delta_1:=[c_1 + \sigma, c_2-\sigma], \dots, \Delta_{k-1} := [c_{k-1}+\sigma, c_k - \sigma], \delta_k:=[c_k-\sigma, c_k]$. Define a diffeomorphism 
 $\psi_j:\p_-C_j\times\Delta_j\to C_j$ by sending vertical intervals to the corresponding trajectories of $v$ parameterized by $\ell$, so that we have $\ell(x,u)=u,\;\; x\in\p_-C_j,u\in\Delta_j$.
 Denote $h(x,u):=d_{x,u}\ell(v)$. 
We have $d(\theta\circ\ell)(v)=\theta'(u)h(x,u).$ 
We will successively define $\theta$  on intervals. Assuming $\sigma$ so small that $B_0\subset U$ we define $\theta(t)=t$ on $\delta_0$. Identifying $C_0$ with $\p_-C_0\times\Delta_0$ via $\psi_0$, the inequality \eqref{eq:grad} takes the form
$$\theta'(u)h(x,u)>C\theta(u)$$
or
$$\frac{d\ln\theta(u)}{du}>\frac C{h(x,u)}.$$
Choose a  smooth  positive function $\wt h(u)<\mathop{\min}\limits_xh(x,u)$.
Solving the  equation
$$\frac{d\ln\theta(u)}{du}=\frac C{\wt h(u)}.$$ we get
$\theta(s)=\theta(c_0+\sigma)e^{H(u)}$, where $H(u)=\int_0^u\wt h(s)ds.$
 Note that we can assume that all the trajectories of $v$ in  $B_j\setminus U$ begin at $\{\ell=c_j-\eps\}$ and end at $\{\ell=c_j+\eps\}$, i.e. there exists a diffeomorphism $(\p_- B_j\setminus U)\times\delta_j\to  B_j\setminus U$.  Hence we can define $\theta$ on $[c_0,c_k]$, by   repeating the same argument  successively  to $\delta_0,\Delta_0,\delta_1,\Delta_1, \dots,\delta_k. $ 
\end{proof}

\begin{lemma}\label{lm:taming}
 Let $Z=p\frac{\p}{\p p}+X_H$ be a  Liouville field on $T^*Y$. Suppose the vector field $Z'=Z|_Y$   is gradient like for some function $\ell:Y \to \R$.  Suppose for all zeroes $a \in Y$ of the vector field $Z'$ the eigenvalues of the differential $d_aZ'$ have real part $<1$. 
 Then  there exists a Riemannian  metric on $Y$ such that    $Z\dot |p|^2\geq \eps |p|^2$ on a neighborhood of $Y$.   \end{lemma}
   \begin{proof}
   Let us start with any background metric on $Y$. Recall that the vector field $Z'$ is gradient like for a function $\ell:Y\to\R$, which can be assumed positive.
  According to Lemma \ref{lm:taming}, near each zero we have
$$Z=\left(\sum\limits_{i,j=1}^nb_{ij}q_j+\sum\limits_{i,j=1}^na_{ij}p_j\right)\frac{\p}{\p q_i}+\sum(\delta_{ij}-b_{ij})p_i\frac{\p}{\p p_j}+O(|p|^2+|q|^2).$$  By assumption, all eigenvalues of the matrix $B$ have real parts $<1$, and hence  the matrix $I-B^T$ has eigenvalues with the real part $>0$, which  implies
\begin{align}\label{eq:bound-near-0}Z_\cdot|p|^2>\eps |p|^2
\end{align} on a 
neighborhood $U$ of the  compact set of  all zeroes of $Z$.  

Let $\wh Z$ be the Hamiltonian extension of $Z'$ to $T^*Y$ with the Hamiltonian
$p(Z'(q))$.
Note that $Z=\wh Z+\wt Z$, where $\wt Z=O(|p|)$. Hence $\wt Z\cdot|p|^2=O(|p|^2)$.
We also note
that $\wh Z\cdot |p|^2=O(|p|)^2$. Hence, $Z\cdot |p|^2=O(|p|)^2$. Denote $\Delta_1:=\max(|\frac{Z\cdot|p|^2}{|p|^2}|)$. 
 Given any point $a\in Y\setminus U$ we choose  local coordinates centered at $a$ in a neighborhood $U_a$ and write $|p|^2=\sum c_{ij}(q)p_ip_j$. 

Denote $$\Delta_2 := \mathop{\max}\limits_{a\in Y\setminus U}\mathop{\max} \limits_{i,j; q\in U_a}|Z'\cdot c_{ij}(q)|. $$ Set $K=\max(\Delta_1,\Delta_2)$.
According to Lemma \ref{lm:grad-like} there exists a function $\theta:\R\to[1,\infty]$ such that the function $\theta\circ\ell$ satisfies on $Y\setminus U$ the condition
$$Z'\cdot(\theta\circ\ell)>3K \, \theta\circ\ell.$$
Let us pull-back  the function $\ell$ to $T^*Y$ via the projection $T^*Y\to Y$. We will keep the notation $\ell$ for the extended function. Note that
 $\wh Z(p,q)\cdot \ell=Z'(q)\cdot\ell$ and $\wt Z\cdot\ell=O(|p|)$.

 Then we have
\begin{align*}
|Z\cdot(\theta\circ\ell(q))|p|^2\geq \left(3K\theta\circ\ell(q) -2K\theta\circ\ell(q)\right)|p|^2+\frac{\p}{\p p}\cdot |p|^2+o(|p|^2)\geq \sigma|p|^2,\;\;\sigma>0,
\end{align*}
if $|p|$ is small enough.

We need to check that the conformal scaling did not destroy the estimate  on $U$.\begin{align*}
&Z\dot( \theta\circ\ell|p|^2)\geq \eps \theta\circ\ell|p|^2+ (Z\cdot\theta\circ \ell)|p|^2\\
&= \eps \theta\circ\ell|p|^2+(\wh Z\cdot\theta\circ \ell)|p|^2+(\wt Z \cdot\theta\circ \ell)|p|^2.
\end{align*}
 But $$\wh Z\cdot\theta\circ \ell)(p,q)=\wh Z'\cdot\ell(q)\geq 0,$$ while $$|\wt Z\cdot(\theta\circ \ell)|\leq C|p|.$$
 Hence, we get 
 \begin{align*}
Z  \cdot( \theta\circ\ell |p|^2)\geq  \eps\theta\circ\ell |p^2| -C|p|^3,
 \end{align*}
 and for $|p|<\frac1{2C}$ we get $Z  \cdot( \theta\circ\ell |p|^2)\geq  \eps\theta\circ\ell |p^2|$ for a reduced $\eps$.
Hence,  rescaling the metric $|p|^2$ with the conformal factor $\theta\circ\ell$ and possibly reducing the size of the neighborhood of $Y$  we get the required metric which satisfies the inequality  $Z\dot |p|^2\geq \eps |p|^2$.
 \end{proof}
  
  \begin{lemma}\label{lm:taming2} Let $Z,\ell$  and the norm $
  | \cdot |$  be as in Lemma \ref{lm:taming}
  We pull-back $\ell $ to $T^*Y$  via the  projection $T^*Y\to Y$ and will keep the notation $\ell$ for the extended function.
  Then 
    there exists a constant $K>0$ such that the function $\wh\phi:=\ell+K|p|^2$ is a Lyapunov function for $Z$ on $\Op Y$.
  \end{lemma}
  \begin{proof}
  It is sufficient to check the Lyapunov condition only in s neighborhood of the zero locus $A$ of $Z$. Indeed, elsewhere it is true by assumption along $Y$, and hence, by openness of the taming condition in its neighborhood.
  As in the proof of Lemma \ref{lm:taming} we  extend  $ Z'$   to $T^*Y$ as a Hamiltonian vector  field $\wh Z$ with the Hamiltonian function $\zeta(p,q)=p(Z'(q))$.    Denote $\wt Z:=Z-\wh Z-\frac{\p}{\p p}$. We also assume that  the function    $\ell$
 is extended to $T^*Y$ by pulling it back from $Y$ by the projection $T^*Y\to Y$.
  We have 
  \begin{align*} Z\cdot \wh\phi =KZ\cdot|p|^2+Z\cdot\ell\geq K\eps|p|^2+\wh Z\cdot\ell+
  \wt Z\cdot\ell.  
  \end{align*}
  
  Furthermore, $$\wh Z(p,q)\cdot\ell=Z'(q)\cdot\ell\geq \eps_1(|d\ell|^2+|Z'(q)|^2),$$ as $\ell$  is Lyapunov for  $Z'$.
  
  We also have $$ |\wt Z\cdot\ell|\leq C_1|d\ell||p|,\;\;|\wh Z(p,q)-Z'(q)|\leq C_2|d\ell(q)||p|. $$
  Thus, 
   \begin{align*} &d\wh\phi(Z)\geq  K\eps|p|^2+ \eps_2\left(|Z'(q)|^2+|d\ell(q)|^2\right) -C_3|d\ell| |p|\\
   &\geq \eps_3(|p|^2+|d\ell(q)|^2+|Z'(q)|^2),
  \end{align*}
  if $K$ is chosen  large enough.
  We also have
  $$||Z||^2\leq C_2(|p|^2+|Z'(q)|^2)\;\;\hbox{ and}\;\;
  |d\wh\phi|^2\leq C_3( K^2|p|^2+|d\ell(q)|^2).$$
  Hence,  $$d\wh\phi(Z)>\eps_4(|Z|^2+|d\phi|^2).$$
  \end{proof}
  
     \subsubsection{Deformations}
\begin{proof}[Proof of Theorem~\ref{sec:unique-W-thick}(ii)]

We will construct a Weinstein homotopy between the given Weinstein structure 
$(W, \omega, \lambda, Z, \phi)$ and the cotangent building structure constructed in  the proof of (i).
Again, we will argue by induction in the number $k$ of blocks. The inductive argument will simultaneously establish the base case of the induction, which is simpler.
 
According to Lemmas \ref{lm:arb-core-eigenvalue} and \ref{lm:taming} we have $Z\dot|p|^2>\eps|p|^2$ on $U\setminus L_1$ for a neighborhood $U\supset L_1$. Choose  $\delta>0$  to have $\{|p|\leq6\delta\}\subset U$. Denote $U_\delta:= \{\delta\leq \phi\leq6\delta\}$,

Assuming $\phi>0$ and choosing a sufficiently large $C$ we  can  find a Lyapunov   function $\psi$  for $Z$  such that
  \begin{itemize}
\item[-] $\psi=C_1|p|^2$ on $\{2\delta<|p|<5\delta\}$;
 \item[-]  $\psi =C_2\phi$ on $W\setminus \{|p|>5\delta\}$;
\item[-] $\psi= \phi  $  on  $\{|p|< \delta\}$;
\item[-] $d\psi(Z)>0$ on $U_\delta\setminus \{2\delta<|p|<5\delta\}$.
\end{itemize}
 
 We can apply the same argument to the function $\wh\phi =\ell+K|p|^2$  constructed in Lemma \ref{lm:taming2} to get a function $\wh\psi$ which coincide with $\wh\phi$ on $\{|p|<\delta\}$ and with $C_1|p|^2$ on  $\{2\delta<|p|<3\delta\}$.
 
  A convex combination of Lyapunov functions for $Z$ is again a Lyapunov function for $Z$, and hence there is a fixed on $\{|p|\geq2\delta\}$ deformation of $\psi$ to a Lyapunov function which is equal to $\wt\psi$ on $\{|p|<\delta\}$. We will keep the notation $\psi$ for the deformed function.
 
  Recall that 
  $\lambda|_{\Op L_1}=
 pdq - dH$ for a function $H$ vanishing on $L\cap\Op L_1$. Take a   function $\theta:[0,6\delta]\to\R_+$    which is equal to $0$ on $[0,2\delta]\cup [5\delta, 6\delta]$ and equal to $1$ on $[3\delta, 4\delta]$. Define a   family of functions $H_t:  U_\delta\to\R$  by the formula
 $H_t:=(1-t\theta(|p|))H.$  Thus, $H_0=H$, $H_1=0$  on  $\{3\delta<|p|<4\delta\}$ and $H_t=H$ on $\{0\leq |p|\leq 2\delta\}\cup \{5\delta\leq |p|\leq  6\delta\}$ for all $t\in[0,1]$. Denote  by $v$ the Hamiltonian field of $H$,  by $v_t$ the Hamiltonian field of $H_t$ and by $Z_t$ the Liouville field $v_t+p\frac{\p}{\p p}$. Then  $v_t=v $ on  $\{0\leq |p|\leq 2\delta\}\cup \{5\delta\leq |p|\leq  6\delta\}$ and   $v_t=(1-t\theta(|p|))v+H\theta'(|p|)w$ on   $\{3\delta<|p|<4\delta\}$, where we denoted  by  $w$  the Hamiltonian vector field of the Hamiltonian $|p|^2$. Let us observe that the Liouville field $Z_t$ is tangent to the arboreal $L$.
 Then we have $v_t\cdot\psi=v\cdot\psi $ on  $\{0\leq |p|\leq 2\delta\}\cup \{5\delta\leq |p|\leq  6\delta\}$ and $$v_t\cdot |p|^2=(1-t\theta(|p|))v\cdot|p|^2+H\theta'(|p|)w\cdot |p|^2=(1-t\theta(|p|))v\cdot|p|^2$$  on $\{3\delta<|p|<4\delta\}$.
  On the other hand, on  $\{3\delta<|p|<4\delta\}$  we have 
  $$Z_t\cdot |p|^2=(1-t\theta(|p|))Z\cdot|p|^2+2t\theta(|p|)|p|^2>0.$$ Thus, $Z_t$   is gradient like  for $\psi$ for all $t\in[0,1]$.

  According to Lemma \ref{lm:taming} the vector  field $Z $ also satisfies the condition $Z\cdot|p|^2\geq \eps|p|^2.$  Hence, the vector field $p\frac{\p}{\p p}+v_1$ is gradient like for the function $ |p|^2 + \psi$. Define $Z_t$ for $t\in[1,2]$ on $\{|p|\leq 4\delta\}$ by the formula
  $$Z_t:= p\frac{\p}{\p p}+(2-t)v_1.$$ 
  
   We claim that     the Liouville  field $Z_t$ is  gradient like for the function $\psi_t:= |p|^2+(2-t)\psi)$.    Away from the $0$-section both vector fields,
   $(2-t)Z_1$ and $(t-1)p\frac{\p}{\p p}$ are  gradient like for $|p|^2$, and hence, so is $Z_t=(2-t)Z_1+(t-1)p\frac{\p}{\p p}$. Hence, it is sufficient to check the Lyapunov condition in an arbitrary small neighborhood of the $0$-section.
  
   Note that  we have $p\frac{\p}{\p p}\cdot\ell=0$. Hence,
   \begin{align*}
   &Z_t\cdot\psi_t=((t-1)p\frac{\p}{\p p}+(2-t)Z_1)\cdot(|p|^2+(2-t)\psi)\\
   &= (t-1)K(t)|p|^2+ (2-t)^2Z_1\cdot\psi\geq 2(t-1)K(t)|p|^2+(2-t)^2\eps(|Z_1|^2+|d\psi|^2).
   \end{align*}
  Furthermore,
  $|Z_t|^2\leq 2\left((2-t)^2|Z_1|^2+(t-1)^2|p|^2\right)$ and $|d\psi_t|^2\leq 2\left(|p|^2+(2-t)^2|d\psi_1|^2\right)$. Therefore,
  $$Z_t\cdot\psi_t\geq \eps_1(|Z_t|^2+|d\psi_t|^2).$$
   We also note that the Liouville field $Z_t$ is tangent to the arboreal $L$ for all $t$.

    All functions $\psi_t$ are proportional to  $|p|^2$ on $\{3\delta<|p|<4\delta\}$,
   and hence can be extended as  proportional to  $\psi$ Lyapunov functions for $Z_t$ for all $t\in[0,2]$. 
   
   Finally, we observe that $Z_2$ is equal to $Z_{\st}=p\frac{\p}{\p p}$ on $\{|p|\leq 4\delta\}$,   and therefore one can use splitting procedure, as it is described in Section \ref{sec:Lbc-gluing} to split  the Weinstein manifold along the hypersurface
   $\{|p|=4\}$ into the block $T^*L_1$ and a Wc-manifold with the skeleton $L_{>1}$.
   Hence the induction hypothesis completes the proof. \end{proof}

\section{Positive cotangent buildings}\label{sec:positivity}
In this section we discuss the positivity relation, first for Lagrangian planes and then for cotangent buildings.

\subsection{The positivity relation}

\subsubsection{Polarizations}

Let $(V, \omega)$ be a symplectic vector space and let $\tau, \nu \subset V$ be transverse linear Lagrangian subspaces. We will refer to the pair $(\tau, \nu)$ as a {\it polarization} of $V$, and call $\tau$ and $\nu$ the respective {\it horizontal} and {\it vertical} spaces of the polarization.
A polarization $(\tau, \nu)$ provides a canonical linear isomorphism $\nu \simeq \tau^*$ given by $\nu\ni  v\mapsto \iota(v)\omega|_\tau\in\tau^*,$
and linear symplectomorphisms $V\simeq  \tau \oplus \nu \simeq T^*\tau$, where we define $T^*\tau = \tau \oplus \tau^*$.

Let $L \subset V$ be a Lagrangian subspace transverse to $\nu$.   Via the identification $V \simeq T^*\tau$, we can regard $L$ as the graph of the differential of a quadratic form denoted by $L^{(\tau,\nu)} :\tau\to \R$. By construction, this differential is the composition of the canonical maps
$$dL^{(\tau,\nu)} :\tau\simeq V/\nu \simeq L \to V\to V/\tau \simeq \nu.$$

\begin{definition}
Let $(V, \omega)$ be a symplectic vector space with polarization $(\tau,\nu)$. Let $L \subset V$ be a Lagrangian subspace transverse to the vertical space $\nu$. We write $$L\psucc_\nu \tau  \qquad \text{(resp.}\, \,  L \pprec_\nu \tau)$$
 when the quadratic form $L^{(\tau,\nu)} :\tau\to \R$ is positive (resp. negative) definite. 

\end{definition}

       \begin{figure}[h]
\includegraphics[scale=0.5]{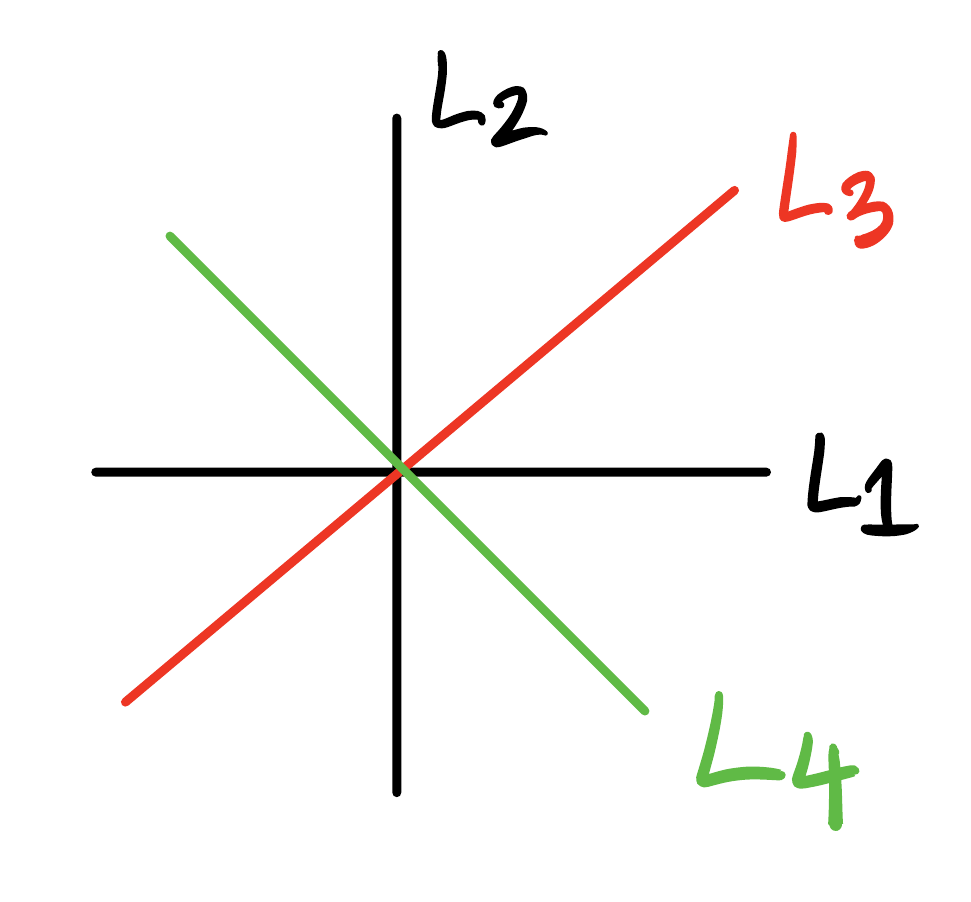}
\caption{The above Lagrangians $L_3$ and $L_4$ in $\R^2$ are given by positive-definite and negative-definite quadratic forms respectively, hence we have $L_3 \psucc_{L_2} L_1$ and $L_4 \pprec_{L_2}L_1$.}
\label{fig:posed}
\end{figure}
Note that $L\psucc_\nu \tau$ or $L \pprec_\nu \tau$  implies that $L$ and $\tau$ are transverse. 

One can check positivity by Hamiltonian reduction to smaller symplectic vector spaces, in particular two-dimensional symplectic planes.
Given a coisotropic $W\subset V$, and  any subspace $L \subset V$, denote the symplectic  reduction by  $[L]^W = (L \cap W)/W^\perp \subset W/W^\perp = [V]^W$.  Note  that  if $L$  is Lagrangian then  $[L]^W$ is Lagrangian in  $[V]^W$  even when the intersection $L\cap  W$  is not transverse.  
\begin{lemma}\label{lm:cont-red} Let $\Lambda(V)$ denote the  Lagrangian Grassmanian of  a symplectic space $V$. Then 
the map $\pi_W:\Lambda(V)\to\Lambda([V]^W)$  defined  by the formula $\pi_W(L)=[L]^W$, $L\in\Lambda(V)$, is continuous.
\end{lemma}

\begin{proof} This is clear. \end{proof}

\begin{lemma}\label{lemma:pos classic}

Let $V$ be a symplectic vector space of dimension $2n>0$, $(\tau,\nu)$ a polarization of $V$ and $L \subset V$ a linear Lagrangian subspace. The following conditions are equivalent:
\begin{enumerate}
\item Positivity: $$L \psucc_\nu \tau.$$
\item Positivity restricted to subspaces: $$[ L]^W \psucc_{[\nu]^W }  [\tau]^W,$$ for all coisotropics $ W\subset V$ containing $\nu$ with $\dim W> n $.

\item Positivity restricted to lines: 
$$[L]^W \psucc_{ [\nu]^W}  [\tau]^W,$$ for all coisotropics $ W\subset V$ containing  $\nu$ with $\dim W = n+1$.

\item Positivity of all reductions:
$$[L]^W \psucc_{ [\nu]^W}  [\tau]^W,$$ for all coisotropics  $W\subset V$.

\item Sylvester's criterion: $$\wedge^{top} [L]^{W_i} \psucc_{\wedge^{top} [\nu]^{W_i}} \wedge^{top}[ \tau]^{W_i},$$
for any flag of coisotropics $$\nu \subset W_{n+1} \subset \cdots \subset W_{2n-1} \subset W_{2n} =  V$$ with $\dim W_i = i$.

\end{enumerate}
\end{lemma}

\begin{proof} Equivalence  of (i) with (ii) (resp. (iii)) means that 
  a  quadratic form is  positive definite if and only if it is   positive definite on  all  subspaces (resp. on  all lines), which is clear.
Clearly  (iv)  implies (iii), and the converse implication  follows from Lemma  \ref{lm:cont-red} together with the two fact that  the positivity condition    $L \psucc_\nu \tau$ implies  that  $[\nu]^W, [L]^W$ and $[\tau]^W$ are pairwise transverse. Finally, (v) is a reformulation of
  Sylvester's criterion that a quadratic form is positive definite if and only if its leading principal minors have positive determinant.
\end{proof}


\subsubsection{Cyclic symmetries}

Note that given a polarization $(\tau, \nu)$ of $V$, we also have the polarization $(\nu, \tau)$.
Additionally, if $L \subset V$ is transverse to $\nu$, we  can regard $(L, \nu)$ as a polarization.

\begin{lemma}\label{lemma:pos basic} The following properties hold.
\begin{enumerate}

\item Duality of polarization: $L\psucc_\nu \tau \iff  L \pprec_\tau \nu$.

\item Exchange of graphs: $L\psucc_\nu \tau \iff  \tau \pprec_\nu L$.

\item Transitivity: suppose $K \subset V$ is  another Lagrangian subspace transverse to $\nu$. Then
$$ K\psucc_\nu L , \, \, \, L\psucc_\nu \tau  \implies K\psucc_\nu \tau$$

\end{enumerate}

\end{lemma}

\begin{proof}
%

Let $\omega$ denote the symplectic form of $V$. For (i), by construction, $L^{(\tau, \nu)}(X) = \omega( dL^{(\tau,\nu)} X , X)$, $X \in \tau$. Since $dL^{(\tau, \nu)} = (dL^{(\nu, \tau)})^{-1}$, for $Y \in \nu$ we have 
$$L^{(\nu,\tau)}(Y) = \omega( (dL^{(\tau,\nu)})^{-1}Y,Y) = - \omega(Y, (dL^{(\tau,\nu)})^{-1} Y) = -L^{(\tau,\nu)}(dL^{(\tau,\nu)}Y).$$

 Hence if $L^{(\tau,\nu)}$ is positive definite, then $L^{(\nu,\tau)}$ is negative definite and vice versa. Next, let $K \subset V$ be a Lagrangian  transverse to $\nu$. Note that $K^{(\tau, \nu)} = K^{(L, \nu)} + L^{(\tau, \nu)}$ under the 
  isomorphism $\tau\simeq V/\nu \simeq L$. Both (ii) and (iii) follow immediately (for (ii) take $K=\tau$ so that $K^{(\tau,\nu)}=0$ and hence $\tau^{(L,\nu)} = - L^{(\tau,\nu)}$).
\end{proof}



\begin{lemma}\label{lemma:pos sym}
Suppose $L_1, L_2, L_3 \subset V$ are Lagrangian subspaces transverse to~$\nu$.
Then for any permutation $\sigma \in \Sigma_3$ we have
$$
L_1 \psucc_{L_3} L_2 \iff
L_{\sigma(1)} \psucc^\sigma_{L_{\sigma(3)}} L_{\sigma(2)}
$$ 
where $\psucc^\sigma$ denotes $\psucc$ for $\sigma$ an even composition of transpositions, and $\pprec$ for $\sigma$ an odd composition of transpositions.
\end{lemma}

\begin{proof}
The case of $\sigma = (23)$ (resp. $\sigma = (12)$) is part (i) (resp. (ii)) of Lemma~\ref{lemma:pos basic}.
These permutations generate.
\end{proof}

One can interpret Lemma~\ref{lemma:pos basic}, and in turn Lemma~\ref{lemma:pos sym}, as saying  the ternary relation $ \psucc$  provides a {\em partial cyclic order} on triples of Lagrangian subspaces.

\begin{definition}
We say an ordered list $L_1, \ldots, L_m \subset V$ of Lagrangian planes is {\em $\psucc$-cyclically ordered} (resp.  {\em $\pprec$-cyclically ordered})
when they are pairwise transverse and satisfy
$$ 
\xymatrix{
L_{i_2} \psucc_{L_{i_1}} L_{i_3}  & \mbox{(resp. $L_{i_2} \pprec_{L_{i_1}} L_{i_3}$)}  
}$$
whenever $ i_1,i_2,i_3 \in \{1, \ldots , m\} $ are cyclically ordered in $\Z/m$. We will also write $L_1 \psucc L_2 \psucc \cdots \psucc L_m$ (resp. $L_1 \pprec L_2 \pprec \cdots \pprec L_m$).
\end{definition}

         \begin{figure}[h]
\includegraphics[scale=0.5]{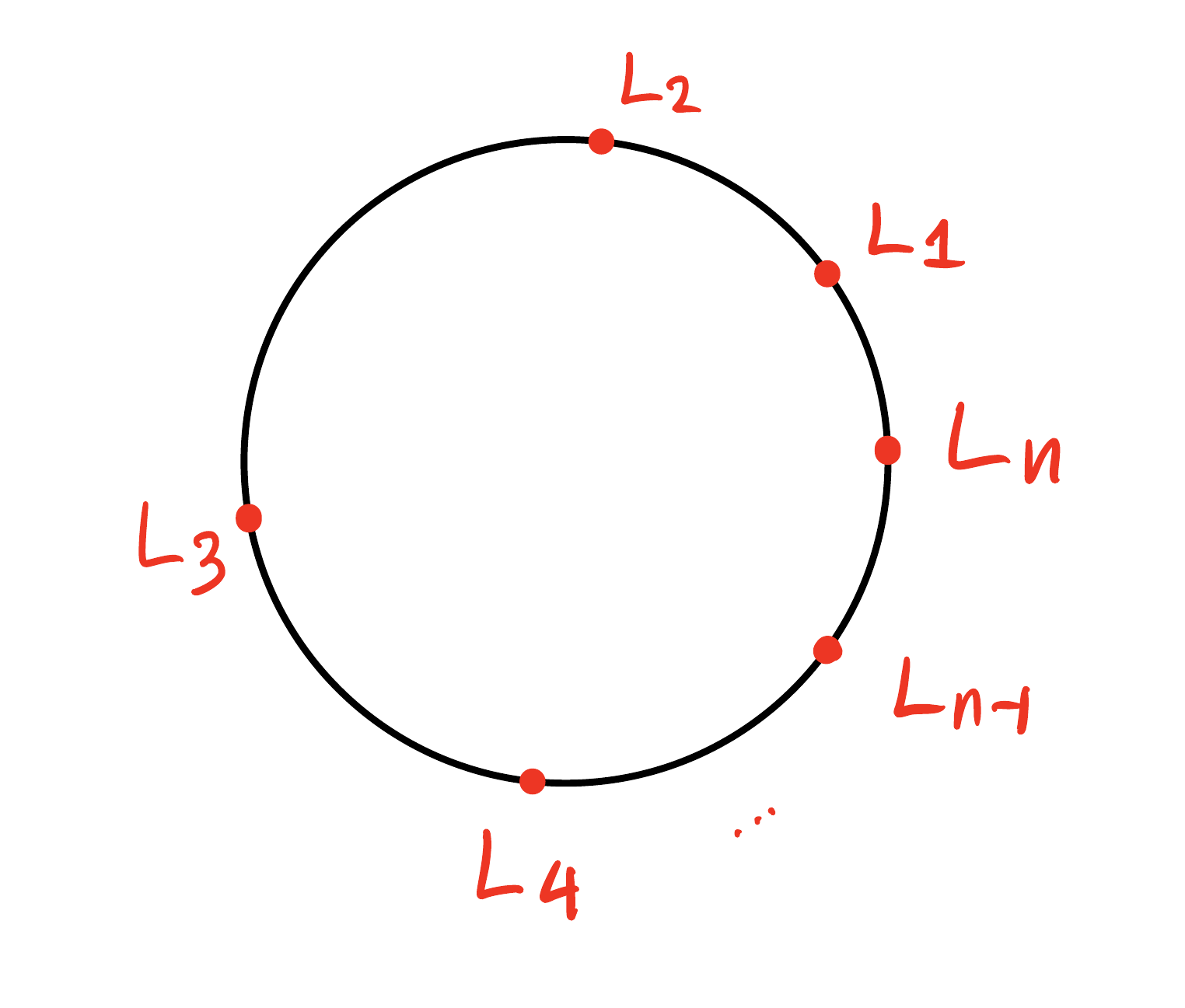}
\caption{A cyclic ordering on a tuple $(L_1, \ldots , L_n)$ can be thought of as an embedding into $S^1$ preserving the cyclic order of all triples.}
\label{fig:circle}
\end{figure}

We record the following useful assertions for future reference. The assertions and their proofs hold in any partial cyclic order.
  
\begin{lemma}\label{lemma:pos basic} The positivity relation satisfies the following properties.

\begin{enumerate}
\item

Suppose $L_1 \pprec_\nu L_2$ and either $\tau \pprec_\nu L_1$ or $L_2 \pprec_\nu \tau$.
Then $L_1 \pprec_\tau L_2$.

\item

Suppose
  $L_3\psucc_{L_4} L_2 \psucc_{L_4}L_1$.
  Then  $L_2 \psucc_{L_3}L_1$. If additionally $L \psucc_{L_3} L_2$ then $L \psucc_{L_4} L_1$.

\item

If $L_- \psucc_{L_2} L_1$ and $L_+\psucc_{L_1} L_2$, then $L_-\psucc_{L_2}L_+$.

\end{enumerate}

\end{lemma}

\begin{proof}
(i) 
By cyclic symmetry applied to $L_1 \pprec_\nu L_2$, we have $L_2 \pprec_{L_1} \nu$ and 
$\nu \pprec_{L_2} L_1$.
Suppose $\tau \pprec_\nu L_1$. Then by cyclic symmetry, we have  $\nu \pprec_{L_1} \tau$ and hence
by transitivity 
$L_2 \pprec_{L_1} \tau$. By transpositional symmetry, we conclude 
$L_2 \psucc_{\tau} L_1$.

Similarly, suppose $L_2 \pprec_\nu \tau$. Then by cyclic symmetry, we have  $\tau \pprec_{L_2} \nu$ and hence
by transitivity
$\tau \pprec_{L_2}L_1$. By transpositional symmetry, we conclude 
$L_1 \pprec_{\tau} L_2$.

(ii) For the first assertion, $L_3\psucc_{L_4} L_2$ (resp. $L_2\psucc_{L_4} L_1$) implies $L_3\pprec_{L_2} L_4$ (resp. $L_4\pprec_{L_2} L_1$) by transpositional symmetry. Hence by transitivity we have $L_3 \pprec_{L_2} L_1$, and so by transpositional symmetry  $L_2 \psucc_{L_3}   L_1$.

  For the second, suppose $L$ satisfies $L \psucc_{L_3} L_2$. 
  By the previous part, we have
  $L_2 \psucc_{L_3} L_1$,  hence  $L \psucc_{L_3}  L_1$ by cyclic symmetry. By transpositional symmetry, we then have $L \pprec_{L_1}  L_3$.
  By assumption, we have $L_3 \psucc_{L_4}  L_1$ and hence by transpositional symmetry $L_3 \pprec_{L_1}  L_4$ so by transitivity $L \pprec_{L_1}  L_4$. Finally, by transpositional symmetry, we obtain $L \psucc_{L_4} L_1$.
  
  (iii) By cyclic symmetry, we have $L_1 \psucc_{L_2} L_+$, and hence by transitivity  $L_- \psucc_{L_2} L_+$.
\end{proof}

\subsubsection{Positive zones}
It is useful to reformulate positivity more geometrically.
  
\begin{definition}
Let $(V, \omega)$ be a symplectic vector space with polarization $(\tau,\nu)$.
Let $\Lambda(V)$ be the Lagrangian Grassmannian of $V$. Denote  by   $C(\tau,\nu)\subset \Lambda(V)$ the subset consisting of those Lagrangians $L\subset V$ transverse to $\nu$
and satisfying $L\psucc_{\nu} \tau$.  
We refer to   $C(\tau,\nu)$   as a {\em positive zone},
 and its closure, denoted by  $\oC(\tau,\nu)$, as a {\em closed positive zone}. 
 
\end{definition}

       \begin{figure}[h]
\includegraphics[scale=0.6]{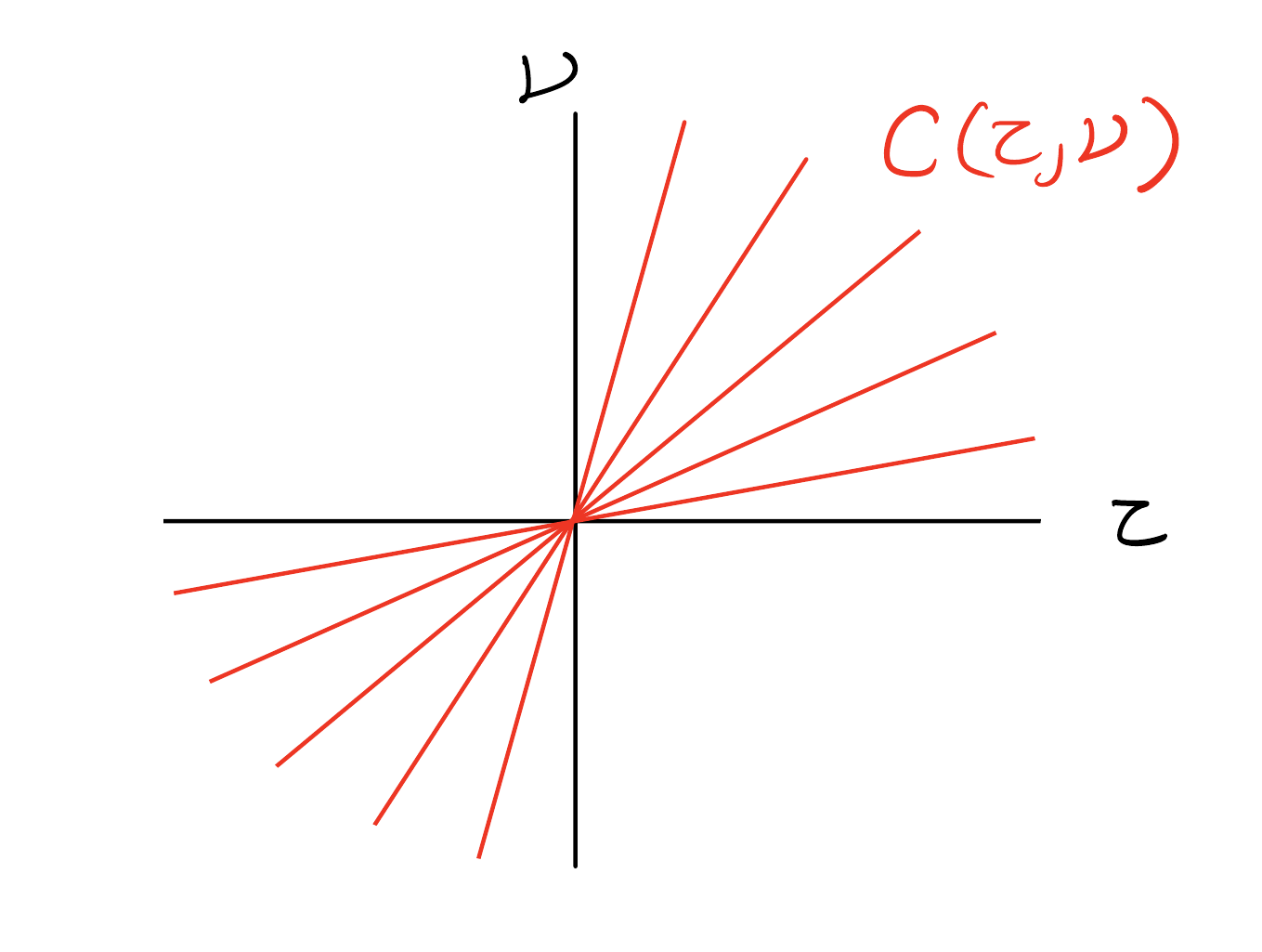}
\caption{The {\em positive zone} of a polarization $(\tau,\nu)$}
\label{fig:Positivezone}
\end{figure}

 \begin{remark}
 As justified by part (ii) of the following lemma, we will regard any single Lagrangian $L \in \Lambda(V)$ itself as a (degenerate) closed positive zone. 
 \end{remark}

\begin{lemma} The positive zone $C(\tau,\nu) \subset \text{Gr}(V)$ satisfies the following properties.
\begin{enumerate}
\item $C(\tau,\nu)$ (resp. $\oC(\tau,\nu)$) is non-empty, open (resp. closed), and geodesically convex for any homogenous metric on $\mathrm Gr(V)$, hence contractible.

\item If $\tau$ approaches $\nu$ along a geodesic in $\Lambda(V)$, then  $C(\tau,\nu)$ and $\oC(\tau,\nu)$ limit to $\nu$ itself.

\end{enumerate}
\end{lemma}

\begin{proof}
(i) is immediate from usual statements about quadratic forms. For (ii), we can put everything into a standard position so that we are contracting positive definite quadratic forms towards zero.
\end{proof}


\begin{lemma}\label{lemma:pos reform} The following properties hold.

\begin{enumerate}
\item

Suppose $L_2 \in C(L_1, \nu)$ and either $L_1 \in C(\tau, \nu)$ or $L_2 \in C(\tau, \nu) $.
Then $L_2\in C(L_2, \tau)$.

\item

Suppose
  $L_3 \in C( L_2, L_4), L_2\in C(L_1, L_4)$.
  Then  $L_2 \in C(L_1, L_3)$ 
  and $C(L_2,L_3)\subset C(L_1,L_4)$.

\item

If $L_- \in C(L_1, L_2)$ and $L_+ \in C(L_2, L_1)$, then $L_- \in C(L_+, L_2)$.

\end{enumerate}

\end{lemma}

\begin{proof} This follows immediately from Lemma~\ref{lemma:pos basic} by unwinding the definition. \end{proof}

 \begin{lemma}\label{lm:extend-cone2}
 Let $\cL \subset \Lambda(V)$ be a non-empty   set of Lagrangian planes each transverse to a fixed Lagrangian plane $L \subset V$. The set $\cS(\cL, L) \subset \Lambda(V)$ of Lagrangian planes $L_- \subset V$ such that $\cL\subset C(L_-,L)$
  is convex. If $\cL$ is compact, then $\cS(\cL, L)$ is non-empty, hence contractible.
 \end{lemma}
 \begin{proof}
 The condition $T\in C(L_-,L)$ is equivalent to $L_-\in C(L,T)$, and thus
  $$
 \cS(L, \cL) = \bigcap\limits_{T\in\cL}C(L,T).
 $$
 Fix a Lagrangian plane $H\subset V$ transverse to $L$. Then we can identify Lagrangian planes $T\subset V$
  transverse to $L$ with quadratic forms $Q_T$ on $H$, and specifically, the set of Lagrangian planes
  $C(L,T) \subset \Lambda(V)$ with the convex   set of those quadratic forms $Q$ on $H$ satisfying $Q<Q_T$. 
  Hence  $ \cS(L, \cL)$ is  an intersection of convex sets so itself convex. 
  
  Finally,  when $\cL$ is compact, the set of quadratic forms $Q_T$, for $T\in\cL$, is bounded, hence we may choose a quadratic form $Q_-$ such that  $Q_-<Q_T$, for all $T\in\cL$. Thus 
  $ \cS(L, \cL)$ is non-empty.
   \end{proof}


\begin{lemma} Let $W\subset V$ be a coisotropic subspace. For any polarization $(\tau,\nu)$ of $V$, the reduction map $\pi_W$ projects the positive zone $C(\tau,\nu) \subset \Lambda(V)$ to the positive zone $C([\tau]^W, [\nu]^W) \subset \Lambda([V]^W)$.
\end{lemma}

\begin{proof}
For any $L \in \Lambda(V)$ transverse to $\nu$, the pre-composition of the quadratic form $([L]^W)^{([\tau]^W, [\nu]^W)}$ with the reduction map $\tau \to [\tau]^W$ is equal to the quadratic form $L^{(\tau,\nu)}$. Since $\tau \to [\tau]^W$ is surjective, the lemma follows.
\end{proof}


 \subsection{Positivity of polarizations}
 
 \subsubsection{Polarized Legendrians}\label{sec:Leg-in-blocks}

Recall that a Legendrian embedding   $\Lambda \subset \sT^*M\setminus M$   is said to be
  {\em adapted}  to the block $\sT^*M$ if    $\Op\p \Lambda \subset \Op\p M$ and  for each stratum $P\subset\p_kM$  we have $\Lambda \cap(\Op P=\sT^*P\times \sT^*\II^k) =\Lambda_k\times \II^k $ for a Legendrian $\Lambda_k\subset \sT^*P$. Recall also that an adapted Legendrian $\Lambda$ is called {\em regular} if $\pi|_\Lambda:\Lambda\to M$ is an   immersion with transverse self-intersections, where $\pi:\sT^*M\to M$ is the cotangent projection.

A {\em polarization}  of a Lagrangian submanifold $L$ of a symplectic manifold $(M,\omega)$ is a field $\eta \subset TM|_L$ of Lagrangian planes transverse to $L$. A polarization of a Legendrian $\Lambda$ in a contact manifold $(V,\xi)$  is a Lagrangian field  $\tau \subset \xi|_\Lambda$ transverse to $\Lambda$ (recall that $\xi$ has a well-defined conformal symplectic structure; hence, it makes sense to talk about Lagrangian fields in $\xi$).  

\begin{definition}A polarization of   a Legendrian $\Lambda\subset \sT^*M\setminus M$ is a field of  Lagrangian planes $\mu \subset \sT^*M|_\Lambda$ which projects to a polarization of the  projection of $\Lambda$ to $S^\infty M$. 
\end{definition}
Note that $\mu$ has to be tangent to the Liouville field $Z$.

In particular, the Lagrangian distribution $\nu_M$ which projects to   the Legendrian distribution $\ell_M$ tangent to  the spherical Legendrian fibration  of $S^\infty M$, defines a canonical polarization of any {\em regular} Legendrian in $\sT^*M\setminus M$. We will refer to this polarization as   {\em tautological}. 
  
The space of polarizations of a given Legendrian  is contractible. Any given polarization of a Legendrian admits a ribbon, whose tangent planes to fibers along the $0$-section form the given polarization. The space of germs of ribbons for a given polarization is contractible. Hence the space of germs of ribbons for a given Legendrian is contractible.
   
  \begin{definition}A polarization $\mu$ of a {\em regular} Legendrian $L\subset \sT^*M\setminus M$ is called {\em positive} if $\ell_M \in C(TL,\mu)$ at any point of $L$, where $\ell_M$ is the tautological polarization. \end{definition}
  
 \subsubsection{Positivity implies transversality of conormals}  
 
 \begin{definition} Let $(W, \lambda)$ be a Wc-manifold. A {\em polarization} of $W$ is a global field of Lagrangian planes $ \eta \subset TW$. \end{definition}
 
 Let $T^*M$ be the cotangent bundle of a compact manifold with corners, $\tau = TM$ the tangent field along $M$ and $\nu = \ker(d \pi)$ the vertical field, where $\pi: T^*M \to M$ is the cotangent bundle projection.
  
 \begin{lemma}\label{cor:cone-over-regular}
Let $L\subset  \sT^*M\setminus M$ be the Lagrangian cone over a regular Legendrian, and $\eta$ a polarization of $\sT^*M$ whose restriction to the $0$-section $M$ is positive with respect to the polarization $(\tau, \nu)$. Then $\eta$ is transverse to  the  cone $L$  on a sufficiently small neighborhood of $M\subset \sT^*M$.
\end{lemma}
 The statement is a corollary of the following linear algebra lemma.
\begin{lemma}\label{lm:pos-quadr-form}
Let $Q$ be a  quadratic form and $L_Q=\{q=Ap\}\subset T^*\R^n$ the corresponding linear Lagrangian.
If $Q$ is positive or negative definite then for any co-oriented hyperplane $H \subset\R^n$ through the origin, its conormal $T^*_H\R^n \subset T^*\R^n$ is transverse to $L_Q$. Conversely, if $L_Q$ is transverse to
conormal $T^*_H\R^n$ for all co-oriented hyperplanes $H \subset\R^n$ through the origin, then $Q$ is positive or negative  definite.
\end{lemma}
 
\begin{proof}
Indeed,  the condition $(p,q)\in L_Q\cap\tau^*$ and $(p,q)\neq0$ is equivalent  to $Q(p)=0$, \end{proof}

\begin{proof}[Proof of Lemma \ref{cor:cone-over-regular}] Indeed, the Liouville cone is the conormal of the immersed front projection of $\Lambda$, and hence we can apply Lemma \ref{lm:pos-quadr-form} in a sufficiently small neighborhood of the $0$-section. \end{proof}

      \begin{figure}[h]
\includegraphics[scale=0.5]{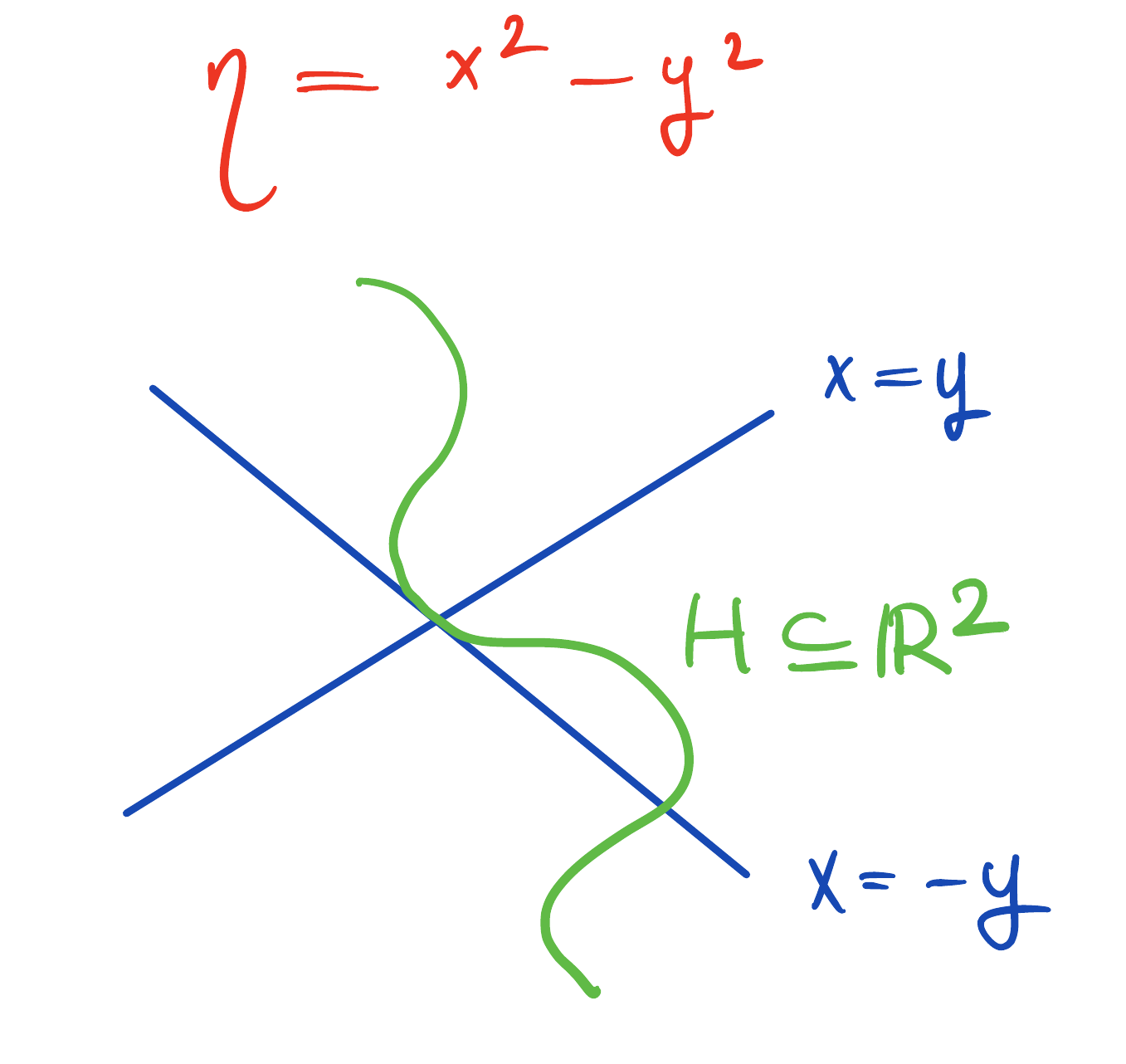}
\caption{For a non positive-definite quadratic form such as $\eta = x^2-y^2$, the conormal $T^*_H\R^2$ of a smooth hypersurface $H \subset \R^2$ will be tangent to $\eta$ whenever $TH$ is tangent to the light-cone $\{x=y\}$.}
\label{fig:positivity}
\end{figure}

 \subsection{Positive cotangent buildings}\label{sec:pos-complexes}
 
\subsubsection{Reductions}

 Let  $W$ be a cotangent building with blocks $B_j  = \sT^*M_j, j\in J$.   The notions we introduce in this section will depend on sizes of blocks, so we deviate here from the point of view of germs and take $B_j$ to be a fixed defining domain for $T^*M_j$.
 Set $B_j^\circ = B_j \setminus M_j$, for $j\in J$. For each $i \in J$,  consider  the line bundle  $\Span(Z_i) \subset T B_i^\circ$ generated by $Z_i$,
 let $\zeta_i = \Span(Z_i)^\perp/\Span(Z_i) \to B_i^\circ$ denote the corresponding symplectic normal bundle, and consider the natural reduction diagram
 $$
 \xymatrix{
 T B_i^\circ & \ar@{_(->}[l]_-\fri \Span(Z_i)^\perp \ar@{->>}[r]^-\frp & \zeta_i 
 }
 $$
 Given a linear subspace $\nu \subset TB_i^\circ$, let $[\nu]^i:=[\nu]^{\Span(Z_i)^\perp }\subset \zeta_i$ be  the reduction
 along the above correspondence
 $$
 [\nu]^i = \frp(\fri^{-1}(\nu)) = ((\nu + \Span(Z_i)) \cap  \Span(Z_i)^\perp)/\Span(Z_i).
 $$
We will  refer to $ [\nu]^i $ as the $i$-reduction of $\nu$. 
 More generally,  given a multi-index $I=(i_1<\dots<i_m) \subset J$ (which could be empty), 
consider the natural reduction diagram
$$
 \xymatrix{
 T (B_{i_1}^\circ \cap \cdots \cap B^\circ_{i_m}) & \ar@{_(->}[l]_-\fri \Span(Z_{i_1},\ldots,  Z_{i_m})^\perp \ar@{->>}[r]^-\frp & \zeta_I 
  }
 $$
 $$
\zeta_I  = \Span(Z_{i_1},\ldots,  Z_{i_\ell})^\perp/ \Span(Z_{i_1},\ldots,  Z_{i_m}).
 $$
 Given a linear subspace  $\nu \subset TB_i^\circ$, we refer to $[\nu]^I:=[\nu]^{ \Span(Z_{i_1},\ldots,  Z_{i_\ell})^\perp} = \frp(\fri^{-1}(\nu)) \subset \zeta_I$ as the $I$-reduction of $\nu$.  If $I=\varnothing$, then $[\nu]^I=\nu$.

 \subsubsection{Positivity}
 
 Let $W=\bigcup_j B_j$ be a cotangent building. Since the sizes of the blocks $B_j$ are fixed, we can assign a {\em type} to any point of $\Skel(W)=\bigcup_j \mM_j$ as follows:

   \begin{definition}
 We say that a point $a\in \mM_i$ is of {\em  type} $I=(i_1<\dots<i_m),$ where
$ i_m<i$, if $a\in B_{i_1}\cap\dots \cap B_{i_m} $
and $a\notin B_j$, for $j\notin I$ and $j<i$.    \end{definition}
 
 Note that we allow the   type $I$ to be empty.  For any  multi-index $I=(i_1<\dots<i_m) \subset J$ and integer $s$, where $1\leq s  \leq m$, denote $I(s)=(i_s<\dots<i_{t}) \subset I$.

\begin{definition} A cotangent building $W=\bigcup_{j=1}^k B_j$  is  {\em positive} if 
  for     any point $a\in\mM_j $  of type $I=(i_1<\dots<i_m )$, $I\neq\varnothing$ and any $s\leq m$  the tuple of Lagrangian planes
$$[T_aM_j]^{I(s)},[\nu_j(a)]^{I(s)},[\nu_{i_m}(a)]^{I(s)},\dots,  [\nu_{i_s}(a)]^{I(s)} $$ is $\pprec$-cyclically ordered.
\end{definition}

         \begin{figure}[h]
\includegraphics[scale=0.65]{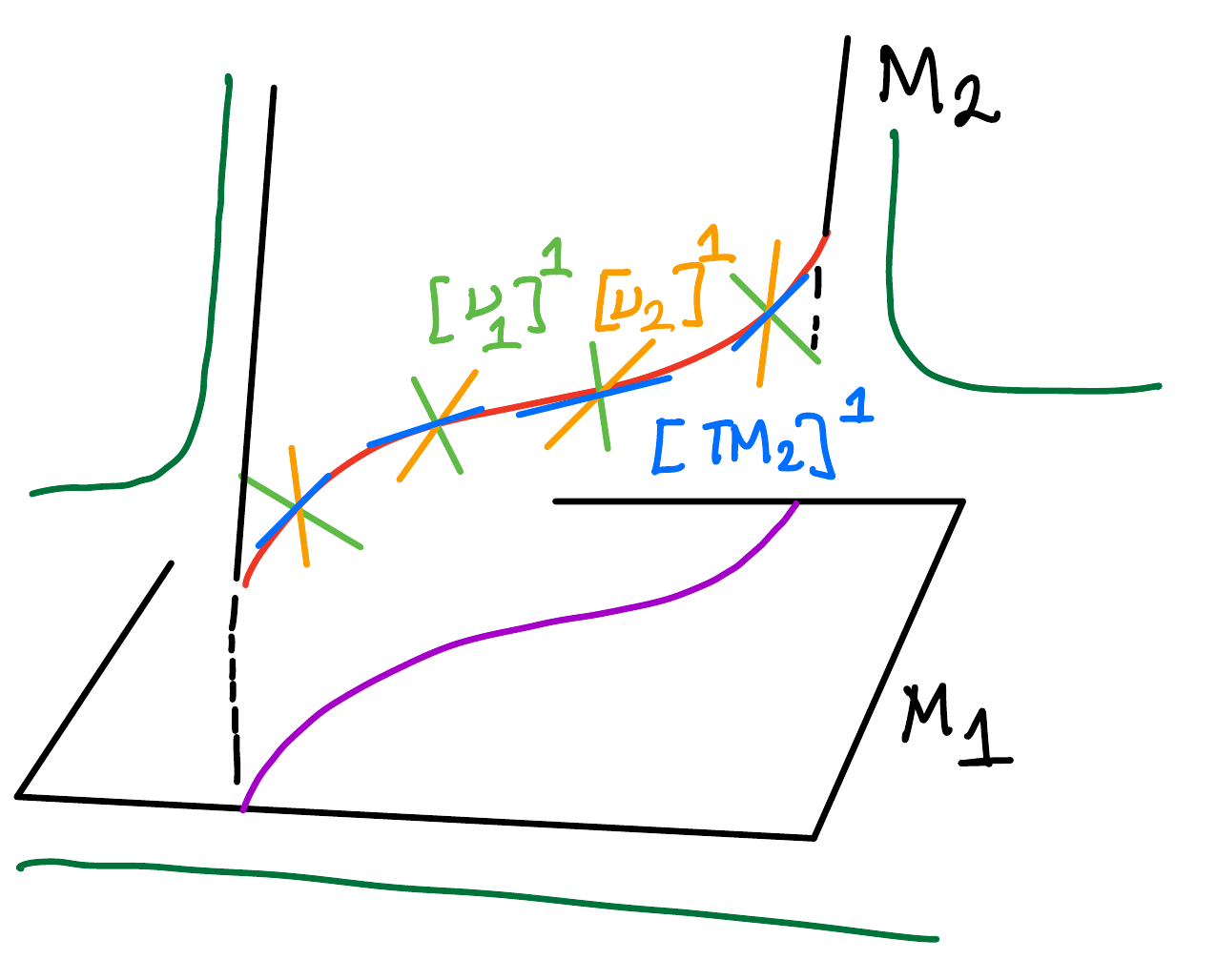}
\caption{The positivity condition for the interaction of two blocks.}
\label{fig:alt positivityblocks}
\end{figure}

 \begin{definition} Let $W$ be a cotangent building $W$ with the notation as above. A Lagrangian distribution $\eta$ on $W$ is called {\it  positive}  if it is transverse to all the 0-sections $M_i$  and at any point $a\in\mM_i$ of type $I$ (which could be empty), we have $[\nu_i(a)]^{I }  \in C([T_aM_{i}]^{I } , [\eta(a)]^I).$
\end{definition}



\subsubsection{Existence and uniqueness of positive distributions} Positive distributions exist and are unique up to a contractible choice on any positive cotangent building. Before we prove this we need two lemmas.

   \begin{lemma}\label{lm:tau-angle}
  Let $B=\sT^*M$ be a cotangent block and $\tau$ any Lagrangian distribution extending
  $TM$. Let $\Sigma\subset \sT^*M$ be a conormal of a co-oriented hypersurface $H \subset M$. For any Lagrangian plane $T\in T_aB$, $a\in B\setminus M$ denote by $\ol T:=[T]^{\Span(Z)^\perp}$ the  reduction with respect to the conormal of the Liouville field $Z$.
  Then   the angle between $ \ol{T\Sigma(a)}$ and $\ol{\tau(a)}$ converges to $0$ as $a\to M$.
    \end{lemma}
  
  \begin{proof} 
The statement is local so we can assume that   $M=\R^n$, $H=\{q_n=0\}$, $a=0$ and $\Sigma = \{ q_n = 0, \, p_j = 0 , j<n  \}$. Moreover, it is sufficient to prove the lemma for a single particular extension $\tau$, and hence we can choose $\tau = \Span(\p_{q_1} , \ldots , \p_{q_n})$. The Liouville field is $Z= \sum_i p_i \p_{p_i}$, hence $Z|_\Sigma = p_n \p_{p_n}$. Therefore $\Span(Z)^{\perp}|_\Sigma = \Span( \partial_{q_1}, \ldots , \partial_{q_{n-1}}, \partial_{p_1} , \ldots,  \partial_{p_n} )$, $T\Sigma \cap Z^{\perp}|_\Sigma = \Span(\p_{q_1} , \ldots , \p_{q_{n-1}} , \partial_{p_n})$ and $\tau  \cap Z^{\perp}|_\Sigma = \Span( \p_{q_1} , \ldots , \p_{q_{n-1}})$ so in fact $\ol{T \Sigma}$ and $\ol{\tau}$ are identical along $\Sigma$.
  \end{proof}  

\begin{lemma}\label{lm:transitivity-of-positivity}
Let $B_1 = \sT^* M_1$ be the base block of a positive cotangent building $W$ and $\eta$ a distribution on $\Op M_1$ such that at any point  $a\in M_1$ the triple $T_aM_1,\nu_1(a),\eta(a)$ is cyclically $\pprec$-ordered. Then $\eta$ is positive for $W$ on $\Op M_1$.
\end{lemma}
    \begin{proof}
     Take any  extension $\tau$ of the distribution  $TM_1$ to $B_1=\sT^*M_1$.  For any point $a\in B_1\cap\mM_j$ of type $I=(i_1=1<i_2<\dots<i_k), i_k<j,$ the tuple $$[T_aM_j]^I, [\nu_j(a)]^I,\dots,  [\nu_1]^I $$ is cyclically  $\pprec$-ordered. On the other hand, if the block $B_1$ is sufficiently thin then the triple  $$[\tau(a)]^1,[\nu_1(a)]^1, [\eta(a)]^i$$ is also  cyclically  $\pprec$-ordered.
     Hence Lemma  \ref{lm:tau-angle} implies that $$[T_aM_j]^1,[\nu_1(a)]^1, [\eta(a)]^i$$ is cyclically  $\pprec$-ordered as well. Therefore,
     $$[T_aM_j]^I, [\nu_j(a)]^I,\dots,  [\nu_1]^I,[ \eta(a)]^I$$ is cyclically  $\pprec$-ordered, i.e.~$\eta$ is positive for $W$ on $B_1$.     \end{proof}

Let us denote by $\Pos(W)$ the space of positive distributions on a positive cotangent building~$W= \bigcup_j B_j$.    
     \begin{prop}
 \label{prop:distr-pos} 
For any positive cotangent building $W$ the space $\Pos(W)$ is non-empty and (weakly) contractible.  \end{prop}
\begin{proof}
We will  show   that $\Pos(W)\neq\varnothing$ using the convexity of the space of positive definite quadratic forms. The weak contractibility claim is a parametric version of the same argument. We argue by induction on the number of blocks.

For $W_{\leq 1}=\sT^*M_1$ a  positive distribution  from $\Pos(W)$ restricted to $M_1$ can be viewed 
as a field of positive  definite quadratic forms on the canonical polarization $\nu_1$. Take any such field $\eta$ and its arbitrary extension to $\sT^*M_1$.   According to Lemma \ref{lm:transitivity-of-positivity} the distribution $\eta$  satisfies the positivity condition  for $W$ if the block $B_1$ is chosen sufficiently thin.

Suppose that we have already constructed $\eta$ on $W_{<m}=\bigcup_{j<m}B_j$ and let us consider the next block $B_m=\sT^*M_m$. Recall that $M_m$ is a manifold with corners. Its corners are enumerated by the types of points in their neighborhood, i.e. if $b\in\p_kM_m$ then points $a\in\mM_m$ in a sufficiently small neighborhood of $b$ have a fixed type $I=(i_1<\dots<i_k)$, where $i_k<m$.
The Liouville fields $Z_{i_j}$ of blocks $B_{i_j}, j=1,\dots, k$, are tangent to $\mM_m$ and yield the  canonical splitting $\Op_{M_m}b=\Op_{\p_k M_m}\times \cI^k$. 
 Consider the restriction of $\eta$ to $\mM_m\cap W_{<m}$. It is  transverse to $M_m$ and can therefore be  viewed as  a field of quadratic forms on $\nu_m$.     The positivity  condition means that:
 
 \begin{itemize}
 
 \item[(\dag)]
 for any  point  $a\in \mM_m$   of type $I$, the quadratic form $\eta(a)$ is positive definite on the subspace of $\nu_m(a)$ dual to $\Span(Z_{i_1}(a),\dots, Z_{i_k}(a))$.
\end{itemize}

 Recall that the notion of  type  depends on the thickness of the blocks $B_1,\dots, B_{m-1}$. Let us take slightly thinner blocks $B_i'\Subset B_i$ and choose a Lagrangian field along $\mM_m$ that
  
 \begin{itemize}
 \item[(i)] coincides on $\mM_m\cap\left(W'_{<m}:= \bigcup_1^{m-1}B_i'\right)$ with the restriction of $\eta$ from $W'_{<m}$,
 \item[(ii)] is given by a field of positive  definite quadratic forms on $\nu_m$.
 \end{itemize} 
 Note in particular that (ii) implies that condition (\dag) holds on $W_{<m}\cap\mM_m$. Then choosing any extension to $\Op\mM_m$, and choosing the block $B_m$ sufficiently thin, we get the required  positive distribution on $W_{\leq m}$.\end{proof}

 \subsection{Positive cotangent buildings and positive arboreals}\label{sec:pos-W-arb}
 
  We   prove  in this  section that the skeleton of a positive cotangent building is generically a positive arboreal,
 where throughout this section {\em generically} means after a $C^\infty$-small perturbation of the building structure. This corresponds to a $C^\infty$-small perturbation of the underlying Weinstein structure.
 Conversely, if a cotangent building has a positive  arboreal skeleton, we show the building structure can be adjusted to be made positive without changing the skeleton. 
 
 \subsubsection{From positive buildings to positive arboreals}
 
The key definition is the following:
 
\begin{definition}\label{def:pos lag dist} Given an arboreal Lagrangian $L$ in a symplectic manifold $X$, a Lagrangian distribution $\eta$ along $L$ is called {\em positive with respect to $L$} if the following condition is satisfied. Take any singular point   $a\in L$. Let $T$ be the tangent space to the root Lagrangian at $a$, and $T'$ a tangent plane to any other smooth piece adjacent to $a$. Then $T'\in C(T,\eta(a)).$ \end{definition}

We will later have use for the following reduced variant:

 \begin{definition}\label{def: red pos lag dist} Let $L\subset X$ be a  positive arboreal Lagrangian and $Z$ a non-zero Liouville vector field tangent to  $L$. A Lagrangian distribution $\mu$ along $L$ is called reduced positive with respect to $(L,Z)$ if for 
  any singular point   $a\in L$  the following  condition is satisfied.
  Let $T$ be the tangent space to the root Lagrangian at $a$, and $T'$ a tangent plane to any other smooth piece adjacent to $a$. Then $[T']^{\zeta}\in C([T]^{\zeta},[\eta(a)]^{\zeta}), $ where $\zeta=\Span(Z)^{\perp_\om}.$ \end{definition}


  \begin{lemma}\label{lm:arb-skel-charact}
 Let $W$  be a  W-complex $B_k\to\dots\to B_0$ and denote by $\nu_j$ the polarization of the block $B_j$,  $j=0,\dots, k$.
\begin{enumerate}
\item If the skeleton $\Skel(W)$ is  arboreal, then   for each  point  $a\in  \mB_j \cap M_i$,  $0\leq j<i \leq k$, the Lagrangian  distribution $[\nu_j]^j$ is transverse to $[TM_j]^j$. Conversely, if  for each  point  $a\in  \mB_j \cap M_i$,  $0\leq j<i \leq k$, the Lagrangian  distribution $[\nu_j]^j$ is transverse to $[TM_j]^j$ then $\Skel(W)$ is  generically arboreal. 
\item If the  skeleton $\Skel(W)$ is positive arboreal, then for each point $a\in  \mB_j\cap M_i  \cap M_\ell$,  $0\leq j<i<\ell\leq k$, the triple  $[T_aM_\ell]^N, [T_aM_i]^N, [\nu_j]^N$, where $N:=\Span(T_aM_i,T_aM_\ell)$, is $\pprec$-cyclically ordered. Conversely, if for each point $a\in  \mB_j\cap M_i  \cap M_\ell$,  $0\leq j<i<\ell\leq k$, the triple  $[T_aM_\ell]^N, [T_aM_i]^N, [\nu_j]^N$ is $\pprec$-cyclically ordered, then $\Skel(W)$ is generically positive arboreal.
\end{enumerate}
 \end{lemma}
 \begin{proof}  The top block $B_k$ is attached to $B_{k-1}$ along a ribbon of a smooth Legendrian, and the transversality condition implies that the front projection to $M_{k-1}$ is an immersion. Generically it has transverse self-intersections, and hence $\Skel(W_{\geq k-1})$ is arboreal in the case (i) and positive arboreal in the case (ii). Continuing by induction we will assume that $\Skel(W_{>j})$ is arboreal. Then the skeleton of the attaching hypersurface $V_j$ of $W_{>j}$ to $B_j$ is arboreal as well. By the axiom on transverse Liouville cones in the definition of arboreal singularities, the transversality condition then  implies  that generically $\Skel(W_{\geq j})$ is arboreal. In the case (ii) the skeleton is positive arboreal, as can be seen directly from inspection of the normal forms given in \cite{AGEN20a}.
 The converse direction is straightforward.
   \end{proof}

\begin{cor}\label{cor:posW-posL}
The skeleton of a positive cotangent building is generically a positive arboreal.
\end{cor}

\subsubsection{More on polarizations}
 
 We will need the following lemma  from linear algebra.
 \begin{lemma}\label{lm:pos-red-space}
 Let $\tau,\nu$ denote  $0$-section and the cotangent fiber in  $T^*\R^n$. Given a
collection of  hyperplanes $\Pi=\{\Pi_1,\dots, \Pi_k\}\subset\R^n$  denote by $\nu_1,\dots, \nu_k$ their Lagrangian conormals, and set $P_j:=\Span(\tau,\nu_j)$.
Denote by $\cP(\Pi)$ the space of 
  Lagrangian planes $\eta$ transverse to $\tau$ such  that  the triples
of lines $([\tau]^{P_j}, [\eta]^{P_j},[\nu]^{P_j})$ are $\prec$-cyclically ordered for
all $j=1,\dots,k$. Denote by $\cP_+$ the subspace of $\cP(\Pi)$ consisting 
   of  those Lagrangian planes such that the triple  $(\tau,\nu, \eta)$ is  $\prec$-cyclically ordered.
Then   $\cP(\Pi)$ and $ \cP_+\subset \cP(\Pi)$ are non-empty  convex subsets of the affine space of Lagrangian planes transverse to $\tau$.
\end{lemma}
   \begin{proof}
   Any  Lagrangian plane   transverse to $\tau$ can be viewed as  a quadratic form on $\nu$. The subspace $\cP_+$ consists of positive definite quadratic forms, while $\cP(\Pi)$ consists of quadratic forms which take positive values on 
 the covectors $A_j\in\nu=\tau^*$ which define the hyperplanes $\Pi_i$, $j=1,\dots, k$.
Both spaces are non-empty and convex.
    \end{proof}
 
 Below in this section we view   polarizations of a  Lagrangian as transverse Lagrangian foliations  on its neighborhood.  Similarly, we assume that polarizations of a Legendrian   $\Lambda$ 
 are extended to its ribbon, i.e. a given embedding of $\sT^*\Lambda$ as  a Weinstein hypersurface.
 \begin{lemma}\label{lm:char-pol}
 Let us fix a reference polarization $\nu$ of a Lagrangian $L_0 \subset (X,\omega)$.
 Then the space $\Pol(L_0)$ of all polarizations on $L_0$ can be viewed as the space of  germs along $L_0$ of symplectomorphisms $T^*L_0\to T^*L_0$ which  fix $L_0$, or alternatively as the space of germs along $L_0$ of functions on $T^*L_0$ which vanish on $L_0$ together with their differential.
 \end{lemma}
 \begin{proof}
   For any   polarization $\eta\in\Pol(L_0)$ there exists  a unique symplectomorphism of $T^*L_0$ which fixes $L_0$ and which sends $\nu$ onto $\eta$ as foliations. Each function on $H$ on $T^*L_0$ vanishing on $L_0$ together with its differential generates a symplectomorphism fixing $L_0$ as the time 1 map of its Hamiltonian flow. Conversely,   each symplectomorphism germ $\phi$ fixing $L_0$ defines the required function $H$ by the conditions $\phi^*pdq= pdq+dH$, 
   $H|_{L_0}=0$.
  \end{proof}

Let   $L=C(\Lambda)\cup L_0$  be an arboreal Lagrangian, where  $\Lambda\subset \p_\infty T^*L_0$ is an arboreal Legendrian.  A polarization $\eta$ is called {\em positive} for $L$ if for any  point $\lambda\in L\cap \ol{C(\Lambda)}$ and any Lagrangian plane $T\in T_\lambda(T^*L_0)$ tangent to one  of  smooth pieces of $L$   different from $L_0$ the triple $([T_\lambda L_0]^N,[T]^N,[\eta(\lambda)]^N)$ is  $\prec$-cyclically ordered, where  $N=\Span(T_\lambda L_0, T)$. Such a polarization is automatically transverse to $L$ on $\Op L_0$ (comp.
Lemma \ref{cor:cone-over-regular}).
 
  \begin{lemma}\label{lm:distr-pos-Lag}
 Let   $L=C(\Lambda)\cup L_0$  be an arboreal Lagrangian, where  $\Lambda\subset \p_\infty T^*L_0$ is an arboreal Legendrian and  $\eta$    a  positive polarization of $L_0$.  Denote   by $\Pol(L,\eta)$ the space of polarizations  $\mu$ of $L_0$  for which the corresponding Liouville field is tangent to the cone  $C(\Lambda)$ and the  triple $(T_\lambda L,\mu(\lambda),\eta(\lambda))$ is  $\prec$-cyclically ordered for any $\lambda\in L_0$.
 Then $\Pol(L,\eta)$ is  non-empty and convex.
 \end{lemma}
 \begin{proof} Viewing  $\mu\in\Pol(L,\eta)$ as a function on $T^*L_0$  the condition
 $\mu\in\Pol(L,\eta)$ means that $\mu$ vanishes on $L$ and that the quadratic part of $\eta-\mu$ along $L_0$ is positive definite on cotangent fibers. Hence, the convexity claim for $\Pol(L,\eta)$ is straightforward.   Let us  show that the space $\Pol(L,\eta)$ is non-empty.
Take any polarization $\eta'$ such that the triple $(T_\lambda L_0, \nu(\lambda_0),\eta'(\lambda))$ is $\prec$-cyclically ordered for all $\lambda\in L_0$. In particular, $\zeta'$ is transverse to $L$ on $\Op L_0$.

The tangent cone to $\ol{L\setminus L_0}$ at a point $\lambda\in L_0$ is the union of  Lagrangian conormals  $\nu_{j,\lambda}$ to a collection  $T_{j,\lambda} $ of transverse hyperplanes    $\Pi_{j,\lambda}\subset  T_\lambda L_0$, $j=1,\dots, k$. Using Lemma \ref{lm:pos-red-space} we can find a homotopy $\eta_t$ connecting $\eta_0=\eta$ and $\eta_1=\eta'$ such that the triple $([T_\lambda]^{N_{j,\lambda}},[\nu]^{N_{j,\lambda}},[\eta_t]^{N_{j,\lambda} }) $ is $\prec$-cyclically ordered for all $0 \leq t \leq 1$. Here we denoted
$N_{j,\lambda}:=\Span(T_\lambda L_0,\nu_{j,\lambda})$.

 In particular, $\eta_t$ is transverse to the germ of $L$  along $L_0$ for all $t\in[0,1]$.  There exists a Hamiltonian isotopy $\psi_t$  defined on $\Op L_0$   such that $\psi_t|_{L_0}=L_0$ and  $\psi_t(\eta_t)=\eta$, $t\in[0,1]$.
Denote $L^t:=\psi_t(L)$. Using a parametric version of the stability theorem which keeps track of polarizations (this is Corollary 3.14 in \cite{AGEN20a}) we can construct    for each point $\lambda\in L$   a Hamiltonian isotopy  $\phi_{t,\lambda}$ on $\Op\lambda$ such that \begin{itemize}
\item [(i)] $\phi_{t,\lambda}|_{L_0\cap\Op\lambda}=\Id$;
\item[(ii)] $\phi_t(\eta|_{\Op\lambda})=\eta|_{\Op\lambda}$;
\item[(iii)] $\phi_{t,\lambda}(L^t\cap\Op\lambda)=L\cap \Op\lambda$.
 \end{itemize}
Then the polarization $\mu_{\lambda}:=\phi_{1,\lambda}\circ \psi_1(\mu)$ belongs to the space $\Pol(L\cap\Op\lambda,\eta|_{\Op\lambda}).$ In view of compactness of $L_0$ we can cover $L_0$ by finitely many balls  $U_j$, $j=1,\dots, N,$ centered at the points $\lambda_j$ such that $\mu_{\lambda_j}\in \Pol(L\cap U_j,\eta|_{U_j})$. Given a partition of unity $\sum_1^N\theta_j=1$ subordinate to the covering, we define the required element of $\Pol(L,\eta)$ by the formula
$\mu=\sum_1^N\theta_j\mu_{\lambda_j}$.
 \end{proof}

    
    \subsubsection{From positive arboreals to positive buildings}

  \begin{prop}\label{prop:two-positivities}
 If the  skeleton of a cotangent     building is   positive  arboreal,  then the cotangent building structure can be deformed to be positive without changing the skeleton.  Similarly, if $\eta$ is a Lagrangian distribution which is  positive   for the skeleton,  then it can be made positive for the building without changing the skeleton. \end{prop}
 
 \begin{proof}
  We  argue by induction on blocks beginning with  the bottom block $B_1$. For a cotangent building consisting of a single block $B_1=\sT^*M_1$ the positivity condition is vacuous and to make $\eta$ positive we just need to choose the polarization  $\nu_1$, so that
  along $M_1$ we have  $\nu_1\in C(TM_1,\eta)$.

  Suppose we  have  already constructed a positive building structure on $W_{<j}=\bigcup_{i<j}B_i$ such that $\eta|_{W_{<j}}$ is positive.  
Let $L=\Skel(W)$ and denote $L_j:=L\cap (B_j=T^*M_j)$.  Suppose the block $B_j$ intersects with blocks $B_{j_1},\dots, B_{j_m}$, $1\leq j_1<\dots<i_m<j$, and let $P_i$ be the face of $M_j$ such that $B_j$ is attached to  $B_{j_i}$ along a ribbon $V_i=\sT^*P_i\subset B_{j_i}\setminus M_{j_i}$. The arboreal Lagrangian $L_j$ intersects the ribbon $V_i$ along an arboreal Legendrian $\Lambda_i$ containing $P_i$. Denote by $\mu_i$ the polarization of $P_i$ induced by the polarization $\nu_j$ of $M_j$ and by $\ol\nu_{j_i}$ the polarization obtained  by reducing the polarization $\nu_{j_i}$ of $B_{j_i}$ with respect to the symplectic conormal of the Liouville field $Z_{j_i}$.

Arguing by induction beginning with $i=m$, we first use Lemma \ref{lm:distr-pos-Lag} to deform $\mu_{m}$ to a polarization $\mu'_{m}$ of $P_m$ such that its Liouville field is tangent to $\Lambda_j$ and the triple $(TP_m,\mu'_m, \ol\nu_{j_m})$ is $\prec$-cyclically ordered. We extend $\mu'_m$ to $B_j$ as a polarization such that its Liouville field is tangent to $L_j$, which we will still denote by $\nu_j$, forgetting the previous polarization of $M_j$.
Note that on $\Op P_m\cap P_{m-1}$ the reduced triple $([TP_{m-1}]^{\zeta_m}, [\mu_{m-1}]^{\zeta_{i_m}}, [\nu_{j_{m-1}}]^{\zeta_{i_m}})$ is $\prec$-ordered, thanks to the positivity of the building $W_{<j}$.
Applying Lemma  \ref{lm:distr-pos-Lag}  again, we can deform the polarization $\mu_{m-1}$ to a polarization $\mu_{m-1}'$ whose Liouville field is tangent to $\Lambda_{j-1}$, which coincides with 
$\ol\nu_j:=[\nu_j]^{\zeta_{i_{m-1}}}$ on a neighborhood $U\supset P_m\cap P_{m-1}$, and such that  the triple $(TP_{m-1},\mu'_{m-1}, \ol\nu_{j_{m-1}})$ is $\prec$-cyclically ordered outside $U$. The extension claim follows from  convexity: one  first constructs $\mu'_{m-1}$ satisfying the positivity condition in the complement of a neighborhood $U'\Subset U$. Then, using a cut-off function $\theta$ equal to $1$ on $U'$ and $0$ outside $U$, one redefines $\mu_{m-1}$ as $(1-\theta)\mu_{m-1}' +\theta\ol\nu_j$. Again we extend to a polarization of $B_j$ such that its Liouville field is tangent to $L_j$ and we still call the new polarization $\nu_j$. Continuing this process, we obtain a   polarization $\nu_j$ on $B_j$ making the building $W_{\leq j}$ positive. Finally, to ensure that $\eta$ is positive along $M_j$ we can further deform $\nu_j$ via a homotopy fixed on $W_{<j} \cap B_j$ by final application of Lemma \ref{lm:distr-pos-Lag}.
\end{proof}

  \section{Ridgification of Lagrangians}\label{sec: ridgy}
  
  In this section we recall the ridgification theorem, as well as its formal analogue, and describe the canonical Wc-structure in the neighborhood of a ridgy Lagrangian.

   \subsection{Geomorphology of Lagrangian ridges}\label{sec:geo}
   
   \subsubsection{Ridgy Lagrangians}
As a stepping stone towards arboreal skeleta it will be useful to consider a class of singular Lagrangian and Legendrian submanifolds, called {\it ridgy}, which were introduced in \cite{AGEN19}. Let us recall their definition. In the standard  symplectic $(\R^2, dx\wedge dy)$ consider the subset
    $R=R_{1,2}=\{xy=0, x\geq 0,y\geq 0\}$.
\begin{definition} The {\em model ridge} $R_{k,n} \subset \R^{2n}$ of order $k$ is the product 
  $R_{k,n}=R^k\times\R^{n-k}\subset (T^*\R)^k\times T^*\R^{n-k}=T^*\R^{n}$. \end{definition}
  
  \begin{example}
The order $n$ model ridge $R_{n,n} \subset T^*\R^n$ is the union to all the inner conormals of the faces of a quadrant in $\R^n$, hence is the union of the $2^n$ linear Lagrangians $\{ p_j=q_k =0, \, q_j, p_k \geq 0 , \, j \in I, \, k \not \in I \}$, where $I \subset \{ 1, \ldots , n \}$. 
\end{example}

 \begin{definition} An $n$-dimensional {\it ridgy Lagrangian} in a symplectic manifold $(X,\omega)$ is
  a closed subset $L\subset X$ which is covered  by open neighbohoods $U_i$ such that each $(U_i,U_i\cap L)$ is symplectomorphic to some $(B, B\cap R_{k,n})$, $0 \leq k \leq n$.  \end{definition}
  
          \begin{figure}[h]
\includegraphics[scale=0.5]{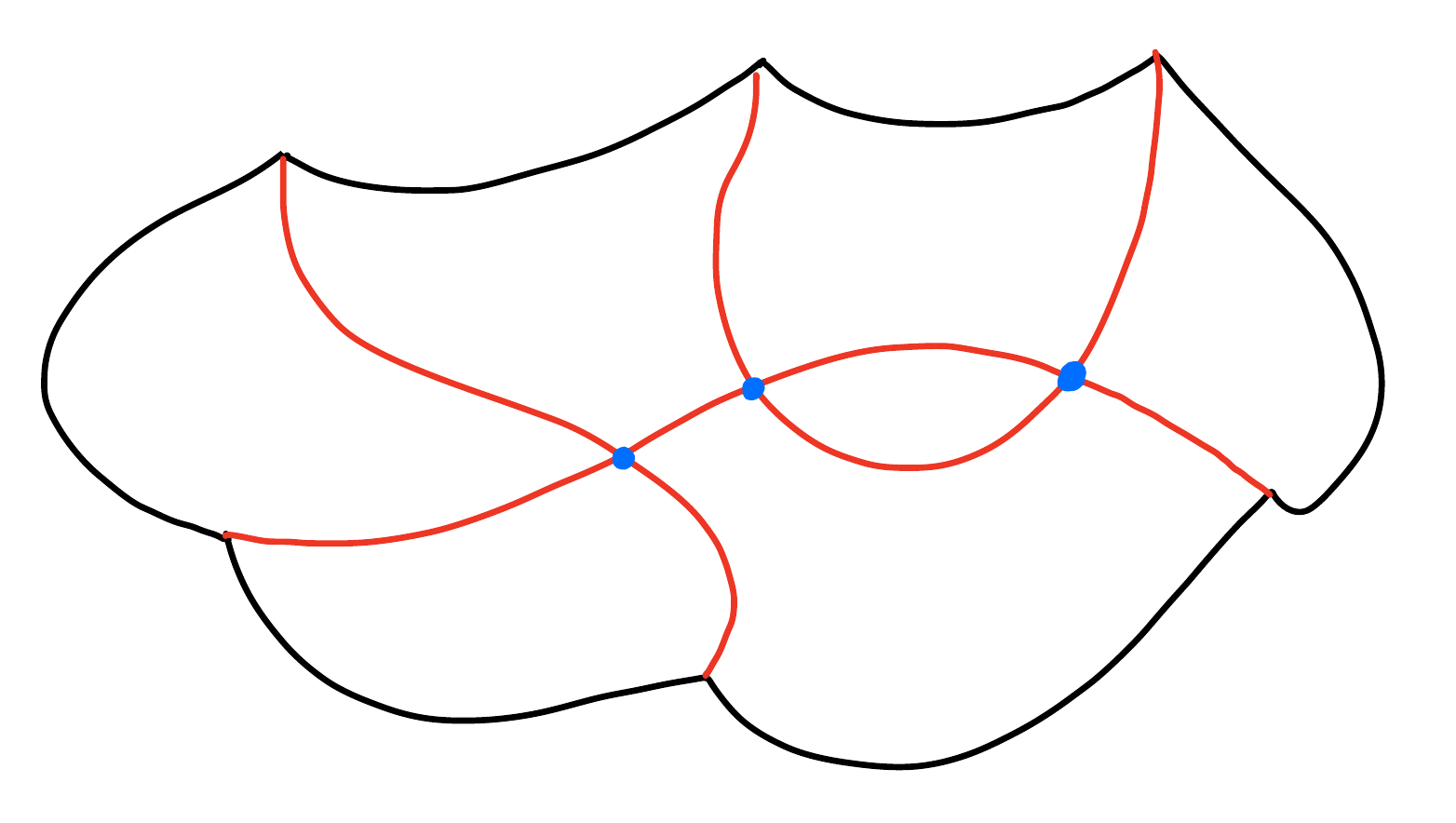}
\caption{A 2-dimensional ridgy Lagrangian has order 1 ridges along a union of separating curves, which intersect in order 2 ridges along a discrete set of points.}
\label{fig:ridgylagrangians}
\end{figure}

 A ridgy Lagrangian $L$ admits a natural stratification 
according to the order of ridges, so that the top dimensional stratum is the smooth part of the ridgy Lagrangian. 
Any ridgy Lagrangian can be viewed as the limit as $\eps \to 0$ of a family of smooth Lagrangians $L(\eps)$ homeomorphic to $L=L(0)$; they are obtained by using the model $R(\eps)=\{xy=\eps, x\geq 0,y\geq 0\}$ instead of $R$. 

\begin{definition} Given a contact manifold $(Y,\xi)$, a ridgy Legendrian $\Lambda\subset \xi$ is defined as a ridgy Lagrangian in a Weinstein hypersurface $\Sigma\subset Y$. 
\end{definition}

Note that $R_{k,n}\subset \R^{2n}$ is invariant with respect to the radial contracting vector field $Z_n= \sum_{j=1}^n q_j \p/\p q_j -  p_j\p/\p p_j $. Hence its link $\p R_{k,n}:=R_{k,n}\cap S^{2n-1}$ in the  standard contact sphere is an $(n-1)$-dimensional ridgy Legendrian.  Similarly, $R_{k,n}\subset\R^{2k}\times T^*\R^{n-k}$ is invariant with respect to the contracting field $Z_k -\sum_{j=1}^{n-k}  p_j \p/\p p_j$, where $\sum_{j=1}^{n-k}  p_j \p/\p p_j$ is the canonical Liouville field for the factor $T^*\R^{n-k}$.

         \begin{figure}[h]
\includegraphics[scale=0.6]{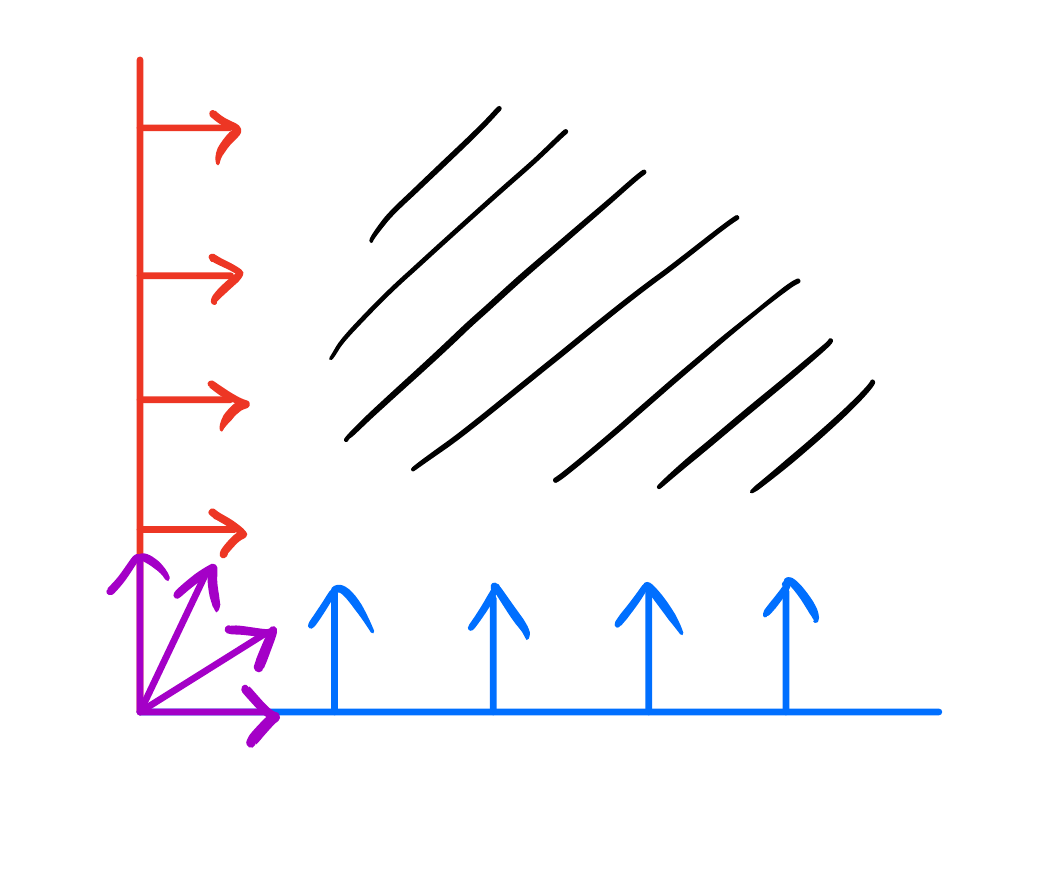}
\caption{The 2-fold model ridge $R^2=R \times R \subset T^*\R \times T^*\R$ is symplectomorphic to the union in $T^*\R^2$ of the first quadrant in $\R^2$, the inner conormals to the positive $x$ and $y$ axes and the quarter-conormal to the origin lying in the first quadrant.}
\label{fig:altmodel}
\end{figure}







 \subsubsection{The formal ridgification theorem}\label{thm:ridgy-main}
First, we recall from \cite{AGEN19} the notion of a tectonic field, which is the formal (i.e. non-integrable) analogue of a ridgy Lagrangian.
  Let $M$ be a manifold with corners and $B=\sT^*M$ the corresponding cotangent block.
  Recall that a Lagrangian plane field $\eta$ along $M$ which is transverse   to the vertical distribution $\nu$ can be viewed as a field of quadratic forms on $TM$. To simplify the notation we will use the same symbol to denote the Lagrangian distribution and the corresponding field of quadratic forms.
 
 By a  {\em dividing hypersurface} $N\subset M$ we mean a  properly embedded, co-oriented codimension 1 submanifold with corners such that $(N,\p N)\subset (M,\p M)$ is homologically trivial.   A collection of  dividing hypersurfaces $\{N_j\}_j$ is said to be in {\em in general position} if each $N_j$ is transverse to the intersection $N_{i_1} \cap \dots \cap N_{i_k}$ of any subset of the other hypersurfaces.
 
  Given a  collection  of  dividing hypersurfaces $N_1,\dots, N_k \subset M$  in general position, a  {\em tectonic field}  $\lambda$ with {\it faults} along $\{N_j\}_j$ is  a collection of smooth graphical Lagrangian plane fields $\{\sigma_Q\}_Q$ defined   over the closures $\overline{Q}$ of the  connected components $Q$ of $M \setminus \bigcup_{j=1}^kN_j$, such that  for every $j=1,\dots, k$ there exists a field  $\ell_j$ of  of non-vanishing 1-forms  along $N_j$ for which the following conditions are satisfied:   
     \begin{itemize}
 
\item[(i)] For each component  $P$ of $N_i \setminus \bigcup_{j\neq i} N_j$  adjacent to components $Q_\pm$  of $M \setminus \bigcup_j N_j$ we have that the difference $\lambda_{Q_+} - \lambda_{Q_-}$ is the rank $1$ quadratic form $\eta_i=\ell_i^2$, where the co-orientation of $N_j$ points into $Q_+$;
 
\item[(ii)] Along each intersection $N_{j_1} \cap \dots \cap N_{j_m}$ the hyperplane fields $\tau_{j_s}=\ker(\eta_{j_s})$, $s=1, \ldots ,m$, are transverse to all possible intersections of the $\tau_{j_r}$, $r \neq s$.
\end{itemize}

 A tectonic field $\eta$  is called {\it $\eps$-small} if it  deviates from $TM$ by an angle $<\eps$, where a fixed but arbitrary Riemannian metric is understood.
 A tectonic field $\lambda$ is called {\it collared} if for each $k$-face $F$ of $M$ the tectonic field $\lambda$ splits as a product of a tectonic field on $F$ and the trivial (tangent) field in the collar directions. In terms of quadratic forms, this means that the quadratic form is the sum of a quadratic form on $TF$ and the zero form in the collar directions. In particular this gives the notion of a smooth Lagrangian plane field which is collared with respect to a corner structure. We can now state the formal version of the ridgification theorem.

\begin{thm}[\cite{AGEN19}] \label{prop:formal-ridgy}
Let $M$ be a manifold with corners and  $\eta$ any Lagrangian plane field in $\sT^*M$.  Then there exists a tectonic field $\lambda$ transverse to $\eta$. 
\end{thm}

\begin{remark}\label{rmk: extension-form-formal-ridgy}
The formal transversalization theorem holds in several variants, which will be needed for our applications below,.
\begin{itemize}
\item[(i)] $C^0$-control: for any fixed $\eps>0$ we may demand that $\lambda$ is $\eps$-small.
\item[(ii)] Relative form: if for a compact set $A \subset  M$ we are already given an $\eps$-small tectonic field $\lambda_0$ over $\Op A$ transverse to $\eta$, then $\lambda$ can be chosen equal to $\lambda_0$ over $\Op A$. In this case the $C^0$-control means that $\lambda-\lambda_0$ can be made $\eps$-small.
\item[(iii)] There is also an extension formulation as follows: if $K_1,K_2 \subset M$ are disjoint compact subsets and $\zeta$ is any tectonic field, then we may find a $\lambda$ satisfying the conclusion of the theorem on $Op(K_1)$ and such that $\lambda=\zeta$ on $Op(K_2)$. In this case the $C^0$-control means that $\lambda-\zeta$ can be made $\eps$-small. 
\item[(iv)] The relative and extension forms can be combined if we assume that $\lambda_0=\zeta$ on $Op(A \cap K_2)$.
\item[(v)] Collared version: if we assume that in the neighborhood $F \times \cI^k$ of a $k$-face $F$ the Lagrangian field $\mu$ splits as the product of $\mu_F \subset T^*F$ and the cotangent (vertical) distribution on $T^*\cI^k$, then $\lambda$ can be taken to be collared. The collared version holds with $C^0$-control and with the relative and extension forms.
\end{itemize}
\end{remark}

A tectonic field is called {\em aligned} if for each fault $N_i$ we have $\ker(\eta_i)=TN_i$, where $\eta_i$ is the rank 1 form giving the discontinuity of the tectonic field along $N_j$. Given a Lagrangian plane field $\eta$ in $\sT^*M$ it is in general not possible to find an aligned tectonic field which is transverse to $\eta$, but it is possible to find an aligned tectonic field which is transverse to a Lagrangian field homotopic to $\eta$. More precisely, we have:

\begin{thm}[\cite{AGEN19}]\label{prop:formal-ridgy-aligned}
Let $M$ be a manifold with corners and  $\eta$ any Lagrangian field in $\sT^*M$.   Then there exists an aligned tectonic field $\lambda$ transverse to a Lagrangian field $\wh \eta$ homotopic to $\eta$. 
\end{thm}

\begin{remark}\label{rmk: extension-form-formal-ridgy-aligned}
The aligned formal transversalization theorem holds in several variants, which will be needed for our applications below,.
\begin{itemize}
\item[(i)] $C^0$-control: for any fixed $\eps>0$ we may demand that $\lambda$ is $\eps$-small.
\item[(ii)] Relative form: if for a compact set $A \subset  M$ we are already given an $\eps$-small aligned tectonic field $\lambda_0$ over $\Op A$ transverse to $\eta$, then $\lambda$ can be chosen equal to $\lambda_0$ over $\Op A$. Moreover we may assume that the homotopy between $\wh \eta $ and $\eta$ is constant on $Op(A)$. In this case the $C^0$-control means that $\lambda-\lambda_0$ can be made $\eps$-small.
\item[(iii)] There is also an extension formulation as follows: if $K_1,K_2 \subset M$ are disjoint compact subsets and $\zeta$ is any aligned tectonic field, then we may find a $\lambda$ satisfying the conclusion of the theorem on $Op(K_1)$ and such that $\lambda=\zeta$ on $Op(K_2)$. In this case the $C^0$-control means that $\lambda-\zeta$ can be made $\eps$-small. 
\item[(iv)] The relative and extension forms can be combined if we assume that $\lambda_0=\zeta$ on $Op(A \cap K_2)$.
\item[(v)] Collared version: if we assume that in the neighborhood $F \times \cI^k$ of a $k$-face $F$ the Lagrangian field $\mu$ splits as the product of $\mu_F \subset T^*F$ and the cotangent (vertical) distribution on $T^*\cI^k$, then $\lambda$ can be taken to be collared and we may moreover assume the homotopy between $\wh \eta$ and $\eta$ to be through collared distributions. The collared version holds with $C^0$-control and with the relative and extension forms.
\end{itemize}
\end{remark}

\subsubsection{The ridgification theorem}

Given a tectonic field $\lambda$ in $\sT^*M$, a ridgy Lagrangian $ L\subset \sT^*M$ is called $\delta$-close to $\lambda$ if for any point $a\in L$ and any tangent plane $T$ to $L$ at $a$ there exists a non-fault point $b\in M$, $\delta$-close to $a$ such that the angle between $\lambda(b)$ and the plane $T$ parallel transported to $b$ along a geodesic is $<\delta$. As before, a fixed but otherwise arbitrary Riemannian metric on $M$ is understood.

A ridgy Lagrangian in $ L\subset \sT^*M$ is called {\it collared} if for each collar $Q \times \cI^k \subset M$, where $Q$ is a $k$-face, we have $L \cap T^*(Q \times \cI^k)  = L_Q \times \cI^k$ for $L_Q \subset T^*Q$ a ridgy Lagrangian and $\cI^k \subset T^*\cI^k$ the zero section.

\begin{remark} If $L \subset \sT^*M$ is a collared ridgy Lagrangian, then it is adapted for the Wc-structure of $\sT^*M$ inherited from the collar structure of $M$.
\end{remark}

The ridgy Lagrangians we will consider are of a special type: they are all obtained from the zero section $M \subset \sT^*M$ by means of a ridgy isotopy, which is defined as follows.

\begin{enumerate}
\item
Let $N_1 , \ldots , N_m \subset M$  be co-oriented separating   hypersurfaces   defined by equations $\phi_j=0$ for some $C^\infty$-functions $\phi_j:M\to\R$ without critical points on $N_j$.  We assume that the $N_j$ are co-oriented by the outward transversals to the domains $\{\phi_j\leq 0\}$. Denote $\phi_j^+=\max(\phi_j,0)$ and choose a cut-off function $\theta_j$ which is equal to 1 on $N_j$ and to $0$ outside a neighborhood of $N_j$.  Define a function $\Phi:M \to \R$ (which is $C^1$ and piecewise $C^\infty$) by the formula
 $$\Phi:=\sum\limits_{j=1}^m\theta_j\left(\phi_j^+\right)^2.$$ An {\em earthquake isotopy} with faults $N_j$  is defined as a family of Lagrangians $L_t \subset \sT^*M$ given by the homotopy of generating functions $t\Phi$, i.e. $L_t=\{p=td\Phi\}$, $t\geq 0$.
\item A {\em ridgy isotopy} is   an earthquake isotopy followed by an ambient Hamiltonian isotopy.
 \end{enumerate}
 
           \begin{figure}[h]
\includegraphics[scale=0.5]{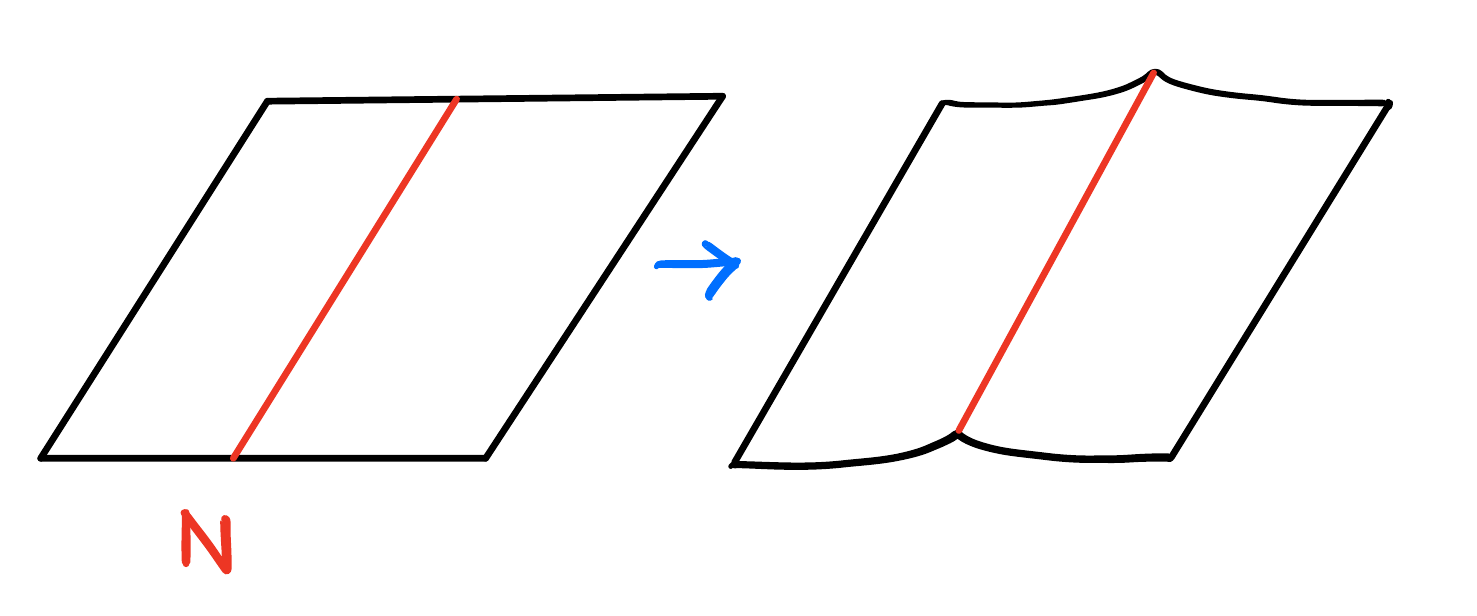}
\caption{An earthquake isotopy.}
\label{fig:ridgyisotope}
\end{figure}

Note that $L_0=M$ and the earthquake isotopy can be realized by an ambient Hamiltonian homotopy beginning from any $t>0$. We can now state the integrable version of the ridgification theorem.

\begin{thm}[\cite{AGEN19}]\label{thm:ridgification}
For any Lagrangian field $\eta$ in $\sT^*M$ there exists a ridgy isotopy of $M$ to a ridgy Lagrangian $L$ which is transverse to $\eta$.
  \end{thm}  
  
  \begin{remark}
  The ridgification theorem also holds in several variants:
  \begin{enumerate}
  \item $C^0$-control: we may assume that the ridgy isotopy is $C^0$-close to the identity.
  \item Relative form: if $\eta$ is transverse to $M$ on $\Op A$ for $A \subset M$ a closed subset, then we may assume that the ridgy isotopy is constant on $\Op A$.
  \item Collared version: if we assume that in the neighborhood $F \times \cI^k$ of a $k$-face $F$ the Lagrangian field $\mu$ splits as the product of $\mu_F \subset T^*F$ and the cotangent (vertical) distribution on $T^*\cI^k$, then we may demand that the ridgy isotopy is collared.
  \end{enumerate}
  \end{remark}
  
   
The proof of Theorem \ref{thm:ridgification} is given in \cite{AGEN19}. The argument is roughly the following: first one constructs an aligned tectonic field $\lambda$ transverse to $\eta$ using Theorem \ref{prop:formal-ridgy-aligned}. Then one integrates $\lambda$ to a ridgy Lagrangian which is transverse to $\eta$ along its ridge locus using the refined holonomic approximation \cite{AG18a}. In the complement of the ridge locus there is no homotopical obstruction to removing the tangencies so one can achieve transversality using the wrinkling theorem for Lagrangian submanifolds \cite{AG18b}. Finally, one exchanges the wrinkles for order 1 ridges to complete the proof. 

For our applications in the present article it will be useful to first inductively apply the holonomic approximation step in extension form and only wrinkle at the end, so we quote here Lemma 4.8 from \cite{AGEN19} which we will use below.

\begin{lemma}[\cite{AGEN19}]\label{lemm: holonomic-approx-extension}
Let $\Lambda \subset T^*L$ be a ridgy Lagrangian, $R \subset \Lambda$ its ridge locus and $\gamma$ a Lagrangian field. Suppose that $\gamma$ is homotopic to a Lagrangian field $\wh \gamma$ which is transverse to $\Lambda$. Then there exists a Hamiltonian isotopy $\Lambda_t$ of $\Lambda$  and a Lagrangian field $\widetilde{\gamma}$ such that: \begin{enumerate}
\item $\Lambda_1 \pitchfork \gamma$ on $Op(R)$.
\item $ \Lambda_1 \pitchfork \widetilde \gamma$ everywhere.
\item $\widetilde{\gamma}$ is homotopic to $\gamma$ by a homotopy fixed on $Op(R)$.
\end{enumerate}
\end{lemma}

\begin{remark}
We have the following variants:
\begin{enumerate}
\item  Relative version: if $A \subset \Lambda$ is a closed subset and the homotopy between $\gamma$ and $\wh \gamma$ is fixed on $Op(A)$, then the Hamiltonian isotopy and the homotopy between $\widetilde \gamma$ and $\gamma$ can both be chosen to be fixed on $Op(A)$.
\item Extension form: if $K_1,K_2 \subset L$ are closed subsets we may find $\Lambda_t$ and $\widetilde \gamma$ such that the conclusions of the lemma are satisfied on $Op(K_1)$, the ridgy isotopy is constant on $Op(K_2)$ and the homotopy between $\widetilde{\gamma}$ and $\gamma$ is constant on $Op(K_2)$. Furthermore, we can combine this extension version with the relative version. 
\item Collared version: if we assume that in the neighborhood $F \times \cI^k$ of a $k$-face $F$ the Lagrangian field $\gamma$ splits as the product of $\gamma_F \subset T^*F$ and the cotangent (vertical) distribution on $T^*\cI^k$ and we moreover assume the homotopy between $\gamma $ and $\wh \gamma$ is through such distributions, then we may assume that $\Lambda_t$ is collared and the homotopy between $\widetilde{\gamma}$ and $\gamma$ is collared.
\end{enumerate}
\end{remark}
  
We also record for future reference the intermediate step in which one integrates an aligned tectonic field, before applying holonomic approximation.

\begin{lemma}[\cite{AGEN19}]\label{lem: intermediate}
Let $\gamma$ be a Lagrangian field on $T^*L$ and let $\lambda$ be an aligned tectonic field on $L$ which is transverse to a Lagrangian field homotopic to $\gamma$. Then there exists a ridgy isotopy $L_t$ such that $L_1$ is transverse to a Lagrangian field homotopic to $\gamma$ on a neighborhood of its ridge locus.
\end{lemma}  

As before, this holds in $C^0$-small, relative, extension and collared versions.

\subsection{Wc-structures for ridgy Lagrangians}\label{sec:ridgy-Wc-building}

\subsubsection{Wc-building structure associated with a ridgy Lagrangian}

Next we discuss a Darboux/Weinstein type theorem for the symplectic structure in the neighborhood of a ridgy Lagrangian.  Denote by $L_k$ the locus of $k$-dimensional locus of $(n-k)$-fold ridges in $L$. This stratifies $L$ as a union of smooth, relatively open submanifolds $L=L_0 \cup \cdots \cup L_n$.

Given a manifold with corners $M$ we denote by $M^{\triangle,j}$ the manifold with corners obtained from $M$ by truncating all corners of  dimension $\leq  j$. Thus, each $i$-face $P$ from $\p_iM$ for $i<j$ yields a 1-face $P^{\blacktriangle, i}:=
 P^{\triangle,i}\times\Delta^{j-i}\subset\p_1M^{\triangle,j}$, where  we denote by $\Delta_{j-i}$ an open $({j-i})$-dimensional simplex. Denote by $\sR^{2k}$ the germ of the standard symplectic $\mathbb{R}^{2k}$ at the origin.

\begin{lemma}
\label{lm:ridgy-W-structure}
Let $L$ be a ridgy Lagrangian in a symplectic manifold $(X, \omega)$ such that for each $k \geq 1$ the symplectic normal bundle to the order $k$ ridge locus is trivial. Then a neighborhood of $L$ admits a structure of a Wc-building $W=(U_0\to\dots \to U_n)$, where $$U_j:=\bigcup\limits_Q\sT^*{Q^{\triangle,j-1}}\times\sR^{2j}$$
  and the union is taken over all components $Q$ of $L_j$, such that the following properties hold: 
  \begin{itemize}
\item[(i)] The inclusion of $W$ into $X$ is a symplectic embedding.
\item[(ii)] The skeleton of $W$ is $L$.
\item[(iii)] The Weinstein structure underlying $W$ is Weinstein homotopic to the cotangent bundle structure $\sT^*L(\eps)$ for a smoothing $L(\eps)$ of $L$. 
\end{itemize}
\end{lemma}

 \begin{figure}[h]
\includegraphics[scale=0.45]{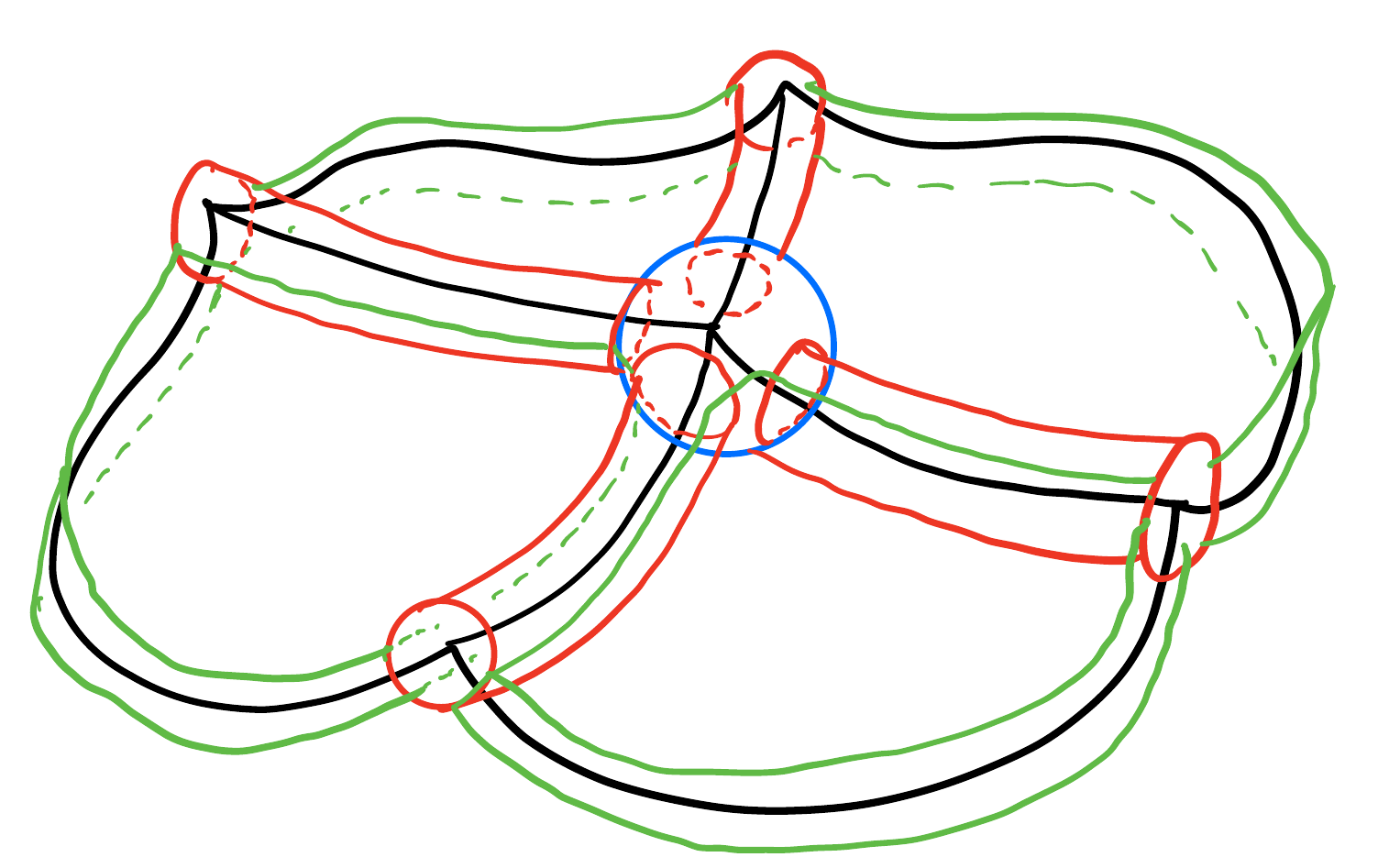}
\caption{The Wc-building structure in the neighborhood of a ridgy Lagrangian}
\label{Wcstructure}
\end{figure}

\begin{remark} 
The hypothesis on the symplectic normal bundles is included only for the sake of simplicity. It is obviously satisfied by the ridgy Lagrangians produced by the ridgification theorem and these are the only ridgy Lagrangians we care about for our applications. In general one should replace $\sT^*{Q^{\triangle,j-1}}\times\sR^{2j}$ by the appropriate symplectic bundle over $\sT^*{Q^{\triangle,j-1}}$ and the result is proved in the same way. Note that in either case the Wc-building structure is {\em not }a cotangent building.
\end{remark}

\begin{proof} 
Recall that $L_k$ denotes  the locus of $k$-dimensional locus of $(n-k)$-fold ridges in $L$. Denote by $U_0$ the union of so small    closed disjoint balls  $B_j$ centered at the points of $L_0$ that $L \cap B_j$ is invariant with respect to the radial contracting vector field inherited from the local model. Denote $L_j^{>0}:=L_j\setminus U_n$. Choose a small (in a similar sense) closed tubular neighborhood $U_{1}$ of $\oL_{1}^{>0}$. Set $L_j^{>1}=L_j^{>0}\setminus U_{1}$ for $ j>1$. 

Continuing this process we define $U_j\supset \oL_j^{>j-1}$ for $j=0,1,\dots, n$. Note that $\oL_j^{>j-1}$ is a manifold with corners of dimension $j$. Note also that $U_j$ can be presented as a symplectic fibration $\pi_j:U_j\to \sT^*\oL_j^{>j-1}$. The restriction $\pi_j|_{\oL_j^{j-1}}$ contains a subfibration $L\cap U_j\to\oL_j^{j-1}$ with the fiber
$R_{n-j,n-j}=R^{n-j}j$. 

The structural group of the fibration $\pi_j$ reduces to the discrete group of symmetries of the model ridge $R^{n-j}$. These symmetries commute with the action of the radial contracting field $Z_j$, and hence $U_j$ admits a global contracting field $Z_j^{j-1}$ which  restricts to the radial field $Z_j$ on fibers over $\oL_j^{j-1}$ and which vanishes along $\oL_j^{j-1}$.
This makes  $(U_j, Z_j^{>j-1})$ a Wc-manifold.

The required Wc-structure on $\Op L$ can now be constructed by successive vertical gluing.  Starting with $(U_0, Z_0)$ we vertically attach $(U_{1}, Z_{1}^{>0})$ using a Wc-hypersurface $U_0$, namely the ribbon of $\p_\infty U_0\cap L$, and the similar hypersurface in $\p U_{1}$ which plays the role of the nucleus of one of the boundaries of $\p U_{1}$.
Continuing this process we obtained the required Wc-structure.

Properties (i) and (ii) are immediate from the construction and for (iii) the homotopy between the constructed Wc-structure and $\sT^*L(\eps)$ can be obtained by repeating the whole construction for the smoothed $L(\eps)$ with the stratification inherited from the order of ridges in $L$. Finally the description of the $U_j$ given in the statement of the lemma follows from the usual Darboux/Weinstein theorem for isotropic submanifolds with trivial symplectic normal bundles.
\end{proof}


  \subsubsection{Ridgy Lagrangians in Wc-buildings}\label{sec:ridgy-W}
   Let $W= \bigcup_j B_j$ be a cotangent building, with $\nu_j$ the vertical distribution for $B_j$.   
  Given a Liouville cone $L$ over an   adapted  Legendrian in $W \setminus \Skel(W)$ we say that
  $L$ is {\em reduced transverse} to all $\nu_i$ defined near its corners
  if for each tangent plane $T$ to $L\cap B_{i_1}\cap\dots \cap B_{i_m}$ at a point $a\in B_{i_1}\cap\dots \cap B_{i_m}$ the reduction $[T]^{\zeta_I}\subset\zeta_I$ is transverse to the reductions $[\nu_{i_s}]^{\zeta_I}$, where $I=(i_1<\cdots <i_m)$.
 
  \begin{prop}\label{prop:main-ridgy}
 Let $\wh W$ be a symplectic manifold and  $W\subset \wh W$ a symplectic embedding of a  positive cotangent building. Let $L\subset \wh W  \setminus \Skel ( W)$ be a Lagrangian with corners such that $L\cap W$ is a Liouville cone over an adapted Legendrian in $W\setminus\Skel(W)$.
Let  $\nu_{-1}$  be a positive distribution which is     extended to $\wh W$.
  Then there exists a $C^0$-small ridgy isotopy which keeps $L\cap W$ adapted   and deforms $L$ to a ridgy Lagrangian $\wh L$ which is transverse   to $\nu_{-1}$ and reduced transverse
  to all the corresponding distributions $ \nu_{i} $ defined near its corners.   
  \end{prop}

\begin{proof} Consider $L$ as the  $0$-section of a block $\sT^*L$. Let $P_1$ be a face of $L$ which is a Legendrian in one of the blocks $B_k$. The distribution $\nu_k$ is adapted for this block structure and is invariant (and tangent) to the contracting field $Z_{P_1}$  adapted to this block. 
We apply the collared version of the formal ridgification theorem \ref{prop:formal-ridgy} to deform the horizontal distribution on $Op(P_1)=P_1\times \cI$ to an $\eps$-small collared tectonic field $\sigma'_1$ in $\sT^*P \times \sT^*\cI$, achieving transversality to the product of $[\nu_k]^{\zeta_{1}}$ and the vertical distribution on $T^*\cI$. Note in particular that $\sigma'_1$ is reduced transverse to $\nu_k$, i.e. $[\sigma'_1]^{\zeta_1}$ is transverse to $[\nu_k]^{\zeta_1}$. Furthermore, by the extension form of the formal transversalization theorem we may assume that $\sigma'_1=0$ outside of a neighborhood of $P_1$ in $L$.

 Note that  according to Lemma \ref{cor:cone-over-regular}, for any other index $j$ we have that $[\sigma_1]^{\zeta_j}$ is transverse to $[\nu_j]^{\zeta_j}$, or in other words, that
 $\sigma_1$ is reduced transverse to $\nu_j$. Using the extension form of the formal ridgification theorem \ref{prop:formal-ridgy} we can further inductively deform the tectonic field $\sigma_1$ into $\sigma_2,\dots, \sigma_k$ which is $\eps$-small, collared and reduced transverse to $\nu_{k-1}, \dots, \nu_1$ and finally construct an $\eps$-small collared tectonic field  $\sigma$ which is in addition is transverse to $\nu_{-1}$.
 
 To conclude the proof we want to apply the ridgification theorem to integrate the tectonic field $\sigma$. It will be easier to work one boundary face of $L$ at a time. We first note that by the same argument as before we can inductively use Theorem \ref{prop:formal-ridgy-aligned} instead of Theorem \ref{prop:formal-ridgy} to replace $\sigma$ with an $\eps$-small, collared, aligned tectonic field $\wh \sigma$, which is no longer reduced transverse to the $\nu_i$ or transverse to $\nu_{-1}$ but  has the property that there exists a family of symplectic bundle isomorphisms $\Phi_t : T(T^*L) \to T(T^*L)$ such that $\wh \sigma$ is reduced transverse to each $\Phi_1(\nu_j)$ and transverse to $\Phi_1(\nu_{-1})$. Then we integrate $\wh \sigma$ to a ridgy Lagrangian $\wh L$ using Lemma \ref{lem: intermediate} which still has the property that it is reduced transverse to each $\Phi_1(\nu_j)$ and transverse to $\Phi_1(\nu_{-1})$. 
 
 Now, starting with $P_1$, we may use Lemma \ref{lemm: holonomic-approx-extension} to first deform $\wh L$ along its ridge locus in a neighborhood of $P_1$ so that it becomes reduced transverse to $\nu_k$ in a neighborhood of the ridge locus. The deformation is achieved by a Hamiltonian isotopy which is $C^0$-close to the identity. Moreover, by the collared version of that lemma we may assume that the resulting ridgy Lagrangian, which now and it what follows we abuse notation and still denote by $\wh L$, is still collared. 
 
 As before, Lemma \ref{cor:cone-over-regular} assures us that $\wh L$ is reduced transverse to $\nu_j$ near the boundary of the next face $P_2$ in a neighborhood of its ridge locus, and our previous application of Lemma \ref{lemm: holonomic-approx-extension} guaranteed that the homotopical condition necessary to keep applying holonomic approximation is still satisfied, even in its collared form. Therefore we can use Lemma \ref{lemm: holonomic-approx-extension} once more to deform $\wh L$ along its ridge locus in a neighborhood of $P_2$ so that it becomes reduced transverse to $\nu_j$ in a neighborhood of its ridge locus and we can continue the process until we have achieved reduced transversality to $\nu_{k-1}, \cdots , \nu_1$ along the ridge locus in a neighborhood of the boundary for a collared $\wh L$. The deformation is still $C^0$-small. Moreover, again by Lemma \ref{cor:cone-over-regular} we have that $\wh L$ is transverse to $\nu_{-1}$ near its boundary, and $\nu_{-1}$ is homotopic to a Lagrangian field transverse to $\wh L$ by a homotopy fixed near the boundary.
 
 At this point we have a ridgy Lagrangian $\wh L \subset T^*L$ which is reduced transverse to the $\nu_i$ and which is collared with respect to the cotangent bundle structure $T^*L$. However, the intersection of $\wh L$ with $W$ is no longer a Liouville cone over an adapted Legendrian, so we need to fix this. Since all the above steps can be performed by $C^0$-small perturbations, recall that we may assume that $\wh L$ is in an $\eps$-neighborhood of the zero section in $T^*L$. Therefore, by a straightforward cutoff at the level of generating functions one may interpolate between our old $\wh L$ and a new $\wh L$ whose intersection with $W$ is a cone over an adapted Legendrian, while maintaining the reduced transversality to the $\nu_{k-1}, \ldots , \nu_1$. Moreover, the new $\wh L$ still has the property that it is transverse to a Lagrangian field homotopic to $\nu_{-1}$, with the homotopy fixed in a neighborhood of its boundary.
  
By the same argument as before, we may use Lemma \ref{lemm: holonomic-approx-extension} to further deform $\wh L$ to obtain a ridgy Lagrangian which is collared, reduced transverse to $\nu_{k-1}, \ldots , \nu_1$ and transverse to $\nu_{-1}$ on a neighborhood of the ridge locus, and such that $\nu_{-1}$ is homotopic to a Lagrangian field transverse to $\wh L$ via a homotopy fixed in a neighborhood of the boundary of $ \wh L$ and fixed in a neighborhood of its ridge locus. Hence we may apply the wrinkling theorem \cite{AG18b} in its relative form to further deform $\wh L$ relative to its boundary and ridge locus to achieve transversality to $\nu_{-1}$ everywhere by the introduction of wrinkles, which can easily be exchanged for order 1 ridges, see for example Corollary 4.13 of \cite{AGEN20a}, or Figure \ref{swap} below.  \end{proof}

         \begin{figure}[h]
\includegraphics[scale=0.5]{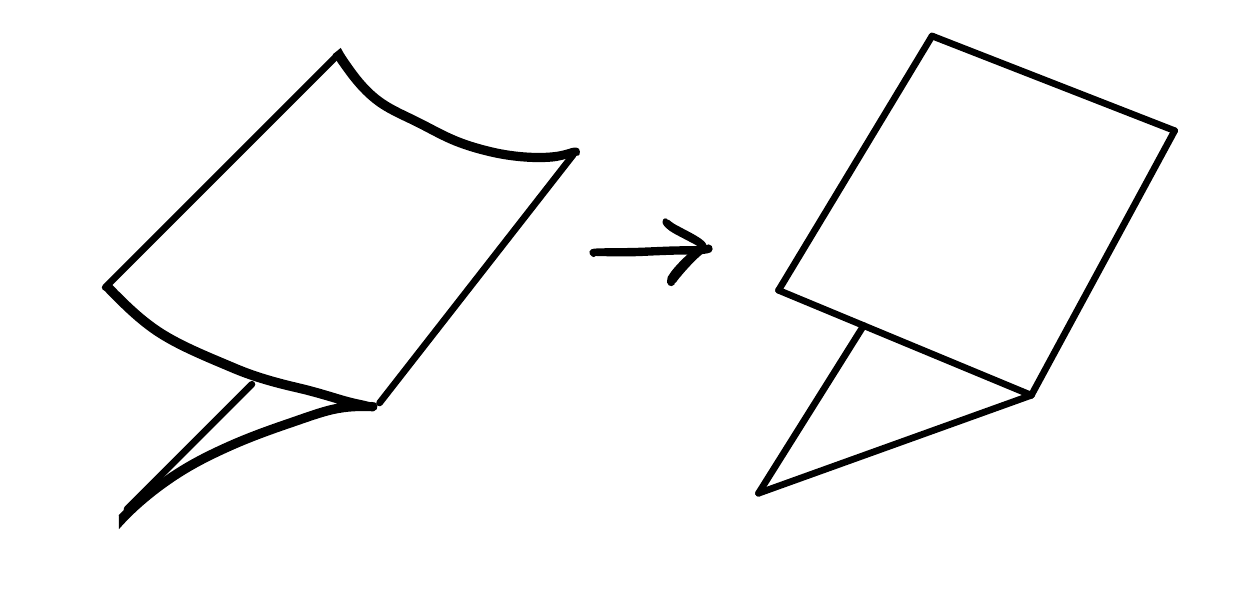}
\caption{The wrinkling theorem allows us to achieve transversality by the introduction of cuspidal singularities, but these can be easily replaced by order 1 ridges by simply replacing the unique local models.}
\label{swap}
\end{figure}



\section{Arborealization of skeleta}\label{sec:corner-arboreal}

In this final section we complete the proof of our main result: a Weinstein manifold admitting a polarization can be deformed to have a positive arboreal skeleton.

   \subsection{Immersions into arboreals}\label{sec:genuine-trans}
   
In this section we will present an inductive scheme to arborealize a ridgy Lagrangian. The most direct way of doing so produces an arboreal Lagrangian with boundary, i.e.~whose boundary is itself arboreal. However, we can do slightly better by inductively capping off the new boundary components that arise  so that the end result is an arboreal Lagrangian with smooth boundary. We therefore begin with a preliminary discussion which describes the capping process.

\subsubsection{Capping arboreal Legendrians}
   
   \begin{lemma} Let $M$ be a smooth manifold, possibly with boundary and corners, and $\Lambda\subset (\sT^*M\setminus M) \cap \sT^* \mM$ an arboreal Legendrian with smooth boundary $\p \Lambda$. Suppose that $\Lambda$ is positive and that it is also positive with respect to the vertical polarization $\nu$ of $\sT^*M$.
   Let $\wh\Lambda\supset\Lambda$ be an arboreal space without boundary such that $\wh \Lambda\setminus \Lambda$ is a manifold. Suppose that there exists an immersion $h:\wh\Lambda\setminus\Lambda\to\mM$, transverse to $\pi(\Lambda)$, extending the co-oriented front projection $\pi:\Op\p\Lambda\to\mM$. Let $H\subset \sT^*M\setminus M$ be the Legendrian positive conormal lift of the immersion $h$.
   Then the  closed arboreal Legendrian $\wt\Lambda:=\Lambda\cup H\subset \sT^*M\setminus M$ is positive arboreal and is positive with respect to $\nu$.
   \end{lemma}

   \begin{proof} Indeed, we just added   to $\Lambda$ a smooth Legendrian transverse to $\nu$. 
   \end{proof}
   
   \begin{remark}
   To be more precise, we should say that there exists a Legendrian lift $H \subset \sT^*M \setminus M$ of $h$ such that the conclusion of the lemma holds, or equivalently we could work in the idealized boundary $S^*M$ from the onset, where the lift $H$ is unique.
   \end{remark}
   
   \subsubsection{Genuine transversality}
   
   Let  $\eta$ be  a Lagrangian distribution in a symplectic manifold  $W$. We introduce a notion of transversality for piecewise smooth Lagrangians which can be intuitively thought of as ``transversality even after smoothing''. 
   
 \begin{definition} A piecewise smooth Lagrangian $L \subset W$ is called {\em genuinely transverse} to $\eta$ if for each point  $a\in L$ there exist Darboux coordinates $(p,q)$ near $a$, with $\eta( a )$ is  tangent to the cotangent fiber $q=0$ at the origin $0$, such that the Lagrangian $ \Op a\subset L $ is generated by a $C^1$-function $H(q)$.    
 \end{definition}

 \begin{remark}
 The ridgy Lagrangians produced by the ridgification theorem are transverse to the distribution used as input in the theorem, but in general are not genuinely transverse.
 \end{remark}
 
 Similarly, a piece-wise smooth Lagrangian map $f:L\to W$, i.e. not necessarily an embedding, is said to be genuinely transverse to a Lagrangian distribution $\eta$ if for every point $a\in L$ there exists a neighborhood $\Op a\subset L$   such that $f|_{\Op L}$ is a piecewise smooth Lagrangian embedding and $f(\Op L)$ is genuinely transverse to $\eta$. Note that any piece-wise smooth Lagrangian which is genuinely transverse  to $\eta$ can be approximated   by a family of  Lagrangians
    $L_t$, $t>0$, which are transverse to $\eta$ and such that   $L_t\mathop{\to}\limits^{C^0}L$ as $t \to 0$, where convergence  is smooth where $L$ is smooth.   Indeed, this follows from the corresponding approximation property for $C^1$-smooth generating functions.
 %

\begin{definition}
Let $L$ be an arboreal space and $M$ a smooth manifold of the same dimension $n$. A map $f:M \to L$ is called an {\em immersion} if there exists a stratification $M=M_0\cup M_1\cup\cdots \cup M_n$ by manifolds  $M_j$ of dimension $(n-j)$ such that $f$ is an immersion on each stratum.
\end{definition}

Recall that as an arboreal space, $L$ is equipped with an orientation structure $\kappa$, which is a line bundle equipped with identifications with the orientation line bundles of the smooth pieces and compatibility at the singularities determined by symplectic geometry. The pullback $f^*\kappa$ is a line bundle on $M$, with fixed identifications with the orientation bundle $\land^n TM_0$ for $P$ on $M_0$.

\begin{definition}
An immersion $f:M \to L$ is {\em oriented} if $f^* \kappa$ is isomorphic to the orientation bundle of $M$ relative to the fixed identification on $M_0$ . \end{definition}

The condition that $f:M \to L$ is oriented can be understood as follows. Let $A$ be a component of $M_1$ and $P,Q$ be components of $M_0$ adjacent to $A$. At a point $f(x) \in L$ in the closure of both $f(P)$ and $f(Q)$, the singularity of $L$ is modeled on signed rooted tree $\cT$. If $f(P)$ is contained in a smooth piece of $L$ closer to the root $\rho$ of $\cT$ than $f(Q)$ then the outward boundary coorientation of $A\subset\p P$ coincides with the coorientation of $f(A)$ in $P$. See also Remark~\ref{rem: pos char}.   
   
    
Let $L \subset W$ be a positive arboreal Lagrangian and $\eta$ a positive distribution for $L$. Then the symplectic vector bundle $TW|_L \to L$ is isomorphic to $\eta \oplus \eta^*$. We can realize $\eta^*$ as a Lagrangian plane field on $TW|_L$ which is transverse to $\eta$ and the space of such is contractible.

\begin{lemma}
Let $L\subset W$ be a positive arboreal Lagrangian and $\eta$  a positive distribution for $L$. For any oriented immersion $f:N \to L$ and any point $x \in N$, the germ of $f(L)$ at $f(x)$ is graphical with respect to the polarization $(\eta^*,\eta)$. 
\end{lemma}

\begin{proof}
We first check the condition along $M_1$. Let $x \in M_1$, so $f(x) \in L$ is an $A_2$ type singularity. There are two possibilities for the germ of $f:M \to L$ near $x$, excluding the trivial case where $f(\Op(x))$ stays in the smooth part of $L$, i.e. the zero section piece corresponding to the root $\rho$ of the signed rooted tree $\cT$ which classified the singularity of $L$ at $f(x)$. The first possibility is that $f(\Op(x)$ is the conormal of a cooriented hyperplane together with the half-plane lying in the direction of the coorientation. The second possibility is to take the other half-plane. By inspection of the local model, the former is graphical with respect to $(\eta^*,\eta)$ and the latter is not. On the other hand, the former gives an oriented immersion while the latter does not. 
The condition on $M_k$, $k >1$ can be verified inductively using the above argument, or more directly one can inspect the explicit local model, in which the polarization $\eta$ can also be taken to be canonical as proved in \cite{AGEN20a}.
\end{proof}

 \begin{proposition}\label{lm:imm-to-pos-arb}  Let $L\subset W$ be a $n$-dimensional positive arboreal Lagrangian and $\eta$  a positive distribution for $L$. Then for any oriented immersion $f:N\to L$ of an $n$-dimensional manifold $N$ the piecewise smooth  immersion $N\mathop{\to}L\hookrightarrow W$ is genuinely transverse to $\eta$.
\end{proposition}

\begin{proof}
Near each point $x \in N$ the germ of $f(L)$ is graphical with respect to $(\eta^*,\eta)$ so near $x$ we can generate $f(L)$ by a piecewise smooth function on $\eta^*$. Moreover, the generating function can be smoothed and the result is still transverse to $\eta$ since it remains graphical with respect to $(\eta^*,\eta)$. To achieve a global smoothing of $L$ one can patch up the local smoothings with a partition of unity.
\end{proof}

         \begin{figure}[h]
\includegraphics[scale=0.5]{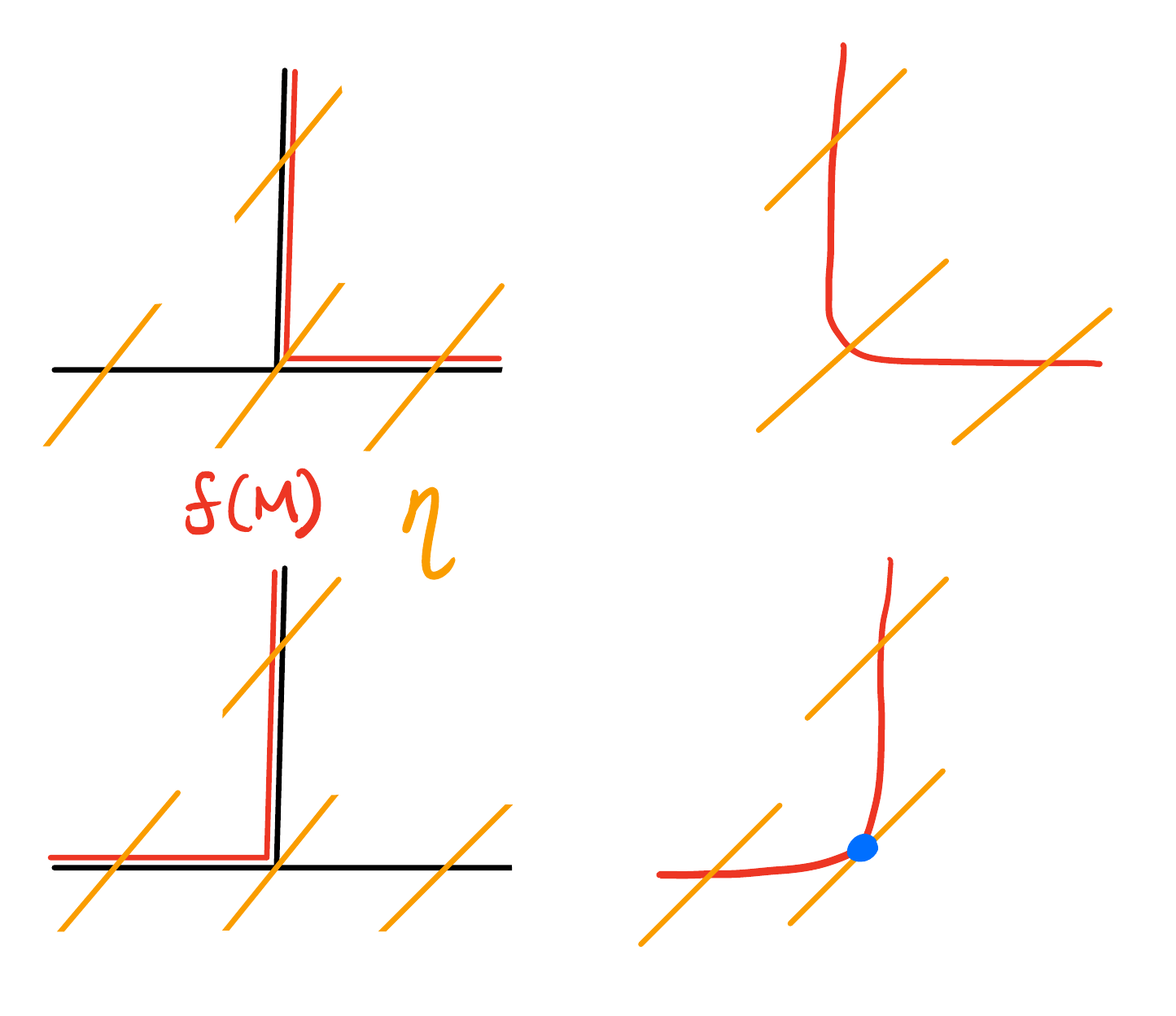}
\caption{On the left: the upper immersion is oriented, while the lower immersion is not. On the right: the upper smoothed Lagrangian is transverse to $\eta$, while the lower one is not.}
\label{fig:genuinetransversality}
\end{figure}

\subsubsection{Capping arboreal Lagrangians}
 
Let $L$ be an arboreal space     and $\Lambda$ a   smooth  component of its boundary $\p L$.  We call the component  $\Lambda$ {\em bounding} if there exists
a manifold $M$ with $\p M=\Lambda$  and an oriented immersion $M\to L$ extending the inclusion $\Lambda\hookrightarrow L$.

\begin{lemma}\label{lm:bound-bound}
Let $\Lambda\subset \sT^*M\setminus M$ be a positive arboreal Legendrian which is positive with respect to the vertical distribution $\nu$.  Suppose that the boundary of $\Lambda$ is    smooth and  bounding.
Then there exists a closed  arboreal Legendrian $\wh\Lambda\subset \sT^*M\setminus M$ which is transverse to $\nu$ and satisfies the following properties, where  we denote by $\wh L \subset \sT^*M$ the arboreal Lagrangian formed by the union of the Liouville cone  of $\wh \Lambda$ and the $0$-section.
\begin{itemize}
\item[(i)] $\Lambda':=\wh\Lambda\setminus \Lambda$ is smooth; 
\item[(ii)] $\wh \Lambda$ bounds in $\wh L$ an immersed submanifold which is genuinely transverse to any Lagrangian distribution   $\eta$    in $\sT^*M$ such that $TM,\nu,\eta$ are $\pprec$-ordered. 
\end{itemize}
\end{lemma}
\begin{proof}
Let $f:N\to\Lambda$ be an immersion bounding $\p\Lambda$ in $\Lambda$. Denote by $\pi:\sT^*M\to M$ the front projection. Then the image $\Sigma:=\pi\circ f(N)\subset M$ is a $C^1$-smooth immersed  cooriented hypersurface. Hence it has a $C^\infty$-smooth, fixed on $\p\Sigma$   push-off $  \Sigma'       $ in the direction of the co-orientation, which together with $\Sigma$ bounds an immersed domain $\Omega\subset M$ with a $C^1$-smooth boundary $\p\Omega=\Sigma\cup \Sigma'$.  Let $\Lambda'$ be the  conormal lift   of $  \Sigma'$. Then $\wh \Lambda:=\Lambda\cup   \Lambda'$ is the required Legendrian, where property (ii) follows from Proposition \ref{lm:imm-to-pos-arb}, 
\end{proof}

          \begin{figure}[h]
\includegraphics[scale=0.5]{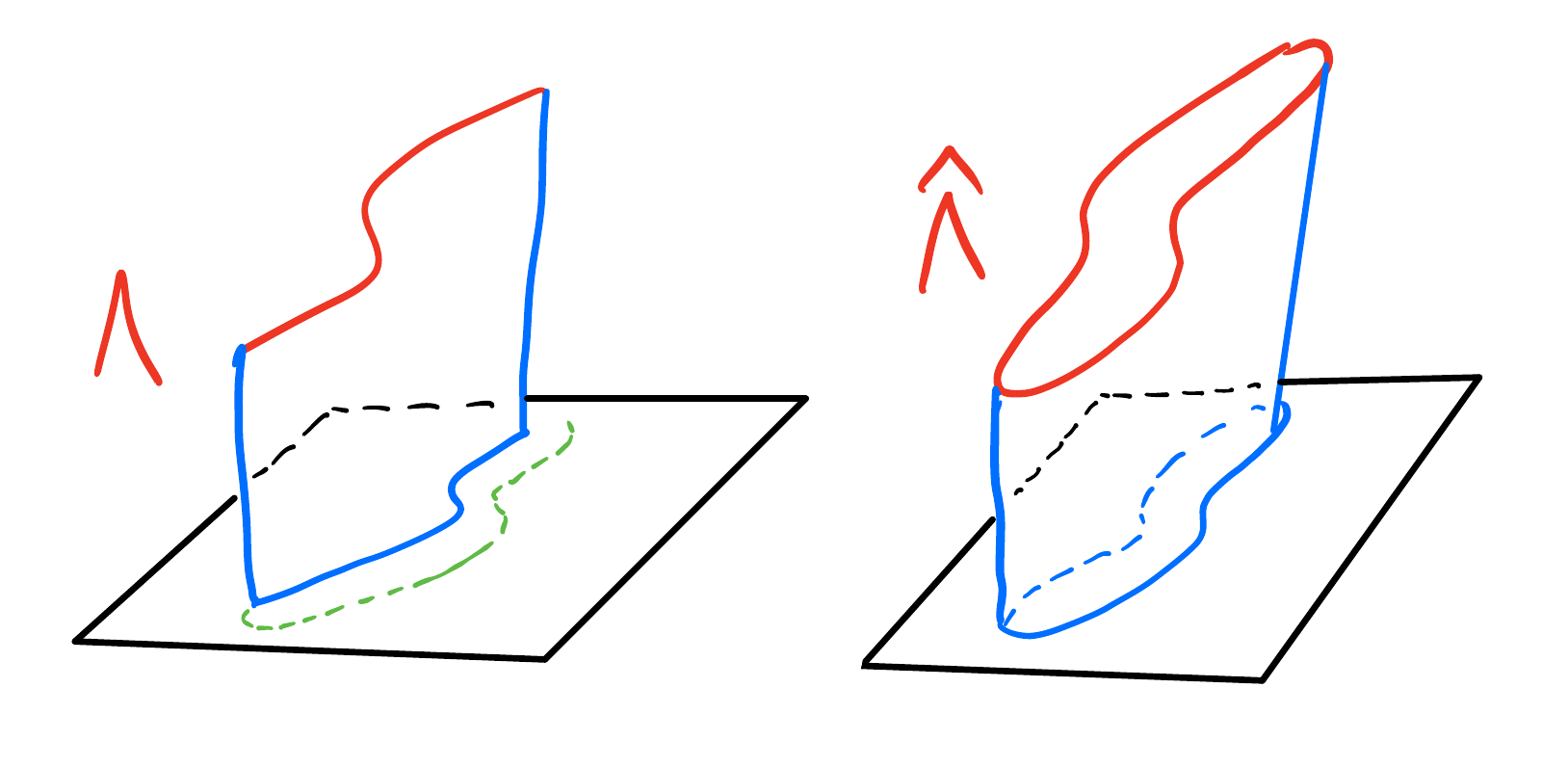}
\caption{Closing up a positive Legendrian whose boundary is smooth and bounding.}
\label{fig:closingup}
\end{figure}

\subsection{Cones over arboreals}\label{sec:cones-arb}

We now discuss the key lemma which will be used for the inductive arborealization of ridgy Lagrangians.

\subsubsection{Arborealization of radial cones on arboreals}

\begin{definition} An arboreal Lagrangian in a Wc-manifold $W$ is called {\it asymptotically conical} if it coincides  outside of a neighborhood of $\Skel(W)$ with  a  Liouville cone over a Legendrian. \end{definition}

Note that the boundary of an asymptotically conical  arboreal Lagrangian $L$ is also asymptotically conical. We call a smooth component $C\subset\p L$ {\it bounding} if there exists  a possibly non-compact manifold $N$  with $\p N=C$ and an  asymptotically conical immersion $f:N\to L$ bounding $C$. 

\begin{lemma}\label{lm:arb-vertex}
 Let $\Lambda$ be a positive arboreal Legendrian in $S^{2n-1}$ endowed with the standard contact structure.   Let $L\subset\R^{2n}$ be the Lagrangian cone over $\Lambda$ centered at $0$, i.e. with respect to the radial Liouville structure on the unit ball $B \subset \R^{2n}$. Let $(\tau,\nu)$ be a polarization of $\R^{2n}$ such that every Lagrangian plane $T$ tangent to $L$ at $0$ satisfies the condition $T^{\tau,\nu}>0$.
 Let $L_0=\tau(0)$ be the Lagrangian  plane through the origin, and $\Lambda_0\subset S^{2n-1}$ its Legendrian  link.
 Suppose that  each   boundary component  of $\Lambda$ is bounding.
 Then
 there exist:
 \begin{itemize}
 \item[(i)] a closed  arboreal Legendrian $\wh\Lambda\supset \Lambda$   such that  $\Lambda':=\wh\Lambda\setminus\Lambda$ is smooth and such that  every Lagrangian plane $T$ tangent to the Liouville cone  $L'$  over $\Lambda'$ satisfies the condition $T^{\tau,\nu}>0$;
 \item[(ii)] an asymptotically conical arboreal Lagrangian $\wh L$ in $\R^{2n}$ endowed with the standard Liouville structure such that:
 \begin{itemize}
 \item  $\wh L$ coincides at infinity with the Lagrangian cone  over  $\Lambda\cup \Lambda_0$;
 \item $\wh L\supset\wh \Lambda$ and $\wh L \supset L_0$;
 \item $\wh L$ has smooth boundary and every boundary component is bounding;
 \item  $\wh L $ is transverse  to any Lagrangian distribution   $\eta\in C(\nu,\tau)$, i.e negative with respect to $(\tau, \nu)$.
 \end{itemize}
 \item[(iii)] a compactly supported homotopy $\lambda_t=\lambda_0 + d f_t$ of $\lambda_0=pdq$ such that the Wc-manifold obtained from $(\R^{2n}, \lambda_1)$ by converting the ribbons of the Legendrians $\Lambda \cup \Lambda_0$ into nuclei of boundary faces has the structure of a positive cotangent building with skeleton $\wh L$.
 \end{itemize}
 \end{lemma}
 
          \begin{figure}[h]
\includegraphics[scale=0.5]{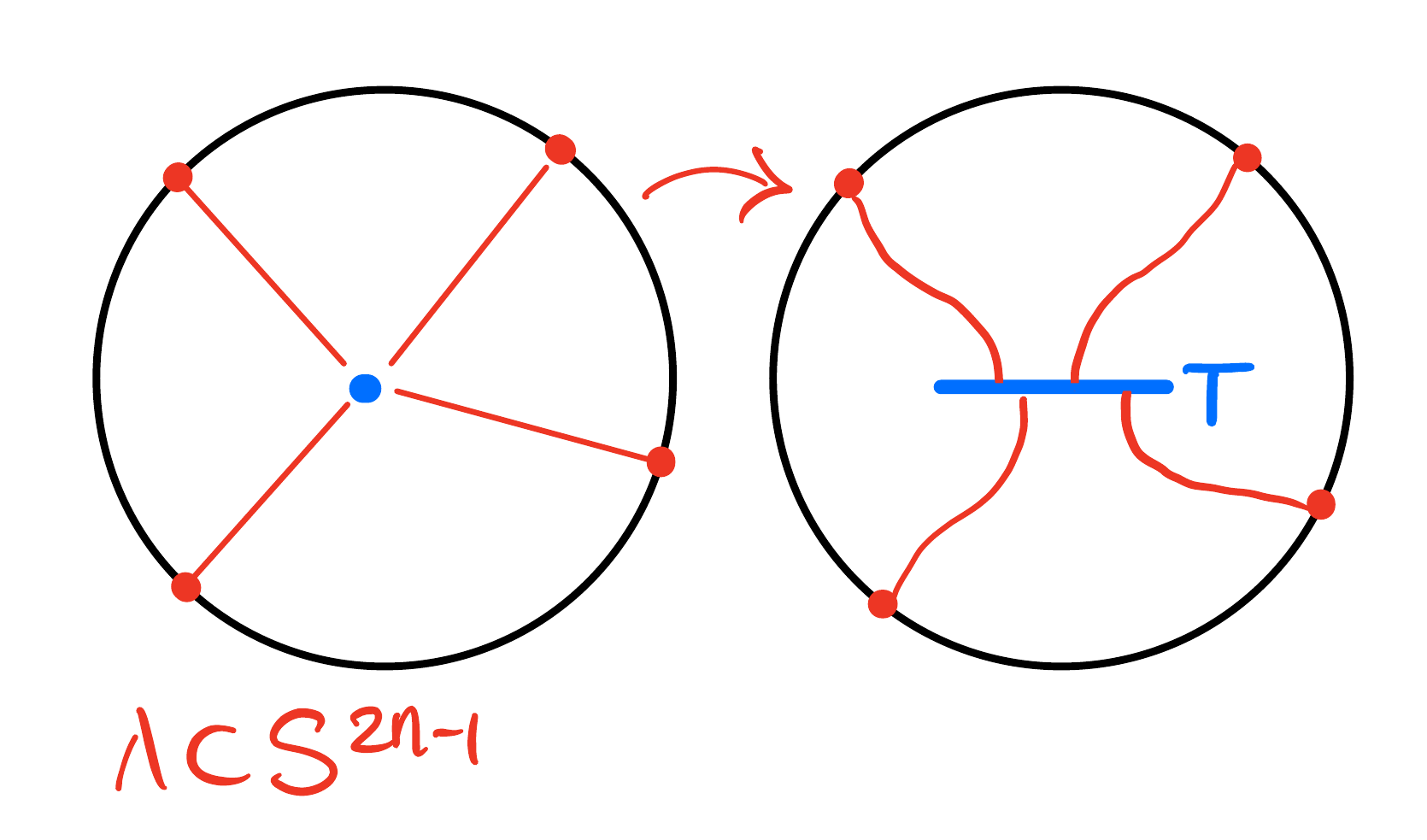}
\caption{The goal of Lemma \ref{lm:arb-vertex} when $n=1$.}
\label{fig:arborealization}
\end{figure}

         \begin{figure}[h]
\includegraphics[scale=0.5]{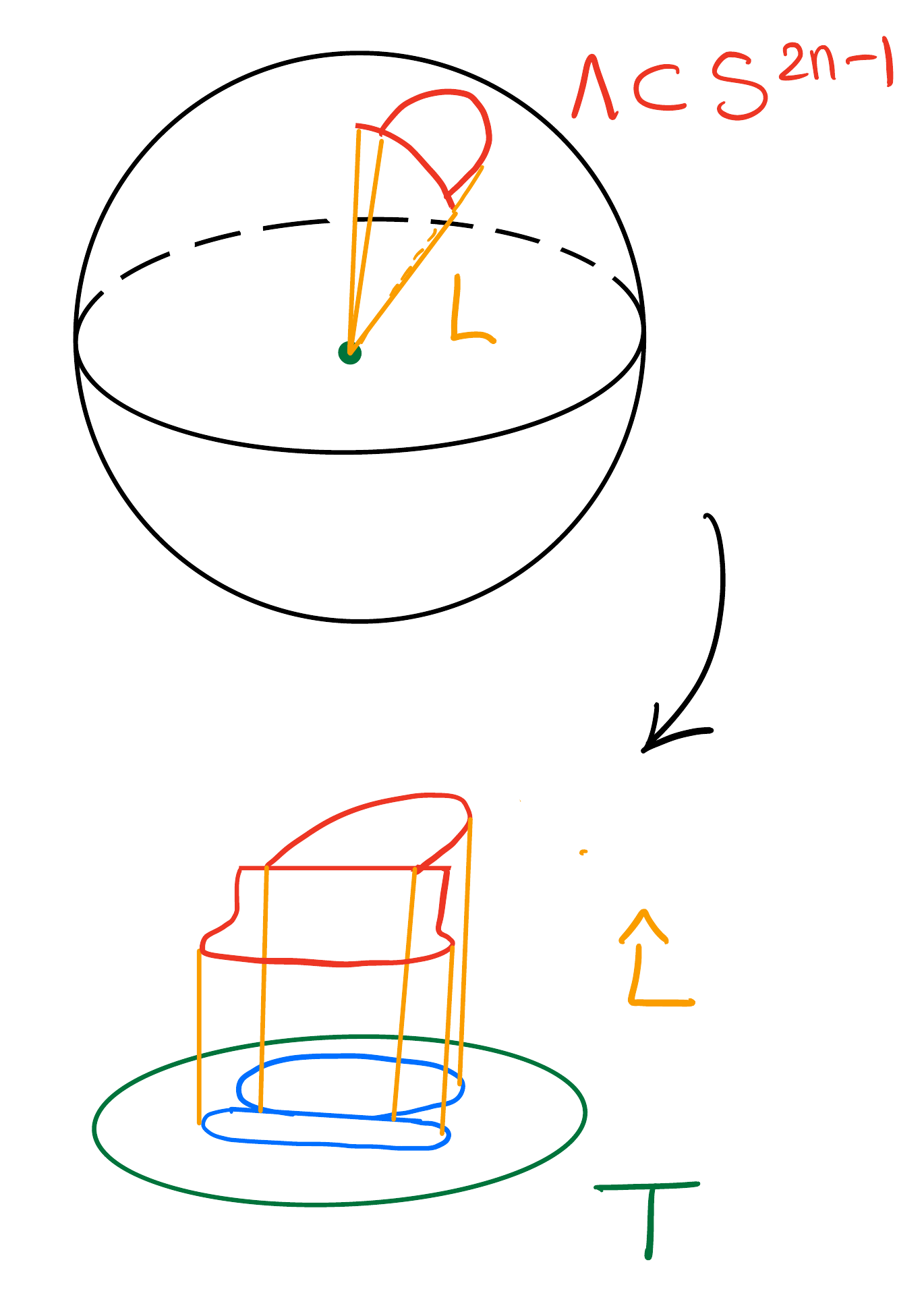}
\caption{The goal of Lemma \ref{lm:arb-vertex} when $n=2$.}
\label{fig:Bounding}
\end{figure}
  
 \begin{proof}
 (i) Let $p,q$ be canonical coordinates for the polarization $(\tau,\nu)$
 Note that the function $H=pq$ restricts to   $L\setminus 0$ as a positive  function which satisfies the condition $H|_L\geq c|p|^2$ for some $c>0$.  Take the function $G=\max(H,\frac{1}{2}c|p|^2)$. Then $\{H=\eps\}\cap L=\{G=\eps\}\cap L$ for any $\eps>0$. The function  $G$ has  star-shaped level sets with corners away from $L$. Hence, it can be smoothed to a function with the same property keeping the restriction to $L$ equal to $pq|_L$. We abuse notation and denote de smoothed function by the same symbol $G$. Note that in particular we have $$(pdq=\frac12(pdq-qdp)+d(pq))|_{L\cap\{G=\eps\}}=0,$$ i.e.  $L\cap\{G=\eps\}$ is Legendrian in the level  set $G=\eps$ for both contact structures, defined by the Liouville forms $pdq$ and   $\frac12(pdq-qdp)$.
 Consider the Liouville  form
 $$\lambda=\frac12(pdq-qdp)+\frac{1}{2}d(\alpha(G))=\frac12(pdq-qdp)+ \frac{1}{2}\alpha'(G) d(pq),$$ where
 $\alpha:\R_+\to[0,1]$ a monotone cut-off function which is equal to the identity on $[0,\eps]$ and to $0$ on $[ \wt\eps:=\eps+\delta,\infty)$ for $\delta\ll\eps$. Then $\lambda=pdq$ in  $\{G<\eps\}$ and $\lambda=\frac12d(pdq-qdp)$ in $\{G\geq\wt\eps \}$.  Let $\wt L=\wt L_{\eps,\delta}=
 $ be the saturation        of $\Lambda_{\wt\eps}=L\cap\{G=\wt\eps\}$ by the forward and backward  Liouville trajectories of  the Liouville field of the form $\lambda$.
   
  As $\delta\to 0$, note that $\wt L_{\eps,\delta}$ converges to the union of the  forward $pdq$-cone  and the backwards $\frac12(pdq-qdp)$ cone over $\Lambda_\eps$.  Hence for small $\delta$ a plane $\wt T$  tangent to $\wt L$ near $\wt\Lambda_\eps$ is close to a plane spanned by a  tangent plane $\eta$ to $\wt \Lambda_\eps$ and  a convex linear combination
  $(1-t)R+tV$, $0\leq t\leq 1$ of the contracting vector fields  $R=-\frac12(p\frac{\p}{\p p}+q\frac{\p}{\p q})$ and $V=-p\frac{\p}{\p p}$.  When $t=0$ the plane $\wt T$ coincides with a plane $T$ tangent to $L$, and hence corresponds to  a positive definite quadratic form $Q$. Increasing $t$ corresponds to adding to $Q$ the quadratic form $\frac{t}{1-t}\ell^2$, where the hyperplane  $\{\ell=0\}$ is the front projection of $\eta$. On the other hand, the distribution $\nu$ is given by our  assumption by a negative definite  quadratic form $P$. The intersection of $\wt T\cap \nu$ corresponds to critical points of the quadratic form $Q+  \frac{t}{1-t}\ell^2-P$ which is positive definite. This assures the transversality of $\nu$ to $\wt L$.
  
  Recall that each  boundary component $C=\p\Lambda$ is bounding. Let $f:N\to \Lambda$, where $\p N=C$, be the bounding immersion. Let $ Z^t$ be the Liouville flow of the Liouville vector field dual to $\lambda$.
  Consider the conical immersion  $F:N\times\R\to \R^{2n}$   given by
  $$F(x,t)= Z^t(x),\; x\in N, t\in\R.$$ Then the limit $$F_\infty(\cdot)=\lim\limits_{t\to-\infty}F(\cdot,t)$$ defines a cooriented $C^1$-smooth immersion $F_\infty: N\to L_0$. Let
  $F_\infty':N\to L_0$ be a $C^\infty$-smooth normal push-off of $N$ in the direction of its co-orientation. By modifying $F_\infty'$ near $\p N=C$ we can arrange that  the immersion $F_\infty'$ together with $F'$ define a $C^1$-smooth immersion $\wh F:=F_\infty\cup F_\infty':\wh N:=N\mathop{\cup}_C N\to L_0$ of the closed manifold $\wh N$ obtained by gluing two copies of  $N$  along the boundary $C$. The manifold $\wh N$ bounds a manifold $M$, namely the product $N\times[0,1]$ with a smoothed boundary. Note that we can arrange that the immersion $\wh F$ extends to $M$. There exists a Legendrian embedding $\phi:N\to S^{2n-1}$ such that $\phi(\p N)=C$ and such that
  $$\lim\limits_{ \t \to \infty} Z^{-t}\circ\phi = 
   F_\infty'.$$ Set $\breve N:=\phi(N)$ and denote by  $\breve L$ the backward $\lambda$-Liouville cone of $\breve N$.
  Then $\wh L:= \wt L\cup\breve L\cup L_0$ is the required  arboreal Lagrangian with boundary.  Its boundary (which can be smoothed) is the union of $\Lambda_0$ and  the forward Liouville cone of $C$ glued along $C$ to $\breve N$.  

The required Liouville form $\lambda_1$ can then be constructed as follows: starting with $\lambda$, convert the ribbon of $\wh \Lambda$ to a boundary nucleus and then back to a Wc-hypersurface. This has the effect of adding the backwards Liouville cone of $\wh \Lambda$ to the skeleton. Then, up to smoothing, further conversion of the ribbons of $\Lambda \cup \Lambda_0$ to boundary nucleus will produce a Wc-structure with skeleton equal to $\wh L$. The positive cotangent building structure is obtained by stabilizing the positive cotangent building structure for the ribbon of $\Lambda$ and adding the block $\sT^*\Lambda_0$. \end{proof}

\subsubsection{Parametric version}
Lemma \ref{lm:arb-vertex} also holds in a parametric form. Namely:

\begin{lemma}\label{lm:arb-vertex-param}
 Let $\Lambda^z\subset S^{2n-1}$ be a a family of positive arboreal Legendrians in $S^{2n-1}$ endowed with the standard contact structure, parametrized by $z \in Z$ for $Z$ a compact manifold. Let $L^z\subset\R^{2n}$ be the Lagrangian cone over $\Lambda^z$ centered at $0$, i.e. with respect to the radial Liouville structure on the unit ball $B \subset \R^{2n}$. Let $(\tau_z,\nu_z)$ be a family of polarizations of $\R^{2n}$ such that every Lagrangian plane $T_z$ tangent to $L^z$ at $0$ satisfies the condition $T_z^{\tau_z,\nu_z}>0$.
 Let $L^z_0=\tau_z(0)$ be the Lagrangian  plane through the origin, and $\Lambda^z_0\subset S^{2n-1}$ its Legendrian  link.
 Suppose that  each   boundary component  of $\Lambda^z$ is bounding and moreover that the bounding manifolds $N^z$ form a fibre bundle over $Z$.
 Then
 there exist:
 \begin{itemize}
 \item[(i)] a family of closed  arboreal Legendrians $\wh\Lambda^z \supset \Lambda^z$   such that  ${ \Lambda'}^z= \wh\Lambda^z \setminus\Lambda^z $ is smooth and such that  every Lagrangian plane $T_z$ tangent to the Liouville cone  ${L'}^z$  over ${\Lambda'}^z$ satisfies the condition $T_z^{\tau_z,\nu_z}>0$;
 \item[(ii)] a family of asymptotically conical arboreal Lagrangians $\wh L^z$ in $\R^{2n}$ endowed with the standard Liouville structure such that:
 \begin{itemize}
 \item  $\wh L^z$ coincides at infinity with the Lagrangian cone  over  $\Lambda^z\cup \Lambda^z_0$;
 \item $\wh L^z\supset\wh \Lambda^z$ and $\wh L^z \supset L^z_0$;
 \item $\wh L^z$ has smooth boundary and every boundary component is bounding, moreover with bounding manifolds forming a fibre bundle over $Z$;
 \item  $\wh L^z $ is transverse  to any family of Lagrangian distributions   $\eta_z\in C(\nu_z,\tau_z)$, i.e negative with respect to $(\tau_z, \nu_z)$.
 \end{itemize}
 \item[(iii)] a family of compactly supported homotopies $\lambda^z_t=\lambda_0 + d f^z_t$ of $\lambda_0=pdq$ such that the Wc-manifold obtained from $(\R^{2n}, \lambda^z_1)$ by converting the ribbons of the Legendrians $\Lambda^z \cup \Lambda^z_0$ into nuclei of boundary faces is a positive cotangent building with skeleton $\wh L^z$.
 \end{itemize}
 \end{lemma}
 
 \begin{proof} The proof proceeds exactly as in the non-parametric case by adding a parameter everywhere. \end{proof}
 
  \subsection{From ridgy to arboreal}\label{sec:arb-rdges}
   
  
  \subsubsection{Recap on reduced positivity}
 We recall the Definition \ref{def: red pos lag dist} of reduced positivity. Let $L\subset X$ be a  positive arboreal Lagrangian and $Z$ a non-zero Liouville vector field tangent to  $L$. A Lagrangian distribution $\mu$ along $L$ is called reduced positive with respect to $(L,Z)$ if for 
  any singular point   $a\in L$  the following  condition is satisfied.
  Let $T$ be the tangent space to the root Lagrangian at $a$, and $T'$ a tangent plane to any other smooth piece adjacent to $a$. Then $[T']^{\zeta}\in C([T]^{\zeta},[\eta(a)]^{\zeta}), $ where $\zeta=\Span(Z)^{\perp_\om}.$

\subsubsection{The distribution $\nu_\infty$}\label{sec:nu-infty-2}
  We consider in this section the following setup.
  Let $M$ be a bc-manifold and $X=\sT^*M$. Suppose the boundary faces of $M$ are ordered: $\p_1M=P_1\cup\dots\cup P_k$. Let $u_j$ be defining coordinates near faces $P_j$.
  Suppose we are given the following objects for each $j=1,\dots, k$:
  \begin{itemize}
  \item[-] a neighborhood  $V_j\supset P_j$;
   \item[-] the canonical Liouville  fields $Z_j$  on $\sT^*P_j$ and the Liouville field   $\wh Z_j:=u_j\frac{\p}{\p u_j} + Z_j$,    on $V_j$.
\item[-] a  Lagrangian distribution  $\nu_j$ on $ V_j$,    which is tangent to $\wh Z_j$ and  invariant with  respect to its negative flow.

\end{itemize}
Suppose, in addition, we are given
\begin{itemize}
\item[-] a ridgy Lagrangian $L\subset \sT^*M$    which on $V_j\cap L$ is tangent to $\wh Z_j$;
\item[-]  a Lagrangian distribution $\nu_{-1}$   on $\sT^*M$ which is transverse  to $L$ and all  distributions $\nu_j$. 
\end{itemize}
Suppose that the following conditions are satisfied:
  \begin{itemize}
  \item[-] for any multi-index $I=\{i_1<\dots< i_\ell\}$ we have
   $$ [\nu_{i_\ell}(a)]^{I'}\pprec \dots\pprec  [\nu_{i_1}(a)]^{I'}\pprec [\nu_{-1}(a)]^{I'}\;\; \hbox{on}\;\;
   \bigcap\limits_{1}^\ell V_{i_j}\setminus \bigcup\limits_{j\notin I}V_j;$$ where $I'=\{i_1<\dots<i_{\ell-1}\}$.
  \item[-]  
 for any point $a\in L\cap V_j$ and any tangent plane $T$ to $L$ at $a$ we have
 $$[T]^j\pprec [\nu_j]^j\pprec [\nu_{-1}]^j.$$\end{itemize}
 
           \begin{figure}[h]
\includegraphics[scale=0.7]{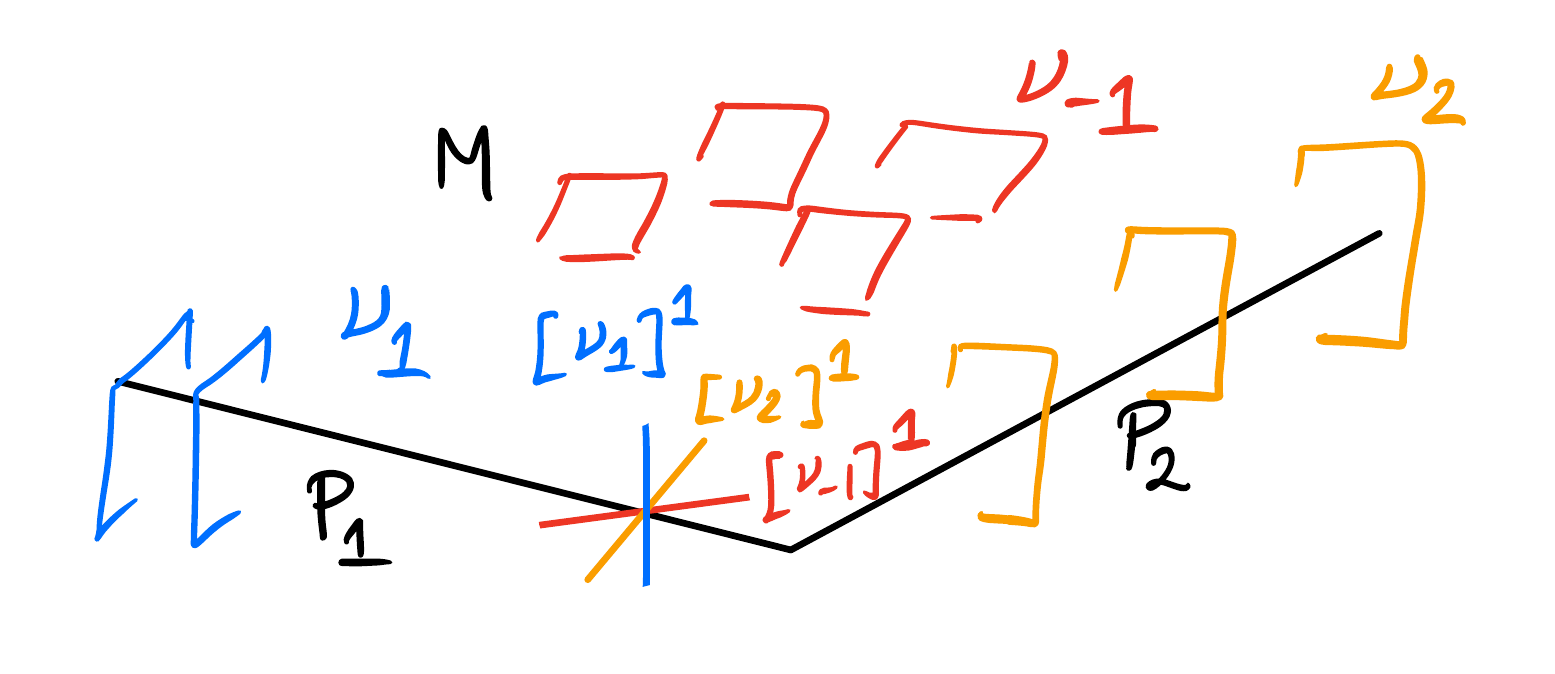}
\caption{The condition on $[\nu_{-1}]^1$.}
\label{fig:hypoth}
\end{figure}

         \begin{figure}[h]
\includegraphics[scale=0.7]{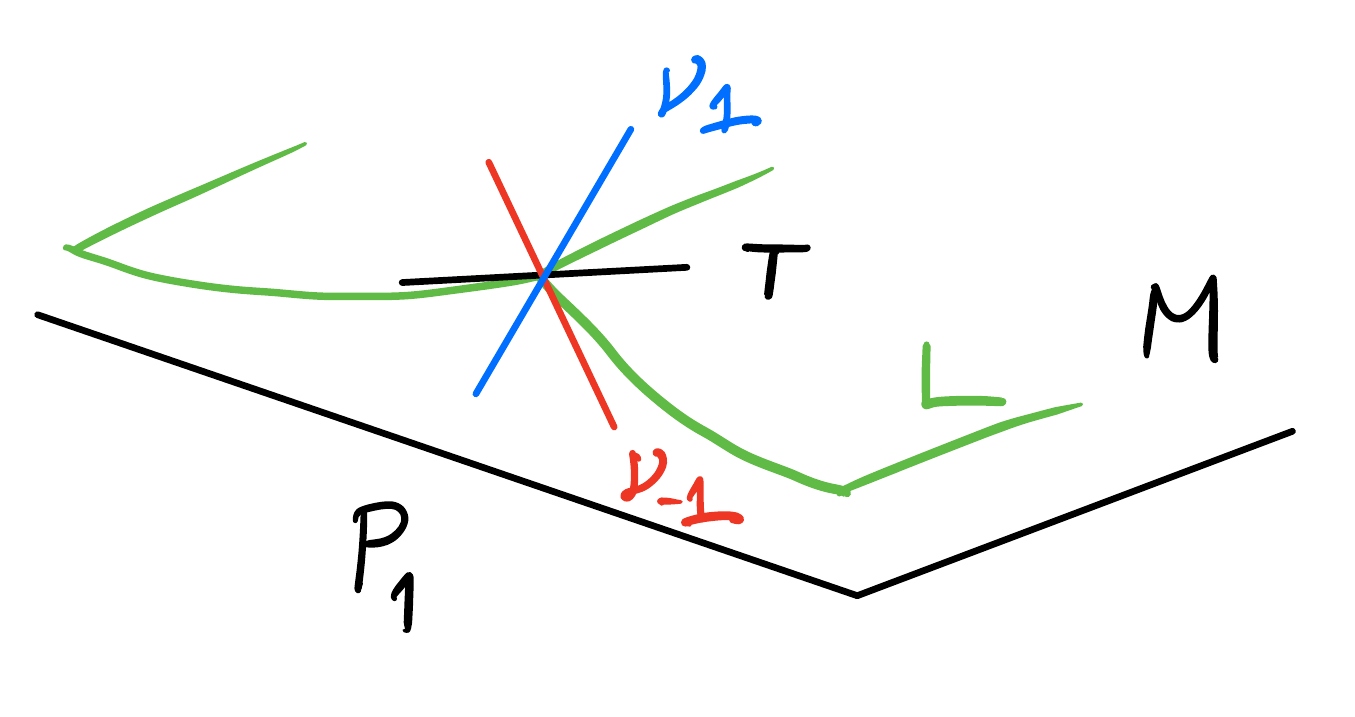}
\caption{The condition on a Lagrangian plane $T$ tangent to $L$.}
\label{fig:hypoth2}
\end{figure}
 
\begin{lemma}\label{lm:nu-infty-2} 
There exist neighborhoods $V_j\Supset V_j'\supset P_j $, $j=1,\dots, k$, 
and a distribution $\nu_\infty$ on $\sT^*M$ which satisfies the following conditions:
\begin{enumerate}
\item $\nu_\infty=\nu_m$ on $V'_{m}\setminus   \bigcup\limits_{j>m} V_j$, $m=1,\dots, k;$
\item for any  $1\leq i<j\leq k$ and any point $a\in V'_{i}\cap (V_{j}\setminus V'_j) $ 
 we have $$[\nu_\infty(a)]^I\in C([\nu_i(a)]^I,[\nu_j(a)]^I),$$ where $I$ is the multi-index $\{i_0=i<i_1<\dots<i_\ell\}, i_\ell<j$, such that  $$a\in \bigcap \limits_1^\ell V_{i_s}\setminus \bigcup \limits_{m\notin I, m<j}V_m .$$
  \item  $\nu_\infty$ is transverse to $\nu_{-1}$ and for any point $a\in L\setminus\bigcup\limits_{1}^k V_j$ and any tangent plane $T$ to $L$ at $a$ we have $T\pprec\nu_\infty(a) \pprec\nu_{-1}(a).$
\end{enumerate}
\end{lemma}

         \begin{figure}[h]
\includegraphics[scale=0.6]{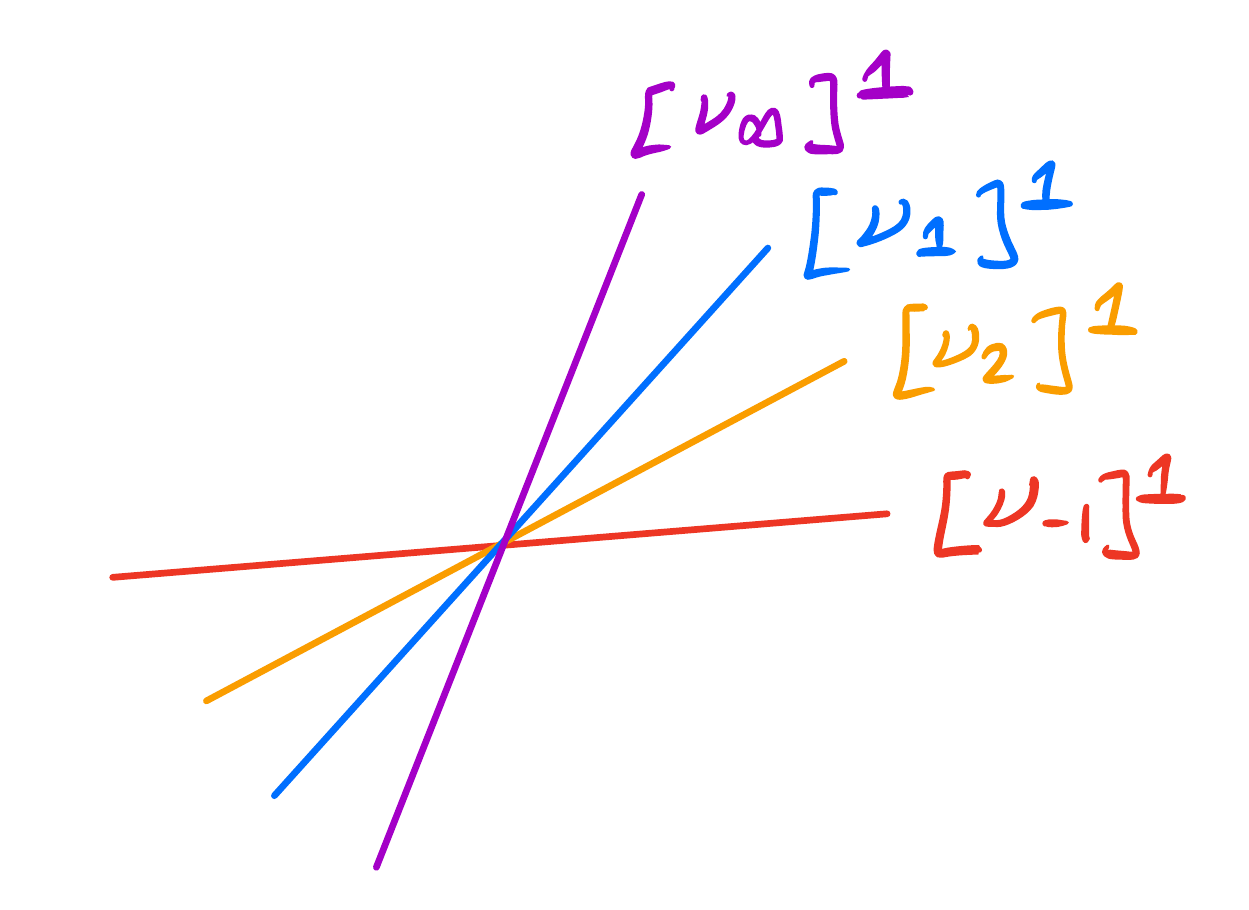}
\caption{The conclusion on $[\nu_\infty]^1$.}
\label{fig:conc1}
\end{figure}

         \begin{figure}[h]
\includegraphics[scale=0.6]{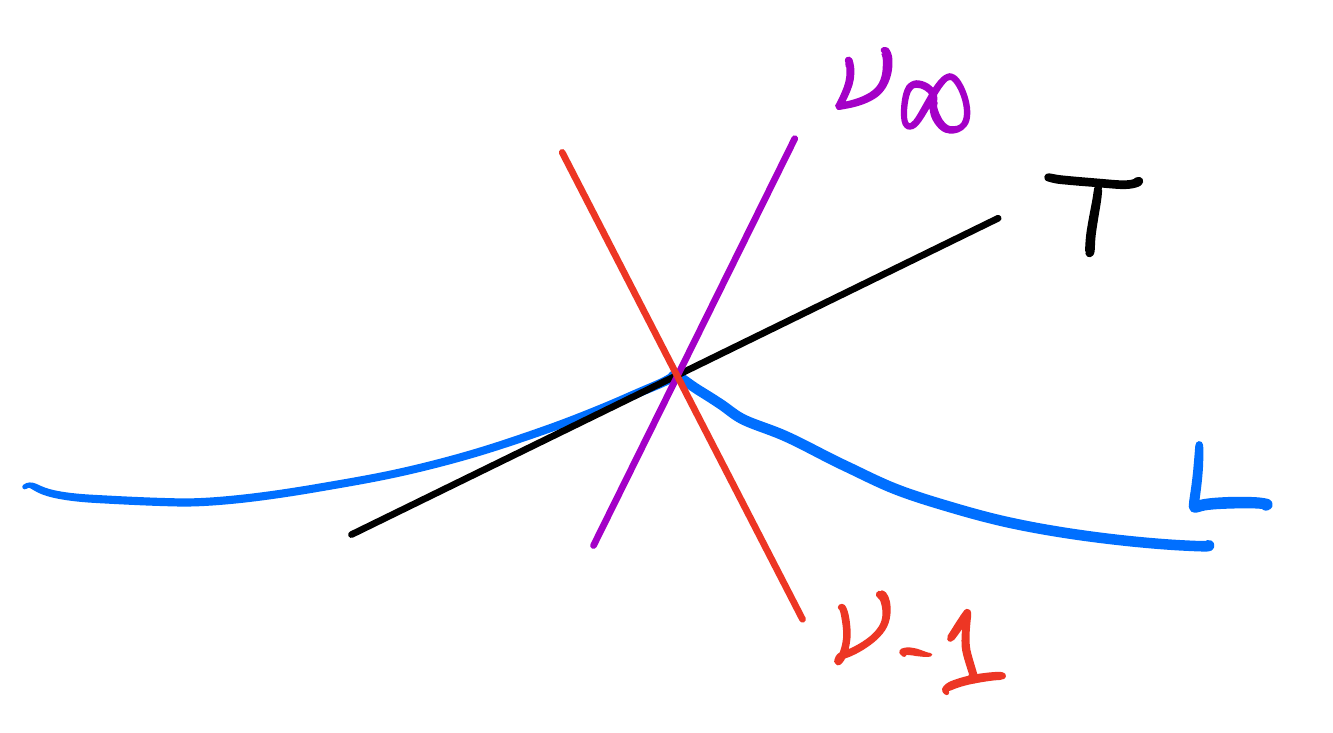}
\caption{The conclusion on a Lagrangian plane $T$ tangent to $L$.}
\label{fig:conc2}
\end{figure}

\begin{proof}
Define the required distribution as equal to $\nu_m$ on $V'_{m}\setminus\bigcup_{j>m} V_j$, $m=1,\dots, k,$ and any  distribution  on 
$\sT^*M\setminus\bigcup_1^k V_j$ which is transverse to $\nu_{-1}$ and satisfies condition (iii) along $L$.
Next, we  successively extend $\nu_\infty$ to $V_k$:   first to $(V_k\setminus V_k')
\cap V'_{k-1}$  to satisfy the condition $[\nu_\infty]^{k-1}\in C([\nu_{k}]^{k-1},[\nu_{k-1}]^{k-1})$, then continuing the process by extending  to $V_k\setminus (V_k'\cup V_{k-1}'))\cap V_{k-2}'$
 to satisfy the condition  $[\nu_\infty]^{k-1}\in C([\nu_{k}]^{k-2,k-1},[\nu_{k-1}]^{k-2,k-1})$, etc. Next, we similarly  successively extend $\nu_\infty$ to $V_{k-1}\setminus (V_{k-1}'\cap V_k)$, $V_{k-2}\setminus(V_{k-2}'\cup V_{k-1}\cup V_k),\dots$, $V_1\setminus (V_1'\cup V_2\cup\dots\cup V_k)$.
   Each extension  is possible because it amounts to  an extension of a  section of a fibration with open convex (and therefore contractible) fiber.
\end{proof}
\subsubsection{Arborealization of ridges}
Recall that for an isotropic submanifold $C$ of a symplectic manifold, its {\em symplectic normal bundle},   is defined
as $(TC)^{\perp_\om}/TC$. Denote by $\pi_C=[\cdot]^C$ the reduction map. We consider a ridgy Lagrangian $L \subset \sT^*M$ with the notation as in Section \ref{sec:nu-infty-2} above. Consider the stratification $L=L_0\cup L_1\cup\dots\cup L_n$, where $L_j$ is the $j$-dimensional locus of order $(n-j)$ ridges.
 Thus the top-dimensional stratum $L_n$ is the smooth locus of $L$.
 
For each component $C_j$ of $L_j$, we denote  by  $N(C_j)$ the  $2j$-dimensional symplectic normal bundle to $\oC_j$. A Lagrangian tangent plane $T$ to $L$ at a point of $\oC_j$ reduces to a Lagrangian plane $ [T]^{C_j}:=\pi_{C_j}(T)\subset N(C_j)$. If moreover $a\in C_{j-1}\subset \oC_j$  for $C_{j-1}$ a component of $L_{j-1}$ and $T$ a Lagrangian plane at $a$, then the plane $[T]^{C_{j-1}}$ further reduces to $[T]^{C_{j-1}}$. 

Note that if $L$ is tangent to $\wh Z_j$ on $V_j$, $j=1,\dots, k$, then so are all $L_i, i=1,\dots, k$. In particular, if $i<\ell$ we have $L_i\cap V_{i_1}\cap \dots\cap V_{i_\ell}=\varnothing.$

Let $U_n \to \cdots \to U_0$ be the canonical Wc-structure associated to a ridgy Lagrangian $L$.

   \begin{definition}
  An {\it arborealization} of the ridgy Lagrangian $L$ is the structure of a cotangent building $B_m \to \dots \to B_0$ on a  neighborhood of $L$ whose skeleton is arboreal and whose underlying Weinstein manifold is homotopic to that of $U_n \to \cdots \to U_0$.
  \end{definition}

We say that the arborealization is positive when the resulting arboreal Lagrangian is positive, which implies $B_m \to \dots \to B_0$ is a positive complex.   

In the next proposition we continue using the notation and assumptions introduced in~Section \ref{sec:nu-infty-2}.

   \begin{prop}\label{prop:slanted-ridgy-arb2} The ridgy Lagrangian  $L$ can be arborealized to a positive       arboreal Lagrangian $\wh L$    that  is  positive with respect to  $\nu_{-1}$, and moreover, over $V_j'$ is invariant with respect to the negative flow of $\wh Z_j$ and reduced positive with respect to $\nu_j$,
   for $j=1,\dots, k$.
  
   \end{prop}
\begin{proof}
Any Lagrangian distribution can be viewed as a field of quadratic forms with respect to the polarization $(\nu_\infty,\nu_{-1})$. Fix a field of   positive definite quadratic forms $Q$. Denote by $\eta_t$ the Lagrangian distribution generated by $-tQ$, and by $\zeta_t$ the distribution generated by $\frac1tQ$. Choose $\eps>0$ so small that the following conditions are satisfied:
\begin{itemize}
\item[-]  $\nu_j\pprec\zeta_\eps\pprec\nu_{-1}$ on $V_j$, $j=1,\dots, k$;
\item[-]   for any point $a\in L\cap V_j $ and any tangent plane $T$ to $L$ at $a$   we have
$$[T|^j\pprec[\eta_\eps]^j\pprec[\nu_{\infty}]^j.$$
  \end{itemize}
 
For each component $C_j$ of $L_j$
 define   Lagrangian distributions in the bundle  $N(C_j)$ over $\oC_j$,  $j=0,\dots, n.$
 \begin{align*}
&\ul s_j:=[\eta_{\frac{\eps}{j+1}}|_{C_j}]^{C_j},\\
& \ol s_j=[ \zeta_{\frac{\eps}{n-j+1}}|_{C_j}]^{C_j}, 
  \end{align*}  
  Note that if $C_{j-1}\subset \oC_j$ then 
  $$[\ul s_{j-1}]^{C_j}\pprec \ul s_{j}|_{C_{j-1}}  \pprec  \ol s_{j}|_{C_{j-1}} \pprec[\ol s_{j-1}]^{C_j}.$$
 We will refer to the system of distributions $\ul s_j,\ol s_j$ as a {\em frame of arborealization} for $L$. 
          \begin{figure}[h]
\includegraphics[scale=0.6]{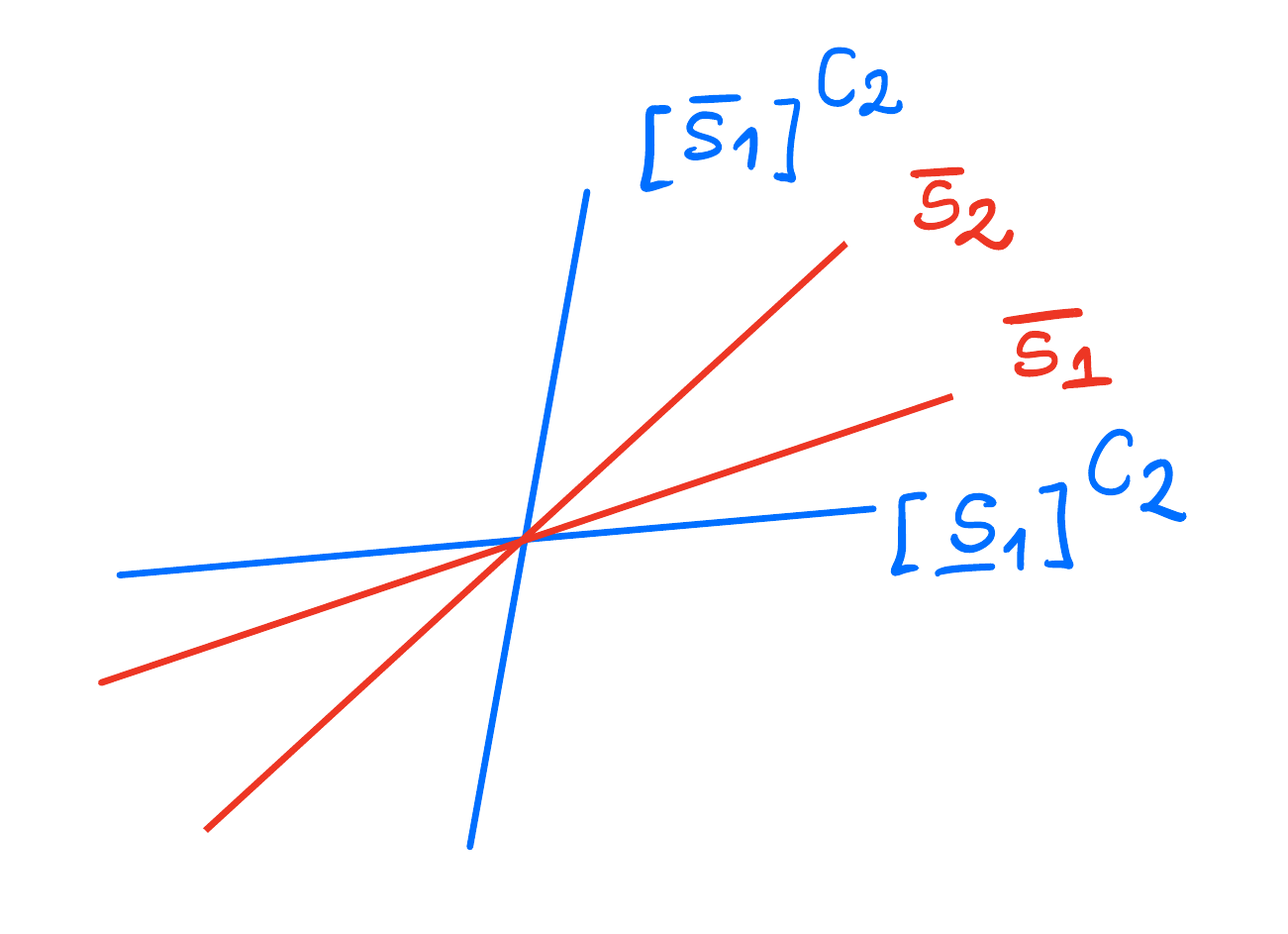}
\caption{A frame of arborealization is nested as illustrated in the figure.}
\label{fig:order}
\end{figure}
  
Consider the Wc-building structure $U_n\to\dots \to U_0$ on a neighborhood of $L$ where  the germ $\sU_j$ is decomposed as $\bigcup_Q \sT^*{Q^{\triangle,j-1}}\times\sR^{2(n-j)}$.  
  The Wc-manifold $W$ is  obtained by successively gluing $U_n$ to  $U_{n-1}$, the result of that to $U_{n-2}$, the result of that to $U_{n-3}$, etc. In this process we get a sequence of Wc-manifolds $W^{>j}= \bigcup_{i> j} U_i$ with skeleta $L^{>j}=L\setminus \bigcup_{i\geq j}{ U_i } $.
    We will be proving the proposition by  induction  for $W^{>n-j}$, $j=1,\dots, n+1$. 
     We will also include into the induction hypothesis the following additional property.   Recall that $W^{>j}$ is attached vertically to $U_j$ along a Wc-hypersurface  $\Sigma$ (which is the nucleus of the attaching face) whose skeleton is fibered over  $L_j^{\triangle,j-1}$ with the fiber which is the link $\p R^{n-j}$ of the ridgy singularity $R^{n-j}$. 
     We will require that the arborealization $\wh L^j$ intersects the attaching face
  in an arboreal skeleton of $\Sigma$ which is fibered  over $L_j^{\triangle,j-1}$, whose fiber  is the arborealization of the  link $\p R^{n-j}$.
  
       The base of induction $j=0$ is trivial because $L_n$ in this case is smooth.
Suppose that we already  arborealized the ridgy Lagrangian  $L^{>j}$ to $\wh L^{>j}$ and deformed correspondingly the Weinstein  structure on $W^{>j}$. i.e. made $\wh L^{>j}$ the skeleton of the new structure so that, by inductive assumption, (i) the arboreal $\wh L^{>j}$ has a smooth boundary $\p\wh L^{>j}$ which bounds an immersed $n$-dimensional manifold $A\subset\wh L^{>j}$, (ii) $\wh L^{>j}$  intersects the nucleus $\Sigma$ in an arborealization of the skeleton of $\Sigma$ and (iii) the skeleton is Legendrian for $U_j$.    We also observe that the condition that each boundary component is bounding is inherited by the skeleton of $\Sigma$. 
   Hence,  using the fiberwise  polarization $(\tau=\ol{s}_j,\nu=\ul{s}_j)$  we can  apply Lemma \ref{lm:arb-vertex-param} to arborealize the   fiberwise Lagrangian cone   over the arboreal Legendrian.
   
   Note that by our construction, the distribution $\nu_{-1}$ is positive for the constructed arboreal $\wh L$, and $\nu_j$ are reduced positive on $V_j'$.  \end{proof}
   
                \begin{figure}[h]
\includegraphics[scale=0.6]{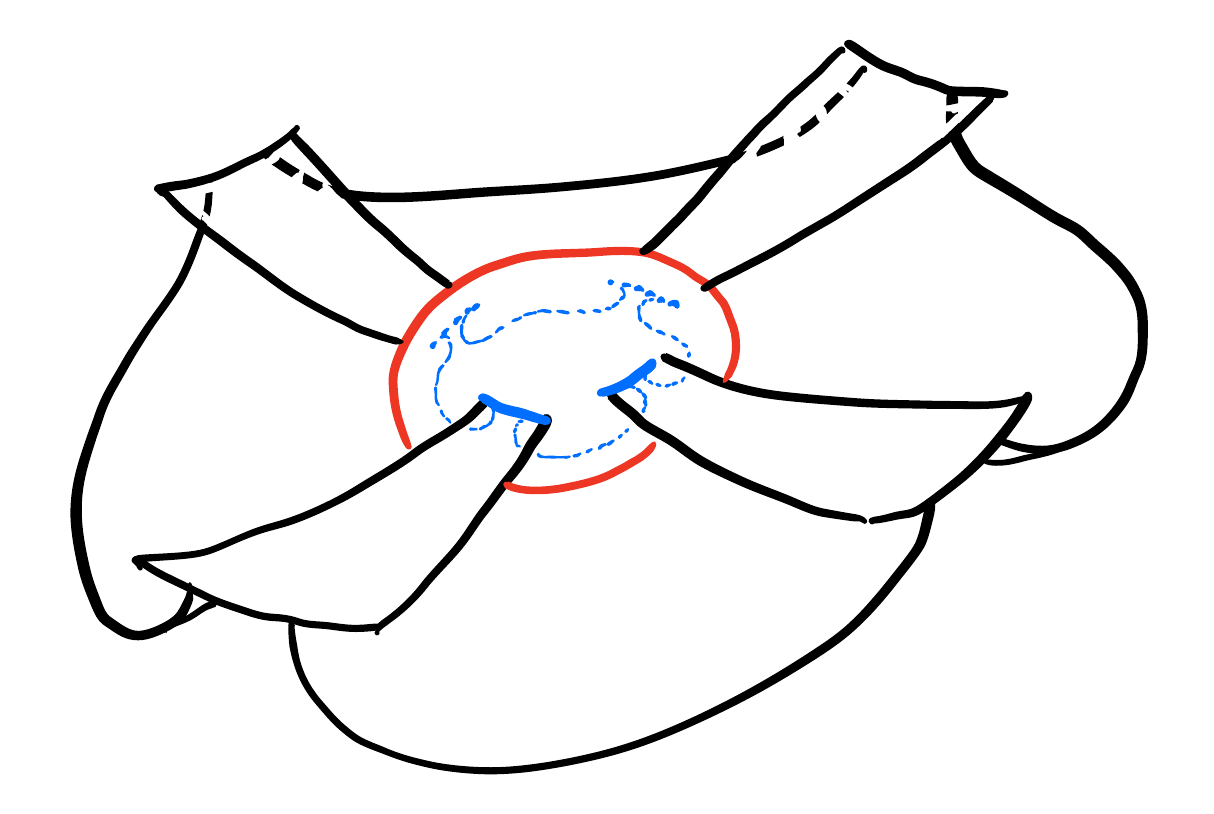}
\caption{In the case $n=2$ the inductive procedure has 2 steps. The figure illustrates the moment between the two steps: one has arborealized the order 1 ridges and it remains to arborealize a radial Liouville cone over an arboreal Legendrian with smooth boundary in a sphere around point corresponding to a order 2-ridge.}
\label{fig:inductivestep}
\end{figure}

 \subsection{Conclusion of the proof}

We are now ready to prove our main result.

\subsubsection{Proof of the main theorem}

\begin{thm}\label{thm:existence-arb-pos}
 Let $(X,\lambda)$ be a Wc-manifold and $\nu\subset TW$ a Lagrangian plane field. Then there exists a Weinstein homotopy $\lambda_t$ of $\lambda_0=\lambda$ and a Wc-hypersurface $A \subset X \setminus \Skel(X,\lambda_1)$ such that the result of converting $A$ to a boundary nucleus yields a Wc-manifold $(X, \lambda_1^A)$ which admits a structure of a positive cotangent building with a minimal distribution $\nu_{-1}$ equal to $\nu$. In particular, $\Skel(X, \lambda_1^A)$ is a positive arboreal Lagrangian (with boundary) transverse to $\nu$. Moreover, we can arrange it so that $\Skel(A,\lambda)$ is smooth, i.e. $\Skel(X, \lambda_1^A)$ has smooth boundary.
 \end{thm}
     
 \begin{proof}  We choose a presentation of $(X,\lambda)$ as a cotangent building  $(B_k\to \dots  \to B_0)$, $B_i=\sT^*M_i$, see Proposition \ref{prop:Wctocot}. Without loss of generality we may assume that the restriction of $\nu$ to $B_0$ serves as $\nu_{-1}$,  i.e. satisfies the condition $\nu_0\in C(TM_0,\nu_{-1})$ over $M_0$. Indeed, $M_0$ is a disk, $B_0=T^*M$ is a ball and we are free to change polarization if desired.
  The cotangent block $B_1$ is attached to $B_0$ along $\sT^*P$, where $P$ is a boundary face of $M_1$.
  
 Consider  the block $B_1$ and the restrictions of the distributions $\nu_0$ and $\nu$  to $B_1$. Applying Theorem \ref{thm:ridgification} we can find a ridgy isotopy of $M_1$ to a ridgy Lagrangian $M_1'$ transverse to $\nu$, reduced transverse to $\nu_0$, invariant with respect to the negative flow of $Z_0$, and such that at each point $a\in M_1'$ and any tangent plane $T$ at $a$  we have $[T]^0\pprec[\nu_0]^0\pprec[\nu_{-1}]^0.$
 
 Moreover, this deformation can be realized by a deformation of skeleta of Weinstein structures of $B_1$. 
  Next,  we apply to  the ridgy Lagrangian $M_1'$ Proposition  \ref{prop:slanted-ridgy-arb2} to  further deform the Weinstein structure on $B_1$ to a structure of a  cotangent building,   which is positive with respect to $\nu_{-1}$   and reduced positive with respect to $\nu_0$. 
    
This yields a  building structure on $B_0\cup B_1$ with positive skeleton, positive with respect to $\nu_{-1}$.
     Using Proposition
    \ref{prop:two-positivities} we can deform the building structure on  this building to make it positive and    positive for $\nu_{-1}$.
    
Before we proceed with the next block, we note that the Lagrangian  $\mM_2\cap  W^{(1)}$ is no longer conical for the constructed Wc-structure of $W^{(1)}$. So we apply Lemma
\ref{lm:Leg-adjustment} to arrange for this property to hold.  Once this is achieved we can proceed as before to produce a Wc-manifold $(W^{(2)}, \lambda^{(2)})$ with a positive W-complex structure for which $\nu$ is a minimal distribution and whose skeleton $L^{(2)}$ is a positive arboreal Lagrangian with smooth boundary. Arguing by induction in the same way over the rest of the attaching blocks we conclude the proof.
 \end{proof}
 
 \subsubsection{Variants of the main theorem}
 
Next we state our main theorem for Weinstein pairs:

\begin{thm}\label{thm:existence-arb-pos-pair}
 Let $(X,\lambda)$ be a Wc-manifold , $A_0 \subset X \setminus \Skel(X,\lambda)$ a Wc-hypersurface and $\nu\subset TW$ a Lagrangian plane field. Then there exists a Weinstein homotopy $\lambda_t$ of $\lambda_0=\lambda$ and a Wc-hypersurface $A_1 \subset X \setminus \Skel(X,\lambda_1)$ disjoint from $A_0$ such that the result of converting $A=A_0 \cup A_1$ to a boundary nucleus yields a Wc-manifold $(X, \lambda_1^A)$ which admits a structure of a positive cotangent building with a minimal distribution $\nu_{-1}$ equal to $\nu$. In particular, $\Skel(X, \lambda_1^A)$ is a positive arboreal Lagrangian (with boundary) transverse to $\nu$. Moreover, we can arrange it so that $\Skel(A_1,\lambda)$ is smooth, i.e. $\Skel(X, \lambda_1^A)$ has smooth boundary away from $A_0$.
 \end{thm}
 
 \begin{proof} The same proof works applied to the cotangent building structure associated to a Weinstein pair. \end{proof}
 
Finally, we state our main theorem in relative version:

\begin{thm}\label{thm:existence-arb-pos-pair-rel}
 Let $(X,\lambda)$ be a Wc-manifold , $A_0 \subset X \setminus \Skel(X,\lambda)$ a Wc-hypersurface and $\nu\subset TW$ a Lagrangian plane field. Suppose that $A_0$ is endowed with the structure of a positive cotangent building with minimal distribution equal to the reduction of $\nu|_{A_0}$. Assume moreover that $\nu$ is transverse to the Liouville field $Z$ of $X$ near $A_0$. Then there exists a Weinstein homotopy $\lambda_t$ of $\lambda_0=\lambda$, fixed near $A_0$, and a Wc-hypersurface $A_1 \subset X \setminus \Skel(X,\lambda_1)$ disjoint from $A_0$ such that the result of converting $A=A_0 \cup A_1$ to a boundary nucleus yields a Wc-manifold $(X, \lambda_1^A)$ which admits a structure of a positive cotangent building with a minimal distribution $\nu_{-1}$ equal to $\nu$. In particular, $\Skel(X, \lambda_1^A)$ is a positive arboreal Lagrangian (with boundary) transverse to $\nu$ and $\Skel(X, \lambda_1^A) \cap A_0 = \Skel(A_0, \lambda)$. Moreover, we can arrange it so that $\Skel(A_1,\lambda)$ is smooth, i.e. $\Skel(X, \lambda_1^A)$ has smooth boundary away from $A_0$.
 \end{thm}

 \begin{proof} All the intermediary steps hold in relative form, so one may apply the same proof to the cotangent building structure associated to the Weinstein pair keeping everything fixed near $A_0$ at every stage of the argument. . \end{proof}
 
Theorem \ref{thm:existence-arb-pos} implies our main Theorem \ref{intro:main thm} as stated in the Section \ref{sec:main res}, simply by taking the underlying Weinstein manifold of the Wc-manifold $(X,\lambda_1^A)$, i.e. converting $A$ back to a Wc-hypersurface. Similarly, Theorems \ref{thm:existence-arb-pos-pair} and \ref{thm:existence-arb-pos-pair-rel} imply the variants outlined in the remarks below the statement. For the concordance statement of Theorem \ref{intro:conc} simply apply Theorem \ref{thm:existence-arb-pos-pair-rel} to the Wc-pair obtained from $W \times T^*[0,1]$ by converting the face nuclei $A_0=(W \times 0) \cup (W \times 1)$ to Wc-hypersurfaces, where the Liouville form is $\lambda = \lambda_t + udt$ for $\lambda_t$ the homotopy between $\lambda_0$ and $\lambda_1$. 
 

 \end{document}